\RequirePackage{silence}
\documentclass[oneside,numbers=noenddot,headinclude,footinclude,cleardoublepage=empty,abstract=on,BCOR=5mm,paper=letter,fontsize=11pt]{scrreprt}

\PassOptionsToPackage{utf8}{inputenc}
  \usepackage{inputenc}

\PassOptionsToPackage{T1}{fontenc}
  \usepackage{fontenc}

\PassOptionsToPackage{
  drafting=false,    
  tocaligned=false, 
  dottedtoc=true,  
  eulerchapternumbers=true, 
  linedheaders=false,       
  floatperchapter=true,     
  eulermath=false,  
  beramono=true,    
  palatino=true,    
  parts=false,
  style=classicthesis
}{classicthesis}

\newcommand{\myTitle}{Stratified Pseudobundles and Quantization\xspace}
\newcommand{\myName}{Ethan Ross\xspace}
\newcommand{\myDegree}{Doctor of Philosophy\xspace}
\newcommand{\myDepartment}{Department of Mathematics\xspace}
\newcommand{\myUni}{University of Toronto\xspace}
\newcommand{\myTime}{2025\xspace}

\PassOptionsToPackage{english}{babel}
\usepackage{babel} 
\usepackage{enumitem} 
\usepackage{amsmath,amsthm,amssymb} 
\usepackage{thmtools} 
\usepackage{thm-restate}
\usepackage[onehalfspacing]{setspace} 
\usepackage{graphicx} 
\usepackage{xspace} 
\usepackage{calc} 
\usepackage{amsmath,amssymb,amsthm,amsfonts,amscd, url,mathtools,mathrsfs}
\usepackage{tikz-cd}
\usepackage[table,dvipsnames]{xcolor}
\definecolor{myblue}{HTML}{1A85FF}
\definecolor{myred}{HTML}{D41159}

\usepackage{classicthesis}

\hypersetup{%
  colorlinks=true, linktocpage=true, pdfstartpage=3, pdfstartview=FitV,%
  breaklinks=true, pageanchor=true,%
  pdfpagemode=UseNone, %
  plainpages=false, bookmarksnumbered, bookmarksopen=true, bookmarksopenlevel=1,%
  hypertexnames=true, pdfhighlight=/O,
  urlcolor=CTurl, linkcolor=CTlink, citecolor=CTcitation, 
  pdftitle={\myTitle},%
  pdfauthor={\textcopyright\ \myName, \myDepartment, \myUni},%
  pdfsubject={},%
  pdfkeywords={},%
  pdfcreator={pdfLaTeX},%
  pdfproducer={LaTeX with hyperref and classicthesis}%
}

\makeatletter
\@ifpackageloaded{babel}%
  {%
    \addto\extrasenglish{%
    }%
    }{\relax}
\makeatother

\makeatletter
\ifthenelse{\boolean{ct@linedheaders}}%
{
    \titleformat{\chapter}[display]%
    {\relax}{\raggedleft{\color{CTsemi}\chapterNumber\thechapter} \\ }{0pt}%
    {\titlerule\vspace*{.9\baselineskip}\raggedright\Large\color{CTtitle}\spacedallcaps}[\normalsize\vspace*{.8\baselineskip}\titlerule]%
}{
    \titleformat{\chapter}[display]%
    {\relax}{\mbox{}\oldmarginpar{\vspace*{-3\baselineskip}\color{CTsemi}\chapterNumber\thechapter}}{0pt}%
    {\raggedright\huge\color{CTtitle}\spacedallcaps}[\normalsize\vspace*{.8\baselineskip}\titlerule]%
}
\makeatother

\newtheorem{theorem}{Theorem}[chapter]
\newtheorem{lem}[theorem]{Lemma}
\newtheorem{prop}[theorem]{Proposition}

\newtheorem{cor}[theorem]{Corollary}

\theoremstyle{definition}
\newtheorem{defs}[theorem]{Definition}

\newtheorem{egs}[theorem]{Example}

\theoremstyle{remark}
\newtheorem{remark}[theorem]{Remark}

\newenvironment{itemize*}
  {\begin{itemize}[topsep=-\parskip+\jot,itemsep=-\parskip-\jot]}
  {\end{itemize}}
  
\newenvironment{enumerate*}
  {\begin{enumerate}[label=(\alph*),topsep=-\parskip+\jot,itemsep=-\parskip-\jot]}
  {\end{enumerate}}
  
\newenvironment{enumerate**}
  {\begin{enumerate}[label=(\roman*),topsep=-\parskip+\jot,itemsep=-\parskip-\jot]}
  {\end{enumerate}}
  
\newenvironment{enumerate***}
  {\begin{enumerate}[label=(\alph*'),topsep=-\parskip+\jot,itemsep=-\parskip-\jot]}
  {\end{enumerate}}

\newcommand{\bR}{\mathbb{R}}
\newcommand{\bZ}{\mathbb{Z}}
\newcommand{\bC}{\mathbb{C}}
\newcommand{\bT}{\mathbb{T}}
\newcommand{\bN}{\mathbb{N}}
\newcommand{\bK}{\mathbb{K}}

\newcommand{\mfX}{\mathfrak{X}}
\newcommand{\mfg}{\mathfrak{g}}
\newcommand{\mfk}{\mathfrak{k}}
\newcommand{\mfp}{\mathfrak{p}}
\newcommand{\mft}{\mathfrak{t}}
\newcommand{\mfh}{\mathfrak{h}}
\newcommand{\mfn}{\mathfrak{n}}
\newcommand{\mfz}{\mathfrak{z}}

\newcommand{\id}{\text{id }}

\newcommand{\into}{\hookrightarrow}
\newcommand{\bra}{\langle}
\newcommand{\ket}{\rangle}

\newcommand{\substeq}{\subseteq}

\newcommand{\End}{\text{End}}

\usetikzlibrary{3d}

\usepgflibrary{shadings}

\usepackage{etoolbox} 

\tikzcdset{every picture/.style={remember picture,baseline=(current bounding box.center)}}

\makeatletter
\pretocmd{\@thmwritebody}{\let\begin=\relax \let\end=\relax }{}{}
\makeatother

\ofoot*{}
\ohead*{}

\begin{document}

\frenchspacing
\raggedbottom

\selectlanguage{english}

\pagenumbering{roman}
\pagestyle{plain}

\begin{titlepage}
    \pdfbookmark[1]{\myTitle}{titlepage}

    \begin{center}
        \large

        \hfill

        \vfill

        \begingroup
            \color{CTtitle}\spacedallcaps{\myTitle}
        \endgroup

		\vfill
		\spacedlowsmallcaps{by}\bigskip
		
        \spacedlowsmallcaps{\myName}

		\vfill
        
        A thesis submitted in conformity with\\
        the requirements for the degree of\smallskip
        
        \myDegree \smallskip
        
        \myDepartment\\
        \myUni\bigskip
        
        \textcopyright{} Copyright by \myName \myTime 

    \end{center}
\end{titlepage}

\cleardoublepage
\pdfbookmark[1]{Abstract}{Abstract}

\begingroup
\let\clearpage\relax
\let\cleardoublepage\relax
\let\cleardoublepage\relax

\chapter*{Abstract}

\doublespacing

\begin{center}
	\myTitle\\
\myName\\
\myDegree\\
\myDepartment\\
\myUni\\
\myTime
\end{center}

Geometric Quantization is a term used to describe a wide collection of techniques dating back to the 1960s in the work of Kirillov, Kostant, and Souriau, which take symplectic manifolds and produce complex vector spaces. The name comes from the natural interpretation of symplectic manifolds as the phase spaces of classical mechanical systems and complex vector spaces as the natural domains of wave functions in quantum mechanics.

In this thesis, I extend the classical framework of Geometric Quantization to handle a class of singular spaces called Symplectic Stratified Spaces, which date back to the work of Sjamaar and Lerman in the 1990s. As part of this work, I develop the theory of stratified pseudobundles to serve as singular replacements for important auxiliary information in Geometric Quantization: prequantum line bundles and polarizations. I then use this formalism to provide [Q,R]=0 results for singular quotients of toric manifolds and cotangent bundles. I also provide an example of singular Geometric Quantization that does not arise from singular reduction.

\endgroup
\cleardoublepage
\phantomsection
\pdfbookmark[1]{Dedication}{Dedication}

\vspace*{3cm}

\begingroup
\setlength{\parindent}{0cm}
\textit{To my wife Carrie Clark,\\
\hspace*{5mm}and to Hermes the Thrice Great}
\endgroup

\cleardoublepage
\pdfbookmark[1]{Acknowledgements}{acknowledgements}

\begingroup
\let\clearpage\relax
\let\cleardoublepage\relax

\chapter*{Acknowledgements}

This thesis was an enormous undertaking and would not have been possible without the help of many people along the way. First and foremost, I would like to thank my wife, Carrie, as I most certainly would not have gotten this far without her. I would also like to thank her for making many of the figures in TikZ.

\

I am grateful to my supervisor, Prof. Lisa Jeffrey, for all her help in guiding me over the years, for giving me a long leash to work on whatever problems I found most interesting, and for always diligently providing me with advice and corrections. Thank you to Prof. Yael Karshon for her words of encouragement and guidance, especially when I felt like everything I was doing was trivial. I would also like to thank Prof. Eckhard Meinrenken, David Miyamoto, and Maarten Mol. The many helpful conversations I had with you all informed a great deal of the content and presentation of this thesis. Thank you also to the external examiner Prof. Reyer Sjamaar for all your helpful feedback and corrections.

\

Thank you to my friends and family for all their support over the years, especially my friend group here in Toronto: Dan, Jane, Kristen, Adam, Eli, Curtis, Nick, and Julia, who all helped make these some of the best years of my life. Thank you also to Victoria and Luke for your companionship from afar.

\

Finally, thank you to the many kind people I met in the math department over the years who helped make my time at the University of Toronto quite enjoyable. In particular, I would like to thank the administrative staff, Jemima Merisca and Sonja Injac, as well as my fellow graduate students, Alice, Dinushi, Valia, Eva, Jonas, Peter, Thomas, Aidan, and Emily, who made attending conferences and student events a great experience.

\endgroup

\clearpage
\setlength{\cftbeforetoctitleskip}{100em}
\begingroup
\pagestyle{scrheadings}
\pdfbookmark[1]{\contentsname}{tableofcontents}
\setcounter{tocdepth}{3} 
\setcounter{secnumdepth}{3} 
\manualmark
\markboth{\spacedlowsmallcaps{\contentsname}}{\spacedlowsmallcaps{\contentsname}}
\tableofcontents
\automark[section]{chapter}
\renewcommand{\chaptermark}[1]{\markboth{\spacedlowsmallcaps{#1}}{\spacedlowsmallcaps{#1}}}
\renewcommand{\sectionmark}[1]{\markright{\textsc{\thesection}\enspace\spacedlowsmallcaps{#1}}}
\endgroup

\cleardoublepage
\pagestyle{scrheadings}
\pagenumbering{arabic}
\cleardoublepage
\chapter{Introduction}
\label{ch:intro}

Geometric quantization is an area of symplectic geometry initiated by Kirillov \cite{kirillov_unitary_1962}, Kostant \cite{kostant_orbits_2009}, and Souriau \cite{souriau_quantification_1967} in the 1960's and 1970's. Initially, it arose as a method for obtaining unitary representations of Lie groups before being realized as a potential candidate for solving the Quantization Problem—that is, how to systematically take a classical physics system (described via symplectic geometry) and produce a quantum mechanical system (described via Hilbert spaces and operators). In its original form and in various generalizations, geometric quantization has developed into a rich mathematical theory.

\ 

The idea is to take a symplectic manifold $(M,\omega)$—to be interpreted as the phase space of a classical mechanical system—and to associate to it two auxiliary pieces of data: a polarization $\mathcal{P} \substeq T^\bC M$ and a prequantum line bundle $(\mathcal{L},\nabla) \to (M,\omega)$. These elements are then used to produce a vector space $\mathcal{Q}(M)$, usually by computing “polarized sections” or some variants such as cohomological wave functions or the index of an operator.

\

What these constructions attempt to emulate is that, for physicists, symplectic $\bR^{2n}$ is expected to have quantization $L^2(\bR^n)$, a space of complex-valued functions on a half-dimensional space. The sections of the line bundle $\mathcal{L}$ are then the complex-valued functions, $\mathcal{P}$ represents the “momentum directions” in phase space, and the connection $\nabla$ provides a way to differentiate the sections of $\mathcal{L}$.

\

The most classical definition of quantization given this data is the space of polarized sections:
\begin{equation}\label{eq: classic def of quantization}
\mathcal{Q}(M):={\sigma\in \Gamma(\mathcal{L}) \ | \ \nabla_{\mathcal{P}}\sigma=0}.
\end{equation}

As was stated above, other definitions of geometric quantization exist in the literature. Given the same quantization data, we can produce a sheaf-theoretic definition (e.g., \cite{sniatycki_bohr-sommerfeld_1975}), an index-theoretic definition (e.g., \cite{meinrenken_singular_1999}), or a shifted-stack theoretic definition (e.g., \cite{safronov_shifted_2023}). In this thesis, we will focus on the classical definition in Equation (\ref{eq: classic def of quantization}).

\ 

One of the central problems in geometric quantization is the \textit{Quantization Commutes With Reduction} problem. Due to Marsden, Weinstein \cite{marsden_reduction_1974}, and Meyer \cite{meyer_symmetries_1973}, there is a natural way to reduce a symplectic manifold by a nice symmetry group and obtain another symplectic manifold. The setup is as follows: we are given a symplectic manifold $(M,\omega)$, a compact Lie group $K$ acting symplectically on $M$, and an equivariant momentum map $J:M\to \mfk^*$, which represents ``conserved quantities'' à la Noether’s Theorem. This collection of data, summarized by the symbol $J:(M,\omega)\to \mfk^*$, and is called a Hamiltonian $K$-space.

\

Marsden, Weinstein, and Meyer proved that if $K$ acts freely on $J^{-1}(0)$, then $J^{-1}(0)$ is a $K$-invariant submanifold of $M$, and that the reduced space $M_0 := J^{-1}(0)/K$ comes equipped with a unique symplectic form $\omega_0$, called the reduced symplectic form, such that if $\pi_0: J^{-1}(0) \to M_0$ is the quotient map, then
\begin{equation}
\pi_0^*\omega_0 = \omega|_{J^{-1}(0)}.
\end{equation}

Returning to quantization, if a Hamiltonian $K$-space $J:(M,\omega)\to \mfk^*$ is equipped with an equivariant polarization and a prequantum line bundle, the quantization $\mathcal{Q}(M)$ becomes a linear, complex $K$-representation. Furthermore, provided the quantization data on $(M,\omega)$ is sufficiently regular, we obtain induced quantization data on the reduced space $M_0$. Finally, we obtain a canonical linear map:
\begin{equation}\label{eq: [Q,R] map}
\kappa:\mathcal{Q}(M)^K\to \mathcal{Q}(M_0).
\end{equation}

In the case where $\kappa$ is a linear isomorphism, we say Quantization Commutes With Reduction, and write $[Q,R] = 0$. In the original paper by Guillemin and Sternberg \cite{guillemin_geometric_1982}, they showed that if $M$ is compact and $\mathcal{P}$ is a K\"{a}hler polarization, then $\kappa$ is indeed a bijection. However, as was shown by Sjamaar and Lerman \cite{sjamaar_stratified_1991}, if the symmetry group $K$ does not act nicely, the resulting reduced space $M_0$ can be highly singular, and many of the usual quantization techniques break down. Nonetheless, the reduction $M_0$ has a natural partition into symplectic manifolds that fit together in a manner analogous to the cells of a CW complex—an object called a Symplectic Stratified Space.

\

The goal of this thesis is to provide a framework for quantizing symplectic stratified spaces that is robust enough to prove quantization commutes with reduction results. In particular, we give definitions for stratified versions of the prequantum line bundle (Definition \ref{def: strat prequantum}), polarizations (Definition \ref{def: strat pol}), and quantization (Definition \ref{def: stratified quantization}) in a manner analogous to the classical definition of quantization in Equation (\ref{eq: classic def of quantization}). We then show that, under reduction, equivariant quantization data descends to stratified quantization data (Theorem \ref{thm: stratified reduction of prequantum} and Theorem \ref{thm: singular reduction of polarizations}), and we still obtain a canonical linear map as in Equation (\ref{eq: [Q,R] map}) (Theorem \ref{thm: sing [Q,R] map}). Finally, we prove that in this formalism, $[Q,R] = 0$ holds for singular quotients of toric manifolds (Theorem \ref{thm: [Q,R] for toric}) and for general cotangent bundles (Theorem \ref{thm: sing [Q,R] cotangent}). This provides a proof of concept that this approach is sensible for quantizing singular spaces.

\ 

This thesis is not the first attempt to quantize symplectic stratified spaces. Indeed, there exist index-theoretic versions of quantization (e.g., \cite{meinrenken_singular_1999}), a module-theoretic approach known as “algebraic quantization” \cite{sniatycki_reduction_1983}, and others. Aside from employing different flavours of quantization than the classical definition in Equation (\ref{eq: classic def of quantization}), these approaches differ from the one presented in this thesis in that they are defined only for singular reduced spaces. In contrast, the formalism we develop can be applied to any stratified symplectic space. The most similar approach to ours is the singular quantization theory of Śniatycki \cite{sniatycki_differential_2013}. However, this formalism also differs from ours in that Śniatycki works stratum-wise, whereas our approach adopts a more global perspective.

\ 

Although the main aim of this thesis was to develop a theory of quantization for general symplectic stratified spaces, we found it necessary to further develop both the non-singular theory of quantization and the theory of singular spaces—namely, subcartesian and stratified spaces. This was done not only to increase the generality of the theory of quantization for singular spaces, but also to establish a convenient framework with standard techniques for formulating and proving results in this setting.

\ 

First, on the non-singular side. Much of the literature on geometric quantization focuses on only two kinds of polarizations: real and Kähler. This is understandable, as all the natural examples of polarizations are of either type. Nevertheless, there exists a large class of polarizations that are neither real nor Kähler—the so-called “mixed” polarizations. Thus, in Chapter \ref{ch: polarizations}, we re-developed the theory of reduction of polarizations, allowing for arbitrary types and only properly acting groups. This pays dividends, as we show in Chapter \ref{ch: geometric quantization} that the $[Q,R]=0$ map in Equation (\ref{eq: [Q,R] map}) extends naturally to the mixed case as well.

\ 

This thesis is broken into three parts, with each part in turn being broken up into further chapters. Let us now give a summary of each part and chapter along with the major results along the way.

\section*{Part \ref{pt: nonsing}: Non-Singular Preliminaries}

The goal of Part \ref{pt: nonsing} is both to set up much of the notation that will be used later on and to discuss the non-singular versions of quantization that will be generalized in Part \ref{pt: strat}. Many of the initial proofs of quantization commutes with reduction relied on objects like differential forms, which are not very well-behaved in the subcartesian setting. Thus, where necessary, we provided new proofs that do not rely on such objects. The upshot is that these new non-singular proofs can be readily generalized to the singular setting and are quite instructive for the techniques we use in a more familiar smooth context.

\subsection*{Chapter \ref{ch: equivariant} Background in Equivariant Geometry}

Chapter \ref{ch: equivariant} begins this thesis with a discussion of group actions, focusing on proper actions by connected Lie groups. This chapter is mostly a review of existing material aimed at bringing the reader up to speed on the basics of proper group actions. Section \ref{sec: proper action} provides the main basic recurring definitions and elementary results. Sections \ref{sec: local normal form proper action} and \ref{sec: slice theorem} are dedicated to setting up and stating one of the most important results in the theory of proper group actions, namely Palais' Slice Theorem (Theorem \ref{thm: slice theorem}). In Section \ref{sec: tangent spaces mfls symmetry and orbit-type}, we then use the Slice Theorem to compute the tangent spaces of a distinguished class of submanifolds of a proper $G$-space.

\

Finally, in Section \ref{sec: equivariant vector bundles}, we conclude with a discussion on proper equivariant vector bundles. We provide a normal form result (Lemma \ref{lem: local normal form proper vb}), use this to give a new proof of the averaging theorem (Theorem \ref{thm: averaging over G}), and conclude with a characterization of when the quotient of a proper equivariant vector bundle is a vector bundle once again (Theorem \ref{thm: special subbundle regular case}).

\subsection*{Chapter \ref{ch: symp} Symplectic Geometry}

Chapter \ref{ch: symp} provides a review of basic definitions and results in symplectic geometry, with a view towards singular reduction and quantization. Section \ref{sec: symplectic vspaces} summarizes elementary definitions and results from linear symplectic geometry. Section \ref{sec: symp and poisson} is then dedicated to giving basic definitions, examples, and theorems for symplectic manifolds, especially regarding the relationship between symplectic manifolds and Poisson manifolds. In Section \ref{sec: ham space}, we review the basic definitions of Hamiltonian spaces. This leads to the discussion in Section \ref{sec: linear symp reps} of the theory of linear symplectic representations and quadratic momentum maps. Section \ref{sec: nonsing reduction} is then used to state the main reduction theorems for Hamiltonian spaces (Theorem \ref{thm: weinstein reduction} and Corollary \ref{cor: more general nonsing reduction}) as well as a normal form theorem (Theorem \ref{thm: local normal form for Hamiltonian spaces}).

\

Section \ref{sec: non-sing reduction examples} then closes off the chapter by illustrating the principles of symplectic reduction for two non-trivial families of examples: the Delzant construction for toric manifolds in Subsection \ref{sub: delzant construction} and the non-singular reduction theory for cotangent bundles in Subsection \ref{sub: cotangent reduction}.

\subsection*{Chapter \ref{ch: polarizations}: Polarizations}

Polarizations are one of the main ingredients for geometric quantization, and so in Chapter \ref{ch: polarizations} we provide a summary of the basic theory of polarizations. Section \ref{sec: reduce lagrangians} returns us to the linear symplectic setting from Section \ref{sec: symplectic vspaces} in Chapter \ref{ch: symp}, but now to study how Lagrangian subspaces interact with linear symplectic reduction. However, this section goes beyond the traditional emphasis on reduction of coisotropic subspaces and provides, in Theorem \ref{thm: linear model of singular reduction}, a linear model for reduction of polarizations which will be used in both the singular and non-singular settings.

\

Section \ref{sec: polarizations} is then dedicated to the basic theory of polarizations—namely involutive and Lagrangian complex distributions of a symplectic manifold. As an aside, in Section \ref{sec: vector bundle images}, we study when the image of a vector bundle under a vector bundle morphism is again a vector bundle. Lemma \ref{lem: characterization when image of vector bundle is a vector bundle again} gives the precise formulation, which is then used in Lemma \ref{lem: surjective submersion gives us distributions} to characterize when a smooth map between manifolds can push forward an involutive distribution. This theory is then used in Section \ref{sec: nonsing reduce polarizations} to prove Theorem \ref{thm: nonsing reduction of polarizations}, which states when a polarization on a free Hamiltonian $G$-space can be reduced. This theory is illustrated in Section \ref{sec: reducing polarization cotangent bundles} with the reduction of the canonical polarization of cotangent bundles.

\subsection*{Chapter \ref{ch: geometric quantization}: Geometric Quantization}

In Chapter \ref{ch: geometric quantization}, we close out Part \ref{pt: nonsing} by giving a summary of the classical theory of geometric quantization. Section \ref{sec: prequantum} is dedicated to motivating the definition of prequantum line bundles, as well as providing some elementary results. In Section \ref{sec: geo quant}, we then state the definition of geometric quantization and provide early examples, specifically applied to symplectic $\bR^{2n}$ and to cotangent bundles.

\

From here, in Section \ref{sec: [Q,R]}, we discuss the reduction theory of prequantum line bundles and apply this to obtain a new proof of the existence of the $[Q,R]=0$ map for arbitrary polarizations in Theorem \ref{thm: non-singular [Q,R]=0}, extending the classical theory of Guillemin and Sternberg. Finally, Section \ref{sec: [Q,R] examples} is dedicated to applying these results to compute the quantization of toric manifolds and to providing a quantization commutes with reduction statement for the non-singular reduction of cotangent bundles.

\section*{Part \ref{pt: strat} Differentiable Stratified Spaces}

As was stated earlier, the general reduction of a symplectic manifold need not be a manifold once again. Indeed, in the case of proper group actions, the result is a symplectic stratified space. Part \ref{pt: strat} is then dedicated to developing the differential geometry of stratified spaces, with an aim towards applying this theory to singular reduction.

\

There are three main formalisms for discussing the differential geometry of stratified spaces: the subcartesian formalism (e.g. \cite{sniatycki_differential_2013}), the smooth atlas formalism (e.g. \cite{pflaum_analytic_2001}), and the reduced differentiable structure formalism (e.g. \cite{mol_stratification_2024}). This thesis adopts the subcartesian formalism as it is the most economical approach in regards to both stating and proving singular differential geometric results. This choice is mostly a matter of taste, as Theorem \ref{thm: charts of subcartesian spaces are locally diffeomorphic}, which is based on a result of Mol \cite{mol_stratification_2024}, allows us to easily translate between each formalism.

\subsection*{Chapter \ref{ch: subcartesian}: Subcartesian Spaces}

Chapter \ref{ch: subcartesian} is dedicated to setting up the formalism for singular differential geometry used in this thesis. Subcartesian spaces are generalizations of manifolds where the chart maps need not have open image; instead, they are allowed to be any subset of a Euclidean space. This is easy enough to state, but it requires some set-up using the language of Differential, AKA Sikorski, spaces. These are topological spaces $X$ equipped with a distinguished set of functions $C^\infty(X)$ which generate the topology, are locally determined, and are closed under composition with smooth functions on Cartesian spaces. Section \ref{def: differential space} is dedicated to developing the elementary theory of differential spaces, including generating sets for differential structures in Subsection \ref{sub: induced differential}, differentiable maps in Subsection \ref{sub: diff maps}, and induced differential structures on subsets in Subsection \ref{sec: subspaces}. With this, in Section \ref{sec: subcartesian} we can finally define Subcartesian spaces, which are simply differential spaces $(X,C^\infty(X))$ with local embeddings into Euclidean spaces. In Subsection \ref{sub: quotients of compact groups}, the Mathers–Schwarz construction of a subcartesian space structure on $C^\infty(M/G)$ \cite{schwarz_lifting_1980}, where $G$ is a connected Lie group and $M$ a proper $G$-space, is provided. In this subsection, we also provide, in Proposition \ref{prop: subcartesian on quotient of closed subset}, a generalization to construct a subcartesian structure $C^\infty(A/G)$ for a distinguished class of subsets $A$ of a proper $G$-space called sliceable subsets.

\

The last sections of Chapter \ref{ch: subcartesian} are dedicated to a discussion of the natural notion of tangent vectors, tangent bundles, and vector fields in the subcartesian setting, namely Zariski tangent vectors and Zariski tangent bundles in Section \ref{sec: zariski tangent vectors}, and Zariski vector fields in Section \ref{sec: zariski vfields}. These are derivations of the ring of functions $C^\infty(X)$ of a subcartesian space and in many ways naturally extend the usual smooth theory of tangent vectors and vector fields. Using this language, we prove Theorem \ref{thm: charts of subcartesian spaces are locally diffeomorphic}, which can be used to show that the three formalisms mentioned above for the differential geometry of stratified spaces are nearly equivalent to one another. Finally, in Section \ref{sec: restricting vfields}, we conclude with a discussion on restricting Zariski vector fields to subspaces and in Corollary \ref{cor: sufficient for restriction to be surjective} provide a sufficient condition for the restriction map being a surjection.

\subsection*{Chapter \ref{ch: diff strat spaces}: Differentiable Stratified Spaces}

Chapter \ref{ch: diff strat spaces} is dedicated to the theory of differentiable stratified spaces. Section \ref{sec: top strat spaces} gives a basic introduction to stratified spaces, which are topological spaces equipped with a locally finite partition into topological manifolds in the subspace topology, called a stratification, subject to the frontier condition. This condition states that the closure of an element of the stratification, called a stratum, is the disjoint union of other strata, resembling the relationship between the open cells of a CW complex. Section \ref{sec: diff strat spaces} then defines differentiable stratified spaces as subcartesian spaces equipped with a compatible stratification, in the sense that the induced subcartesian structure on each stratum is a smooth structure. Subsection \ref{sub: foliations and strats} introduces the main (dare I say polemical) viewpoint of this thesis on stratifications: they are to be viewed as a kind of singular foliation in the sense of Stefan \cite{stefan_accessible_1974} and Sussmann \cite{sussmann_orbits_1973}, another kind of partition by manifolds which, thanks to Corollary \ref{cor: sfs satisfy the axiom of the frontier}, also satisfy the frontier condition. As discussed, there are examples of stratifications which are not singular foliations and singular foliations which are not stratifications. Nonetheless, provided certain conditions are met, Corollary \ref{cor: sfs satisfy the axiom of the frontier} shows that demonstrating a partition of a manifold is a singular foliation can also establish that it is a stratification. In Subsection \ref{sub: SRFs}, this viewpoint is exploited to prove in Theorem \ref{theorem: dimension-type parititon is a locally finite sf} that a certain class of singular foliations, called Singular Riemannian Foliations, induce a canonical stratification called the dimension-type stratification. This work is part of a forthcoming paper with David Miyamoto \cite{miyamoto_srfs_2025}.

\

With this preliminary out of the way, Section \ref{sec: strats and group actions} discusses various stratifications arising from proper group actions. Namely, in Subsection \ref{sub: orbit-type strat} we discuss the orbit-type stratification and in Subsection \ref{sub: infnitesimal stratification} the reduced orbit-type and infinitesimal stratifications. At the manifold level, we give a new proof in Theorem \ref{thm: orbit-type is sf} of a result of Jotz-Ratiu-Sniatycki \cite{jotz_singular_2011} that the orbit-type stratification is a singular foliation and hence a stratification. This also shows that the quotient space of a proper group action inherits a natural stratification called the canonical stratification. Leveraging results of Posthuma-Tang-Wang \cite{posthuma_resolutions_2017}, Corollary \ref{cor: infinitesimal partition is an sf and a stratification} establishes that the reduced orbit-type and infinitesimal stratifications are stratifications and foliations.

\

Section \ref{sec: strat vfs} gives a brief account of stratified vector fields, which can be viewed as Zariski vector fields tangent to each of the strata. Here, we provide an alternate construction of the stratified tangent bundle due to Pflaum \cite{pflaum_analytic_2001} and prove elementary properties of these vector fields. This leads to Theorem \ref{theorem: equivariant vector fields project to stratified vector fields} of Schwarz \cite{schwarz_lifting_1980}, stating that all stratified vector fields on a quotient space arise from equivariant vector fields on the original space.

\

To conclude, Section \ref{sec: regularity conds} presents an extended discussion on regularity conditions for stratified spaces. General stratified spaces can be somewhat pathological from the viewpoint of smooth geometry, so various regularity conditions have been introduced to control singularities. Subsection \ref{sub: Whitney} reviews the Whitney conditions, arguably the most common regularity conditions imposed on stratified spaces. These allow the construction of control data that facilitates using normal bundles to study the stratification. Since the Whitney conditions may seem unfamiliar to smooth geometers, Proposition \ref{prop: enough sections implies Whitney A} proves that if the stratified tangent bundle of a stratified space is spanned by stratified vector fields, then the space is Whitney (A). Next, Subsection \ref{sub: local trivial} discusses local triviality, another condition reminiscent of singular foliations. Proposition \ref{prop: locally trivial strat implies enough sections} shows that locally trivial stratified spaces are automatically Whitney (A). Finally, Subsection \ref{sub: quasi-homo} discusses locally conical and quasi-homogeneous stratifications, introduced by Zimhony \cite{zimhony_commutative_2024} as Smoothly Locally Trivial with Conical Fibres and Smoothly Locally Trivial with Quasi-Homogeneous Fibres, respectively. These are locally trivial stratified spaces which can be locally realised as subsets of graded vector bundles in the sense of Grabowski-Rotkiewicz \cite{grabowski_graded_2012}. Adapting an argument by David Miyamoto, Theorem \ref{thm: quasi-homogeneous implies Whitney regular} shows that quasi-homogeneous stratified spaces satisfy the Whitney conditions. We conclude by demonstrating that most major examples of stratifications encountered so far are quasi-homogeneous and hence Whitney regular.

\subsection*{Chapter \ref{ch: stratpseud}: Stratified Pseudobundles}

Chapter \ref{ch: stratpseud} is partially based on my paper \cite{ross_stratified_2024}. Pseudobundles are a natural generalization of vector bundles for subcartesian spaces where fibre ranks may vary. They have previously been studied by Marshall \cite{marshall_calculus_1975}, who used linear versions of the singular charts of subcartesian spaces. In Section \ref{sec: pseudo}, I give a new, more general definition of pseudobundles that is more modern in the sense that the pseudobundle structure is encoded via linear functions. All of the “singular vector bundles” encountered so far—such as Zariski, foliation, and stratified tangent bundles—are in fact pseudobundles. In Section \ref{sec: pseudo quotient}, we prove Theorem \ref{thm: equivariantly generated implies quotient is vector bundle}, which constructs differentiable vector bundles over the quotient of a sliceable subset of a proper equivariant vector bundle, generalizing Theorem \ref{thm: special subbundle regular case} from Chapter \ref{ch: equivariant}. Section \ref{sec: strat pseud} then introduces stratified pseudobundles, which are pseudobundles stratified by smooth vector bundles. By a result of Pflaum \cite{pflaum_analytic_2001}, the stratified tangent bundle of a stratified space is a stratified pseudobundle only if the underlying stratified space is Whitney A. This is generalized to more general pseudobundles over stratified spaces via the regularity condition Linear Whitney (A), introduced in Definition \ref{def: linear whitney}. As with stratified spaces, a stratified pseudobundle spanned by its sections is automatically Linear Whitney A. The final section, Section \ref{sec: strat dist}, addresses the complexification of pseudobundles and its specific application to subbundles of the complexified stratified tangent bundle. Key results here include Theorem \ref{thm: complexification of pseudobundles}, which shows that fibrewise complexification yields a pseudobundle as regular as the input, and Lemma \ref{lem: a stratified distribution can be reconstructed from the strata}, giving a condition for when complex distributions on a Euclidean space define stratified complex distributions on a stratified subset.

\section*{Part \ref{pt: quant} Singular Quantization}

In Part \ref{pt: quant}, we complete the set-up of the theory of differentiable stratified spaces and are now ready to study symplectic stratified spaces and their quantization. A key goal of this thesis so far has been to provide a framework that allows the classical theory of quantization to be applied in this more general singular setting without undue complication. This part begins with a discussion of the Sjamaar-Lerman construction \cite{sjamaar_stratified_1991} of symplectic stratified spaces arising from singular reduction, followed by the generalization of polarizations and prequantum line bundles using the language of pseudobundles developed in Chapter \ref{ch: stratpseud}.

\subsection*{Chapter \ref{ch: sing quant}: Symplectic Stratified Spaces}

Symplectic stratified spaces are differentiable stratified spaces equipped with a Poisson bracket and compatible symplectic forms on each stratum. They can be viewed as generalizations of Poisson manifolds, except that the symplectic foliation is required to be locally finite. Chapter \ref{ch: sing quant} studies these spaces and their appearance through singular reduction. In Section \ref{sec: sing reduce cotangent}, we define symplectic stratified spaces and provide initial examples coming from Poisson manifolds. Sections \ref{sec: orbit symp strat}, \ref{sec: reduce mfld symmetry}, and \ref{sec: poisson on reduced space} develop the setup and proof of Theorem \ref{thm: singular reduction of symplectic manifolds} due to Sjamaar and Lerman \cite{sjamaar_stratified_1991}, which states that symplectic reduction of a general proper Hamiltonian $G$-space yields a symplectic stratified space. Section \ref{sec: orbit symp strat} proves Theorem \ref{thm: quasi-homo reduced spaces} by Zimhony, showing that the orbit and canonical stratifications of the zero-level set of the momentum map and its quotient, respectively, are quasi-homogeneous stratified spaces, and that each stratum of the reduced space is naturally a symplectic manifold. In Section \ref{sec: reduce mfld symmetry}, we outline Sjamaar and Lerman’s proof that each stratum in the reduced space can be realized via the classical Marsden-Weinstein-Meyer reduction of the manifolds of symmetry. Finally, Section \ref{sec: poisson on reduced space} completes the proof of Theorem \ref{thm: singular reduction of symplectic manifolds} by constructing the Poisson bracket on the reduced space.

\

So far this chapter has reviewed known results. In Section \ref{sec: Hamiltonian module} I introduce a new definition: Hamiltonian modules (Definition \ref{def: hamiltonian module}). These form a distinguished class of generators for the stratified vector fields of a stratified symplectic space, containing the Hamiltonian vector fields, and are crucial for defining a notion of “differentiability” for stratified prequantum connections. Theorem \ref{thm: Hamiltonian module for singular reduction} then provides a canonical construction of Hamiltonian modules for singular reduced spaces. Finally, Section \ref{sec: sing reduce cotangent} summarizes the singular reduction of cotangent bundles due to Perlmutter-Rodriguez-Olmos-Sousa-Dias \cite{perlmutter_geometry_2007}. The last new result in this chapter, due to me, is Lemma \ref{lem: pullback of Liouville}, which shows that the Liouville 1-form of a cotangent bundle canonically induces a primitive of the symplectic forms on the strata of the singular reduced space.

\subsection*{Chapter \ref{ch: sing prequantum}: Singular Reduction of Prequantum Line Bundles}

In Chapter \ref{ch: sing prequantum}, we discuss prequantum line bundles on stratified spaces. Section \ref{sec: strat prequantum} introduces stratified prequantum line bundles (Definition \ref{def: strat prequantum}) as differentiable complex line bundles over stratified spaces equipped with a connection that defines a differentiable pairing between sections of the bundle and elements of a Hamiltonian module. This structure imposes a global “differentiability” condition ensuring these bundles are more than just prequantum line bundles restricted to each stratum. Section \ref{sec: sing reduce prequantum} then presents Theorem \ref{thm: stratified reduction of prequantum}, which constructs the reduction of a prequantum line bundle on a singular reduced space, generalizing the non-singular case. Example \ref{eg: contangent prequantum still reducible for singular cotangent} illustrates that the canonical prequantum line bundle over a cotangent bundle can always be reduced. Finally, Section \ref{sec: sing prequantum toric} applies Theorem \ref{thm: stratified reduction of prequantum} together with the Delzant construction to prove Theorem \ref{thm: reducible prequantum for toric}, providing a method to construct stratified prequantum line bundles over singular quotients of toric manifolds.

\subsection*{Chapter \ref{ch: sing polarization}: Singular Reduction of Polarizations}

Now that we have symplectic stratified spaces and stratified prequantum line bundles, the final component needed for singular quantization is a suitable notion of a singular polarization. Chapter \ref{ch: sing polarization} addresses this subtle issue. In Section \ref{sec: strat polarization}, we define stratified polarizations (Definition \ref{def: strat pol}) as complex stratified distributions, in the sense introduced in Chapter \ref{ch: stratpseud}, which restrict to genuine polarizations on each stratum. Example \ref{eg: stratified polarizations on toric manifolds} provides nearly natural examples arising from toric manifold theory—“nearly” because although the stratifications consist of symplectic manifolds, they do not arise from a global Poisson structure. Section \ref{sec: sing reduce polarization} then proves Theorem \ref{thm: singular reduction of polarizations}, which characterizes when a polarization on a Hamiltonian $G$-space induces a stratified polarization on the reduced space. However, as shown by Example \ref{eg: cant always reduce}, the natural polarization on a cotangent bundle almost never induces a stratified polarization upon reduction. Section \ref{sec: type and reduction} discusses preservation of polarization type under reduction, demonstrating that real and complex polarizations induce real and complex polarizations stratum-wise on the reduced space. Moreover, Section \ref{sec: when we can reduce pols} shows that K\"{a}hler polarizations can always be reduced. Finally, Section \ref{sec: non-reducible pols} returns to cotangent bundles, where Theorem \ref{thm: cotangent polarization almost reducible} establishes that the canonical polarization of a cotangent bundle induces a pseudobundle on the reduced space which forms a polarization on an open dense subset of each stratum.

\subsection*{Chapter \ref{ch: sing quant}: Singular Reduction and Quantization}

In the final chapter, Chapter \ref{ch: sing quant}, we arrive at the definition of quantization for symplectic stratified spaces. Section \ref{sec: quant symp strats} presents Definition \ref{def: stratified quantization}, which closely parallels the classical definition of quantization, but replaces all standard quantization data with their stratified counterparts. Section \ref{sec: [Q,R] sing} is devoted to establishing the existence of a canonical $[Q,R]=0$ map relating the quantization of a Hamiltonian $G$-space to that of its singular reduced space. Building on this, Section \ref{sec: toric [Q,R]} proves in Theorem \ref{thm: [Q,R] for toric} that the $[Q,R]=0$ map is an injection for singular quotients of toric manifolds. Finally, Section \ref{sec: sing cotangent [Q,R]} shows that, despite the polarization not inducing a stratified polarization on the singular reduced space, the $[Q,R]=0$ result still holds in this setting as demonstrated by Theorem \ref{thm: sing [Q,R] cotangent}.
\part{Non-Singular Preliminaries}\label{pt: nonsing}
\cleardoublepage
\chapter{Background in Equivariant Geometry}\label{ch: equivariant}

Proper actions by Lie groups run at the very heart of this thesis. For this reason, in this Chapter we give a thorough introduction to the basic concepts of proper actions, most important of all being Palais' Slice Theorem \ref{thm: slice theorem}. The curious reader is encouraged to consult standard references like \cite{lee_introduction_2013} or \cite{michor_topics_2008} for more information.

\section{Proper Group Actions and Notation}\label{sec: proper action}

Let $G$ denote a Lie group and $M$ a smooth manifold equipped with a smooth (left) $G$-action
\begin{equation}\label{eq: G-space}
    G\times M\to M;\quad (g,m)\mapsto m\cdot m.
\end{equation}

\begin{defs}[$G$-space]\label{defs: G-space}
    Call a manifold $M$ with an action of $G$ as in Equation (\ref{eq: G-space}) a \textbf{$G$-space}.
\end{defs}

\begin{egs}\label{eg: adjoint and coadjoint actions}
    Let $G$ be a Lie group and $\mfg=\text{Lie}(G)$ its Lie algebra. Then $G$ has a natural action on $\mfg$ called the \textbf{adjoint action}. Viewing $\mfg$ as the tangent space to the identity element $e$, we can define the adjoint action as the derivative of the conjugation map. That is, for each $g\in G$ let
    $$
    c_g:G\to G;\quad a\mapsto gag^{-1}.
    $$
    This is a diffeomorphism, hence differentiable at the the identity. Now define
    $$
    \text{Ad}_g:=T_ec_g:\mfg\to \mfg.
    $$
    It's then an exercise in the chain rule to show that
    $$
    G\times \mfg\to \mfg;\quad (g,\xi)\mapsto \text{Ad}_g(\xi)
    $$
    makes $\mfg$ into a $G$-space. 

    \ 

    Dualizing the adjoint action, we also obtain the \textbf{coadjoint action}. Write 
    $$
    \bra \cdot,\cdot\ket:\mfg^*\times \mfg\to \bR;\quad (\alpha.\xi)\mapsto \bra \alpha,\xi\ket
    $$
    for the natural pairing between $\mfg$ and its dual $\mfg^*$. Given now $\alpha\in\mfg^*$ and $g\in G$, define $\text{Ad}_g^*(\alpha)$ by
    $$
    \bra \text{Ad}_g^*(\alpha),\xi\ket=\bra \alpha,\text{Ad}_{g^{-1}}(\xi)\ket.
    $$
    Then the map
    $$
    G\times \mfg^*\to \mfg^*;\quad (g,\alpha)\mapsto \text{Ad}_g^*(\alpha)
    $$
    makes $\mfg^*$ also into a $G$-space.
\end{egs}

Another core definition we will use is that of equivariance.

\begin{defs}[Equivariant Map]\label{def: equivariance}
    Let $G$ be a Lie group and $M$ and $N$ two $G$-spaces. Say a smooth map $f:M\to N$ is \textbf{equivariant} if for all $g\in G$ and $m\in M$, we have
    $$
    f(g\cdot m)=g\cdot f(m).
    $$
\end{defs}

Now fix a Lie group $G$ and a $G$-space $M$. Given a point $m\in M$, we will write
\begin{equation}
    G\cdot m=\{g\cdot m \ | \ g\in G\}
\end{equation}
for the orbit through $m$. We will also write
\begin{equation}
    G_m:=\{g\in G \ | \ g\cdot m=m\}
\end{equation}
for the stabilizer group of $m$. Finally, let us formally define three kinds of distinguished subsets as these will be central to our later discussions.

\begin{defs}[fixed point Set, Manifold of Symmetry, Orbit-Type Class]\label{def: fixed point, symmetry, and orbit-type}
    Let $M$ be a proper $G$-space and $H\leq G$ a subgroup.
    \begin{itemize}
        \item[(1)] Define the \textbf{$H$-fixed point set} by
        \begin{equation}\label{eq: fixed point}
            M^H:=\{m\in M \ | \ H\subseteq G_m\}
        \end{equation}
        \item[(2)] Define the \textbf{$H$-manifold of symmetry} by
        \begin{equation}\label{eq: manifold of symmetry}
            M_H:=\{m\in M \ | \ H=G_m\}\\
        \end{equation}
        \item[(3)] Define the \textbf{$H$ orbit-type class} by
        \begin{equation}\label{eq: orbit-type class}
            M_{(H)}:=\{m\in M \ | \ G_m\text{ is conjugate to }H\}.
        \end{equation}
    \end{itemize}
\end{defs}
As we will see later on, provided the action of $G$ on $M$ is sufficiently nice, we can leverage many properties out of the subsets of the form $M_{(H)}$.

\ 

Finally, let's establish some standard notation for the induced Lie algebra action of $G$ on $M$. Write $\mfg=\text{Lie}(G)$ for the Lie algebra of $G$. Given $\xi\in \mfg$ we define the \textbf{fundamental vector field associated to $\xi$}, denote $\xi_M\in \mfX(M)$ by
$$
(\xi_M f)(m)=\frac{d}{dt}\bigg|_{t=0} f(\exp(-t\xi)\cdot m),
$$
for $f\in C^\infty(M)$. With this definition, the map
$$
\mfg\to \mfX(M);\quad \xi\mapsto \xi_M
$$
is a Lie algebra homomorphism. 

\ 

Now, in general, we will only be interested in a particular type of Lie group action, namely proper group actions.

\begin{defs}[Proper Action/ Proper $G$-Space]
    Let $M$ be a $G$-space. We say the action of $G$ on $M$ is \textbf{proper} if the map
    \begin{equation}
        a:G\times M\to M\times M;\quad (g,m)\mapsto (g\cdot m,m)
    \end{equation}
    is proper. That is, for each compact $C\subseteq M\times M$, the pre-image $a^{-1}(C)\subseteq G\times M$ is compact. If the action of $G$ on $M$ is proper, we will call $M$ a \textbf{proper $G$-space}.
\end{defs}

Let us now record a handy fact that we will be using for our discussion of equivariant vector bundles.

\begin{prop}[{\cite[Proposition 21.5]{lee_introduction_2013}}]\label{prop: equivalent characterization of proper}
    Let $M$ be a $G$-space. Then the following are equivalent.
    \begin{itemize}
        \item[(1)] $M$ is a proper $G$-space.
        \item[(2)] For any sequences $\{m_i\}\subseteq M$ and $\{g_i\}\subseteq G$, if both $\{m_i\}$ and $\{g_i\cdot m_i\}$ are convergent, then $\{g_i\}$ has a convergent subsequence.
    \end{itemize}
\end{prop}

\begin{cor}\label{cor: compact action is proper}
    If $M$ is a $G$-space and $G$ is compact, then the action of $G$ is proper.
\end{cor}

\begin{proof}
    Since $G$ is compact, any sequence has a convergent subsequence. Thus the second condition in Proposition \ref{prop: equivalent characterization of proper} will automatically hold.
\end{proof}

\begin{cor}\label{cor: if target is proper, so is domain}.
    Let $M$ and $N$ be two $G$-spaces and $f:M\to N$ an equivariant map. If $N$ is a proper $G$-space, then $M$ is also a proper $G$-space.
\end{cor}

\begin{proof}
    Let $\{m_i\}\subseteq M$ be a convergent sequence and $\{g_i\}\subseteq G$ a sequence so that $\{g_i\cdot m_i\}$ is convergent. By continuity, $\{f(m_i)\}$ and $\{f(g_i\cdot m_i)\}$ are convergent sequences in $N$. By equivariance, $f(g_i\cdot m_i)=g_i\cdot f(m_i)$ for all $i$. Thus, since the action of $G$ on $N$ is proper, we deduce that there exists a convergent subsequence of $\{g_i\}$. Hence, the action of $G$ on $M$ is proper as well.
\end{proof}

Let us now recall some standard facts about proper group actions.

\begin{prop}\label{prop: fundamental facts about proper G-spaces}
    Let $M$ be a smooth proper $G$-space and $m\in M$. 
    \begin{itemize}
        \item[(1)] The stabilizer group $G_m$ to $m$ is a compact subgroup of $G$.
        \item[(2)] The orbit $G\cdot m$ is a closed embedded submanifold of $M$ with the map
        $$
        G/G_m\to G\cdot m;\quad gG_m\mapsto g\cdot m
        $$
        being an $G$-equivariant diffeomorphism.
        \item[(3)] The tangent space to $G\cdot m$ has the form
        $$
        T_m(G\cdot m)=\{\xi_M(m) \ | \ \xi\in \mfg\}.
        $$
    \end{itemize}
\end{prop}

\section{Local Normal Form for Proper Actions}\label{sec: local normal form proper action}
Let us now discuss the local normal form for proper $G$-space which we will be discussing later on. Let $K\leq G$ be a compact subgroup and let $V$ be a real linear $K$-representation. Define
\begin{equation}
    G\times _K V:=(G\times V)/K,
\end{equation}
where the $K$-action on $G\times V$ is given by
$$
K\times (G\times V)\to G\times V;\quad (k,(g,v))\mapsto (gk^{-1},k\cdot v).
$$
This is a free action by a compact group, and hence the quotient is a smooth manifold. We will be using these normal forms for a strengthened version for proper actions, so I will defer the proof to a standard reference.

\begin{prop}[{\cite{lee_introduction_2013}}]\label{prop: local model for proper G space}
    Let $K\leq G$ and $V$ be as above.
    \begin{itemize}
        \item[(1)] The space $G\times_K V$ is a smooth manifold and the projection
        $$
        \pi:G\times V\to G\times_K V;\quad (g,v)\mapsto [g,v]
        $$
        is a surjective submersion.
        \item[(2)] The action
        $$
        G\times (G\times_K V)\to G\times_K V;\quad (g,[a,v])\mapsto [ga,v]
        $$
        makes $G\times_K V$ into a proper $G$-space.
    \end{itemize}
\end{prop}

With this in hand, we can observe that the special submanifolds we introduced in Definition \ref{def: fixed point, symmetry, and orbit-type} are actually submanifolds of proper $G$-spaces of the form $G\times_K V$.

\begin{prop}\label{prop: local normal form manifold of symmetry and orbit type}
    Let $G$ be a Lie group, $K\leq G$ a compact subgroup, and $V$ a linear $K$-representation. Write
    $$
    N_G(K)=\{g\in G \ | \ gKg^{-1}=K\}
    $$
    for the normalizer subgroup of $K$.
    \begin{itemize}
        \item[(i)] Have equalities
        \begin{equation}\label{eq: local normal form manifold of symmetry}
            (G\times_K V)_K=N_G(K)\times_K V^K
        \end{equation}
        and
        \begin{equation}\label{eq: local normal form orbit-type}
            (G\times_K V)_{(K)}=G\times_K V^K
        \end{equation}
        \item[(ii)] The map
        \begin{equation}\label{eq: local normal form msym embedding}
            (N_G(K)/K)\times V^K\to G\times_K V;\quad (gK,v)\mapsto [g,v]
        \end{equation}
        is an $N_G(K)$-equivariant embedding with image $(G\times_K V)_K$ and 
        \begin{equation}\label{eq: local normal form otyp embedding}
            (G/K)\times V^K\to G\times_K V;\quad (gK,v)\mapsto [g,v]
        \end{equation}
        is a $G$-equivariant embedding with image $(G\times_K V)_{(K)}$.
    \end{itemize}
\end{prop}

\begin{proof}
    First, observe two key facts that we will make good use of throughout this proof.

    \ 

    \noindent\textbf{Fact 1}: For any $[g,v]\in G\times_K V$ we have
    \begin{equation}\label{eq: stablizer in local normal form}
    G_{[g,v]}=gG_{[e,v]}g^{-1}=gK_v g^{-1},
    \end{equation}
    where $K_v\leq K$ is the stabilizer group of the linear $K$-action on $V$.

    \ 

    \noindent\textbf{Fact 2}: For any $(g,v)\in G\times K$, 
    \begin{equation}\label{eq: equal to conjugate subgroup}
        K\subseteq gK_vg^{-1}\quad \Rightarrow \quad K=K_v.
    \end{equation}

    \ 

    The first fact is immediate from the definition of the action of $G$. As for the second, note that $K_v\subseteq K$ and $K\subseteq gK_vg^{-1}$ implies that $K\subseteq gKg^{-1}$. It follows that $K=gKg^{-1}$ and hence $K=K_v$. 

    \ 

    With these claims out of the way, we can now prove items (i) and (ii).

    \begin{itemize}
        \item [(i)] First we prove Equation (\ref{eq: local normal form manifold of symmetry}). The inclusion $N_G(K)\times_K V^K\subseteq (G\times_K V)_K$ is immediate from Equation (\ref{eq: stablizer in local normal form}). For the reverse, observe that if $[g,v]\in (G\times_K V)_K$, then $gK_vg^{-1}=K_v$ and hence by Equation (\ref{eq: equal to conjugate subgroup}), $K=K_v$ and $g\in N_G(K)$.

        \ 

        Equation (\ref{eq: local normal form orbit-type}) is proven almost verbatim the same way.

        \item[(ii)] We only show the map from Equation (\ref{eq: local normal form msym embedding}) is an $N_G(K)$-equivariant embedding with image $(G\times_K V)_K$. The proof of Equation (\ref{eq: local normal form otyp embedding}) is similar. To that end, let us give the map from Equation (\ref{eq: local normal form msym embedding}) a name, say
        $$
        \phi:(N_G(K)/K)\times V^K\to G\times_K V;\quad (gK,v)\mapsto [g,v]
        $$
        Observe that $\phi$ makes the diagram commute
        $$
        \begin{tikzcd}
            N_G(K)\times V^K\arrow[r,hook]\arrow[d] & G\times V\arrow[d]\\
            (N_G(K)/K)\times V^K\arrow[r,"\phi"] & G\times_K V
        \end{tikzcd}
        $$
        where the vertical arrows are the quotient maps by the $K$-action. Hence, $\phi$ is trivially an $N_G(K)$-equivariant smooth map. Furthermore, its a triviality to see that the derivative of $\phi$ must be injective everywhere. Finally $\phi$ is an embedding since the top map is a $K$-equivariant embedding and $K$ is compact. Clearly the image of $\phi$ is $(G\times_K V)_K$.
        
    \end{itemize}
\end{proof}

\section{The Slice Theorem for Proper Group Actions}\label{sec: slice theorem}

One of the foundational results about proper group actions is the existence of slices. These are submanifolds of a proper $G$-space which meet the orbits of the action in a transverse fashion which allows for a form of local splitting. We will not be going in depth into the theory of slices, but the curious reader can refer to the original paper of Palais \cite{palais_existence_1961} or standard references like \cite{lee_introduction_2013} or \cite{michor_topics_2008} for more information.

\ 

Fix a Lie group $G$. In the particular form of the slice theorem we will be utilizing makes use of a local normal form for proper $G$-spaces. 

Given now a proper $G$-space $M$, as we saw in Proposition \ref{prop: fundamental facts about proper G-spaces}, the stabilizer group $G_m\leq G$ is compact. Since $G_m$ fixes $m$, we get a canonical linear $G_m$ action on the tangent space $T_m M$ given by
$$
G_m\times T_mM\to T_mM;\quad (g,v)\mapsto T_m g(v),
$$
where $T_mg:T_mM\to T_mM$ is the derivative of the diffeomorphism
$$
M\to M;\quad m\mapsto g\cdot m.
$$
Not only does this linearized action of $G_m$ preserve the tangent space of $m$, it also preserves the tangent space of the orbit $G\cdot m$ through $m$. Thus, we get a canonical action on the normal space
$$
G\times (T_mM/T_m(G\cdot m))\to T_mM/(T_m(G\cdot m)).
$$
For our convenience, let us now introduce some notation.

\begin{defs}[Normal Space to Orbit]\label{def: normal space to orbit}
    For any point $m\in M$, define the \textbf{normal space to the orbit through $m$} by
    \begin{equation}
        \nu_m(M,G):=T_mM/(T_m(G\cdot m)).
    \end{equation}
\end{defs}

Thus, we have shown that given any point $m\in M$, we have a canonical linear action by the compact group $G_m$ on the normal space $\nu_m(M,G)$. This now allows us to state the slice theorem.

\begin{theorem}[Slice Theorem {\cite{palais_existence_1961}}]\label{thm: slice theorem}
Let $M$ be a proper $G$-space. For any $m\in M$ there exists
\begin{itemize}
    \item[(1)] $G$-invariant neighbourhoods $U\subseteq M$ of $m$ and $V$ of $[e,0]$ in $G\times_{G_m}\nu_m(M,G)$.
    \item[(2)] A $G$-equivariant diffeomorphism $\phi:U\to V$ mapping $m$ to $[e,0]$.
\end{itemize}
We will call the data $\phi:U\to V$ a \textbf{slice neighbourhood centred on $m$}.
\end{theorem}

Although this is technically a consequence of results we will be presenting later, now is a most convenient time to state a standard result we will get much use out of later on.

\begin{cor}[{\cite[Remark B.25]{guillemin_moment_2002}}]\label{cor: maximal slice neighbourhood}
    Let $M$ be a proper $G$-space and $m\in M$. Then there exists a \textbf{maximal slice neighbourhood centred on $m$}. That is, a slice neighbourhood
    $$
    \phi:U\subseteq M\to V\subseteq G\times_{G_m}\nu_m(M,G)
    $$
    as in Theorem \ref{thm: slice theorem}, but where
    $$
    V=G\times_{G_m}\nu_m(M,G).
    $$
\end{cor}

The main benefit of Palais' Slice Theorem is that to prove local properties about proper group actions, it suffices to assume we are in a maximal slice neighbourhood as in Corollary \ref{cor: maximal slice neighbourhood} and prove our properties there. Let us illustrate 

One immediate consequence of Palais' slice theorem that the connected components of the distinguished subsets we introduced in Definition \ref{def: fixed point, symmetry, and orbit-type} are embedded submanifolds!

\begin{cor}\label{cor: manifold of symmetry and orbit-type are manifolds}
    Let $M$ be a proper $G$ space and $K\leq G$ a compact subgroup.
    \begin{itemize}
        \item[(1)] $M_K$ and the connected components of $M_{(K)}$ are embedded submanifolds.
        \item[(2)] $M_K$ is a free $N_G(K)/K$-space and the connected components of $M_{(K)}$ are closed under the action of $G$.
        \item[(3)] $G\cdot M_K=M_{(K)}$.
    \end{itemize}
\end{cor}

\begin{proof}
\begin{itemize}
    \item[(1)] If $M_K=\emptyset$, then  $M_{(K)}=\emptyset$ as well so nothing to prove here. So now suppose $x\in M_K$. Since this is a local question, we may pass to a maximal slice neighbourhood as in in Corollary \ref{cor: maximal slice neighbourhood}. Now it's a simple matter of applying Proposition \ref{prop: local normal form manifold of symmetry and orbit type}.

    \item[(2)] Since for any $x\in  M$ and $g\in G$ we have
    \begin{equation}\label{eq: stabilizers get conjugated}
    G_{g\cdot x}=gG_xg^{-1}
    \end{equation}
    it follows that $M_{(K)}$ is closed under the $G$-action. Since $G$ is connected, so are the connected components of $M_{(K)}$. As for $M_K$, Equation (\ref{eq: stabilizers get conjugated}) also shows that $M_K$ is closed under the action of $N_G(K)$ since if $x\in M_K$ and $g\in N_G(K)$ we have
    $$
    G_{g\cdot x}=gKg^{-1}=K.
    $$
    Furthermore, given $g\in N_G(K)$ and $x\in M_K$ we have $g\cdot x=x$ if and only if $g\in K$. Hence, $N_G(K)/K$ acts freely.

    \item[(3)] Clearly $G\cdot M_K\substeq M_{(K)}$. If $x\in M_{(K)}$, then $G_x=gKg^{-1}$ for some $g\in G$. Thus, $g^{-1}\cdot x\in M_K$.
\end{itemize}
\end{proof}

We will make much use of this fact later on. For now, let us give a condition for when orbit spaces are manifolds.

\begin{prop}
    Let $M$ be a proper $G$-space and suppose $M=M_{(K)}$ for some $K$. Then $M/G$ is canonically a smooth manifold and the quotient map
    $$
    \pi:M\to M/G
    $$
    is a surjective submersion.
\end{prop}

\begin{proof}
    Let $x\in M$ and let $K=G_x$. By Corollary \ref{cor: maximal slice neighbourhood}, we can find an open neighbourhood $U\subseteq M$ of $x$ and a $G$-equivariant diffeomorphism
    \begin{equation}\label{eq: chart for quotient}
    \phi:U\to G\times_K \nu_x(M,G).
    \end{equation}
    Since $M=M_{(K)}$, by Proposition \ref{prop: local normal form manifold of symmetry and orbit type} we have $\nu_x(M,G)=\nu_x(M,G)^K$ and 
    $$
    G\times_K \nu_x(M,G)\cong (G/K)\times \nu_x(M,G).
    $$
    Hence, $U/G\cong \nu_x(M,G)$. Since $\nu_x(M,G)$ is a vector space, this provides us with a chart for $M/G$. Clearly any two slice neighbourhoods of this form will be compatible, providing us with an atlas for $M/G$. 

    \ 

    The quotient $\pi:M\to M/G$ is also clearly a surjective submersion since restricting to a slice neighbourhood as in Equation (\ref{eq: chart for quotient}), we have the diagram commutes
    $$
    \begin{tikzcd}
        U\arrow[r,"\phi"]\arrow[d,"\pi"] & (G/K)\times \nu_x(M,G)\arrow[d,"pr_2"]\\
        U/G\arrow[r,"\widetilde{\phi}"] & \nu_x(M,G)
    \end{tikzcd}
    $$
    where the second vertical arrow is the projection onto the second factor.
\end{proof}

This of course recovers a classical result about free actions.

\begin{cor}
    Suppose $M$ is a proper $G$ space and $G$ acts freely on $M$. Then $M/G$ is canonically a smooth manifold and $\pi:M\to M/G$ is a surjective submersion.
\end{cor}

\begin{proof}
    In this case, $M=M_{(\{e\})}$, where $e\in G$ is the identity.
\end{proof}

\begin{remark}
    Of course this is somewhat cheating as we implicitly already used this statement for compact groups in Proposition \ref{prop: local model for proper G space}. The non-trivial part is that we have extended to potentially non-compact groups, provided they act properly.
\end{remark}

\section{Tangent Spaces of Manifolds of Symmetry and Orbit-Type Classes}\label{sec: tangent spaces mfls symmetry and orbit-type}

Let $M$ be a proper $G$-space. As we saw in Corollary \ref{cor: manifold of symmetry and orbit-type are manifolds} the connected components of the manifolds of symmetry and the orbit-type classes are embedded submanifolds. With a little more set-up, we will be able to actually compute the tangent spaces of these manifolds.

\ 

Let $G$ be a Lie group and fix a compact subgroup $K\leq G$ a compact subgroup and write $\mfk=\text{Lie}(K)$ for the Lie algebra of $K$. Since $\mfk$ is closed by the adjoint action of $K$ on $\mfg$, we get an induced action
$$
K\times (\mfg/\mfk)\to \mfg/\mfk;\quad (k,\xi+\mfk)\mapsto \text{Ad}_k(\xi)+\mfk.
$$
We now seek to determine the fixed point set of this action. To do this, let us introduce some notation. Write
$$
N_G(K)=\{g\in G \ | \ gKg^{-1}\subseteq K\}
$$
for the normalizer of $K$ in $G$ and write
$$
Z_G(K)=\{g\in G \ | \ gk=kg\text{ for all }k\in K\}
$$
for its centralizer. These are both closed subgroups of $G$, hence have Lie algebras. Write
$$
\mfn_G(K)=\text{Lie}(N_G(K))\text{ and }\mfz_G(K)=\text{Lie}(Z_G(K)).
$$
Note that in the case where $K$ is connected, $\mfn_G(K)$ is equal to the Lie algebra normalizer 
$$
\mfn_\mfg(\mfk)=\{\xi\in \mfg \ | \ [\xi,\mfk]\subseteq \mfk\}
$$

and $\mfz_G(K)$ is equal to the Lie algebra centralizer
$$
\mfz_\mfg(\mfk)=\{\xi \in \mfg \ | \ [\xi,\eta]=0\text{ for all }\eta\in\mfk\}.
$$
In general the Lie algebras of the normalizer and centralizer subgroups will not be equal to the Lie algebra normalizers and centralizers as the following example shows.

\begin{egs}
    The idea for this counter-example came from the following Math StackExchange post \cite{schoeneberg_answer_2023}.Let $G=\text{GL}_2(\bR)$ and consider the subgroup
    $$
    K=\bigg\{\begin{pmatrix}
        a & 0\\
        0 & b
    \end{pmatrix}\in \text{GL}_2(\bR) \ \bigg| \ a,b=\pm 1 \bigg\}
    $$
    As $K$ is a discrete subgroup, its Lie algebra $\mfk=\{0\}$. Hence,
    $$
    \mfz_{\mathfrak{gl}_2(\bR)}(\mfk)=\mfn_{\mathfrak{gl}_2(\bR)}(\mfk)=\mathfrak{gl}_2(\bR)
    $$
    On the other hand, an easy computation shows
    \begin{align*}
        N_{\text{GL}_2(\bR)}(K)&=\bigg\{\begin{pmatrix}
        a & b\\
        c & d
    \end{pmatrix} \in \text{GL}_2(\bR) \ \bigg| \ a=d=0\text{ or }b=c=0\bigg\}\\
        Z_{\text{GL}_2(\bR)}(K)&=\bigg\{\begin{pmatrix}
        a & 0\\
        0 & b
    \end{pmatrix} \in \text{GL}_2(\bR) \ \bigg| \ a,b\in\bR\bigg\}
    \end{align*}
    In which case, 
    $$
    \mfn_{\text{GL}_2(\bR)}(K)=\mfn_{\text{GL}_2(\bR)}(K)=\bigg\{\begin{pmatrix}
        a & 0\\
        0 & b
    \end{pmatrix} \in \mathfrak{gl}_2(\bR) \ \bigg| \ a,b\in\bR\bigg\}
    $$
    Hence, the Lie algebra normalizers and centralizers need not be equal to the Lie algebras of the Lie group normalizers and centralizers.
\end{egs}

None the less, we still get a useful relation.

\begin{prop}[{\cite[Corollary B]{hazod_normalizers_1998}}] \label{prop: Lie algebra of normalizer}
For any compact subgroup $K\leq G$ with Lie algebra $\mfk$, we have an equality
$$
\mfn_G(K)=\mfk+\mfz_G(K).
$$
\end{prop}

Using this, we can now show the following.

\begin{prop}\label{prop: fixed point set of quotient of normal lie algebra}
If $G$ is connected and $K\leq G$ is a compact subgroup, then
$$
(\mfg/\mfk)^K=\mfn_G(K)/\mfk.
$$
\end{prop}

\begin{proof}

Let us first show 
$$
\mathfrak{n}_G(K)/\mfk\subseteq (\mfg/\mfk)^K.
$$
Suppose that $\xi\in \mathfrak{n}_G(K)$ and write
$$
\xi=\eta+\zeta,
$$
for $\eta\in \mfk$ and $\zeta\in \mathfrak{z}_G(K)$. Since $H$ acts trivially on $Z_G(K)$ via conjugation, we then trivially have for any $k\in K$ that
$$
\text{Ad}_k(\xi)=\text{Ad}_k(\eta)+\zeta.
$$
Therefore,
$$
\text{Ad}_k(\xi)-\xi=\text{Ad}_k(\eta)-\eta\in \mfk.
$$

Conversely, to show that
$$
(\mfg/\mfk)^K\subseteq \mathfrak{n}_G(K)/\mfk,
$$
fix a $\xi\in \mfg$ with 
$$
\text{Ad}_k(\xi)-\xi\in \mfk
$$
for all $k\in K$. Define now a map
$$
\Phi:\bR\times G\to G;\quad (t,g)\mapsto \Phi_t(g):=\exp(t\xi)g\exp(-t\xi).
$$
It's then a straightforward exercise is Lie theory to show that
$$
\frac{d}{dt}\Phi_t(g)\bigg|_{t=0}=T_eL_g(\text{Ad}_g(\xi)-\xi).
$$
In particular, for any $k\in K$, we then get that 
$$
\frac{d}{dt}\Phi_t(k)\bigg|_{t=0}=T_eL_k(\text{Ad}_k(\xi)-\xi)\in T_kK.
$$
Since $K$ is a closed subgroup of $G$, it then follows that $\Phi_t(k)\in K$ for any $t\in \bR$ and $k\in K$. In particular, setting $t=1$, we get 
$$
\Phi_1(K)=\exp(\xi)K\exp(-\xi)\subseteq K
$$
and thus $\xi\in \mathfrak{n}_G(K)$.
\end{proof}

As an application of this, let us now return to the manifolds of symmetry from Definition \ref{def: fixed point, symmetry, and orbit-type}

\begin{prop}\label{prop: tangent space to manifold of symmetry and orbit-type}
Fix $x\in M$, write $K=G_x$, and let $S\subseteq M_{(K)}$ be the connected component containing $x$. Then,
\begin{align*}
    T_xM_K&=(T_xM)^K\\
    T_xS&=(T_xM)^K + T_x(G\cdot x)
\end{align*}
\end{prop}

\begin{proof}
Since this is a local question, we may pass to a slice neighbourhood as in Theorem \ref{thm: slice theorem} and assume $M=G\times_K V$ for a finite-dimensional $K$-representation and $x=[e,0]$. Observe that via the quotient map
$$
G\times V\to G\times_K V,
$$
we have a canonical identification
$$
T_x(G\times_K V)=\mfg/\mfk\oplus V.
$$
Now, let us first turn to the manifold of symmetry $L\subseteq M_K$. Recall that in Proposition \ref{prop: local normal form manifold of symmetry and orbit type} we showed that
$$
M_K=N_G(K)\times_K V^K
$$
Since $L$ is the connected component of $M_K$ containing $[e,0]$, we immediately obtain
$$
T_xL=\mfn_G(K)/\mfk\oplus V^K.
$$
By Proposition \ref{prop: Lie algebra of normalizer}, we have
$$
(\mfg/\mfk\oplus V)^K=(\mfg/\mfk)^K\oplus V^K=(\mfn_G(K)/\mfk)\oplus V^K.
$$
Thus, $(T_xM)^K=T_xL$ as desired.

\ 

Now for the orbit-type piece $S\subseteq M_{(K)}$. Making use of Proposition \ref{prop: local normal form manifold of symmetry and orbit type} once again, we have
$$
M_{(K)}=G\times_K V^K
$$
and thus, since $S$ too is the connected component containing $x$ we get
$$
T_xS=(\mfg/\mfk)\oplus V^K.
$$
As we already saw, 
$$
(T_xM)^K=(\mfn_G(K)/\mfk)\oplus V^K
$$
whereas
$$
T_x(G\cdot x)=(\mfg/\mfk)\oplus 0
$$
Adding these two together, then yields
\begin{align*}
(T_xM)^K+T_x(G\cdot x)&=(\mfn_G(K)/\mfk)\oplus V^K+(\mfg/\mfk)\oplus 0\\
&=(\mfg/\mfk)\oplus V^K\\
&=T_xS.
\end{align*}
\end{proof}

\section{Equivariant Vector Bundles}\label{sec: equivariant vector bundles}

Now let us discuss an extension of the Slice Theorem to equivariant vector bundles. The pay-off for this discussion is (1) we will get a nice local normal form for equivariant vector bundles and (2) we will get an averaging procedure for proper group actions. 

\begin{defs}[Equivariant Vector Bundle]
    Let $\pi:E\to M$ be a smooth vector bundle and $G$ a Lie group. A \textbf{$G$-equivariant vector bundle structure} on $\pi:E\to M$ is a smooth action of $G$ on both $E$ and $M$ such that
    \begin{itemize}
        \item[(1)] $\pi:E\to M$ is $G$-equivariant.
        \item[(2)] The action of $G$ on $E$ is linear. That is, for each $x\in M$ and $g\in G$, the induced map
        $$
        E_x\to E_{g\cdot x};\quad e\mapsto g\cdot e
        $$
        is linear.
    \end{itemize}
    Call $\pi:E\to M$ a \textbf{proper equivariant vector bundle} if the action of $G$ on $M$ is proper.
\end{defs}

It might seem strange that in our definition of a proper equivariant vector bundle, we only demand that $G$ act properly on the base. As it turns out, this automatically implies that $G$ acts on the total space $E$ as well.

\begin{prop}
    If $\pi:E\to M$ is a proper equivariant vector bundle, then the action of $G$ on $E$ is proper.
\end{prop}

\begin{proof}
    This is a simple application of Corollary \ref{cor: if target is proper, so is domain}.
\end{proof}

\begin{egs}\label{eg: tangent and cotangent equivariant vb structures}
    The standard example of a proper equivariant vector bundle is the tangent bundle of a proper $G$-space. Indeed, if $M$ is a proper $G$-space, then by differentiating the $G$-action on $M$ we get a $G$-action
    $$
    G\times TM\to TM;\quad (g,v)\mapsto Tg(v)
    $$
    where $Tg:TM\to TM$ denotes the derivative of the map
    $$
    M\to M;\quad m\mapsto g\cdot m.
    $$
    This action can be dualized to the cotangent bundle. Indeed, given $m\in M$, $\alpha\in T^*_mM$, and $g\in G$, define $g^*\alpha\in T^*_{g\cdot m}M$ by
    $$
    \bra g^*\alpha,v\ket=\bra \alpha, T_{m}g^{-1}(v)\ket,
    $$
    where $v\in T_{g\cdot m}M$ and $\bra\cdot,\cdot\ket$ denotes the pairing between the tangent and cotangent fibres. 
\end{egs}

\begin{egs}\label{eg: free orbit foliation}
    Another great class of examples coming from actions on manifolds would be the ``orbit foliation'', which we will be exploring in more detail later on. Like before, suppose $M$ is a proper $G$-space and suppose further that $M$ has only one orbit-type, i.e. $M=M_{(K)}$ for some compact $K\leq G$. Write $\mfg$ for the Lie algebra of $G$. Then define a $G$-invariant subbundle $\mfg_M\subseteq TM$ to be the image of the map
    $$
    \mfg\times M\to TM;\quad (\xi,m)\mapsto \xi_M(m),
    $$
    where $\xi_M\in\mfX(M)$ is the fundamental vector field associated to $\xi$. Note that for all $g\in G$, $m\in M$, and $\xi\in \mfg$, we have
    $$
    T_mg(\xi_M(m))=(\text{Ad}_g\xi)_M(g\cdot m)
    $$
    and thus $\mfg_M$ is closed under the action of $G$. To see that $\mfg_M$ is actually a subbundle, let us fix $x\in M$ and pass to a slice neighbourhood centered on $x$. That is, let us assume $M=G\times_K V$ for compact $K\leq G$ and a $K$-representation $V$, and $x=[e,0]$. Since $M=M_{(K)}$, it follows that $V=V^K$ and hence there is a natural $G$-equivariant diffeomorphism
    $$
    M\cong (G/K)\times V.
    $$
    It's then easy to see that under this identification
    $$
    \mfg_M=(G/K)\times V\times (\mfg/\mfk),
    $$
    where $\mfk$ is the Lie algebra of $K$. Hence $\mfg_M$ is a vector subbundle of $TM$.

    \ 

    Of course, thanks to Example \ref{eg: tangent and cotangent equivariant vb structures}, we get a whole bunch of new equivariant vector bundles from $\mfg_M$. For instance, it's straightforward to see that the annihilator
    $$
    \mfg_M^\circ=\{\alpha\in T^*M \ | \ \bra \alpha,v\ket=0\text{ for all }v\in \mfg_M\}
    $$
    is also an equivariant vector bundle with respect the induced cotangent action.
\end{egs}

In what follows, we will need a particular normal form for proper equivariant vector bundles. 

\begin{prop}\label{prop: local model for equivariant vector bundle}
    Let $G$ be a Lie group, $K\leq G$ a compact subgroup, and $V$, $W$ two finite-dimensional $K$-representation. Then
    $$
    \pi:G\times_K (V\times W)\to G\times_K V;\quad [g,(v,w)]\mapsto [g,v]
    $$
    is a proper equivariant vector bundle.
\end{prop}

\begin{proof}
    Since the projection 
    $$
    pr:G\times V\times W\to G\times V;\quad (g,v,w)\mapsto (g,v)
    $$
    is smooth and $K$-equivariant, the induced map on the quotient
    $$
    \begin{tikzcd}
        G\times V\times W \arrow[rr,"pr"]\arrow[dd] && G\times V\arrow[dd]\\
        &&\\
        G\times_K(V\times W) \arrow[rr,"\pi"] && G\times_K V
    \end{tikzcd}
    $$
    is smooth. Clearly $\pi$ is also $G$-equivariant. As for the vector bundle structure, given $[g_0,v_0]\in G\times_K V$, the fibre is given by
    $$
    (G\times_K(V\times W))_{[g_0,v_0]}=\{[g_0,(v_0,w)] \ | \ w\in W\},
    $$
    which is naturally isomorphic to $W$. To get local trivializations, fix $[g_0,v_0]\in G\times_K V$. Note that since the projection $\pi_K:G\to G/K$ is a principal $K$-bundle, we can find a neighbourhood $U\subseteq G/K$ of the coset $g_0K$ and isomorphism $\phi$ of $K$-principal bundles over $U$
    \begin{equation}
    \begin{tikzcd}
        \pi_K^{-1}(U)\arrow[dr,"\pi_K"]\arrow[rr,dashed,"\phi"] && U\times K\arrow[dl,"pr_1"]\\
        &U&
    \end{tikzcd}
    \end{equation}
    For convenience, write $\phi(g)=(gK,f(g))$ for all $g\in \pi_K^{-1}(U)$. Via this isomorphism, we obtain an isomorphism of vector bundles
    \begin{equation}\label{eq: local trivialization of local model}
    \begin{tikzcd}
        \pi_k^{-1}(U)\times_K (V\times W)\arrow[rr,"\phi_1"]\arrow[dd,"\pi"] && U\times V\times W\arrow[dd]\\
        &&\\
        \pi_K^{-1}(U)\times_K V\arrow[rr,"\phi_0"] && U\times V
    \end{tikzcd}
    \end{equation}
    where
    $$
    \phi_0:\pi_K^{-1}(U)\times_K V\to U\times V;\quad [g,v]\mapsto (gK,f(g)\cdot v)
    $$
    and $\phi_1$ is defined similarly. The diagram in Equation (\ref{eq: local trivialization of local model}) is a local trivialization.
\end{proof}

We will now show that every equivariant vector bundle locally has the form in Proposition \ref{prop: local model for equivariant vector bundle}. First, let us recall as result from Segal \cite{segal_equivariant_1968} regarding equivariant vector bundles for compact $G$.

\begin{prop}[{\cite{segal_equivariant_1968}}]\label{prop: equivariance and homotopies}
    Let $K$ be a compact Lie group. Suppose we have two $K$-spaces with $K$-homotopic $K$-equivariant maps $\phi_0,\phi_1:M\to N$. Furthermore, suppose $N$ is compact and $\pi:E\to N$ is a $K$-equivariant vector bundle. Then then there exists $K$-equivariant vector bundle isomorphism between the pullback bundles
    $$
    \begin{tikzcd}
        \phi_0^{-1}E\arrow[rr,dashed,"\exists"]\arrow[dr] && \phi_1^{-1}E\arrow[dl]\\
        &M&
    \end{tikzcd}
    $$
\end{prop}

With this in hand, we can prove the following.

\begin{lem}[Local Normal Form for Proper Equivariant Vector Bundles]
\label{lem: local normal form proper vb}
    Let $G$ be a Lie group, $K\leq G$ a compact subgroup, and $V$ a finite-dimensional $K$-representation. Write $M=G\times_K V$. If $\pi:E\to M$ is a proper $G$-equivariant vector bundle, then there exists a finite-dimensional $K$-representation $W$ and an equivariant vector bundle isomorphism 
    $$
    \begin{tikzcd}
        G\times_K(V\times W)\arrow[rr,dashed,"\exists"]\arrow[dr] && E\arrow[dl]\\
        &M&
    \end{tikzcd}
    $$
\end{lem}

\begin{proof}
    Fix a proper equivariant vector bundle $\pi:E\to G\times_K V$. Observe that we can identify $V$ with $K\times_K V\subseteq G\times_K V$. Define now $A=E|_{V}$ to be the restriction of $E$ to $V\cong K\times_K V$. Since $V$ is closed under the action of $K$, $\pi|_A:A\to V$ is a $K$-equivariant vector bundle. Since the action of $K$ on $V$ is linear, the scalar multiplication map
    $$
    \mu:[0,1]\times V\to V;\quad (t,v)\mapsto tv
    $$
    defines a $K$-equivariant homotopy between $V$ and $\{0\}$. Since both $K$ and $\{0\}$ are compact, Proposition \ref{prop: equivariance and homotopies} produces a $K$-equivariant isomorphism of vector bundles
    \begin{equation}\label{eq: iso of Gvb}
    \begin{tikzcd}
        A\arrow[rr,dashed,"\exists"]\arrow[dr,"\pi|_A"] && V\times W\arrow[dl,"pr_1"]\\
        &V&
    \end{tikzcd}
    \end{equation}
    where $pr_1:V\times W\to V$ is the projection onto the first factor.  Using an extremely similar proof as in Proposition \ref{prop: local model for equivariant vector bundle}, we can show that
    $$
    \widetilde{\pi}:G\times_K A\to G\times_K V;\quad [g,a]\mapsto [g,\pi(a)]
    $$
    is a $G$-equivariant vector bundle over $G\times_K V$. It's then an easy task to show the map
    $$
    G\times_K A\to E;\quad [g,a]\mapsto g\cdot a
    $$
    defines a $G$-equivariant vector bundle isomorphism. Composing with the isomorphism from Equation (\ref{eq: iso of Gvb}) produces the desired result.
\end{proof}

\begin{egs}\label{eg: normal form for tangent bundle of proper space}
    Using the ideas of the proof above, we can obtain a local normal form for tangent bundles of proper $G$-spaces. Indeed, given connected Lie group $G$, compact subgroup $K\leq G$, and a $K$-representation $V$, we have a $G$-equivariant isomorphism of vector bundles
    $$
    T(G\times_K V)=G\times_K(V\times (\mfg/\mfk)\times V).
    $$
    Similarly, we can also compute that the cotangent bundle has the form
    $$
    T^*(G\times_K V)=G\times_K(V\times \mfk^\circ\times V^*),
    $$
    where $\mfk^\circ\subseteq \mfg^*$ in the annihilator of $\mfk=\text{Lie}(K)$.
\end{egs}

The final piece of the puzzle we need for the averaging theorem for proper actions is the averaging theorem for compact groups.

\begin{lem}[Averaging Theorem For Compact Groups {\cite{segal_equivariant_1968}}]\label{lem: average for compact}
    Suppose $K$ is a compact Lie group, $\mu$ a choice of a Haar measure on $K$, and $V$ and $W$ two finite-dimensional real $K$-representations. Given any smooth map $f:V\to W$, there is a canonical $K$-equivariant map $f^K:V\to W$ called \textbf{the $K$-average of $f$} given by
    $$
    f^K(v):=\frac{1}{\text{vol}(K)}\int_K k^{-1}\cdot f(k\cdot v)d\mu(k).
    $$
\end{lem}

With this, we can now prove the following.

\begin{theorem}[Averaging Theorem for Proper Vector Bundles]{\cite{jotz_singular_2011}}]\label{thm: averaging over G}
    Let $\pi:E\to M$ be a proper $G$-equivariant vector bundle. Then there exists a linear surjective map
    \begin{equation}
        \Gamma(E)\to \Gamma(E)^G;\quad \sigma\mapsto \sigma^G,
    \end{equation}
    called an \textbf{averaging map}, such that if $\sigma\in\Gamma(E)^G$, then $\sigma^G=\sigma$.
\end{theorem}

\begin{proof}
    First, we show that locally we can average sections. Suppose $M=G\times_K V$ for compact subgroup $K\leq G$ and a $K$-representation $V$. By Lemma \ref{lem: local normal form proper vb}, we may assume $E=G\times_K(V\times W)$ and 
    $$
    \pi:G\times_K(V\times W)\to G\times_K V;\quad [g,(v,w)]\mapsto [g,v]
    $$
    for some $K$-representation $W$. Fix now a section $\sigma\in\Gamma(E)$. The restriction 
    $$
    \sigma|_{K\times_K V}:K\times_K V\to K\times_K (V\times W)
    $$
    can be identified with the section of the trivial $K$-equivariant vector bundle $V\times W\to V$, with the bundle map being the projection onto the first factor. Since $K$ is compact, we can apply the averaging from Lemma \ref{lem: average for compact} to obtain an equivariant section $\tau$ of $V\times W\to V$. That is, define
    \begin{equation}
    \tau:=(\sigma|_{K\times_K V})^K\in\Gamma(E|_{K\times_K V})^K,
    \end{equation}
    where the superscript indicates that the average over $K$ has been taken. With this, we can now define
    $$
    \sigma^G([g,v]):=g\cdot \tau([e,v]).
    $$
    Since $\tau$ is $K$-equivariant, it easily follows that $\sigma^G$ is well-defined and $G$-equivariant. Note that if $\sigma$ was already equivariant, then $\sigma|_{K\times_K V}$ would also be $K$-equivariant and hence $\tau=\sigma|_{K\times_K V}$. Thus, $\sigma^G=\sigma$ in this case.

    \ 

    Now that the local picture has been established, let us now show how to piece these constructions together to a global construction. First, let us fix a Haar measure $\mu$ on $G$ and let $\{U_i\}_{i\in I}$ be a locally finite cover of $M$ by slice neighbourhoods. Let $\{\phi_i\}_{i\in I}$ be a partition of unity subordinate to the cover $\{U_i\}_{i\in I}$. For any $i\in I$, $\phi_i$ is compactly supported. Hence, via the Haar measure $\mu$ on $G$, we can define
    $$
    \phi_i^G:M\to \bR;\quad x\mapsto \int_G\phi_i(g\cdot x)d\mu(g).
    $$
    Clearly $\phi_i^G$ is equivariant and smooth. Furthermore, the support of $\phi_i^G$ still lies in $U_i$ since $U_i$ is closed under the $G$-action. Finally, we will have
    $$
    \sum_i \phi_i^G=1.
    $$
    Now, given a section $\sigma\in\Gamma(E)$, restricting $\sigma|_{U_i}$ to one of the slice neighbourhoods, we can apply the averaging procedure as above to obtain a $G$-equivariant section $\sigma_i^G\in\Gamma(E|_{U_i})^G$. Now define
    $$
    \sigma^G=\sum_{i\in I}\phi_i\sigma_i^G.
    $$
    Then $\sigma^G\in\Gamma(E)^G$.
\end{proof}

\begin{remark}
    Notice that there is not just one averaging map, there are in principal an infinity of averaging maps. In particular, the local averaging maps relied on the specific stabilizer groups at the centres of the maximal slice neighbourhoods. These were then pieced together making use of a chosen Haar measure on the global group $G$. 
\end{remark}

As an application of averaging and the local normal form we derived for proper vector bundles, we can now construct a family of subbundles which have the property that they can be quotiented to obtain vector bundles over the quotient of the base space. To set up this construction, let $G$ be a connected Lie group and $\pi:E\to M$ a proper equivariant vector bundle. For any $x\in M$, since $\pi:E\to M$ is equivariant, it follows that the stabilizer group $G_x$ acts linearly on the fibre $E_x$. In particular, the fixed point set $E_x^{G_x}$ is defined.

\begin{defs}\label{def: special subbundle}
    Let $G$ be a connected Lie group and $\pi:E\to M$ a proper equivariant vector bundle. Define
    $$
    \widetilde{E}:=\bigcup_{x\in M} E_x^{G_x}.
    $$
\end{defs}

This is the key ingredient for our construction.

\begin{theorem}\label{thm: special subbundle regular case}
    Let $G$ be a connected Lie group and $\pi:E\to M$ a proper equivariant vector bundle. Suppose $M$ has one orbit-type, then the following holds.
    \begin{itemize}
        \item[(1)] $\widetilde{E}\subseteq E$ is a $G$-invariant vector subbundle of $E$.
        \item[(2)] The induced map $\widetilde{\pi}:\widetilde{E}/G\to M/G$ is a smooth vector bundle.
    \end{itemize}
\end{theorem}

\begin{proof}
Let $K\leq G$ be a compact subgroup so that $M=M_{(K)}$. Fix $x\in M$, we will now make use of Lemma \ref{lem: local normal form proper vb} to construct a $G$-equivariant local trivialization of $\widetilde{E}$. Thus, let us assume that $M=G\times_K V$ and $E=G\times_K(V\times W)$ for some $K$-representations $V,W$, $x=[e,0]$, and
$$
\pi:G\times_K(V\times W)\to G\times_KV;\quad [g,(v,w)]\mapsto [g,v].
$$
Since $M=M_{(K)}$ it follows that $V=V^K$. Note that the fibre over $[e,0]$ is isomorphic to $W$ as a $K$-representation. Using the fact that $K$ acts trivially on $V$, we have for any $v\in V$ that
$$
E_{[e,v]}^K\cong W^K.
$$
Using the $G$-action on $M$ and $E$, we then obtain that
$$
\widetilde{E}=G\times_K (V\times W^K)\cong (G/K)\times V^K\times W^K.
$$
This provides the desired $G$-equivariant local trivialization and hence $\widetilde{E}$ is a vector subbundle.

\ 

From our construction, it also immediately follows that $\widetilde{E}/G$ is a vector bundle over $M/G$. Indeed, in the local normal form above, we have
$$
M/G\cong V
$$
while
$$
\widetilde{E}/G\cong V\times W^K,
$$
providing a smooth local trivialization of $\widetilde{E}/G$.
\end{proof}

\begin{cor}
    Let $G$ be a connected Lie group and $\pi:E\to M$ a proper equivariant vector bundle. If $G$ acts freely on $M$, then the quotient $E/G\to M/G$ is a vector bundle.
\end{cor}

\begin{remark}
    One may wonder if the quotient of the total bundle $E/G$ is also a vector bundle over the quotient of the base $M/G$. In general, this is not the case. Indeed, consider the trivial bundle $\pi:\bR^2\to \{pt\}$ where we equip $\bR^2$ with the natural $S^1$ action given by rotations about the origin and $\{pt\}$ the trivial action. Then, $\bR^2/S^1$ is homeomorphic to $[0,\infty)$ which is clearly not homeomorphic to a vector bundle over $\{pt\}/S^1=\{pt\}$ as $[0,\infty)$ is not homeomorphic to any vector space.
\end{remark}

\begin{cor}\label{cor: tangent is special iff normal}
    Let $G$ be a connected Lie group and $M$ a proper $G$-space with only one orbit-type, i.e. $M=M_{(K)}$ for some compact $K\leq G$. Then $\widetilde{TM}=TM$ if and only if $K$ if normal.
\end{cor}

\begin{proof}
    This is a consequence of the normal form in Example \ref{eg: normal form for tangent bundle of proper space}. Passing to the normal form, we have $M=(G/K)\times V$ for a trivial $K$-representation $V$ and 
    $$
    TM=(G\times_K (\mfg/\mfk))\times V\times V.
    $$
    Hence, by Proposition \ref{prop: fixed point set of quotient of normal lie algebra},
    $$
    \widetilde{TM}=(G\times_K (\mfg/\mfk)^K)\times V\times V=(G/K)\times (\mathfrak{n}_G(K)/\mfk)\times V\times V.
    $$
    These two expressions are equal if and only if $\mfn_G(K)/\mfk=\mfg/\mfk$ which can only happen if $\mfn_G(K)=\mfg$. Since $G$ is connected, it follows that $N_G(K)=G$ and thus $K$ is normal.
\end{proof}

One consequence of averaging is that we can obtain an alternate characterization of $\widetilde{E}$ as the subbundle spanned by equivariant sections. Furthermore, we can use this further to determine all the sections of the quotient bundle $\widetilde{E}/G$.

\begin{prop}\label{prop: equivariant sections give all sections below}
    Let $G$ be a connected Lie group and $\pi:E\to M$ a proper equivariant vector bundle. Let $\widetilde{E}\subseteq E$ be the subset defined in Definition \ref{def: special subbundle}. Then the following holds.
    \begin{itemize}
        \item[(1)] The map
        \begin{equation}\label{eq: regular bundle generated by equivariant sections}
        M\times \Gamma(E)^G\to E;\quad (x,\sigma)\mapsto \sigma(x)
        \end{equation}
        has image $\widetilde{E}$. 
        \item[(2)] In the case where $M$ has one orbit-type, write $\pi_M:M\to M/G$ and $\pi_E:\widetilde{E}\to \widetilde{E}/G$ for the quotient maps. Then there exists canonical linear isomorphism
        $$
        (\pi_E)_*:\Gamma(E)^G\to \Gamma(\widetilde{E}/G)
        $$
        such that if $\sigma\in\Gamma(E)^G$, then $(\pi_E)_*\sigma\in\Gamma(\widetilde{E}/G)$ is the unique section making the diagram commute
        \begin{equation}\label{eq: regular quotient diagram}
        \begin{tikzcd}
            M\arrow[d,"\pi_M"]\arrow[r,"\sigma"] & E\arrow[d,"\pi_E"]\\
            M/G \arrow[r,dashed,"(\pi_E)_*\sigma"] & \widetilde{E}/G
        \end{tikzcd}
        \end{equation}
    \end{itemize}
\end{prop}

\begin{proof}
    \begin{itemize}
        \item[(1)] First, let us show the image of the map in Equation (\ref{eq: regular bundle generated by equivariant sections}) lies in $\widetilde{E}$. Fix an equivariant section $\sigma\in\Gamma(E)^G$ and $x\in M$. Then, for any $g\in G_x$, we have
        $$
        g\cdot\sigma(x)=\sigma(g\cdot x)=\sigma(x).
        $$
        Hence, $\sigma(x)\in E_x^{G_x}$. Conversely, suppose $v\in E_x^{G_x}$. Along the orbit $G\cdot x$ define a equivariant section
        $$
        \sigma_0:G\cdot x\to E|_{G\cdot x};\quad g\cdot x\mapsto g\cdot v.
        $$
        Since $v$ is stabilized by $G_x$, it follows that $\sigma_0$ is well-defined and smooth. Working in local coordinates and making use of partitions of unity, we can extend $\sigma_0$ to a global section $\overline{\sigma_0}\in \Gamma(E)$ with $\overline{\sigma_0}|_{G\cdot x}=\sigma_0$. Applying an averaging map to $\overline{\sigma_0}$, we obtain equivariant section $\sigma=(\overline{\sigma_0})^G\in\Gamma(E)^G$. Since $\sigma_0$ is equivariant, it follows that $\sigma|_{G\cdot x}=\sigma_0$. In particular, we have constructed equivariant section $\sigma$ so that $\sigma(x)=v$.

        \item[(2)] We now assume that $M$ has only one orbit-type, that is, $M=M_{(K)}$ for some compact subgroup $K\leq G$. Suppose $\sigma\in \Gamma(E)^G$ is an equivariant section. Define
        $$
        (\pi_E)_*\sigma:M/G\to \widetilde{E}/G;\quad [x]\mapsto [\sigma(x)]
        $$
        Since $\sigma$ is equivariant, it follows that $(\pi_E)_*$ is well-defined. Also, by construction, $(\pi_E)_*\sigma$ satisfies Equation (\ref{eq: regular quotient diagram}). To see that $(\pi_E)_*\sigma$ is smooth, recall that we showed in the proof of Theorem \ref{thm: special subbundle regular case} that around any point $x\in M$, we can find a $G$-invariant neighbourhood $U\subseteq M$ of $x$ and a $G$-equivariant vector bundle isomorphism
        $$
        \widetilde{E}|_U\cong U\times W,
        $$
        where $G$ acts only on the first factor. In this case, $\sigma|_U$ can be identified with a map 
        $$
        U\to U\times W;\quad u\mapsto (u,f(u))
        $$
        for some smooth $K$-invariant map $f:U\to W$. Then in this trivialization, it's easy to see that $(\pi_E)_*\sigma$ becomes the map
        $$
        U/G\to (U/G)\times W;\quad [u]\mapsto ([u],f(u))
        $$
        which is well-defined and smooth since $f$ is $K$-invariant and smooth. Thus, we have a map
        $$
        (\pi_E)_*:\Gamma(E)^G\to \Gamma(\widetilde{E}/G)
        $$
        which is clearly linear and injective.

        \ 

        To show $(\pi_E)_*$ is surjective, let $\eta\in\Gamma(\widetilde{E}/G)$. Observe that for any $x\in M$ that the map
        $$
        (\pi_E)_x:\widetilde{E}_x\to (\widetilde{E}/G)_{[x]}
        $$
        is an isomorphism. Hence, define
        $$
        \sigma:M\to \widetilde{E};\quad x\mapsto (\pi_E)_x^{-1}(\eta([x]).
        $$
        If $\sigma$ is smooth, its a triviality to see that $\sigma$ is equivariant and $(\pi_E)_*\sigma=\eta$. So to show $\sigma$ is smooth, observe that if we pass to a $G$-invariant trivialization $U\subseteq M$ of $\widetilde{E}$ so that $\widetilde{E}|_U\cong U\times W$ with the action of $G$ being only on the first factor. Then, in this trivialization there exists $K$-invariant $f:U\to W$ so that
        $$
        \eta([u])=([u],f(u))
        $$
        for any $u\in U$. After a diagram chase, we obtain
        $$
        \sigma(u)=(u,f(u))
        $$
        and hence $\sigma$ is smooth.
    \end{itemize}
\end{proof}

\begin{cor}\label{cor: pushforwards of equivariant vector fields, regular}
    Let $G$ be a connected Lie group and $M$ a proper $G$-space with only one orbit type. Write $\pi:M\to M/G$ for the quotient map. Then if $V\in \mfX(M)^G$ is an equivariant vector field, the point-wise pushforward $\pi_*V$ defines a vector field on $M/G$. Furthermore, the map
    $$
    \pi_*:\mfX(M)^G\to \mfX(M/G)
    $$
    is a surjection.
\end{cor}

\begin{proof}
    Let $\mfg_M\subseteq TM$ be the subbundle spanned by the fundamental vector fields and consider now the normal bundle
    $$
    \nu(M,G):=TM/\mfg_M.
    $$
    Since $\mfg_M$ is an equivariant subbundle, it follows that the quotient is too. Furthermore, since $\mfg_M$ is the kernel of the projection $T\pi:TM\to T(M/G)$, it follows that the induced map
    \begin{equation}\label{eq: normal bundle give tangent of quotient}
        T\pi:\nu(M,G)\to \pi^{-1}T(M/G)
    \end{equation}
    is a $G$-equivariant isomorphism of vector bundles. Now I claim that
    $$
    \widetilde{\nu(M,G)}=\nu(M,G).
    $$
    In which case, by Theorem \ref{thm: special subbundle regular case}, $\nu(M,G)/G\cong T(M/G)$ as $G$-equivariant vector bundles. Furthermore, by Proposition \ref{prop: equivariant sections give all sections below} the map in Equation (\ref{eq: normal bundle give tangent of quotient}) induces a bijection
    $$
    \Gamma(\nu(M,G))^G\to \Gamma(T(M/G))=\mfX(M/G)
    $$
    and hence a surjection
    $$
    \mfX(M)^G\to \mfX(M/G).
    $$
    So to show that $\widetilde{\nu(M,G)}=\nu(M,G)$, we just need to show that $\nu_x(M,G)^{G_x}=\nu_x(M,G)$ for all $x\in M$. To that end, fix $x\in M$. Passing to the local normal form in Lemma \ref{lem: local normal form proper vb} and using the fact that $M$ has only one orbit-type, we may assume that
    $$
    M=(G/K)\times V
    $$
    for a compact subgroup $K\leq G$ and $x=(K,0)$. Observe that
    $$
    \nu_x(M,G)=\frac{(\mfg/\mfk)\oplus V}{(\mfg/\mfk)\oplus \{0\}}\cong V
    $$
    as $K$-representations. Since $K$ acts trivially on $V$, the result then follows.
\end{proof}

\cleardoublepage
\chapter{Symplectic Geometry}\label{ch: symp}

Symplectic manifolds are a natural generalization of ``phase space'' coming from the Hamiltonian formalism of classical mechanics. The idea here is that we have a classical particle of mass $m$ subject to a force $F$. If the position is parameterized by coordinates $q_1,\dots,q_n$ and momentum by coordinates $p_1,\dots,p_n$, we then call the space of all possible positions and momenta \textbf{phase space}. In our case, this can be viewed simply as $\bR^n\times \bR^n$, or some suitable subset. Writing $F(q)=(F_1(q),\dots,F_n(q))$ for the force, Newton's equation then tell us that the position of the particle $q(t)$ will solve the differential equation
$$
F=ma=m\frac{d^2 q}{dt}=\frac{dp}{dt},
$$
where the derivatives are taken coordinate-wise. The Hamiltonian formalism is a variant of Newton's equation in the case where the force is conservative, i.e. $F$ is the gradient of some function $U(q)$ called ``the potential energy''. Not only will the particle have potential energy, it will also have kinetic energy, i.e the energy due to motion
$$
T=\frac{\|p\|^2}{2m}=\frac{p_1^2+\cdots+p_n^2}{2m}.
$$
Putting these two together, gives the total energy, also known as the Hamiltonian
$$
H(q,p)=T(p)+U(q).
$$
Note that if $q(t)$ is a solution to Newton's equation, then it's straightforward to see that
$$
\frac{d}{dt}H\bigg(q,m\frac{d}{dt}q\bigg)=0
$$
hence the saying that ''energy is a conserved quantity''. Using this (or more traditionally Hamiltonian functionals), we can then deduce Hamilton's equations. Namely a trajectory $(q(t),p(t))$ satisfies Newton's equations if
$$
\begin{cases}
    \frac{dq_i}{dt}&=\frac{\partial H}{\partial p_i}\\
    \frac{dp_i}{dt}&=-\frac{\partial H}{\partial q_i}
\end{cases}
$$

\ 

To begin to get at symplectic geometry, we first need the Poisson bracket. Given two functions $f$ and $g$ on phase space, we define
\begin{equation}\label{eq: first poisson}
\{f,g\}:=\sum_{i=1}^n \bigg(\frac{\partial f}{\partial p_i}\frac{\partial g}{\partial q_i}-\frac{\partial f}{\partial q_i}\frac{\partial g}{\partial p_i}\bigg).
\end{equation}
This may seem like a strange thing to do, but it doesn't take much massaging of Hamilton's equations so show that if $f$ is a function of phase space and $(q(t),p(t))$ is the trajectory of the particle, then
$$
\frac{d}{dt}f(q(t),p(t))=\{H,f\}(q(t),p(t)).
$$
In particular, given that in classical physics measurable quantities, called observables, are modelled by functions on phase space, we see that the dynamics of observables are determined by the algebraic properties of the Poisson bracket. 

\ 

The insight that leads to symplectic geometry, is that the Poisson bracket in Equation (\ref{eq: first poisson}) can actually be encoded by a $2$-form called a symplectic form. Furthermore, symmetries which preserve Hamilton's equations (called ``canonical transformations'' by the physicists), are then encoded as diffeomorphisms of phase space which preserve the symplectic form.

\ 

Now the actual history of symplectic geometry is quite long and complex, starting in some sense with Hamilton's study of optics as detailed in \cite{guillemin_symplectic_1984}. Thanks to the work of Arnold \cite{arnold_mathematical_1978}, Weinstein \cite{weinstein_symplectic_1971}, Souriau \cite{souriau_structure_1970} and many more, since the 1960's, symplectic geometry has blossomed into one of the major areas of differential geometry far beyond any interpretations arising from physics.

\ 

The interested reader is encouraged to consult any of the sources listed above for an introduction or to \cite{da_silva_lectures_2008} for another excellent introduction to symplectic geometry.

\section{Symplectic Vector Spaces}\label{sec: symplectic vspaces}
We begin our study of symplectic geometry with symplectic vector spaces. As we shall see repeatedly throughout this thesis, many proofs pass to the linear case and so it will be necessary to establish a firm foundation here first.

\begin{defs}[Symplectic Vector Space]
    A \textbf{symplectic vector space} is a pair $(V,\omega)$ where $V$ is a real vector space and $\omega\in\bigwedge^2(V^*)$ is a $2$-form which is non-degenerate. That is, the map
    $$
    \omega^\flat:V\to V^*;\quad v\mapsto \omega(v,\cdot)
    $$
    is a linear isomorphism.
\end{defs}

\begin{egs}
    The standard example of symplectic vector spaces are to be found in $\bR^{2n}=\bR^n\times \bR^n$. Write $e_1,\dots,e_n$ and $f_1,\dots,f_n$ for the standard basis vectors on the first and second copies of $\bR^n$, respectively. Now let $e_1^*,\dots,e_n^*$ and $f_1^*,\dots,f_n^*$ be the associated dual basis. Now define the \textbf{standard symplectic form} on $\bR^{2n}$ by
    $$
    \omega_0=\sum_{j=1}^n e_i^*\wedge f_i^*.
    $$
    It's then easy to check that $\omega_0$ is symplectic. Indeed, after identifying $(\bR^{2n})^*\cong \bR^{2n}$ with the canonical inner-product, the matrix form of $\omega_0^\flat$ has the form
    $$
    \omega^\flat=
    \begin{pmatrix}
        0 & -I_n\\
        I_n & 0
    \end{pmatrix}
    $$
    and hence $\omega_0$ is non-degenerate.
\end{egs}

\begin{defs}[Linear Symplectomorphism]
    Let $(V_1,\omega_1)$ and $(V_2,\omega_2)$ be two symplectic vector spaces. A \textbf{linear symplectomorphism} is a linear isomorphism $\Phi:V_1\to V_2$ so that
    $$
    \Phi^*\omega_2=\omega_1.
    $$
\end{defs}

\begin{theorem}\label{thm: Darboux bases exist}
    Every symplectic vector space $(V,\omega)$ admits a \textbf{Darboux basis}. That is, a basis $e_1,\dots,e_n,f_1,\dots,f_n\in V$ such that
    $$
    \omega=\sum_{i=1}^n e_i^*\wedge f_i^*,
    $$
    where $\{e_1^*,\dots,e_n^*,f_1^*,\dots,f_n^*\}$ is the dual basis in $V^*$.
\end{theorem}

\begin{cor}\label{cor: all symplectic vector spaces are the same}
    Let $(V,\omega)$ be a symplectic vector space.
    \begin{itemize}
        \item[(1)] $V$ is even dimensional.
        \item[(2)] If $\dim(V)=2n$ for some $n$, then there exists a symplectomorphism $\Phi:(\bR^{2n},\omega_0)\to (V,\omega)$, where $\omega_0$ is the standard symplectic form on $\bR^{2n}$.
    \end{itemize}
\end{cor}

\subsection{Symplectic Complements and Distinguished Subspaces}\label{sub: complements}

\begin{defs}[Symplectic Complement]\label{def: symplectic complement}
    Let $(V,\omega)$ be a symplectic vector space and $C\subseteq V$ a linear subspace. Define the \textbf{symplectic complement} of $C$, denoted $C^\omega$, by
    $$
    C^\omega=\{v\in V \ | \ \omega(v,c)=0\text{ for all }c\in C\}.
    $$
\end{defs}

Unlike complements for inner-products, a subspace need not be transverse to its symplectic complement. Indeed, in symplectic linear algebra, we are mostly concerned with subspaces which intersect their symplectic complements in interesting ways. Before we examine these subspaces, lets establish some foundational properties of symplectic complements.

\begin{prop}\label{prop: elementary properties of symplectic complements}
    Let $(V,\omega)$ be a symplectic vector space and $A,B\subseteq V$ two linear subspaces.
    \begin{itemize}
        \item[(1)] If $A\subseteq B$, then $B^\omega\subseteq A^\omega$.
        \item[(2)] $(A^\omega)^\omega=A$.
        \item[(3)] $\dim(A)+\dim(A^\omega)=\dim(V)$.
    \end{itemize}
\end{prop}

\begin{defs}[Distinguished Subspaces of Symplectic Vector Space]\label{def: distinguished subspaces in symplectic vector space}
    Let $(V,\omega)$ be a symplectic vector space and $C\subseteq V$ a linear subspace. We say $C$ is 
    \begin{itemize}
        \item[(1)] \textbf{isotropic} if $C\subseteq C^\omega$.
        \item[(2)] \textbf{coisotropic} if $C^\omega\subseteq C$
        \item[(3)] \textbf{Lagrangian} if $C=C^\omega$.
        \item[(4)] \textbf{symplectic} if $C\cap C^\omega=\{0\}$.
    \end{itemize}
\end{defs}

\begin{egs}
    Consider standard symplectic $(\bR^{2n},\omega_0)$ and let $e_1,\dots,e_n,f_1,\dots,f_n$ be the standard Darboux basis. Let's now look at standard examples of each kind of subspace introduced in Definition \ref{def: distinguished subspaces in symplectic vector space}.
    \begin{itemize}
        \item[(1)] Isotropic.
        $$
        A=\text{span}_\bR\{e_1,\dots,e_k\}
        $$
        for $k\leq n$. A quick computation shows
        $$
        A^{\omega_0}=\text{span}_\bR\{e_1,\dots,e_n,f_{k+1},\dots,f_n\},
        $$
        which in particular contains $A$. Thus $A$ is isotropic. Note that due to (1) in Proposition \ref{prop: elementary properties of symplectic complements} it follows that $A^{\omega_0}$ is coisotropic. 
        \item[(2)] Coisotropic.
        $$
        B=\text{span}_\bR\{e_1,\dots,e_n,f_1,\dots,f_k\}.
        $$
        One can also immediately see that
        $$
        B^{\omega_0}=\text{span}_\bR\{e_{k+1},\dots,e_n\}
        $$
        which is contained in $B$. Hence, $B$ is coisotropic. Similar to the above, we also see that property (1) of Proposition \ref{prop: elementary properties of symplectic complements} implies that $B^{\omega_0}$ is isotropic.
        \item[(3)] Lagrangian.
        $$
        C=\{e_1,\dots,e_n\}.
        $$
        Using the computation in (a), we immediately deduce that $C^{\omega_0}=C$ and hence $C$ is Lagrangian.
        \item[(4)] Symplectic.
        $$
        D=\text{span}_\bR\{e_1,\dots,e_k,f_1,\dots,f_k\}.
        $$
        One can check that
        $$
        D^{\omega_0}=\text{span}_\bR\{e_{k+1},\dots,e_n,f_{k+1},\dots,f_n\}.
        $$
        Note that the restriction of $\omega_0$ to both $D$ and $D^{\omega_0}$ defines a symplectic form on each subspace.
    \end{itemize}
\end{egs}

Let's now summarize some easy properties about symplectic vector spaces that we will be freely using later on.

\begin{prop}\label{prop: symplectic subspaces and complements}
    Let $(V,\omega)$ be a symplectic vector space and $C\subseteq V$ a symplectic subspace.
    \begin{itemize}
        \item[(1)] The restrictions $\omega|_C$ and $\omega|_{C^\omega}$ defines symplectic forms on each subspace.
        \item[(2)] $V$ is the internal direct sum of $C$ and $C^\omega$. That is, $V=C\oplus C^\omega$.
    \end{itemize}
\end{prop}

So we see then that symplectic subspaces are precisely the subspaces where symplectic complements act more like complements associated to an inner-product. In particular, to each symplectic subspace we get a canonical projection map.

\begin{defs}[Symplectic Projection]\label{def: symplectic projection}
    Let $(V,\omega)$ be a symplectic vector space and $C\subseteq V$ a symplectic subspace. Let $P_C:V\to C$ be the canonical projection induced by the internal direct sum $V=C\oplus C^\omega$. 
\end{defs}

\subsection{Linear Symplectic Reduction}\label{sub: linear reduction}

Now to close up our introductory discussion on symplectic linear algebra, let us now discuss linear symplectic reduction.

\begin{prop}[Linear Symplectic Reduction]\label{prop: linear reduction}
Let $C\subseteq V$ be a subspace. Then 
$$
V_0:=\frac{C}{C\cap C^\omega}
$$
is a symplectic vector space with a unique symplectic form $\omega_0$ satisfying
$$
\pi^*\omega_0=\omega|_C,
$$
where $\pi:C\to V_0$ is the canonical projection.
\end{prop}

\begin{proof}
Define $\omega_0$ on $V_0$ by
$$
\omega_0([c],[c']):=\omega(c,c'),
$$
where $c,c'\in C$ and $[c],[c']\in V_0$ are their corresponding equivalence classes. First we show that $\omega_0$ is well-defined, then we show $\omega_0$ is non-degenerate.

\begin{itemize}
    \item[(i)] $\omega_0$ is well-defined.

    \ 

    Indeed, if $c,d,c',d'\in C$ such that $c-d, \ c'-d'\in C\cap C^\omega$, then
    $$
    \omega(d,d')=\omega(d+(c-d),d'+(c'-d'))=\omega(c,c').
    $$
    Hence, $\omega_0([d],[d'])=\omega_0([c],[c'])$ as desired.

    \item[(ii)] $\omega_0$ is non-degenerate.

    \ 

    Suppose $c\in C$ so that $[c]\in \ker(\omega_0^\flat)$. That is, for all $d\in C$ we have
    $$
    \omega_0([c],[d])=0.
    $$
    By definition of $\omega_0$, this implies $\omega(c,d)=0$ for all $d\in C$. In particular, this implies $c\in C^\omega$ and hence $c\in C\cap C^\omega$. Therefore, $[c]=0$ and and so $\omega_0^\flat$ is injective hence a linear isomorphism.
\end{itemize}
\end{proof}

\begin{defs}[Linear Reduced Space]\label{def: linear reduction}
Given a subspace $C\subseteq V$, call $V_0=C/(C\cap C^\omega)$ the \textbf{linear reduction} of $V$ at $C$ and call $\pi_0:C\to V_0$ the \textbf{reduction map.}
\end{defs}

In the next chapter, we will discuss what other structures can be ``reduced'' in a symplectic vector space. In particular, we will study when Lagrangian subspaces can be reduced.

\section{Symplectic and Poisson Manifolds}\label{sec: symp and poisson}

We are now ready to define symplectic manifolds. As is quite unorthodox, we will pretty quickly move into even greater generality of Poisson manifolds. This will be done as many features of symplectic manifolds which may appear somewhat mysterious become very natural in the more general Poisson picture. Furthermore, in the singular setting, the Poisson bracket is the more natural global object and so it is worthwhile spending time on both concepts at the same time.

\begin{defs}[Symplectic Manifold]\label{def: symplectic}
    A symplectic manifold is a pair $(M,\omega)$ where $M$ is a manifold and $\omega\in \Omega^2(M)$ is a $2$-form satisfying the following two conditions.
    \begin{itemize}
        \item[(1)] (Closed) $d\omega=0$.
        \item[(2)] (Non-Degenerate) The map 
        $$
        \omega^\flat:TM\to T^*M;\quad v\mapsto \omega(v,\cdot)
        $$
        is an isomorphism of vector bundles.
    \end{itemize}
    If $\omega$ satisfies the above two conditions, we call $\omega$ a \textbf{symplectic form}.
\end{defs}

\begin{egs}
    Of course any symplectic vector space $(V,\omega)$ will be a symplectic manifold. Turning now to standard symplectic $(\bR^{2n},\omega_0)$, if we write $\bR^{2n}=\bR^n\times \bR^n$ and write $x_1,\dots,x_n$ and $y_1,\dots,y_n$ for the coordinates of the first and second copies of $\bR^n$, respectively, then
    $$
    \omega_0=\sum_{i=1}^n dx_i\wedge dy_i
    $$
    which is transparently closed. A simple application of Corollary \ref{cor: all symplectic vector spaces are the same} then immediately shows all symplectic vector spaces are symplectic manifolds.
\end{egs}

\begin{egs}\label{eg: cotangent bundle is symplectic}
    As a generalization of the above, consider now cotangent bundles. Fix a manifold $Q$ and let $\tau_Q:T^*Q\to Q$ be the cotangent bundle of $Q$. Define the \textbf{Liouville $1$-form} $\theta_Q\in \Omega^1(T^*Q)$ by
    \begin{equation}
    (\theta_Q)_\alpha(v)=\bra \alpha, (\tau_Q)_*v\ket
    \end{equation}
    for $\alpha\in T^*Q$ and $v\in T_\alpha(T^*Q)$. Choosing local coordinates $(q_1,\dots,q_n)$ on $Q$ and letting $(p^1,\dots,p^n)$ be the induced cotangent coordinates, it immediately follows that in these coordinates
    \begin{equation}
    \theta_Q =\sum_{i=1}^n p_i dq_i
    \end{equation}
    and hence in these coordinates
    \begin{equation}
    d\theta_Q=-\sum_{i=1}^n dq_i\wedge dp_i
    \end{equation}
    Therefore, 
    \begin{equation}\label{cotangent symplectic form}
    \omega_Q=-d\theta_Q\in\Omega^2(T^*Q)
    \end{equation}
    is a symplectic form.
\end{egs}

Notice how both the symplectic vector space and the cotangent bundle examples were able to be written in the same local form. This is a more general phenomenon of symplectic manifolds.

\begin{theorem}\label{thm: darboux coords for symplectic manifolds}
    Let $(M,\omega)$ be a symplectic manifold with $\dim(M)=2n$ and $x\in M$. Then we can find a neighbourhood $U\subseteq M$ of $x$, a neighbourhood $V\subseteq \bR^{2n}$ of $0$, and a diffeomorphism
    $$
    \phi:U\to V
    $$
    such that $\phi^*\omega_0=\omega|_U$, where $\omega_0$ is the standard symplectic form.
\end{theorem}

In particular, there are no local invariants of symplectic manifolds. This makes symplectic structures a much more rigid structure than say Riemannian structures. Before we continue on with more basic definitions of symplectic geometry, let us now discuss another family of examples of non-trivial symplectic manifolds, namely coadjoint orbits.

\begin{egs}\label{eg: KKS symplectic structure}
    Let $G$ be a connected Lie group with Lie algebra $\mfg$, and let $\mfg^*$ denote its dual. Recall the coadjoint action introduced in Example \ref{eg: adjoint and coadjoint actions}
    $$
    G\times \mfg^*\to \mfg^*;\quad (g,q)\mapsto \text{Ad}_g^*(q).
    $$
    We will now discuss a canonical symplectic structure on the coadjoint orbits, called the Kostant-Kirillov-Souriau symplectic structure. Let us now fix $p\in \mfg^*$ and write 
    $$
    \text{Ad}_G^*(p):=\{\text{Ad}_g^*(p) \ | \ g\in G\}
    $$
    for the coadjoint orbit through $p$. Equip $\text{Ad}_G^*(p)$ with a smooth manifold structure via the bijection
    $$
    G/G_p\to \text{Ad}_G^*(p).
    $$
    Since $\text{Ad}_G^*(p)$ is an orbit of a group action, the its tangent spaces are spanned by fundamental vector fields. In particular, at $p$, define
    \begin{equation}
        \omega_{KKS}(\xi_{\mfg^*}(p),\eta_{\mfg^*}(p)):=\bra p,[\xi,\eta]\ket.
    \end{equation}
    $\omega_{KKS}$ can then be extended by the $G$-action to all of $\text{Ad}_G^*(p)$ to define a $2$-form $\omega_{KKS}\in\Omega^2(\text{Ad}_G^*(p))$. I claim now that $\omega_{KKS}$ is a symplectic form.

    \ 

    Let us assume that $G_p\neq G$ otherwise $\text{Ad}_G^*(p)$ is just a point and we will agree that points are trivial symplectic manifolds. Now, to show $\omega_{KKS}$ is closed, first identify $T_q^*\mfg^*=\mfg$. In this case, we can then easily show that if $\xi\in \mfg$ and $q\in \mfg^*$, then the fundamental vector field $\xi_{\mfg^*}(q)$ at $q$ has the form
    \begin{equation}\label{eq: fundamental vector field on g*}
        \xi_{\mfg^*}(q)=-\text{ad}_\xi^*(q),
    \end{equation}
    where $\text{ad}_\xi^*:\mfg^*\to\mfg^*$ is defined by
    $$
    \bra \text{ad}_\xi^*(q),\eta\ket=-\bra q,[\xi,\eta]\ket
    $$
    Making use of the identity
    $$
    [\text{ad}_{\xi}^*,\text{ad}_{\eta}^*]=\text{ad}_{[\xi,\eta]}^*
    $$
    then implies $d\omega_{KKS}=0$. Non-degeneracy is a consequence of the fact that $G_p\neq G$. 
\end{egs}

Given now a symplectic manifold $(M,\omega)$, by definition each tangent space $T_xM$ is a symplectic vector space. Using this, we can now ``globalize'' the linear definitions of distinguished subspaces of a symplectic vector space to submanifolds of a symplectic manifold.

\begin{defs}[Isotropic/Coisotropic/Lagrangian/Symplectic Submanifolds]\label{def: distinguished submanifolds of symplectic manifold}
    Let $(M,\omega)$ be a symplectic manifold and $C\subseteq M$ a submanifold. For each property $P\in \{\text{isotropic}, \ \text{coisotropic}, \ \text{Lagrangian}, \ \text{symplectic}\}$, say $C$ is a $P$ submanifold if $T_xC\subseteq T_xM$ is a $P$ subspace for each $x\in C$.
\end{defs}

Now symplectic manifolds have a distinguished class of vector fields called Hamiltonian vector fields. These are obtained by the identification of vector fields and $1$-forms induced by the isomorphism $\omega^\flat:TM\to T^*M$. In particular, the Hamiltonian vector fields correspond to the exact $1$-forms. However, before we discuss this any further, it will be most natural for when we get into the singular setting to introduce a more general structure than symplectic forms. Namely, Poisson structures.

\begin{defs}[Poisson Structure]
    A \textbf{Poisson manifold} is a pair $(M,\pi)$, where $M$ is a manifold and $\pi\in\bigwedge^2(TM)$ is a bivector so that the pairing
    $$
    \{\cdot,\cdot\}:C^\infty(M)\times C^\infty(M)\to C^\infty(M);\quad (f,g)\mapsto \{f,g\}:=\pi(df,dg)
    $$
    satisfies the Jacobi identity. That is, for all $f,g,h\in C^\infty(M)$, we have
    $$
    \{f,\{g,h\}\}+\{g,\{h,f\}\}+\{h,\{f,g\}\}=0.
    $$
    We call $\{\cdot,\cdot\}$ a \text{Poisson bracket}.
\end{defs}

\begin{remark}
    Given that the Poisson brackets $\{\cdot,\cdot\}$ arise from biderivations, it follows that they are bilinear and antisymmetric. In particular, this implies a Poisson bracket is a Lie bracket.
\end{remark}

\begin{defs}[Poisson Map]
    Let $(M_1,\pi_1)$ and $(M_2,\pi_2)$ be two Poisson manifolds. A smooth map $\phi:M_1\to M_2$ is said to be \textbf{Poisson} if for all $f,g\in C^\infty(M_2)$ we have
    $$
    \{\phi^*f,\phi^*g\}_1=\phi^*\{f,g\}_2..
    $$
\end{defs}

Now, if we are given a Poisson manifold $(M,\pi)$, there is a very natural way to obtain vector fields from functions. Given a function $f\in C^\infty(M)$, observe that since $\pi$ is a bivector, the map
$$
\{f,\cdot\}:C^\infty(M)\to C^\infty(M)
$$
is a derivation. As we will be exploring in great detail in Chapter \ref{ch: subcartesian}, derivations of the algebra of functions is one way of characterizing vector fields. This then allows us to give the following definition.

\begin{defs}[Hamiltonian Vector Field]\label{def: hamiltonian vf}
    Let $(M,\pi)$ be a Poisson manifold and $f\in C^\infty(M)$ a function. Define the\textbf{ Hamiltonian vector field} of $f$, denote $V_f\in\mfX(M)$, by
    $$
    V_f(g)=\{f,g\}
    $$
    for all $g\in C^\infty(M)$.
\end{defs}

\begin{prop}\label{prop: Hamiltonian vfs compatible with Lie bracket}
    Let $(M,\pi)$ be a Poisson manifold. Then the map
    $$
    C^\infty(M)\to \mfX(M);\quad f\mapsto V_f
    $$
    is a Lie algebra homomorphism. 
\end{prop}

\begin{proof}
    This is a consequence of the Jacobi identity. Given $f,g\in C^\infty(M)$ and a test function $h\in C^\infty(M)$, we have
    \begin{align*}
    [V_f,V_g]h&=V_f(V_gh)-V_g(V_fh)\\
    &=\{f,\{g,h\}\}-\{g,\{f,h\}\}\\
    &=\{f,\{g,h\}\}+\{g,\{h,f\}\}\\
    &=-\{h,\{f,g\}\}\\
    &=\{\{f,g\},h\}\\
    &=V_{\{f,g\}}h,
    \end{align*}
    where on the third and fifth lines we used antisymmetry of the Poisson bracket and on line four we used the Jacobi identity. 
\end{proof}

Returning back to the symplectic case, suppose we have a symplectic manifold $(M,\omega)$. By definition, $\omega$ defines an isomorphism of vector bundles
$$
\omega^\flat:TM\to T^*M
$$
Write $\omega^\#=(\omega^\flat)^{-1}$. Thus, given two functions $f,g\in C^\infty(M)$, we can define
\begin{equation}\label{eq: Poisson bracket of symplectic manifold}
\{f,g\}:=\omega^\#(df)(g)
\end{equation}
It's then easy to see that since $\omega$ is closed, that $\{\cdot,\cdot\}$ defines a Lie bracket. We can summarize our discussion as follows.

\begin{prop}\label{prop: poisson bracket on symplectic manifold}
    Let $(M,\omega)$ be a symplectic manifold.
    \begin{itemize}
        \item[(1)] The pairing
        $$
        \{\cdot,\cdot\}:C^\infty(M)\times C^\infty(M)\to C^\infty(M);\quad (f,g)\mapsto \{f,g\}
        $$
        defined by Equation (\ref{eq: Poisson bracket of symplectic manifold}) defines a Poisson structure on $M$.
        \item[(2)] Given any function $f\in C^\infty(M)$, the Hamiltonian vector field $V_f\in\mfX(M)$ defined by Definition \ref{def: hamiltonian vf} is the unique vector field satisfying
        $$
        df=\omega(V_f,\cdot).
        $$
        \item[(3)] For any two functions $f,g\in C^\infty(M)$, we have
        $$
        \omega(V_f,V_g)=-\{f,g\}.
        $$
    \end{itemize}
\end{prop}

\begin{egs}\label{eg: local form for symplectic Poisson structures}
    Recall that thanks to Theorem \ref{thm: darboux coords for symplectic manifolds} all symplectic manifolds are locally symplectomorphic to standard symplectic $\bR^{2n}$ for some $n$. Thus, to understand what Hamiltonian vector fields look like locally, it suffices to look at standard symplectic $\bR^{2n}$. In which case, fix standard coordinates $x_1,\dots,x_n,y_1,\dots,y_n$ so that
    $$
    \omega_0=\sum_{i=1}^n dx_i\wedge dy_i.
    $$
    If we let $f\in C^\infty(\bR^{2n})$ be a function, then
    $$
    df=\sum_{i=1}^n \bigg(\frac{\partial f}{\partial x_i}dx_i +\frac{\partial f}{\partial y_i}dy_i\bigg).
    $$
    It's easy to see that
    \begin{align*}
        \omega_0^\flat\bigg(\frac{\partial}{\partial x_i}\bigg)&=dy_i\\
        \omega_0^\flat\bigg(\frac{\partial}{\partial y_i}\bigg)&-dx_i\\
    \end{align*}
    Hence, the unique vector field $V_f\in \mfX(\bR^{2n})$ satisfying
    $$
    df=\omega^\flat(V_f)
    $$
    will have the form
    $$
    V_f=\sum_{i=1}^n\bigg(\frac{\partial f}{\partial y_i}\frac{\partial }{\partial x_i} -\frac{\partial f}{\partial x_i}\frac{\partial }{\partial y_i}\bigg)
    $$
    This also allows us to easily determine the local form of the Poisson bracket on a symplectic manifold. Given functions $f,g\in C^\infty(M)$, the standard Poisson bracket on $\bR^{2n}$ has the form
    \begin{align*}
        \{f,g\}_0:=V_f(g)=\sum_{i=1}^n\bigg(\frac{\partial f}{\partial y_i}\frac{\partial g}{\partial x_i} -\frac{\partial f}{\partial x_i}\frac{\partial g}{\partial y_i}\bigg)
    \end{align*}
    In terms of a bi-vector $\pi_0\in\Gamma(\bigwedge^2 TM)$, we can easily deduce
    $$
    \pi_0=-\sum_{i=1}^n \frac{\partial}{\partial x_i}\wedge \frac{\partial}{\partial y_i}.
    $$
\end{egs}

\begin{egs}\label{eg: not symplectic Poisson}
    Not all Poisson manifolds are symplectic. For instance, the bivector $\pi\in \Gamma(\bigwedge^2 T\bR^2)$ given by
    $$
    \pi=r^2\frac{\partial}{\partial x}\wedge \frac{\partial }{\partial y},
    $$
    where $r=\sqrt{x^2+y^2}$. It's easily verified that $\pi$ is a Poisson bivector, however it cannot come from a symplectic form $\pi=0$ at $(x,y)=(0,0)$ which clearly cannot happen given the local form a symplectic Poisson bivector given in Example \ref{eg: local form for symplectic Poisson structures}. However, provided we agree that points are trivial symplectic manifolds, we do have a canonical partition into submanifolds
    $$
    \bR^2 = \{0\}\bigsqcup (\bR^2\setminus\{0\})
    $$
    where the restriction of the Poisson bivector is symplectic. For instance, writing $S=\bR^2\setminus\{0\}$, it's easy to see that $\pi$ is tangent to $S$ and that the induced map
    $$
    \pi^\#:T^*S\to TS;\quad \alpha\mapsto \pi(\alpha,\cdot)
    $$
    is an isomorphism of vector bundles. Writing $\pi^\flat$ for the inverse, map it's easy to see that the induced $2$-form $\omega_S\in\Omega^2(S)$ given by
    $$
    \omega_S(v,w):=\pi(\pi^\flat(v),\pi^\flat(w))
    $$
    is symplectic. In fact, in coordinates
    $$
    \omega_S=-\frac{dx\wedge dy}{r^2}.
    $$
\end{egs}

\begin{defs}[Symplectic Submanifold of Poisson Structure]
    Let $(M,\pi)$ be a Poisson manifold. A submanifold $S\subseteq M$ is said to be \textbf{symplectic} if
    \begin{itemize}
        \item[(1)] $\pi$ is tangent to $S$, i.e. $\pi|_S\in\Gamma(\bigwedge^2TS)$
        \item[(2)] The restriction of the map
        $$
        \pi^\#:T^*M\to TM;\quad \alpha\mapsto \pi(\alpha,\cdot)
        $$
        to $T^*S$ is defines an isomorphism of vector bundles
        $$
        \pi^\#|_{T^*S}:T^*S\to TS
        $$
    \end{itemize}
\end{defs}

This property that a Poisson manifold decomposes into a disjoint union of symplectic submanifolds is a general feature of Poisson structures.

\begin{theorem}[{\cite{weinstein_local_1983}}]\label{thm: existence of symplectic leaves}
    Let $(M,\pi)$ be a Poisson manifold. Then for any point $x\in M$, there exists unique connected symplectic submanifold $S\subseteq M$ containing $x$ with the following properties.
    \begin{itemize}
        \item[(1)] If $R\subseteq M$ is another connected symplectic submanifold containing $x$, then $R\subseteq S$ as an open subset.
        \item[(2)] $T_xS$ is spanned by Hamiltonian vector fields $V_f(x)$ for $f\in C^\infty(M)$.
    \end{itemize}
    We call $S$ the \textbf{symplectic leaf containing $x$}.
\end{theorem}

\begin{remark}
    Of course if $(M,\omega)$ is a symplectic manifold, then the symplectic leaves are simply the connected components of $M$.
\end{remark}

\begin{egs}
    As an interesting and non-trivial example, let's return to coadjoint orbits of a connected Lie group. Fix a connected Lie group $G$ with Lie algebra $\mfg$ and dual $\mfg^*$. There is a canonical Poisson structure on $\mfg^*$ called the Kostant-Kirillov-Souriau (KKS) Poisson structure, denoted $\pi_{KKS}$. To define $\pi_{KKS}$, identify $T^*\mfg^*=\mfg^*\times \mfg$. Using this identification, we can define a bracket on $\Omega^1(\mfg^*)$ given by
    $$
    [\alpha,\beta](p):=[\alpha(p),\beta(p)]_{\mfg},\quad \alpha,\beta\in\Omega^1(\mfg*), \ p\in \mfg^*
    $$
    where $[\cdot,\cdot]_{\mfg}$ is the Lie bracket on $\mfg$. Using this, we now define
    \begin{equation}
    \pi_{KKS}(\alpha,\beta)(p):=\bra p,[\alpha(p),\beta(p)]\ket,\quad \alpha,\beta\in\Omega^1(\mfg^*), \ p\in \mfg^*.
    \end{equation}
    The fact that $\pi_{KKS}$ satisfies the axioms of a Lie bracket on functions is a consequence of the fact that $[\cdot,\cdot]_{\mfg}$ is a Lie bracket. Now, let us see why the coadjoint orbits are the symplectic leaves for this Poisson structure. 
    
    \ 
    
    A choice of a basis $\xi_1,\cdots,\xi_n\in\mfg$ defines a coordinate system on $\mfg^*$. Using this and Equation (\ref{eq: fundamental vector field on g*}), we have for any $f\in C^\infty(\mfg^*)$ and $p\in \mfg^*$ that
    $$
    V_f(p)=\sum_{i=1}^n\frac{\partial f}{\partial \xi_i}\text{ad}_{\xi_i}^*(p).
    $$
    This allows us to conclude that (1) the Hamiltonian vector fields are tangent to the coadjoint orbits and (2) the set of Hamiltonian vector fields is spanned by the fundamental vector fields. Putting these two facts together, we deduce that the coadjoint orbits are precisely the symplectic leaves of $\pi_{KKS}$.
\end{egs}

\begin{egs}
    More generally, give a Lie algebra $(\mfg,[\cdot,\cdot])$ we can define a \textbf{linear Poisson structure} on its dual $\mfg^*$ given as follows. For $f,g\in C^\infty(\mfg^*)$, we identify $df,dg:\mfg^*\to \mfg$. We then define
    $$
    \pi_x(d_xf,d_yg):=\bra x, [d_xf,d_xg]\ket
    $$
    This endows $\mfg^*$ with a Poisson structure called a \textbf{linear Poisson structure}.
\end{egs}

\section{Hamiltonian Spaces and Reduction}\label{sec: ham space}

We now begin to combine the equivariant geometry of Chapter \ref{ch: equivariant} with the symplectic geometry developed thus far. This leads to Hamiltonian spaces. These can be motivated either from a purely symplectic point of view or from a physics point of view via Noether's Theorem. For sake of space, I will confine myself to the symplectic picture. The idea is to suppose we have a symplectic manifold $(M,\omega)$ and a Lie group $G$ acting on $M$ symplectically. That is, for each $g\in G$, we have
$$
g^*\omega=\omega.
$$
Then, as we saw in Chapter \ref{ch: equivariant}, each Lie algebra element $\xi\in \mfg=\text{Lie}(G)$ induces a vector field $\xi_M\in \mfX(M)$. Since the action of $G$ preserves the symplectic form, it follows that the flow of $\xi_M$ preserves $\omega$. That is,
$$
L_{\xi_M}\omega=0,
$$
where $L$ denotes the Lie derivative operator. Another class of vector field which preserves the symplectic form are the Hamiltonian vector fields. Thus, the most natural kind of action is one where the fundamental vector fields are also Hamiltonian, i.e. for each $\xi\in \mfg$ there exists a function $J^\xi\in C^\infty(M)$ so that $V_{J^\xi}=\xi_M$. Combining these functions together and imposing further conditions leads to the following definition.

\begin{defs}[Hamiltonian $G$-space]\label{def: Hamiltonian G-space}
A \textbf{Hamiltonian $G$-space} consists of a symplectic manifold $(M,\omega)$ and a map $J:M\to \mfg$, called the \textbf{momentum map}, such that the following holds.
\begin{itemize}
    \item[(1)] $M$ is a $G$-space and the $G$-action preserves $\omega$.
    \item[(2)] $J$ is $G$-equivariant with respect to the coadjoint action of $G$ on $\mfg^*$.
    \item[(3)] (Hamiltonian Condition) If $\xi\in \mfg$ and we write $\xi_M\in\mfX(M)$ for the induced fundamental vector field and
    $$
    J^\xi:M\to \bR;\quad x\mapsto \bra J(x),\xi\ket,
    $$
    then
    $$
    dJ^\xi-\omega(\xi_M,\cdot)=0.
    $$
\end{itemize}
We will write $J:(M,\omega)\to \mfg^*$ for a Hamiltonian $G$-space. If $G$ acts properly on $M$, say $J:(M,\omega)\to \mfg^*$ is a \textbf{proper Hamiltonian $G$-space}.
\end{defs}

\begin{defs}[Isomorphism of Hamiltonian $G$-Spaces]\label{def: iso of ham G-spaces}
    Let $J_1:(M_1,\omega_1)\to \mfg^*$ and $J_2:(M_2,\omega_2)\to \mfg^*$ be two Hamiltonian $G$-spaces. An \textbf{isomorphism} between them is a $G$-equivariant symplectomorphism $\Phi:M_1\to M_2$ such that the diagram commutes
    $$
    \begin{tikzcd}
        M_1\arrow[rr,"\Phi"]\arrow[dr,"J_1"] && M_2\arrow[dl,"J_2"]\\
        &\mfg^* &
    \end{tikzcd}
    $$
\end{defs}

\begin{egs}
    Let $G$ be a connected Lie group and $p\in \mfg^*$. Then the inclusion 
    $$
    \iota:\text{Ad}_G^*(p)\into \mfg^*
    $$
    of the coadjoint orbit through $p$ is a momentum map with respect to the KKS symplectic structure $\omega_{KKS}\in\Omega^2(\text{Ad}_G^*(p))$ from Example \ref{eg: KKS symplectic structure}. By definition, $\omega_{KKS}$ is preserved by the $G$-action on $\text{Ad}_G^*(p)$, so we only have to show $\iota$ is momentum map. Fixing $\xi\in \mfg$, recall that
    $$
    \xi_{\mfg^*}(q)=-\text{ad}_\xi^*(q).
    $$
    In particular, $\xi_{\mfg^*}$ is linear in $q$. Hence, 
\end{egs}

\begin{egs}\label{eg: cotangent bundle hamiltonian G-space}
Let $G$ be a Lie group, $\mfg=\text{Lie}(G)$ its Lie algebra, and $Q$ a proper $G$-space. The cotangent bundle $T^*Q$ then has a canonical Hamiltonian $G$-space structure
$$
J:(T^*Q,\omega_{can})\to \mfg^*
$$
which we describe now. Write $\tau_Q:T^*Q\to Q$ for the bundle map, $\theta_Q\in\Omega^1(T^*Q)$ for the Liouville $1$-form and $\omega=-d\theta_Q$ for the canonical symplectic form as in Example \ref{eg: cotangent bundle is symplectic}.

\ 

First, recall from Example \ref{eg: tangent and cotangent equivariant vb structures} that $T^*Q$ has a canonical proper $G$-equivariant vector bundle structure given by
$$
G\times T^*Q\to T^*Q;\quad (g,\alpha)\mapsto g\cdot \alpha
$$
where for $\alpha\in T_q^*Q$ and $v\in T_{g\cdot q}Q$
$$
\bra g\cdot \alpha,v\ket=\bra \alpha,T_qg^{-1}(v)\ket.
$$
This action preserves the Lioville $1$-form. Indeed, given $q\in Q$, $\alpha\in T_q^*Q$,  $v\in T_\alpha(T^*Q)$, and $g\in G$, we have
\begin{align*}
(g^*\theta)_{\alpha}(v)&=\theta_{g\cdot \alpha}(T_\alpha g(v))\\
&=\bra g\cdot \alpha, T_{g\cdot \alpha}\tau\circ T_{\alpha}g(v)\ket\\
&=\bra g\cdot \alpha, T_qg\circ T_\alpha\tau(v)\ket\\
&=\bra \alpha, T_\alpha\tau(v)\ket\\
&=\theta_\alpha(v).
\end{align*}
Thus, the action by $G$ also preserves the symplectic form $\omega=-d\theta$. 

\ 

Now to describe the momentum map. Fix $\alpha\in T^*Q$ and define $J(\alpha)\in\mfg^*$ by
\begin{equation}
\bra J(\alpha),\xi\ket=\bra \alpha,\xi_Q(\tau(\alpha))\ket,
\end{equation}
where $\xi\in \mfg$ and  $\xi_Q\in\mfX(Q)$ is the associated fundamental vector field on $Q$. To see that $J$ is a momentum map, fix $\xi\in\mfg$. By definition of the Liouville $1$-form, we have
$$
J^\xi=\theta(\xi_{T^*Q})
$$
where $\xi_{T^*Q}\in\mfX(T^*Q)$ is the fundamental vector field associated to $\xi$. Note that by construction, $\theta$ is an invariant $1$-form and thus,
$$
L_{\xi_{T^*Q}}\theta=0
$$
where $L_{\xi_{T^*Q}}$ is the Lie derivative. By Cartan's magic formula (see for instance \cite[Theorem 14.35]{lee_introduction_2013})
$$
0=d\theta(\xi_{T^*Q},\cdot)+d(\theta(\xi_{T^*Q})
$$
and hence
$$
dJ^\xi=d(\theta(\xi_{T^*Q}))=\omega(\xi_{T^*Q},\cdot).
$$
Thus, $J$ is a momentum map.
\end{egs}

As we will be studying cotangent bundles in great detail, let's give a name to the above construction.

\begin{defs}[Hamiltonian Lift]\label{def: hamiltonian lift}
    Let $G$ be a Lie group and $Q$ a proper $G$-space. Call the proper Hamiltonian $G$-space
    $$
    J:(T^*Q,\omega)\to \mfg^*
    $$
    constructed in the above example the \textbf{Hamiltonian lift} of the action on $Q$.
\end{defs}

\section{Linear Symplectic Representations}\label{sec: linear symp reps}
To understand the local structure of Hamiltonian spaces, it will be necessary to move to describe the infinitesimal structure associated to symplectic representations.

\begin{defs}[Symplectic Representation]
    Let $G$ be a Lie group. A \textbf{symplectic representation} is a symplectic vector space $(V,\omega)$ together with a $G$-representation structure on $V$ which preserves the symplectic form.
\end{defs}

We will be making good use of a modified version of the Slice Theorem (i.e Theorem \ref{thm: slice theorem}) for proper Hamiltonian spaces, and hence our interest will primarily lie with symplectic representations of compact Lie groups. Now let us establish a foundational property of symplectic representations that will be quite prominent when we discuss symplectic reduction.

\begin{prop}\label{prop: fixed point set symplectic rep}
    Let $K$ be a compact Lie group and $(V,\omega)$ a symplectic $K$-representation. Then the fixed point set $V^K$ is a symplectic subspace.
\end{prop}

\begin{proof}
    Since the action by $K$ on $V$ is linear, it immediately follows that $V^K$ is a linear subspace. By assumption, $\omega$ is $K$-invariant. In particular, this means it defines an isomorphism of $K$-representations
    $$
    \omega^{\#}:V\to V^*;\quad v\mapsto \omega(v,\cdot)
    $$
    and hence a linear isomorphism
    $$
    \omega^{\#}:V^K\to (V^*)^K.
    $$
    Using averaging, we can easily show that the natural map
    $$
    (V^*)^K\to (V^K)^*
    $$
    induced by restricting the dual of the inclusion $V^K\into V$ is a linear isomorphism. In particular, if $v\in V$ we can then find $\lambda\in (V^*)^K$ so that $\lambda(v)\neq 0$. Using the isomorphism $\omega^{\#}$, $\lambda$ corresponds to some $w\in V^K$. By definition this then means that
    $$
    \omega(v,w)=-\lambda(v)\neq 0.
    $$
    Thus, $V^K$ is symplectic.
\end{proof}

Continuing on with the theme of reduction, let's now apply our results about linear reduction to invariant subspaces.

\begin{prop}\label{prop: equivariant linear reduction}
    Let $K$ be a compact Lie group, $(V,\omega)$ a finite dimensional symplectic representations, and $C\subseteq V$ a sub $K$-representation. Then the following holds.
    \begin{itemize}
        \item[(1)] The symplectic complement $C^\omega$ is a sub $K$-representation.
        \item[(2)] The linear reduction $V_0=C/(C\cap C^\omega)$ has a canonical symplectic representation structure.
    \end{itemize}
\end{prop}

\begin{proof}
    \begin{itemize}
        \item[(1)] This follows from the fact that $K$ preserves the symplectic form. If $v\in C^\omega$ and $k\in K$, then for any $c\in C$ we have
        $$
        \omega(k\cdot v,c)=\omega(v,k\cdot c)=0
        $$
        since $C$ is closed under the action of $K$. Hence, $k\cdot v\in C^\omega$ and so $C^\omega$ is a sub-representation.

        \item[(2)] Since both $C$ and $C^\omega$ are sub-representations, so is the intersection $C\cap C^\omega$. And thus, we have a canonical action on the reduced space $V_0$ given by
        $$
        K\times V_0\to V_0;\quad (k,[c])\mapsto [k\cdot c].
        $$
        To see that this action preserves the reduced symplectic form $\omega_0$, let $\pi_0:C\to V_0$ be the quotient map. Clearly $\pi_0$ is equivariant, thus if $k\in K$, then
        $$
        \pi_0^*k^*\omega_0=k^*\pi_0^*\omega_0=k^*(\omega|_C)=\omega|_C
        $$
        Thus, by definition of $\omega_0$ we have $k^*\omega_0=\omega_0$. Hence $(V_0,\omega_0)$ is a symplectic $K$-representation.
    \end{itemize}
\end{proof}

\begin{cor}
    Let $K$ be a compact Lie group and $(V,\omega)$ a symplectic $K$-representation. Then the symplectic complement $(V^K)^\omega$ is a symplectic sub-representation of $V$.
\end{cor}

\begin{proof}
    Since $V^K$ is symplectic, then by Proposition \ref{prop: equivariant linear reduction}, so is the complement $(V^K)^\omega$. Furthermore, by item (1) of Proposition \ref{prop: equivariant linear reduction} $(V^K)^\omega$ is a sub-representation. 
\end{proof}

 As it turns out, symplectic representations have a canonical momentum map called the \textbf{quadratic momentum map}. To set this up, suppose $K$ is a compact Lie group and $(V,\omega)$ a symplectic representation. Since $V$ is a vector space, we have a canonical identification $TV=V\times V$. Hence, if $\xi\in \mfk$ is a Lie algebra element and $\xi_V\in \mfX(V)$ the associated fundamental vector field, then for each $v\in V$ we can identify $\xi_V(v)\in V$. This now allows us to show the following.

\begin{lem}\label{lem: quadratic momentum map}
    Let $K$ be a compact Lie group and $(V,\omega)$ a symplectic representation. Then the map $J_V:V\to \mfk^*$ defined by
    \begin{equation}
        \bra J_V(v),\xi\ket=\frac{1}{2}\omega(\xi_V(v),v),\quad \xi\in \mfk, \ v\in V
    \end{equation}
    is a momentum map for the $K$-action. We call $J$ the \textbf{quadratic momentum map}
\end{lem}

\begin{proof}
    To see that this is a momentum map, let $\xi\in \mfk$ and $v\in V$. We show
$$
dJ_V^\xi+\omega(\xi_V,\cdot)=0.
$$
Indeed, fix $v\in V$ and identify $w\in V$ with an element of $T_vV$ by the action
$$
w(f)=\frac{d}{dt}\bigg|_{t=0}f(v+tw),
$$
where $f\in C^\infty(V)$. Then on one hand we have
$$
d_vJ_V^\xi(w)=\frac{d}{dt}\bigg|_{t=0}J_V^\xi(v+tw)=\frac{\omega(\xi_V(v),w)+\omega(\xi_V(w),v)}{2},
$$
where we used the linearity of $\xi_V$. Since the action of $K$ on $(V,\omega)$ is symplectic, we have
$$
\omega(\xi_V(w),v)=-\omega(w,\xi_V(v))
$$
which, together with the antisymmetry of $\omega$ returns
$$
d_vJ_V^\xi(w)=\frac{\omega(\xi_V(v),w)+\omega(\xi_V(v),w)}{2}=\omega(\xi_V(v),w).
$$
Hence
$$
dJ_V-\omega(\xi_V,\cdot)=0.
$$
We also need to show $J_V$ is equivariant, but this is a consequence of the fact that for any $k\in K$, $\xi\in \mfk$, and $v\in V$ we have
$$
k\cdot \xi_V\cdot k^{-1}\cdot v=(\text{Ad}_k\xi)\cdot v.
$$
From this and the $K$-invariance of $J_V$, it immediately follows that
$$
J_V(k\cdot v)=\text{Ad}^*_kJ_V(v).
$$
\end{proof}

Let us now record some facts about linear symplectic representations that will come in handy later on.

\begin{lem}\label{lem: foundational properties of symplectic representations}
Let $K$ be a compact Lie group and $(V,\omega)$ a symplectic $K$-representation. 
\begin{itemize}
    \item[(1)] If $J_V:V\to \mfk^*$ is the quadratic momentum map, then $J_V^{-1}(0)^K=V^K$.
    \item[(2)] If $J_W:W\to \mfk^*$ denotes induced quadratic momentum map, then
    $$
    J_V^{-1}(0)=J_W^{-1}(0)+V^K.
    $$
\end{itemize}
\end{lem}

\begin{proof}
\begin{itemize}
    \item[(1)] First, the statement makes sense since $J_V$ is equivariant. Hence, since $G$ fixes $0\in \mfk^*$, it follows that $K$ acts on $J_V^{-1}(0)$. In particular, we clearly then have $J_V^{-1}(0)^K\subseteq V^K$. On the otherhand, if $v\in V^K$ then we see for any $\xi\in\mfk$ that
    $$
    \xi_V(v)=0.
    $$
    In particular, this implies that
    $$
    \bra J_V(v),\xi\ket=\frac{1}{2}\omega(\xi_V(v),v)=\frac{1}{2}\omega(0,v)=0.
    $$
    Since $v$ and $\xi$ we arbitrary, we conclude that $V^K\subseteq J_V^{-1}(0)$. 

    \item[(2)] Since $W$ and $V^K$ are both symplectic spaces and are are symplectic complements, then they are also linear complements. That is, $V=V^K+W$ and $V^K\cap W=\{0\}$. Hence, for every $v\in V$ there are unique $v^K\in V^K$ and $w\in W$ so that
    $$
    v=v^K+w.
    $$
    Now, observe that since $K$ is a sub-representation, if $\xi\in \mfk$, then
    $$
    \xi_V|_W=\xi_W.
    $$
    With this, we see that if $v\in V^K$ and $w\in W$, then for any $\xi\in\mfk$, we have
    $$
    \bra J_V(v+w),\xi\ket=\frac{1}{2}\omega(\xi_V(v+w),v+w)=\frac{1}{2}\omega(\xi_W(w),w)=\bra J_W(w),\xi\ket.
    $$
    In particular, we see that if $v\in V^K$ and $w\in W$, then $J_V(v+w)=0$ if and only if $J_W(w)=0$. The result then follows.
\end{itemize}
\end{proof}

\section{Reduction and Local Normal Forms}\label{sec: nonsing reduction}

Our main interest for Hamiltonian $G$-spaces is going to be performing symplectic reduction. We now introduce this foundational concept below.

\begin{defs}[Symplectic Reduction]\label{def: symplectic reduction}
Let $J:(M,\omega)\to \mfg^*$ we define the \textbf{symplectic reduction} of $M$ at level $0$ of the momentum map to be the quotient
$$
M_0:=J^{-1}(0)/G.
$$
We will write $\pi_0:J^{-1}(0)\to M_0$ for the quotient map. Call $\pi_0$ the \textbf{reduction map}
\end{defs}

\begin{egs}\label{eg: reduction of cotangent bundle, free}
    Let $G$ be a connected Lie group and $Q$ a proper $G$-space on which $G$ acts freely, and let $J:T^*Q\to \mfg^*$ be the Hamiltonian lift as in Definition \ref{def: hamiltonian lift}. Let's first determine $J^{-1}(0)$. Given $\alpha\in T^*Q$ by definition $\alpha\in J^{-1}(0)$ if and only if $\bra J(\alpha),\xi_Q\ket=0$ for all $\xi\in \mfg$, where $\xi_Q\in\mfX(Q)$ is the associated fundamental vector field. Recall from Example \ref{eg: free orbit foliation} that since $G$ acts freely, the fundamental vector fields span a $G$-invariant subbundle $\mfg_Q\subseteq TQ$. Thus, we have shown that $J^{-1}(0)$ is the annihilator of $\mfg_Q$, that is
    $$
    J^{-1}(0)=\mfg_Q^\circ\subseteq T^*Q,
    $$
    Since $G$ acts freely on $Q$, we also have that the reduced space
    $$
    (T^*Q)_0=J^{-1}(0)/G=\mfg_Q^\circ/G
    $$
    is a vector bundle over $Q/G$. Which vector bundle? Well as we will show in more generality later on, it is precisely $T^*(Q/G)$.
\end{egs}

In the case where a compact Lie group is acting freely, the so-called ``miracle of reduction'' occurs. Let $K$ be a compact Lie group and $J:(M,\omega)\to \mfk^*$ a Hamiltonian $K$-space. Suppose $0$ is a regular value of $J$ and that $K$ acts freely on $J^{-1}(0)$. Then it is a straightforward calculation to show for any $x\in J^{-1}(0)$ that
$$
(T_x J^{-1}(0))^{\omega_x}=T_x(K\cdot x).
$$
This is the ``miracle of reduction'' for the following reasons. First, $J^{-1}(0)$ is a coisotropic submanifold and for each $x\in J^{-1}(0)$, the tangent space of the quotient space has the form
$$
T_{[x]}(J^{-1}(0)/K)=T_xJ^{-1}(0)/T_x(K\cdot x)
$$
which is precisely the linear reduction of $T_xJ^{-1}(0)$. Thus, we have mostly shown the following classic result.

\begin{theorem}[Marsden-Weinstein-Meyer Reduction {\cite{marsden_reduction_1974, meyer_symmetries_1973}}]\label{thm: weinstein reduction}
    Let $K$ be a compact Lie group with Lie algebra $\mfk$ and $J:(M,\omega)\to \mfk^*$ a Hamiltonian $K$-space. Suppose $0$ is a regular value of $J$ and that $K$ acts freely on $J^{-1}(0)$. Then the following holds.
    \begin{itemize}
        \item[(1)] $J^{-1}(0)\subseteq M$ is an embedded coisotropic submanifold.
        \item[(2)] The reduction $M_0=J^{-1}(0)/K$ has a canonical symplectic form $\omega_0\in\Omega^2(M_0)$ such that 
        $$
        \pi_0^*\omega_0=\omega|_{J^{-1}(0)},
        $$
        where $\pi_0:J^{-1}(0)\to M_0$ is the reduction map.
    \end{itemize}
\end{theorem}

Now we are going to be dealing with proper actions by groups which need not act freely. Thus, we will need a local normal form to properly discuss this more general setting. An application of Theorem \ref{thm: weinstein reduction} together with Reduction in Stages (see for instance \cite{marsden_hamiltonian_2007} for a comprehensive source) allows us to show the following.

\begin{lem}[Local Model of Proper Hamiltonian Space {\cite{sjamaar_stratified_1991}}]\label{lem: local normal form for Hamiltonian spaces}
    Let $G$ be a connected Lie group and $K\leq G$ a compact subgroup. Write $\mfk$ for the Lie algebra of $K$ and 
    $$
    \mfk^\circ=\{p\in \mfg^* \ | \ \bra p,\xi\ket=0\text{ for all }\xi\in\mfk\}
    $$
    for the annihilator of $\mfk$. Suppose we have the following data
    Suppose we have the following data.
    \begin{itemize}
        \item[(i)] A linear symplectic representation $(V,\omega_V)$ of $K$.
        \item[(ii)] A $K$-equivariant splitting $\mfp:\mfk^*\to \mfg^*$ of the short exact sequence
        $$
	\begin{tikzcd}
	0\arrow[r]&\mfk^\circ\arrow[r]&\mfg^*\arrow[r]&\mfk^*\arrow[r]&0
	\end{tikzcd}
	$$
    \end{itemize}
    Write $\pi:G\times \mfk^\circ\times V\to G\times_K(\mfk^\circ\times V)$ for the quotient map and $\omega_{T^*G}\in \Omega^2(G\times \mfg^*)$ for the canonical symplectic form.
    \begin{itemize}
        \item[(1)] $G\times_K(\mfk^\circ\times V)$ is canonically a symplectic manifold with unique $G$-invariant symplectic form $\omega$ satisfying
        $$
        \pi^*\omega=\omega_{T^*G}\times \omega_V|_{G\times \mfg^*\times V}.
        $$
        \item[(2)] Writing $J_V:V\to \mfk^*$ for the quadratic momentum map from Example \ref{lem: quadratic momentum map}, the map
        $$
        J:G\times_K(\mfk^\circ\times V)\to \mfg^*;\quad [g,q,v]\mapsto \text{Ad}_g^*(q+\mfp\circ J_V(v))
        $$
        is a momentum map.
    \end{itemize}
\end{lem}

Now that we have our local model, let's now discuss the local normal form and how to obtain it. Fix a Hamiltonian $G$-space $J:(M,\omega)\to \mfg^*$ and let $x\in M$. Then the tangent space $T_xM$ has a natural symplectic $G_x$ representation structure. Clearly this action preserves the tangent space to the orbit $T_x(G\cdot x)$. Thus, by Proposition \ref{prop: equivariant linear reduction} the reduction of the complement $T_x(G\cdot x)^{\omega_x}$ has a canonical symplectic action by $G_x$ as well. This allows us to give the following definition.

\begin{defs}[Symplectic Normal Space]\label{def: symplectic normal space}
    Let $J:(M,\omega)\to \mfg^*$ be a Hamiltonian $G$-space and $x\in M$. Define the \textbf{symplectic normal space} to the orbit through $x$, denoted $S\nu_x(M,G)$, to be the linear reduction of $T_x(G\cdot x)^{\omega_x}$, that is
    \begin{equation}
        S\nu_x(M,G):=\frac{T_x(G\cdot x)^{\omega_x}}{T_x(G\cdot x)^{\omega_x}\cap T_x(G\cdot x)}
    \end{equation}
    Write $J_x:S\nu_x(M,G)\to \mfg_x^*$ for the induced quadratic momentum map as in Lemma \ref{lem: quadratic momentum map}
\end{defs}

We now have enough set-up to state a version of the slice theorem for Hamiltonian spaces.

\begin{theorem}[Local Normal Form for Momentum Map Along Zero Level-Set{\cite{sjamaar_stratified_1991}}]\label{thm: local normal form for Hamiltonian spaces}
    Let $J:(M,\omega)\to \mfg^*$ be a proper Hamiltonian $G$-space and $x\in J^{-1}(0)$. Write $\mfg_x$ for the Lie algebra of the stabilizer group $G_x$ and suppose $\mfp:\mfg_x^*\to \mfg^*$ is a splitting as in Lemma \ref{lem: local normal form for Hamiltonian spaces}. Then there exists the following data.
    \begin{itemize}
        \item[(1)] $G$-invariant neighbourhoods $U\subseteq M$ of $x$ and $V\subseteq G\times_{G_x}(\mfg_x^\circ\times S\nu_x(M,G))$ of $[e,0,0]$
        \item[(ii)] An isomorphism of Hamiltonian $G$-space $\Phi:U\to V$ with respect to the momentum map
        $$
        M\times_{G_x}(\mfg_x^\circ \times V)\to \mfg^*;\quad [g,q,v]\mapsto \text{Ad}_g^*(q+\mfp\circ J_x(v)),
        $$
        where $J_x:S\nu_x(M,G)\to \mfg_x^*$ is the quadratic momentum map, such that $\Phi(x)=[e,0,0]$.
    \end{itemize}
\end{theorem}

\begin{defs}[Symplectic Slice Neighbourhood]
    Let $J:(M,\omega)\to \mfg^*$ be a proper Hamiltonian $G$-space and $x\in M$. Call neighbourhoods $U\subseteq M$ and $V\subseteq G\times_{G_x}(\mfg_x^\circ\times S\nu_x(M,G))$ together with the isomorphism of Hamiltonian $G$-spaces $\Phi:U\to V$ a \textbf{symplectic slice neighbourhood} about $x$.
\end{defs}

As a consequence of this, we get the following extension of Marsden-Weinstein-Meyer reduction.

\begin{cor}\label{cor: more general nonsing reduction}
    Let $G$ be a connected Lie group and $J:(M,\omega)\to \mfg^*$ a proper Hamiltonian $G$-space. Suppose $J^{-1}(0)\neq\emptyset$ and any two points in $J^{-1}(0)$ have conjugate stabilizers. Then the following holds.
    \begin{itemize}
        \item[(1)] $J^{-1}(0)$ is an embedded submanifold.
        \item[(2)] The reduction $M_0=J^{-1}(0)/G$ is a canonically a symplectic manifold with symplectic structure $\omega_0\in\Omega^2(M_0)$ satisfying
        $$
        \pi_0^*\omega_0=\omega|_{J^{-1}(0)},
        $$
        where $\pi_0:J^{-1}(0)\to M_0$ is the reduction map.
    \end{itemize}
\end{cor}

\begin{proof}
Since both statements are local in nature, by Theorem \ref{thm: local normal form for Hamiltonian spaces} we may assume that we are in a local model of a Hamiltonian $G$-space. That is, $M=G\times_K(\mfk^\circ\times V)$ for a compact subgroup $K\leq G$ and symplectic $K$-representation $(V,\omega)$ with $J:M\to \mfg^*$ having the form
$$
J:G\times_K(\mfk^\circ\times V)\to \mfg^*;\quad [g,q,v]\mapsto \text{Ad}_g^*(q+\mfp\circ J_V(v)),
$$
where $\mfp:\mfk^*\to \mfg^*$ is a splitting and $J_V:V\to \mfk^*$ is the quadratic momentum map. 
    \begin{itemize}
        \item[(1)] Since the image of $\mfp$ lies in the kernel of the natural projection $\mfg^*\to \mfk^*$, it follows that
        $$
        J^{-1}(0)=G\times_K(\{0\}\times J_V^{-1}(0)).
        $$
        Since $[e,0,0]\in J^{-1}(0)$, by assumption we must have that $J^{-1}(0)\subseteq M_{(K)}$. As we saw in Proposition \ref{prop: local normal form manifold of symmetry and orbit type},
        $$
        G\times_K(\mfk^\circ\times V)_{(K)}=G\times_K((\mfk^\circ)^K\times V^K)
        $$
        and hence by Lemma \ref{lem: foundational properties of symplectic representations},
        \begin{equation}\label{eq: local form of zero level set}
        J^{-1}(0)=G\times_K(\{0\}\times J_V^{-1}(0)^K)=G\times_K(\{0\}\times V^K).
        \end{equation}
        This is trivially equivariantly diffeomorphic to $(G/K)\times V^K$ and hence $J^{-1}(0)$ is an embedded submanifold.
        
        \item[(2)] Using Equation (\ref{eq: local form of zero level set}), we obtain that
        $$
        M_0=J^{-1}(0)/G\cong V^K.
        $$
        This shows both that the reduced space $M_0$ has a natural smooth manifold structure. Furthermore, since $V$ is symplectic, we have by Proposition \ref{prop: fixed point set symplectic rep} that $V^K$ is a symplectic subspace and thus $M_0$ has a natural symplectic structure.
    \end{itemize}
\end{proof}

As a final consequence of the the normal form theorem, we can give a infinitesimal description to the tangent space of a proper Hamiltonian $G$-space which will come in handy when we pass to the singular setting.

\begin{prop}\label{prop: tangent space of proper hamiltonian space}
    Let $G$ be a connected Lie group, $J:(M,\omega)\to \mfg^*$ a proper Hamiltonian $G$-space, and $x\in J^{-1}(0)$. Then a symplectic slice neighbourhood about $x$ induces a $G_x$-equivariant isomorphism of linear symplectic $G_x$-representations
    $$
    \phi:(T_xM,\omega_x)\to (T^*(\mfg/\mfg_x)\times S\nu_x(M,G), \omega_0)
    $$
    where $S\nu_x(M,G)$ is the symplectic normal space to the orbit through $x$ as in Definition \ref{def: symplectic normal space} and $\omega_0=\omega_{T^*(\mfg/\mfg_x)}\times \omega_{S\nu_x(M,G)}$. 
\end{prop}

\begin{proof}
    Fixing $x\in J^{-1}(0)$, let
    $$
    \Phi:U\substeq M\to V\substeq G\times_{G_x}(\mfg_x^\circ\times S\nu_x(M,G))
    $$
    be a symplectic slice neighbourhood around $x$. In particular, we obtain a $G_x$-equivariant linear symplectomorphism
    $$
    T_x\Phi:T_xM\to T_{[e,0,0]}G\times_{G_x}(\mfg_x^\circ\times S\nu_x(M,G)).
    $$
    Since $G_x$ acts freely on $G\times\mfg_x^\circ\times S\nu_x(M,G)$, we have
    \begin{align*}
    T_{[e,0,0]}G\times_{G_x}(\mfg_x^\circ\times S\nu_x(M,G))&\cong T_{(e,0,0)}(G\times \mfg_x^\circ\times S\nu_x(M,G))/T_{(e,0,0)}(G_x\cdot (e,0,0))\\
    &=(\mfg/\mfg_x)\times \mfg_x^\circ \times S\nu_x(M,G).
    \end{align*}
    The dual of the projection $\mfg\to \mfg/\mfg_x$ provides a $G_x$-equivariant isomorphism $(\mfg/\mfg_x)^*\to \mfg_x^\circ$ and hence $(\mfg/\mfg_x)\times \mfg_x^\circ \cong T^*(\mfg/\mfg_x)$. Since the symplectic form on the local normal form is given by the reduction of the product of the symplectic form on $T^*G$ and $S\nu_x(M,G)$, the result then follows.
\end{proof}

\section{Non-Trivial Examples of Reduction}\label{sec: non-sing reduction examples}

\subsection{Toric Manifolds and the Delzant Construction}\label{sub: delzant construction}

A beautiful application of the theory of reduction is to the theory of toric manifolds. These are a particular kind of Hamiltonian space and, as we shall see, they can be reconstructed from the image of the momentum map. Unless otherwise stated, everything below can be found in Delzant's original paper \cite{delzant_hamiltoniens_1988}. For an excellent overview in English, see Chapters 1-3 in \cite{guillemin_moment_1994}.

\begin{defs}[Toric Manifold]\label{def: toric mfld}
    Let $\bT$ be a torus, i.e. a compact connected abelian Lie group which is isomorphic to $(S^1)^n$ for some $n$. Write $\mft$ for the Lie algebra of $\bT$. A \textbf{toric $\bT$-manifold} is a Hamiltonian $\bT$-space $J:(M,\omega)\to\mft$ such that
    \begin{itemize}
        \item[(1)] $\bT$ acts effectively on $M$, i.e. 
        $$
        \bigcap_{x\in M}\bT_x=\{e\}
        $$
        and
        \item[(2)] $\dim(M)=2\dim(\bT)$.
    \end{itemize}
\end{defs}

Delzant's construction takes as input a particular kind of polytope and produces a toric manifold. In order to set up this construction, we first give a definition of this special class of polytope.

\begin{defs}[Delzant Polytope]
    Consider a subset $\Delta\subseteq \bR^n$. We say $\Delta$ is a \textbf{Delzant polytope} if the following holds.
    \begin{itemize}
        \item[(1)] There exists \textbf{inward normals} $v_1,\dots,v_N\in \bZ^N$ so that 
        $$
        \Delta=\{x\in \bR^N \ | \ \bra x,v_i\ket\geq \lambda_i\text{ for all }i=1,\dots,N\}
        $$
        for some $\lambda_1,\dots,\lambda_N\in \bR$.
        \item[(2)] The map linear map $\pi:\bR^N\to \bR^n$ defined by $\pi(e_i)=v_i$, where $e_1,\dots,e_N\in \bR^N$ is the standard basis, is a surjective linear map such that 
        $$
        \pi(\bZ^N)= \bZ^n.
        $$
    \end{itemize}
\end{defs}

\begin{remark}
    Given a Delzant polytope
    $$
    \Delta=\{x\in \bR^N \ | \ \bra x,v_i\ket\geq \lambda_i\text{ for all }i=1,\dots,N\},
    $$
    we can use the normals to cut out the \textbf{facets}. Namely, given any set of indices $I\substeq \{1,\dots,N\}$, we define
    $$
    F_I:=\{x\in \Delta \ | \ \bra x,v_i\ket=\lambda_i\text{ for all }i\in I\}.
    $$
    If non-empty, this is again another polytope inside of $\bR^n$. Furthermore, it's easy to see that if $I_1\substeq I_2$ are two sets of indices, then
    $$
    F_{I_2}\substeq F_{I_1}.
    $$
    The vertices can be identified as the minimal facets with respect to inclusion. Finally, as it will be helpful later on, let us define the \textbf{open facets} of $\Delta$ to be subsets of the form
    $$
    \mathring{F}_I:=F_I\setminus \bigcup_{I\subsetneq J}F_J,
    $$
    i.e. what remains from a facet after removing smaller facets with respect to inclusion. These are submanifolds of $\bR^n$.
\end{remark}

Let's now fix a Delzant polytope $\Delta\subseteq\bR^n$ with normals $v_1,\dots,v_N\in \bZ^N$ so that
\begin{equation}\label{Delzant polytope}
\Delta =\{x\in \bR^n  \ | \ \bra x,v_i\ket\geq \lambda_i\text{ for all }i=1,\dots,N\}
\end{equation}
for some $\lambda_1,\dots,\lambda_N\in \bR$. Delzant's construction for a toric manifold $M_\Delta$ is as follows. Let $\pi:\bR^N\to \bR^n$ be the linear map defined by
$$
\pi(e_i)=v_i\text{ for all }i=1,\dots,N,
$$
where $e_1,\dots,e_N\in \bR^N$ is the standard basis. Furthermore, let $\mfk=\ker(\pi)$. Then we get an exact sequence
\begin{equation}\label{toric infinitesimal exact}
\begin{tikzcd}
0\arrow[r] & k\arrow[r] & \bR^N\arrow[r,"\pi"] & \bR^n\arrow[r] & 0
\end{tikzcd}
\end{equation}
Since $\pi(\bZ^N)\subseteq \bZ^n$, it follows that $\pi$ induces a group homomorphism
$$
\pi:T^N=\bR^N/\bZ^N\to T^n=\bR^n/\bZ^n
$$
and, furthermore, the kernel $K$ is a subtorus of $T^N$ and its Lie algebra can be canonically identified with $\mfk$. Hence, we get a exact sequence of Lie groups.
\begin{equation}\label{toric global exact}
\begin{tikzcd}
0\arrow[r] & K\arrow[r] & T^N\arrow[r,"\pi"] & T^n\arrow[r] & 0
\end{tikzcd}
\end{equation}

Identify $T^N=(S^1)^N\subseteq (\bC^\times)^N$, and let $\omega$ be the canonical symplectic form on $\bC^N$. The action
$$
T^N\times \bC^N\to \bC^N;\quad ((t_1,\dots,t_N), (z_1,\dots,z_N))\mapsto (t_1z_1,\dots,t_Nz_N)
$$
has a canonical momentum map
\begin{equation}
\bC^N\to \bR^N;\quad (z_1,\dots,z_N)\mapsto \frac{1}{2}(|z_1|^2,\dots,|z_N|^2)
\end{equation}
Since $T^N$ is abelian, we can shift $J$ by a constant in $\bR^N$ and still have a momentum map. For this reason, we will define
\begin{equation}
J:\bC^N\to \bR^N;\quad (z_1,\dots,z_N)\mapsto \frac{1}{2}(|z_1|^2,\dots,|z_N|^2)-\lambda.
\end{equation}
As a final step towards the construction, let $\iota:\mfk\into \bR^N$ be the inclusion and let $\iota^*:(\bR^N)^*\to \mfk^*$ be its dual. Then the map
\begin{equation}
\iota^*\circ J:\bC^N\to \mfk^*
\end{equation}
is a momentum map for the induced $K$-action. 

\begin{lem}[{\cite{delzant_hamiltoniens_1988}}]
$K$ acts freely on $(\iota^*\circ J)^{-1}(0)$.
\end{lem}

\begin{proof}
This amounts to showing that if $z=(z_1,\dots,z_N)\in (\iota^*\circ J)^{-1}(0)$, then $z_j\neq 0$ for all $j$. A nice proof of this fact can be found in \cite[Theorem 1.6]{guillemin_moment_1994}
\end{proof}

Consequently, using Marsden-Weinstein-Meyer reduction (Theorem \ref{thm: weinstein reduction}), $\phi^{-1}(0)/K$ is canonically a symplectic manifold. Furthermore, we see that $(\iota^*\circ J)^{-1}(0)$ is precisely the set of points in $\bC^N$ mapped to $\mfk^\circ\subseteq (\bR^N)^*$, the annhilator of $\mfk$. This can canonically be identified with $(\bR^N/\mfk)^*\cong (\bR^n)^*\cong \bR^n$.

\begin{defs}[Toric Manifold from a Delzant Polytope]
Given a Delzant polytope $\Delta$, let $(M_\Delta,\omega_\Delta)$ denote the symplectic reduction of $\bC^N$ along $\phi^{-1}(0)$. Furthermore, we let
\begin{equation}
J_\Delta:M_\Delta\to \bR^n;\quad [x]\mapsto J(x)
\end{equation}
be the induced map.
\end{defs}

Delzant's construction turns out to produce all toric manifolds.

\begin{theorem}[{\cite{delzant_hamiltoniens_1988}}]\label{thm: delzant's theorem}
$J_\Delta:(M_\Delta,\omega_\Delta)\to \bR^n$ is a toric $T^n$-manifold. If $\Delta$ is compact, then so is $M_\Delta$. Furthermore, if $J:(M,\omega)\to \bR^n$ is any toric $T^n$-manifold, then $\Delta=J(M)$ is a Delzant polytope and as Hamiltonian $T^n$-spaces, $(M,\omega)\cong (M_\Delta,\omega_\Delta)$.
\end{theorem}

\subsection{Reduction of Cotangent Bundles}\label{sub: cotangent reduction}

As was hinted at in Example \ref{eg: reduction of cotangent bundle, free}, the symplectic reduction of a cotangent bundle where the underlying group is acting both properly and freely produces the cotangent bundle of the quotient space. It turns out we can push this much further as we shall see. 

\ 

For the rest of this section, we will let $G$ be a connected Lie group with Lie algebra $\mfg$, $Q$ a proper $G$-space with the property that $Q=Q_{(K)}$ for some compact $K\leq G$, $J:T^*Q\to \mfg^*$ the canonical Hamiltonian $G$-space structure on the cotangent bundle, and we will let 
\begin{equation}\label{eq: orbit foliation cotangent}
\mfg_Q:=\text{span}_\bR\{\xi_Q \ | \ \xi\in \mfg\}\subseteq TQ
\end{equation}
be the invariant subbundle spanned by the fundamental vector fields from Example \ref{eg: free orbit foliation}.

\begin{lem}[\cite{emmrich_orbifolds_1990}] \label{lem: 0 level set nice cotangent}
Suppose the action of $G$ on $Q$ is regular. That is, $Q=Q_{(K)}$ for some compact subgroup $K\leq G$. 
\begin{itemize}
    \item[(i)] $J^{-1}(0)=J^{-1}(0)_{(K)}$
    \item[(ii)] $(T^*Q)_0=J^{-1}(0)/G$ is a smooth manifold and there exists unique symplectic form $\omega_0\in \Omega^2((T^*Q)_0)$ such that
    $$
    \pi_0^*\omega_0=\omega|_{J^{-1}(0)},
    $$
    where $\pi_0:J^{-1}(0)\to (T^*Q)_0$ is the quotient map.
    \item[(iii)] The restriction of the Liouville $1$-form $\theta_Q|_{J^{-1}(0)}$ is basic and hence induces a $1$-form $\theta_0\in\Omega^1((T^*Q)_0)$ such that
    \begin{equation}
    \pi_0^*\theta_0=\theta_Q|_{J^{-1}(0)}.
    \end{equation}
    and $d\theta_0=\omega_0$.
\end{itemize}
\end{lem}

\begin{proof}
\begin{itemize}
    \item[(i)] Suppose $\alpha\in J^{-1}(0)$ and let $g\in G_\alpha$. Writing $x=\tau_Q(\alpha)$, we see that 
    $$
    g\cdot x=g\cdot \tau_Q(\alpha)=\tau_Q(g\cdot \alpha)=\tau_Q(\alpha)=x.
    $$
    Hence, $G_\alpha\subseteq G_{\tau_Q(\alpha)}\sim K$. Thus, we may assume after conjugation that $G_\alpha=K_\alpha\subseteq K$. 

    \

    To show that $K\subseteq K_\alpha$, choose a $G$-invariant metric on $Q$ so that we have a splitting
    $$
    T_xQ=T_x(G\cdot x)\oplus \nu_x(Q,G)
    $$
    Since there is only one orbit-type on $Q$, it immediately follows that $K$ acts trivially on $\nu_x(Q,G)$. Now fix $v\in T_xQ$ and $k\in K$. Since $K$ stabilizes $x=\tau_Q(\alpha)$, it follows that $k\cdot\alpha\in T_x^*Q$. Hence, we may pair $k\cdot \alpha$ and $v$, yielding
    $$
    \bra k\cdot \alpha,v\ket=\bra \alpha,k^{-1}\cdot v\ket
    $$
    Writing $v=\xi_Q(x)+w$ for some $\xi\in \mfg$ and $w\in \nu_x(Q,G)$, we have
    $$
    k^{-1}\cdot v=(\text{Ad}_{k^{-1}}\xi)_Q(x)+w
    $$
    Thus,
    $$
    \bra \alpha, k^{-1}\cdot v\ket=\bra \alpha, w\ket
    $$
    Similarly,
    $$
    \bra \alpha,v\ket=\bra \alpha,w\ket.
    $$
    Thus, $k\cdot \alpha=\alpha$

    \item[(ii)] Since the orbit foliation on $Q$ is regular, $J^{-1}(0)$ is an embedded submanifold of $T^*Q$. Furthermore, $J^{-1}(0)$ has only one orbit-type by Lemma \ref{lem: 0 level set nice cotangent}, hence the quotient $J^{-1}(0)/G$ is a manifold. Therefore, by Marsden-Weinstein reduction, $(T^*Q)_0$ inherits a canonical symplectic form $\omega_0$ with
    $$
    \pi_0^*\omega_0=\omega_Q|_{J^{-1}(0)}.
    $$

    \item[(iii)] This immediately follows from the fact that $\theta_Q$ is $G$-invariant and from the uniqueness of $\omega_0$.
\end{itemize}
\end{proof}

Recall that $\mfg_Q\subseteq TQ$ from Equation (\ref{eq: orbit foliation cotangent}) is defined to be the span of the fundamental vector fields. This is the tangent bundle of foliation by $G$-orbits on $Q$. Hence, if we write $\pi:Q\to Q/G$ for the projection map, we have a short exact sequence of vector bundles over $Q$
$$
\begin{tikzcd}
0\arrow[r] & \xi_Q\arrow[r] & TQ\arrow[r,"T\pi"] & \pi^{-1}T(Q/G)\arrow[r] & 0
\end{tikzcd}
$$
where $\pi^{-1}T(Q/G)\subseteq Q\times T(Q/G)$ is the pull-back bundle. Dualizing the short exact sequence, we get a short exact sequence
$$
\begin{tikzcd}
0\arrow[r] & \pi^{-1}T^*(Q/G)\arrow[r,"\phi"] & T^*Q\arrow[r,"T\pi"] & (\xi_Q)^*\arrow[r] & 0
\end{tikzcd}
$$
where
$$
\phi:\pi^{-1}T^*(Q/G)\to T^*Q
$$
is defined by
$$
\bra\phi(q,\beta),v\ket=\bra \beta, T_q\pi(v)\ket
$$
for $(q,\beta)\in \pi^{-1}T^*(Q/G)$ and $v\in T_qQ$. Working point-wise, we see that the image of $(\pi^{-1}T^*(Q/G))_q$ by $\phi$ is precisely $\mfg_Q(q)^\circ=J^{-1}(0)_q$ for all $q\in Q$. Thus, we have an isomorphism of vector bundles over $Q$
\begin{equation}\label{eq: 0 level set is pullback}
\phi:\pi^{-1}T^*(Q/G)\to J^{-1}(0).
\end{equation}
Now to set up the next statement, write $\tau_{Q/G}:T^*(Q/G)\to Q/G$ for the bundle projection map. By definition,
$$
\pi^{-1}T^*(Q/G)=\{(q,\beta)\in Q\times T^*(Q/G) \ | \ \pi(q)=\tau_{Q/G}(\beta)\}
$$
where $\pi:Q\to Q/G$ is the quotient map. Hence, we have two natural maps
$$
pr_1:\pi^{-1}T^*(Q/G)\to Q;\quad (q,\beta)\mapsto q
$$
and
$$
pr_2:\pi^{-1}T^*(Q/G)\to T^*(Q/G);\quad (q,\beta)\mapsto \beta.
$$

\begin{lem}\label{lem: reduced 1-form is canonical}
In the setting where $Q$ has only one orbit-type, let $\theta_{Q/G}\in\Omega^1(T^*(Q/G))$ denote the Liouville $1$-form on $T^*(Q/G)$. Then,
$$
\phi^*\theta_Q=pr_2^*\theta_{Q/G}.
$$
\end{lem}

\begin{proof}
Let us fix $(q,\beta)\in \pi^{-1}T^*(Q/G)$ and $v\in T_{(q,\beta)}(\pi^{-1}T^*(Q/G))$. We first have
$$
(\phi^*\theta_Q)_{(q,\beta)}(v)=(\theta_Q)_{\phi(q,\beta)}(\phi v)=\bra \phi(q,\beta),(\tau_Q)_*\phi_*v\ket=\bra \beta,(\pi\circ\tau_Q\circ\phi)_*v\ket
$$
After a diagram chase, we have
$$
\pi\circ\tau_Q\circ\phi=\tau_{Q/G}\circ pr_2
$$
Hence,
$$
(\phi^*\theta_Q)_{(q,\beta)}(v)=\bra \beta, (\tau_{Q/G}\circ pr_2)_*v\ket=(\theta_{Q/G})\beta((pr_2)_*v\ket=(pr_2^*\theta_Q)_{(q,\beta)}(v).
$$
\end{proof}

With this fact established, we can now prove the following.

\begin{theorem}\label{thm: reduced symplectomorphism}
Suppose $Q$ has only one orbit-type. Then the map $\phi:\pi^{-1}T^*(Q/G)\to J^{-1}(0)$ from Equation (\ref{eq: 0 level set is pullback}) induces a canonical symplectomorphism
\begin{equation}
\widetilde{\phi}:T^*(Q/G)\to (T^*Q)_0
\end{equation}
\end{theorem}

\begin{proof}
Write $\pi_0:J^{-1}(0)\to (T^*Q)_0$ for the quotient map. Since
$$
\phi:\pi^{-1}T^*(Q/G)\to J^{-1}(0)
$$
is an isomorphism of $G$-equivariant vector bundles, $\phi$ induces an isomorphism between the quotients
$$
\widetilde{\phi}:T^*(Q/G)\to (T^*Q)_0
$$
By Lemma \ref{lem: 0 level set nice cotangent}, the symplectic form on $(T^*Q)_0$ has a canonical primitive $\theta_0$. By Lemma \ref{lem: reduced 1-form is canonical}, we have
$$
\widetilde{\phi}^*\theta_0=\theta_{Q/G}
$$
and hence $\widetilde{\phi}$ is a symplectomorphism.
\end{proof}

\cleardoublepage
\chapter{Polarizations}\label{ch: polarizations}

Given the description of symplectic manifolds at the beginning of Chapter \ref{ch: symp} as generalizations of phase space from classical physics, we may wonder what other physics concepts can be generalized and ported over to the symplectic world. Polarizations represent the generalization of momentum, or rather, momentum directions to symplectic manifolds. Indeed, an arbitrary symplectic manifold $(M,\omega)$ may be describable locally as a phase-space thanks to Darboux's Theorem, but this does not mean the local notion of momentum from this description is globally coherent. The key here is that in phase space, the momentum variables carve out Lagrangian submanifolds. Then thanks to Frobenius' Theorem, the tangent space of this foliation by momentum submanifolds gives an involutive subbundle $\mathcal{P}\substeq TM$ by Lagrangian subspaces. This in turn can be generalized further by demanding not that the polarization lie in the tangent bundle, but rather in the complexified tangent bundle $T^\bC M$, giving us a notion of ``complex momentum''. As we shall see in the next chapter, this concept will be key in formulating geometric quantization.

\ 

Our goal in this chapter is to not only lay out an elementary theory of polarizations, but also to develop the non-singular theory of their reduction. This is an essential idea for fusing Hamiltonian geometry with quantization.

\section{Lagrangian Subspaces and Their Reductions}\label{sec: reduce lagrangians}
Before we can work out the case of general reduction of polarizations, it will be helpful to have some preliminary facts handy.

\

For all that follows, we will be assuming $(V,\omega)$ is a finite dimensional real symplectic vector space, $\omega_\bC$ will be the $\bC$-linear extension of $\omega$ to $V_\bC:=V\otimes \bC$, and $L\subseteq V_ \bC$ will be a complex Lagrangian subspace.  The main question we want to answer in this section is the following: given a subspace $C\subseteq V$ and its reduction $\pi_0:C\to V_0$, where $V_0=C/(C\cap C^\omega)$, when is 
$$
L_0=\pi_0(L\cap C_\bC)\subset (V_0)_\bC
$$
a Lagrangian subspace? In the case where $C$ is coisotropic, i.e. $C^\omega\subseteq C$, the answer is always yes.

\begin{prop}\label{prop: linear coisotropic reduction}
Suppose $C\subseteq V$ is a coisotropic subspace and let $V_0=C/C^\omega$ and $\pi:C\to V_0$ be the reduction map.  Then, 
$$
L_0:=\pi_0(L\cap C_\bC)\subseteq (V_0)_\bC
$$
is a complex Lagrangian subspace. 
\end{prop}

\begin{proof}
$L_0$ being isotropic is immediate by the definition of the reduced symplectic form $\omega_0$. So we just need to show $L_0$ is coisotropic. Let $c\in C$ so that $\pi_0(c)\in L_0^{\omega_0}$. Then, for any $\ell\in L\cap C_\bC$, we have
$$
0=\omega_0(\pi_0(c),\pi_0(\ell))=\omega_0(c,\ell).
$$
Hence,
$$
c\in (L\cap C_\bC)^\omega=L^\omega+C_\bC^\omega=L+C^\omega_\bC
$$
Writing $c=a+b$, $a\in L$ and $b\in C^\omega_\bC$, we have that $a\in L\cap C_\bC$ since $C^\omega\subseteq C$. Thus, 
$$
\pi_0(c)=\pi_0(a)\in L_0.
$$
\end{proof}

A question worth investigating here, is what happens to the type after we do reduction along a coisotropic subspace.

\begin{cor}\label{cor: linear reduction preserves type}
Let $C\subseteq V$ be a coisotropic subspace, $L\subseteq V_\bC$ a complex Lagrangian subspace, and let $\pi_0:C\to V_0$ be as in Proposition \ref{prop: linear coisotropic reduction}. 
\begin{itemize}
    \item[(1)] If $L$ is totally complex, i.e. $L\cap \overline{L}=\{0\}$, then $L_0$ is also totally complex, i.e. $L_0\cap \overline{L_0}=\{0\}$.
    \item[(2)] If $L$ is real, i.e. $L=\overline{L}$, then $L_0$ is also real.
\end{itemize}
\end{cor}

\begin{proof}
Both (1) and (2) are consequences of the fact that $L_0$ is obtained by intersecting a complex subspace with the complexification of a real subspace. 
\end{proof}

\begin{egs}
Let $(V,\omega)$ be a $4$ dimensional symplectic vector space and $\{e_1,e_2,f_1,f_2\}$ a Darboux basis for $\omega$, i.e 
$$
\omega=e_1^*\wedge f_1^*+e_2^*\wedge f_2^*,
$$
where $\{e_i^*,f_i^*\}$ is the dual basis. Consider now the coisotropic subspace
$$
C=\text{span}_\bR\{e_1,e_2,f_1\}.
$$
$C$ is indeed coisotropic with
$$
C^\omega=\text{span}_\bR\{e_2\}
$$
and hence the reduced space $V_0=C/C^\omega$ has a canonical Darboux basis which can be identified with $\{e_1,f_1\}$. Write $\pi_0:C\to V_0$ for the reduction map. Now, consider the following three Lagrangian subspaces,
\begin{align*}
L_1&=\text{span}_\bC\{e_1,e_2\}\\
L_2&=\text{span}_\bC\{e_1+if_1,e_2+if_2\}\\
L_3&=\text{span}_\bC\{e_1+if_1,e_2\}\\
L_4&=\text{span}_\bC\{e_1,e_2+if_2\}
\end{align*}
Note that $L_1$ is real, $L_2$ is totally complex, and $L_3$ and $L_4$ are mixed. Indeed, as we suspect, we see that the reduction of $L_1$ remains real
$$
(L_1)_0=\pi_0(L_1\cap C_\bC)=\text{span}_\bC\{e_1\},
$$
the reduction of $L_2$ remains complex
$$
(L_2)_0=\pi_0(L_2\cap C_\bC)=\text{span}_\bC\{e_1+if_1\}.
$$
Whereas we see that the reduction of $L_3$
$$
(L_3)_0=\pi_0(L_3\cap C_\bC)=\text{span}_\bC\{e_1+if_1\}
$$
becomes totally complex and the reduction of $L_4$
$$
(L_4)_0=\pi_0(L_4\cap C_\bC)=\text{span}_\bC\{e_1\}
$$
becomes real. So the only stable types under reduction are real and totally complex.
\end{egs}

Before we move on to the general case, I would like to recall a result of Guillemin and Sternberg. First, a definition.

\begin{defs}\label{def: positive lagrangian}
    Let $(V,\omega)$ be a real symplectic vector space and $L\subseteq V^\bC$ a complex Lagrangian subspace. Define inner-product
    \begin{equation}\label{eq: Lagrangian inner-product}
        V^\bC\times V^\bC\to \bC;\quad (v,w)\mapsto i\omega(v,\overline{w}),
    \end{equation}
    where $\overline{w}$ denotes the complex-conjugate of $w$. If the inner-product in Equation (\ref{eq: Lagrangian inner-product}) is positive-definite when restricted to $L$, then say $L$ is \textbf{positive}.
\end{defs}

\begin{prop}[\cite{guillemin_geometric_1982}] \label{prop: positive polarization trivially intersects orbits}
Suppose $L\subseteq V\otimes \bC$ is a positive Lagrangian subspace and $C\subseteq V$ a coisotropic subspace. Then,
$$
L\cap C_\bC^\omega =\{0\}.
$$
\end{prop}

\begin{proof}
    This is near verbatim the proof found in \cite{guillemin_geometric_1982}. Let $c_1,c_2\in C^\omega$ so that $\ell=c_1+ic_2\in L$. Then,
    $$
    i\omega(\ell,\overline{\ell})=2\omega(c_1,c_2)=0
    $$
    since $C^\omega$ is isotropic. Since $L$ is assumed to be positive, it then follows that $\ell=0$. 
\end{proof}

\begin{cor}\label{cor: linear reduction of positive is positive}
    Suppose $L\subseteq V\otimes \bC$ is a positive Lagrangian subspace and $C\subseteq V$ is coisotropic. Then the reduced Lagrangian $L_0\subseteq (V_0)_\bC=(C/C^\omega)\otimes \bC$ is a positive polarization.
\end{cor}

\begin{proof}
    Write $\pi_0:C\to V_0$ for the projection. Recall that 
    $$
    L_0=\pi_0(L\cap C_\bC).
    $$
    Write $\omega_0$ for the reduced form on $V_0$. Then, if $\ell_1,\ell_2\in L\cap C_\bC$ we have
    $$
    i\omega_0(\pi_0(\ell_1),\overline{\pi_0(\ell_2)})=i\omega(\ell_1,\overline{\ell_2})
    $$
    Suppose now that $\pi_0(\ell_1)=\pi_0(\ell_2)\neq 0$. Then there exists $c\in C^\omega_\bC$ with 
    $$
    \ell_2=\ell_1+c.
    $$
    Thus,
    $$
    i\omega_0(\pi_0(\ell_1),\pi_0(\ell_2))=i\omega(\ell_1,\overline{\ell_1}+\overline{c})=i\omega(\ell_1,\overline{\ell}_1)>0
    $$
    where we used the fact that $L$ is positive and that $C^\omega_\bC$ is closed under complex conjugation.
\end{proof}

In the case of non-singular reduction, this will immediately imply that a polarization induces a reduced polarization on the reduced space, provided certain regularity conditions hold. However, it will not suffice on its own in the case of singular reduction. For that, we need a few auxiliary results.

\begin{prop}\label{prop: lagrangian decomposition}
Suppose $S\subseteq V$ is a symplectic subspace and suppose that
\begin{equation}\label{eq: lagrangian decomposition}
L=L\cap S_\bC+L\cap S_\bC^\omega.
\end{equation}
Then $L\cap S_\bC$ is a complex Lagrangian subspace of $S_\bC$.
\end{prop}

\begin{proof}
Let $\omega_S:=\omega|_S$. Clearly,
$$
L\cap S_\bC\subseteq (L\cap S_\bC)^{\omega_S}
$$
so $L\cap S_\bC$ is isotropic. To show coisotropic, let $x\in (L\cap S_\bC)^{\omega_S}$ and let $\ell\in L$. By the decomposition in Equation (\ref{eq: lagrangian decomposition}), we can write 
$$
\ell=\ell_1+\ell_2,
$$
where $\ell_1\in L\cap S_\bC$ and $\ell_2\in L\cap S_\bC^\omega$. Observe that
$$
\omega(x,\ell)=\omega(x,\ell_1)+\omega(x,\ell_2).
$$
Since $x\in S_\bC$ and $\ell_2\in S_\bC^\omega$, the second summand vanishes and thus 
$$
\omega(x,\ell)=\omega(x,\ell_1)=\omega_S(x,\ell_1)=0.
$$
Therefore, $x\in L^\omega=L$ and hence $x\in L\cap S_\bC$.
\end{proof}

Recall now that if $S\subseteq V$ is a symplectic subspace, then so is its complement $S^\omega$. Furthermore, we have a canonical internal direct sum $V=S\oplus S^\omega$. Thus, we have a natural projection $P_S:V\to S$ along the complement of $S$. 

\begin{prop}\label{prop: linear singular reduction}
Suppose $L\subseteq V_\bC$ a complex Lagrangian subspace and $C\subseteq V$ a subspace.  Suppose there is a symplectic subspace $S\subseteq V$ such that
\begin{itemize}
    \item[(1)] $C=C\cap S+C\cap S^\omega$
    \item[(2)] $L=L\cap S_\bC+L\cap S_\bC^\omega$
    \item[(3)] $C\cap S$ is coisotropic in $S$.
    \item[(4)] $c-P_S(c)\in C^\omega$ for all $c\in C$.
\end{itemize}
Then the following holds.
\begin{itemize}
    \item[(i)] The symplectic complements of $C\cap S$ in $S$ is given by 
    $$
    (C\cap S)^{\omega|_S}=C\cap C^\omega\cap S
    $$
    and the image of $C\cap C^\omega$ under the symplectic projection $P_S:V\to S$ is given by
    $$
    P_S(C\cap C^\omega)=(C\cap S)^{\omega|_S}.
    $$
    \item[(ii)] Define
    $$
    V_0=\frac{C}{C\cap C^\omega}\quad\text{and}\quad S_0=\frac{C\cap S}{(C\cap S)^{\omega|_S}}
    $$
    and let $\pi_C:C\to V_0$ and $\pi_S:C\cap S\to S_0$ be the respective projections. Then the inclusion $\iota:C\cap S\into C$ defines a symplectomorphism $\phi:S_0\to V_0$ making the diagram commutes
    $$
    \begin{tikzcd}
    C\cap S\arrow[r,"\iota"]\arrow[d,"\pi_S"] &C\arrow[d,"\pi_C"] \\
    S_0\arrow[r,"\phi"] & V_0 
    \end{tikzcd}
    $$

    \item[(iii)] $L_0=\pi_C(L\cap C_\bC)$ is a Lagrangian subspace of $(V_0)_\bC$.
\end{itemize}
\end{prop}

\begin{proof}
\begin{itemize}
    \item[(i)] First, observe that
    $$
    C\cap C^\omega\cap S= (C\cap S)^{\omega|_S}
    $$
    Indeed, since 
    $$
    C=C\cap S+C\cap S^\omega,
    $$
    it is straightforward to see that
    $$
    (C\cap S)^{\omega|_S}=C^\omega\cap S.
    $$
    But we are assuming $C\cap S$ is coisotropic, thus
    $$
    C^\omega\cap S\subseteq C\cap S
    $$
    and hence
    $$
    C^\omega\cap S=(C^\omega\cap S)\cap (C\cap S)=C\cap C^\omega\cap S.
    $$
    Now to show that
    $$
    P_S(C\cap C^\omega)=C\cap C^\omega\cap S.
    $$
    Let $c\in C\cap C^\omega$. Using the identity
    $$
    (C+C^\omega)^\omega=C\cap C^\omega,
    $$
    we show $P_S(c)$ lies in the symplectic complement of $C+C^\omega$. To that end, let $x\in C$ and $y\in C^\omega$. Using the fact that $P_S(C)=C\cap S$, we obtain
    $$
    \omega(P_S(c),x+y)=\omega(P_S(c),x)=\omega(P_S(c),P_S(x))
    $$
    Now, since $c-P_S(c)\in S^\omega$, we obtain
    $$
    \omega(P_S(c),P_S(x))=\omega(P_S(c)+c-P_S(c),P_S(x))=\omega(c,P_S(x)).
    $$
    Once again using the fact that $P_S(C)=C\cap S$ and $c\in C^\omega$, we then obtain
    $$
    \omega(c,P_S(x))=0
    $$
    hence the claim.

    \item[(ii)] To define $\phi$ let $c\in C\cap S$ and let
    $$
    \phi(\pi_S(c))=\pi_C(c).
    $$
    This is well-defined since if $c-c'\in C\cap C^\omega\cap S$, then clearly $\pi_C(c-c')=0$. By construction, it makes the diagram commute. Writing $\omega_0$ for the symplectic form on $V_0$ and $(\omega_S)_0$ for the symplectic form on $S_0$, we observe that
    $$
    \pi_S^*\phi^*\omega_0=\iota^*\pi_C^*\omega_0=\iota^*(\omega|_{C})=\omega|_{S\cap C}=(\omega_S)|_{C\cap S}
    $$
    Since $(\omega_0)_S$ is the unique two form on $S_0$ satisfying $\pi_S^*(\omega_S)_0=\omega_S|_{C\cap S}$, it follows that
    $$
    \phi^*\omega_0=(\omega_S)_0.
    $$
    It remains to show that $\phi$ is a linear isomorphism. Since $\phi$ pulls back a non-degenerate $2$-form to a non-degenerate $2$-form, it immediately follows that $\phi$ is injective. For surjectivity, observe that if $c\in C$, then since
    $$
    c-P_S(c)\in C^\omega,
    $$
    we immediately conclude that
    $$
    \phi(\pi_S(P_S(c))=\pi_C(c).
    $$
    Hence, $\phi$ is a symplectomorphism.
    \item[(iii)] Using the commuting diagram in (ii), we obtain that
    $$
    \pi_C(L\cap C_\bC)=\phi(\pi_S(L\cap C_\bC\cap S_\bC)).
    $$
    Since $C\cap S$ is coisotropic and by Proposition \ref{prop: lagrangian decomposition}, $L\cap S_\bC$ is a complex Lagrangian subspace of $S_\bC$, it follows that $\pi_S(L\cap C_\bC\cap S_\bC)$ is a complex Lagrangian subspace of $(S_0)_\bC$. Hence, by Proposition \ref{prop: lagrangian decomposition}, it follows that $\pi_S(L\cap C_\bC\cap S_\bC)\subseteq (S_0)_\bC$ is a complex Lagrangian and since $\phi$ is a symplectomorphism, so is its image.
\end{itemize}
\end{proof}

It might seem unlikely that we should ever be in the presence of a symplectic subspace as in Proposition \ref{prop: linear singular reduction}. However, in the presence of a linear action of a compact group, much of the above becomes automatic. So henceforth, let us fix a compact (not necessarily connected) Lie group $K$, and suppose $(V,\omega)$ is a symplectic $K$-representation. We will also be assuming that the Lagrangian subspace $L\subseteq V_\bC$ is a complex $K$-representation.

\begin{prop}\label{prop: equivariant set up for reduction}
\begin{itemize}
    \item[(1)] $L=L\cap V^K_\bC+L\cap (V^K)^\omega_\bC$, and 
    \item[(2)] $L^K$ is a Lagrangian subspace of $V^K$.
\end{itemize}
\end{prop}

\begin{proof}
\begin{itemize}
    \item[(1)] Recall from Proposition \ref{prop: fixed point set symplectic rep} that $V^K$ is a symplectic subspace. Hence, we have a natural surjective projection along $(V^K)^\omega$
    $$
    P_{V^K}:V\to V^K\subseteq V.
    $$
    Restricting, we get a map $P_{V^K}|_L:L\to L^K\subseteq L$ and hence we also get a canonical direct sum
    $$
    L=L^K+\ker(P_{V^K}|_L).
    $$
    It's then automatic that $L^K=L\cap V^K_\bC$ and $\ker(P_{V^K}|_L)=L\cap (V^K)^\omega_\bC$. 

    \item[(2)] In light of Proposition \ref{prop: lagrangian decomposition}, $L\cap V^K_\bC$ is a complex Lagrangian subspace of $V^K_\bC$. Furthermore, we have
    $$
    L^K=L\cap V^K_\bC.
    $$
    The result then follows.
\end{itemize}
\end{proof}

With this, we can now do a linear version of singular reduction.

\begin{theorem}\label{thm: linear model of singular reduction}
Let $V$ and $W$ be two finite dimensional symplectic $K$-representations and let $F\subseteq V$ be a $K$-invariant real Lagrangian subspace and define
$$
C=F\oplus W^K
$$
and let 
$$
(V\oplus W)_0:=\frac{F\oplus W^K}{(F\oplus W^K)\cap (F\oplus W^K)^\omega}
$$
and write $\pi:C\to V_0$ for the projection. If $L\subseteq V_\bC\oplus W_\bC$ is a $K$-invariant complex Lagrangian subspace, then 
$$
L_0=\pi(L\cap C_\bC)\subseteq (V\oplus W)_0\otimes \bC
$$
is a complex Lagrangian subspace of $(V\oplus W)_0$.
\end{theorem}

\begin{proof}
Let $S=V^K\oplus W^K=(V\oplus W)^K$. $S$ is a symplectic subspace and the projection map
$$
P_S:V\oplus W\to V^K\oplus W^K;\quad (v,w)\mapsto (v^K,w^K)
$$
is given by averaging over $K$. By Proposition \ref{prop: equivariant set up for reduction},
$$
L=L\cap S_\bC+L\cap S_\bC^\omega.
$$
Since $C$ is closed under the $K$-action, we also have
$$
C=C\cap S+C\cap S^\omega.
$$
Next, 
$$
C\cap S=F^K\oplus W^K
$$
which is clearly co-isotropic in $V^K\oplus W^K$. Finally, for any $(f,w)\in F\oplus W^K$, we have
$$
(f,w)-P_S(f,w)=(f,w)-(f^K,w)=(f-f^K,0)
$$
which lies in $C^\omega$ since $F$ is Lagrangian in $V$. Thus, we can apply Proposition \ref{prop: linear singular reduction} and the result then follows.
\end{proof}

\section{Polarizations}\label{sec: polarizations}

Now that the linear theory of Lagrangian subspaces is complete, we can now go to the global theory of polarizations.

\begin{defs}[Polarization]\label{def: polarization}
Let $(M,\omega)$ be a symplectic manifold. A \textbf{polarization} of $(M,\omega)$ is a complex distribution $P\subseteq T^\bC M$ satisfying the following properties.
\begin{itemize}
    \item $P$ is involutive. That is, if $V,W\in \Gamma(P)$, then $[V,W]\in\Gamma(P)$.
    \item $P$ is Lagrangian. That is, for all $x\in M$, the fibre $P_x\subseteq T^\bC_x M$ is a complex Lagrangian subspace with respect to the complexification $\omega_x^\bC$ of the symplectic form $\omega_x$.
\end{itemize}
Given a polarization $\mathcal{P}\subseteq T^\bC M$, define the polarized functions by
$$
\mathcal{O}_\mathcal{P}(M):=\{f\in C^\infty(M,\bC) \ | \ v(f)=0\text{ for all }v\in \mathcal{P}\}.
$$
We will also write $C_\mathcal{P}^\infty(M)$ for real-valued functions in $\mathcal{O}_\mathcal{P}(M)$.
\end{defs}

We are to think of a polarization as representing the ``momentum directions'' of a symplectic manifold. Locally, this represents half the variables on which a wave-function could depend, so we will only take wave-functions which are constant in those directions. Before we make this notion precise, let us discuss some examples of polarizations and some of their features.

\begin{egs}
Let $(M,\omega)=(\bR^{2n},\omega_{can})$ be standard symplectic $\bR^{2n}$ with Darboux coordinates $(x_1,\dots,x_n,y_1,\dots,y_n)$. This symplectic manifold has a large number of natural polarizations. Viewing the $y$ coordinates as momentum coordinates, we can produce the momentum polarization
$$
\mathcal{P}_{momentum}=\text{span}_\bC\bigg\{\frac{\partial}{\partial y_i} \ \bigg| \ i=1,\dots,n\bigg\}
$$
and, viewing the $x$ coordinates as the position varibles, we get the position polarization
$$
\mathcal{P}_{position}=\text{span}_\bC\bigg\{\frac{\partial}{\partial x_i} \ \bigg| \ i=1,\dots,n\bigg\}.
$$
Both of these polarizations are meaningful for quantization. Note that they are also the complexifications of real distributions and hence are foliations. In both cases, the polarized functions $\mathcal{O}_{\mathcal{P}_{momentum}}(\bR^{2n})$ and $\mathcal{O}_{\mathcal{P}_{position}}(\bR^{2n})$ are equal to the functions on $\bR^{2n}$ which are constant along the leaves and hence can be identified with $C^\infty(\bR^n,\bC)$.

\

We can get other kinds of polarizations via complex structures. Let $I$ denote the standard complex structure on $\bR^{2n}$ identifying it with $\bC^n$. Now consider the holomorphic and anti-holomorphic tangent bundles $T^{0,1}\bR^{2n}$ and $T^{0,1}\bR^{2n}$. These are both involutive Lagrangian subbundles of $T^\bC \bR^{2n}$ and hence are polarizations. In this case, the polarized function $\mathcal{O}_{T^{1,0}\bR^{2n}}(\bR^{2n})$ and $\mathcal{O}_{T^{0,1}\bR^{2n}}(\bR^{2n})$ are equal to the anti-holomorphic and holomorphic functions, respectively. 
\end{egs}

\begin{egs}\label{eg: canonical polarization of cotangent bundle}
    As a generalization of the above, cotangent bundles are also extremely natural places for polarizations to arise. Let $Q$ be a manifold and $\tau:T^*Q\to Q$ the cotangent bundle of $Q$. Since $\tau$ is a surjective submersion  with fibres the cotangent fibres, the complexification of the  kernel
    $$
    \mathcal{P}=\ker(T\tau)\otimes \bC\subseteq T^\bC(T^*Q)
    $$
    is an involutive complex distribution. Furthermore, since for any $\alpha\in T^*Q$ we have
    $$
    \mathcal{P}_\alpha=T_\alpha(T^*_{\tau(\alpha)}Q)\otimes \bC,
    $$
    it immediately follows from the definition of the Liouville form $\theta\in\Omega^1(T^*Q)$ that $\mathcal{P}$ is Lagrangian. Working in cotangent coordinates, we can immediately deduce that the polarized function $\mathcal{O}_{\mathcal{P}}(T^*Q)$ are the complex-valued smooth functions which are constant along the cotangent fibres. Put another way,
    \begin{equation}
    \mathcal{O}_{\mathcal{P}}(T^*Q)=\tau^*C^\infty(Q,\bC).
    \end{equation}
\end{egs}

\begin{egs}\label{eg: toric lagrangian foliation}
    Suppose $T^n$ is an $n$-dimensional torus and $J:(M,\omega)\to\mft^*$ is a toric $T^n$ space with $T^n$ acting freely. Then by Example \ref{eg: free orbit foliation}, the subset $\mft_M\substeq TM$ defined by
    $$
    (\mft_M)_x=\{\xi_M(x) \ | \ \xi\in \mft\}, \quad x\in M
    $$
    is a $T^n$-invariant subbundle. I claim $\mathcal{P}=\mft_M\otimes\bC$ is a polarization. Involutivity follows immediately from $\mft_M$ being the tangent bundle to a foliation. As for Lagrangian, since $T^n$ is abelian we have for any $\xi,\eta\in \mft$ that
    $$
    \omega(\xi_M,\eta_M)=\eta_M(J^\xi)=0
    $$
    since $J^\xi:M\to \bR$ is invariant for any $\xi\in \mft$. Thus, the rank of $\mft_M$ is half the dimension of $M$, it follows that $\mathcal{P}$ is fibre-wise Lagrangian and hence a polarization.
\end{egs}

\begin{prop}\label{prop: polarized functions}
Let $\mathcal{P}$ be a polarization. 
\begin{itemize}
    \item[(1)] Let $f\in C^\infty(M,\bC)$. Then $f\in \mathcal{O}_\mathcal{P}(M)$ if and only if $V_f\in \Gamma(\mathcal{P})$.
    \item[(2)] $\mathcal{O}_\mathcal{P}(M)$ is a Poisson subalgebra of $C^\infty(M,\bC)$.
\end{itemize}
\end{prop}

\begin{proof}
\begin{itemize}
    \item[(1)] Let $f\in C^\infty(M,\bC)$. Then $f\in \mathcal{O}_\mathcal{P}(M)$ if and only if $d_xf$ lies in the annihilator of $\mathcal{P}_x$ for all $x$. Since $df=\omega(V_f,\cdot)$, this occurs if and only if $V_f(x)\in \mathcal{P}_x^\omega=\mathcal{P}_x$ for all $x$. In turn, this occurs if and only if $V_f\in\Gamma(\mathcal{P})$.
    \item[(2)] Let $f,g\in \mathcal{O}_\mathcal{P}(M)$. Then,
    $$
    V_{\{f,g\}}=[V_f,V_g].
    $$
    By (1), $V_f,V_g\in \Gamma(\mathcal{P})$ and so by involutivity, we have $V_{\{f,g\}}\in \Gamma(\mathcal{P})$ which implies $\{f,g\}\in \mathcal{O}_\mathcal{P}(M)$ by (1).
\end{itemize}
\end{proof}

There are two archetypal (and most extensively studied) kinds of polarizations: real and K\"{a}hler. Fix a symplectic manifold $(M,\omega)$ for the following discussion. Now suppose we have a (non-singular) foliation $\mathcal{F}$ of $M$ by Lagrangian submanifolds. Then, by Frobenius' Theorem, the tangent bundle of the foliation $T\mathcal{F}$ is involutive, and clearly $T_x\mathcal{F}\subseteq T_xM$ is Lagrangian. Hence, the complexification
$$
T\mathcal{F}\otimes \bC\subseteq TM\otimes\bC
$$
is a polarization. 

\begin{prop}
Let $\mathcal{F}$ denote a Lagrangian foliation and let $\mathcal{P}=T\mathcal{F}\otimes \bC$. Then the polarized functions $\mathcal{O}_\mathcal{P}(M)$ are precisely the complex functions which are constant along the leaves of $\mathcal{F}$.
\end{prop}

\begin{proof}
By definition, $f\in \mathcal{O}_\mathcal{P}(M)$ if and only if $df|_L=0$ for every leaf $L$. Since the leaves are connected and $df|_L=d(f|_L)$, this implies the result.
\end{proof}

\

Let us now discard the Lagrangian foliation $\mathcal{F}$ and consider the opposite ``extreme'' (in a sense to made precise later): a K\"{a}hler structure. Let $I$ be an integrable complex structure on $M$ which is compatible with $\omega$ in the sense that
$$
g(X,Y):=\omega(X,I(Y))
$$
defines a Riemannian metric on $(M,\omega)$. $I$ naturally extends to an endomorphism
$$
I:TM\otimes \bC\to TM\otimes \bC
$$
which allows us to decompose $TM\otimes \bC$ into its $i$ and $-i$ eigenbundles, respectively
$$
TM\otimes \bC=T_J^{0,1}M\oplus T_J^{1,0}M.
$$
It immediately follows that each of $T_I^{0,1}M,T_I^{1,0}M\subseteq TM\otimes \bC$ are Lagrangian subbundles and they are both involutive by Newlander-Nirenberg \cite{newlander_complex_1957}. Hence, we get two polarizations from $I$: the holomorphic tangent bundle $T_I^{0,1}M$ and the antiholomorphic tangent bundle $T_I^{1,0}M$. 

\begin{prop}
Let $(M,\omega,I)$ be a K\"{a}hler manifold. Then the complex polarized functions with respect to the holomorphic tangent bundle $T^{1,0}_IM$ are the antiholomorphic functions. Conversely, the complex polarized functions with respect to the antiholomorphic tangent bundle $T^{0,1}_IM$ are given by the holomorphic functions on $M$.
\end{prop}

\begin{proof}
This is merely the definition of holomorphic and antiholomorphic functions.
\end{proof}

\

I said that this example coming from K\"{a}hler structures was opposite to that of a Lagrangian foliation. This is in the sense of complex conjugation. Notice that if $\mathcal{F}$ is a Lagrangian foliation, then $T\mathcal{F}\otimes \bC$ is equal to its fibre-wise complex conjugate
$$
\overline{T\mathcal{F}\otimes \bC}=T\mathcal{F}\otimes \bC,
$$
while for a K\"{a}hler structure $I$, the holomorphic and anti-holomorphic tangent bundles $T_I^{1,0}M$ and $T_I^{0,1}M$ intersect trivially and are permuted by complex conjugation. That is,
$$
\overline{T^{0,1}_IM}\cap T_I^{0,1}M=\{0\}.
$$
This leads naturally to the notion of types of polarizations.

\begin{defs}[Polarization Types]
Let $(M,\omega)$ be a symplectic manifold and $\mathcal{P}\subseteq TM\otimes \bC$ a polarization. 
\begin{itemize}
    \item[(1)] $\mathcal{P}$ is called \textbf{real} if $\mathcal{P}$ is equal to its fibre-wise complex conjugate, that is
    $$
    \overline{\mathcal{P}}=\mathcal{P}.
    $$
    \item[(2)] $\mathcal{P}$ is called \textbf{totally complex} if $\mathcal{P}$ trivially intersects its fibre-wise complex conjugate, that is
    $$
    \overline{\mathcal{P}}\cap \mathcal{P}=\{0\}.
    $$
    \item[(3)] If $\mathcal{P}$ is neither real nor totally complex, it is said to be of \textbf{mixed} type.
\end{itemize}
\end{defs}

\begin{remark}\label{rem: not all totally complex are kahler}
In general, totally complex polarizations need not correspond to a K\"{a}hler structure. Indeed, rather they correspond to pseudo-K\"{a}hler structures which means the induced metric $g$ is a pseudo-Riemannian metric. That is, the induced metric $g$ is not demanded to be positive-definite. Indeed, given a totally complex polarization $\mathcal{P}$ we have an internal direct sum
$$
TM\otimes \bC = \mathcal{P}\oplus \overline{\mathcal{P}}.
$$
Define then
$$
I:\mathcal{P}\oplus \overline{\mathcal{P}}\to \mathcal{P}\oplus \overline{\mathcal{P}};\quad v+w\mapsto iv-iw.
$$
If we let $v\in TM$, then we can write $v=p+\overline{p}$ for unique $p\in \mathcal{P}$. By definition,
$$
I(v)=ip-i\overline{p}
$$
Observe that $\overline{I(v)}=I(v)$ which implies that $I(v)$ lies in $TM$. Hence, we get a map
$$
I:TM\to TM
$$
Trivially $I^2=-\id_{TM}$ and $I$ is integrable by Newlander-Nirenberg. Now define symmetric $2$-tensor on $TM$ by
$$
g(v,w)=\omega(v,I(w)).
$$
$g$ need not be positive definite, but it will be non-degenerate, hence defining a pseudo-Riemannian metric.

\

For instance endow $\bC^2$ with complex coordinates $(z_1,z_2)$ and let
$$
\omega=\frac{i}{2}\sum_{j=1}^2 d z_j \wedge d\overline{z_j}
$$
be the canonical symplectic form. Now define
$$
\mathcal{P}=\bigg\langle \frac{\partial}{\partial z_1},\frac{\partial}{\partial \overline{z_2}}\bigg\rangle.
$$
It's then easy to see that $\mathcal{P}$ is involutive and Lagrangian, hence a polarization. Furthermore, we trivially have $\mathcal{P}\cap \overline{\mathcal{P}}=0$. Thus, $\mathcal{P}$ is totally complex. In this setting, we see that the induced symmetric $2$-tensor has the form
$$
g=\frac{\partial}{\partial x_1}\frac{\partial}{\partial x_1}+\frac{\partial}{\partial y_1}\frac{\partial}{\partial y_1}-\frac{\partial}{\partial x_2}\frac{\partial}{\partial x_2}-\frac{\partial}{\partial y_2}\frac{\partial}{\partial y_2}
$$
where concatenation indicates symmetric product, and the $x_j,y_j$ are defined by
$$
x_j=\frac{z_j+\overline{z_j}}{2}\quad\text{ and }\quad y_j=\frac{i}{2}(z_j-\overline{z_j}).
$$
In particular, the induced metric has mixed signature $(++--)$ and thus is pseudo-Riemannian.
\end{remark}

\begin{egs}
Let $(M_1,\omega_1)$ and $(M_2,\omega_2)$ be two symplectic manifolds, $\mathcal{F}$ a Lagrangian foliation of $M_1$ and $J$ a K\"{a}hler structure on $M_2$. Then, the external direct sum
$$
\mathcal{P}=(T\mathcal{F}\otimes \bC)\oplus T_J^{0,1}M_2\subseteq T(M_1\times M_2)\otimes \bC
$$
is a polarization of mixed type.
\end{egs}

As a last property of polarizations, I would like to discuss positivity. Observe that for each $x\in M$, we can define a symmetric sesquilinear form on $\mathcal{P}_x$ as follows:
\begin{equation}\label{polarization pairing}
\langle v,w\rangle : =i\omega_x(v,\overline{w}),
\end{equation}
where $\overline{w}$ denotes the complex-conjugate of $w$.

\begin{defs}\label{def: positive polarization}
A polarization $\mathcal{P}$ is said to be positive if the pairing defined in Equation (\ref{polarization pairing}) is positive-definite.
\end{defs}

As we shall see below, a polarization being positive is equivalent to being K\"{a}hler. The point of providing this alternate characterization is that positivity is an easy condition to verify in the context of taking quotients.

\begin{lem}
    Let $(M,\omega)$ be a symplectic manifold and $\mathcal{P}\subseteq T^\bC M$ a totally complex polarization. Then the following are equivalent.
    \begin{itemize}
        \item[(1)] $\mathcal{P}$ is positive.
        \item[(2)] There exists K\"{a}hler structure $I$ so that $\mathcal{P}=T_I^{0,1}M$.
    \end{itemize}
\end{lem}

\begin{proof}
    First, let us suppose $\mathcal{P}=T_I^{0,1}M$ for some K\"{a}hler structure $I$. Write $g$ for the induced Riemannian metric given by
    $$
    g(v,w)=\omega(v,I(w)).
    $$
    Fixing $x\in M$ and $v\in \mathcal{P}_x$ and plugging $v$ into the inner-product from Equation (\ref{polarization pairing}),
    $$
    \bra v,v\ket=i\omega_x(v,\overline{v})
    $$
    Since $v$ lies in the complexificaiton of $TM$, there exists $w_1,w_2\in T_xM$ so that $v=w_1+iw_2$. From here, it then follows that
    $$
    \bra v,v\ket = 2\omega_x(w_1,w_2).
    $$
    Now, since $I(v)=-iv$, it follows that $I(w_1)=w_2$ and $I(w_2)=-w_1$. In particular,
    $$
    \bra v,v\ket =2\omega_x(w_1,I(w_1))=2g_x(w_1,w_1).
    $$
    In particular, if $w_1\neq 0$, we have $\bra v,v\ket>0$ and hence $\mathcal{P}$ is positive.

    \ 

    For the converse, let us suppose that $\mathcal{P}$ is positive. First I claim that $\mathcal{P}$ is totally complex. Indeed, if $v\in \mathcal{P}_x\cap \overline{\mathcal{P}_x}$ for some $x\in M$, then it follows that $\overline{v}\in \mathcal{P}_x$. Hence,
    $$
    \bra v,v\ket=i\omega_x(v,\overline{v})=0
    $$
    since $\mathcal{P}_x$ is Lagrangian. Since $\bra\cdot,\cdot\ket$ is assumed to be positive-definite, it follows that $v=0$. Thus, $\mathcal{P}_x\cap \overline{\mathcal{P}_x}=0$ and so $\mathcal{P}$ is totally complex. 

    \ 

    Thus, by the discussion in Remark \ref{rem: not all totally complex are kahler}, there exists an integrable complex structure $I:TM\to TM$ such that $\mathcal{P}=T^{0,1}_IM$. To finish, we need to show that
    $$
    g(v,w):=\omega(v,I(w))
    $$
    is positive definite. This is straightforward as for any $v\in TM$ first see that
    $$
    I(v+iI(v))=I(v)-iv=-i(v+iI(v))
    $$
    and hence $v+iI(v)\in \mathcal{P}$ and furthermore
    $$
    g(v,v)=\omega(v,I(v))= \frac{i}{2}\omega(v+iI(v),v-iI(v))=\frac{\bra v+iI(v),v+iI(v)\ket}{2}
    $$
    If $v\neq 0$, it then immediately follows from positivity that $g(v,v)>0$ and thus $I$ is a K\"{a}hler structure.
\end{proof}

\section{Images of Vector Bundles and Complex Distributions}\label{sec: vector bundle images}

In the previous section, we derived the infinitesimal picture needed for singular reduction. Now we will need to piece everything together making use of integrability. The main ideas here are coming from vector bundle and complex distribution theory.

\

First, we need to ask a somewhat simple question. When is the image of a vector bundle again a vector bundle? In particular, if 
$$
\begin{tikzcd}
A\arrow[r,"p"]\arrow[d] & B\arrow[d]\\
M\arrow[r,"f"] & N
\end{tikzcd}
$$
is a vector bundle morphism with $f$ a surjective submersion, when is $p(A)\subseteq B$ a vector subbundle over $N$? Unless $f$ is a diffeomorphism, it does not suffice to just assume that $p$ is constant rank.

\begin{egs}\label{eg: constant rank not enough}
To demonstrate that a vector bundle morphism with constant rank need not have a vector bundle as its image, consider the universal covering map
$$
f:\bR\to S^1\subseteq\bC;\quad \theta\mapsto e^{i\theta}
$$
Let $B=S^1\times \bR^2$ be the trivial rank $2$ bundle over $S^1$ and define $A\subseteq \bR\times \bR^2$ as follows. For each $\theta\in \bR$, let $R_\theta\in \text{SO}(2)$ be the counter-clockwise rotation by $\theta$. Choosing a line
$$
L=\{(x,0) \ | \ x\in \bR\}.
$$
Now define $A\subseteq \bR\times \bR^2$ by
$$
A_\theta:=R_{\theta/4}(L).
$$
The restriction of the projection $pr_1:\bR\times \bR^2\to \bR$ onto the first factor defines a vector bundle structure
$$
\pi_A:A\to \bR.
$$
Write $\pi_B:S^1\times \bR^2\to S^1$ for the projection on the first factor as well. Since $A_\theta\subseteq\bR^2$ for each $\theta$, we get a canonical vector bundle morphism
$$
\begin{tikzcd}
A\arrow[r,"p"]\arrow[d,"\pi_A"] & B\arrow[d,"\pi_B"]\\
M\arrow[r,"f"] & N
\end{tikzcd}
$$
given by fibre-wise inclusion. Then, we see $p(A)$ is not a vector subbundle of $B$. Letting
$$
X=\{(x,y)\in \bR^2 \ | \ x=0\text{ or }y=0\},
$$
we see that 
$$
p(A)\cap (\{(1,0)\}\times \bR^2)=\{(1,0)\}\times X
$$
and over every other point in $S^1$ will be a rotated copy of $X$. Hence, $p(A)$ is not a vector subbundle.
\end{egs}

Constant rank does not suffice to ensure the image of a vector bundle morphism is a vector bundle, so we need to impose a further condition.

\begin{lem}\label{lem: characterization when image of vector bundle is a vector bundle again}
Let 
$$
\begin{tikzcd}
A\arrow[r,"p"]\arrow[d] & B\arrow[d]\\
M\arrow[r,"f"] & N
\end{tikzcd}
$$
be a vector bundle morphism with $f$ a surjective submersion. Then $p(A)\subseteq B$ is a vector subbundle of $B$ with 
$$
p(A)_{f(x)}=p(A_x)
$$
for any $x\in M$ if and only if
\begin{itemize}
    \item[(1)] The map
    $$
    M\to \bZ;\quad x\mapsto \dim p(A_x)
    $$
    is constant.
    \item[(2)] If $f(x)=f(y)$, then $p(A_x)=p(A_y)$.
\end{itemize}
\end{lem}

\begin{proof}
First, suppose $p(A)\subset B$ is a vector subbundle with
$$
p(A_x)=p(A)_{f(x)}
$$
for all $x\in M$. Then clearly the dimension of $p(A_x)$ must be constant and if $x,y\in M$ satisfy $f(x)=f(y)$, then $p(A_x)=p(A_y)$.

\

So let us now assume the converse. That is, assume $p(A_x)$ is constant in $x$ and $p(A_x)=p(A_y)$ for any $x,y\in M$ with $f(x)=f(y)$. To show that $p(A)$ is a vector subbundle, it suffices to show that around any point $n\in N$, we can find a neighbourhood $U\subseteq N$ of $n$ and point-wise linearly independent sections $\tau_1,\dots,\tau_k\in\Gamma(U,B)$ which span $p(A)|_U$.

\

Thus, making use of the fact that $f$ is a submersion, we can pass to local coordinates and assume $M=\bR^n\times \bR^m$, $N=\bR^n$, and $f:\bR^n\times\bR^m\to \bR^n$ is the projection onto the first factor. Note that since $p$ is assumed constant rank, the kernel $\ker(p)\subseteq A$ is a vector subbundle. Thus, we can find sections
$$
\sigma_1,\dots,\sigma_k,\sigma_{\ell+1},\dots,\sigma_{r}\in\Gamma(\bR^n\times \bR^m,A)
$$
so that $\sigma_{\ell+1},\dots,\sigma_{r}$ is a frame for $\ker(p)$. For each $i=1,\dots,k$, define
$$
\tau_i:\bR^n\to B;\quad x\mapsto p\circ\sigma_i(x,0)
$$
By construction, $\tau_1,\dots,\tau_k$ span $p(A)$ and are point-wise linearly independent. Hence the claim.
\end{proof}

\begin{defs}\label{def: constant along fibres}
Let 
$$
\begin{tikzcd}
A\arrow[r,"p"]\arrow[d] & B\arrow[d]\\
M\arrow[r,"f"] & N
\end{tikzcd}
$$
be a vector bundle morphism. If $p(A_x)=p(A_y)$ for any $x,y\in M$ with $f(x)=f(y)$, say $p$ is \textbf{constant along the fibres of }$f$.
\end{defs}

\begin{remark}
    Returning to the vector bundle morphism from Example \ref{eg: constant rank not enough} 
    $$
    \begin{tikzcd}
        A\subseteq \bR\times \bR^2\arrow[r,"p"]\arrow[d] & S^1\times \bR^2\arrow[d]\\
        \bR\arrow[r,"f"] & S^1
    \end{tikzcd}
    $$
    where $f(\theta)=e^{i\theta}$ is the universal covering map and $A$ is defined by
    $$
    A_\theta:=R_{\theta/4}(L)\subseteq \bR^2
    $$
    where $L=\{(x,0)\in\bR^2 \ | \ x\in \bR\}$ and $R_{\theta/4}\in \text{SO}(2)$ is the counter-clockwise rotation around the origin by $\theta/4$.

    \
    
    The reason why $p(A)$ was not a vector subbundle of $S^1\times \bR^2$ was precisely because $p$ violates hypothesis (ii) in Lemma \ref{lem: characterization when image of vector bundle is a vector bundle again}. Indeed, consider the fibres $A_0$ and $A_{2\pi}$. We have
    $$
    p(A_0)=\{f(0)\}\times L=\{1\}\times L
    $$
    whereas
    $$
    p(A_{2\pi})=\{f(2\pi)\}\times R_{\pi/2}(L)=\{1\}\times R_{\pi/2}(L).
    $$
    Clearly $p(A_0)\neq p(A_{2\pi})$ even though $f(0)=f(2\pi)$. 
\end{remark}

\begin{egs}
If we do not demand $p$ be constant along the fibres of the base map, we can also obtain examples where $p(A)$ is a vector bundle. Indeed, consider now $f:S^1\times \bR\to \bR$ given by projection onto the second factor. Let $B=\bR\times \bR^2$ be the trivial rank-2 vector bundle and let us define $A\subseteq (\text{SO}(2)\times \bR)\times \bR^2$ as follows. Let 
$$
L=\{(x,0)\in \bR^2  \ | \ x\in \bR\}
$$
and for each $(R,t)\in \text{SO}(2)\times \bR$, define
$$
A_{(R,t)}:=R(L)\subseteq \bR^2.
$$
These fibres then piece together to a real line bundle $A\to \text{SO}(2)\times \bR$. Letting $p:A\to B$ be given by fibre-wise inclusion, we see that $p$ is constant rank, but if $R,R'\in \text{SO}(2)$ are distinct elements and are not inverses, then for any $t\in \bR$ the two lines $R(L)$ and $R'(L)$ are not equal and thus
$$
A_{(R,t)}\neq A_{(R',t)}
$$
and hence their images are not equal either under $p$. However, it's easy to see that
$$
p(A)=B
$$
and hence is a vector subbundle. Notice that in this case, $p(A_{(R,t)})\neq p(A)_t$ for any $(R,t)\in \text{SO}(2)\times \bR$.
\end{egs}

\begin{cor}\label{cor: lifting sections through splittings}
Suppose $p:A\to B$ is a vector bundle morphism which is constant along the fibres of $f$. Then any choice of splitting
$$
A=\ker(p)\oplus A'
$$
defines an injective map
$$
p^!:\Gamma(p(A))\to \Gamma(A');\quad \sigma\mapsto p^!\sigma
$$
such that for any $\sigma\in \Gamma(p(A))$, the diagram commutes
$$
\begin{tikzcd}
A'\arrow[r,"p"] & p(A)\\
M\arrow[r,"f"]\arrow[u,"p^!\sigma"] & N\arrow[u,"\sigma"]
\end{tikzcd}
$$
\end{cor}

\begin{proof}
Clearly $p$ induces an isomorphism of vector bundles over $M$ between the pullback bundle $f^{-1}(p(A))$ and $A'$,
\begin{equation}\label{eq: pullback iso}
f^{-1}(p(A))\cong A'
\end{equation}
Thus, if $\sigma\in \Gamma(p(A))$ is a section, we get an induced pull-back section $f^{-1}\sigma\in\Gamma(f^{-1}p(A))$ which gets identified with a section $p^!\sigma\in \Gamma(A')$ via the isomorphism in Equation (\ref{eq: pullback iso}). 
\end{proof}

The immediate consequence of our discussion of general vector bundles is that we can now push-forward distributions by surjective submersions provided certain conditions are met.

\begin{lem}\label{lem: surjective submersion gives us distributions}
Suppose $M$ and $N$ are smooth manifold, $P\subseteq T^\bC M$ is a complex involutive subbundle, and $\pi:M\to N$ a surjective submersion. Write $V^\bC M:=\ker(T^\bC\pi)$. Suppose the following conditions hold.
\begin{itemize}
    \item[(i)] $T^\bC\pi|_P:P\to T^\bC N$ is constant along the fibres of $\pi$ and is constant rank.
    \item[(ii)] For any sections $X\in\Gamma(V^\bC M)$ and $Y\in\Gamma(P)$ we have
    $$
    [X,Y]\in\Gamma(P).
    $$
\end{itemize}
Then
$$
\pi_*P:=T\pi(P)\subseteq T^\bC N
$$
is an involutive complex distribution on $N$.
\end{lem}

\begin{proof}
The map
$$
T\pi|_P:P\to T^\bC N
$$
is a constant rank vector bundle morphism which is constant along fibres over a surjective submersion. By Lemma \ref{lem: characterization when image of vector bundle is a vector bundle again}, its image $\pi_*P$ is a complex subbundle of $T^\bC N$. Now to show $\pi_*P$ is involutive. 

\

Let $X,Y\in\Gamma(\pi_*P)$ be two sections and let $X^L,Y^L\in\Gamma(P)$ be two lifts as in Corollary \ref{cor: lifting sections through splittings}. Then $[X^L,Y^L]\in\Gamma(P)$ since $P$ is involutive. Since $X^L\sim_\pi X$ and $Y^L\sim_\pi Y$, we have $[X^L,Y^L]\sim_\pi [X,Y]$. Thus, $[X,Y]\in\Gamma(\pi_*P)$ as well.
\end{proof}

\section{The Non-Singular Reduction of Polarizations}\label{sec: nonsing reduce polarizations}

For all that follows, let $G$ be a connected Lie group, $J:(M,\omega)\to \mfg^*$ a proper Hamiltonian $G$-space, and $\mathcal{P}\subseteq TM\otimes \bC$ a $G$-invariant polarization. 

\begin{theorem}\label{thm: nonsing reduction of polarizations}
Suppose $0$ is a regular value of $J$, $J^{-1}(0)\neq \emptyset$,  and so that $G$ acts freely on $J^{-1}(0)$. Let $M_0=J^{-1}(0)/G$ be the reduced space and $\pi_0:J^{-1}(0)\to M_0$ the quotient map. If 
$$
\mathcal{P}_{J^{-1}(0)}:=\mathcal{P}\cap T^\bC J^{-1}(0)\subseteq T^\bC M
$$
is an embedded submanifold of $T^\bC M$, then
$$
\mathcal{P}_0:=(\pi_0)_*(\mathcal{P}_{J^{-1}(0)})
$$
is a polarization of $M_0$.
\end{theorem}

\begin{proof}
Since $\mathcal{P}_{J^{-1}(0)}$ is an embedded submanifold of $T^\bC M$, it immediately follows that $\mathcal{P}_{J^{-1}(0)}$ is a complex vector subbundle of $T^\bC J^{-1}(0)$. Furthermore, since $\mathcal{P}_{J^{-1}(0)}$ is given by the intersection of two subbundles which are closed under Lie brackets, $\mathcal{P}_{J^{-1}(0)}$ is a complex distribution of $J^{-1}(0)$.

\

We are assuming that $\mathcal{P}$ is $G$-invariant, and so it follows immediately that the vector bundle morphism $T\pi_0:\mathcal{P}_{J^{-1}(0)}\to T^\bC M_0$ as in the below diagram,
\begin{equation}\label{eq: reduction vb morphism}
\begin{tikzcd}
    \mathcal{P}_{J^{-1}(0)}\arrow[r,"T\pi_0"]\arrow[d] & T^\bC M_0\arrow[d]\\
    J^{-1}(0)\arrow[r,"\pi_0"] & M_0
\end{tikzcd} 
\end{equation}
is constant along the fibres of $\pi_0$ in the sense of Definition \ref{def: constant along fibres}. Furthermore, since $J^{-1}(0)$ is a coisotropic submanifold and
$$
T_x (G\cdot x)=(T_x J^{-1}(0))^{\omega_x}
$$
for all $x$, it follows from the linear reduction statement in Proposition \ref{prop: linear coisotropic reduction} that $T_x\pi_0((\mathcal{P}_{J^{-1}(0)})_x)\subseteq T^\bC_{\pi_0(x)}M_0$ is a Lagrangian subspace of $T^\bC_{\pi_0(x)}M_0$ for any $x$. From this, we obtain two facts at once. First, the vector bundle morphism in Equation (\ref{eq: reduction vb morphism}) is constant rank. Making use of Lemma \ref{lem: surjective submersion gives us distributions}, we deduce that $(\pi_0)_*(\mathcal{P}_{J^{-1}(0)})=\mathcal{P}_0\subseteq T^\bC M_0$ is a complex distribution on $M_0$. Second, since $\mathcal{P}_0$ is fibre-wise Lagrangian, it must be the case that $\mathcal{P}_0$ is a polarization.

\end{proof}

\begin{defs}\label{def: reducible polarization}
Given a $G$-invariant polarization $\mathcal{P}$, say $\mathcal{P}$ is \textbf{reducible} $\mathcal{P}_{J^{-1}(0)}:=\mathcal{P}\cap T^\bC J^{-1}(0)$ is an embedded submanifold of $T^\bC M$.
\end{defs}

In light of this theorem, from now on we shall be assuming that $0$ is a regular value of $J$ and that $G$ acts freely on $J^{-1}(0)$. The main question we will be investigating further is when $\mathcal{P}$ is reducible and the implications therein.

\begin{prop}\label{prop: positive polarization intersect orbits trivially}
Let $\mfg_{J^{-1}(0)}\subseteq TJ^{-1}(0)$ denote the tangent bundle of the foliation of $J^{-1}(0)$ by orbits of the free $G$-action. If $\mathcal{P}$ is fibre-wise positive as in Definition \ref{def: positive polarization}, then
$$
\mathcal{P}\cap \mfg_{J^{-1}(0)}=\{0\}.
$$
\end{prop}

\begin{proof}
This follows from the fact that for each $x\in J^{-1}(0)$, we have
$$
(T_xJ^{-1}(0))^{\omega_x}=T_x (G\cdot x)=(\mfg_{J^{-1}(0)})_x.
$$
So, by Proposition \ref{prop: positive polarization trivially intersects orbits}, $\mathcal{P}_x\cap T^\bC_x (G\cdot x)=\{0\}$ since $\mathcal{P}$ is positive.
\end{proof}

\begin{cor}[\cite{guillemin_geometric_1982}]\label{cor: positive polarizations are reducible non-singular}
If $\mathcal{P}$ is a positive polarization, then it is reducible. Furthermore, the reduced polarization
$$
\mathcal{P}_0=(\pi_0)_*(\mathcal{P}_{J^{-1}(0)})
$$
is a positive polarization on $M_0$.
\end{cor}

\begin{proof}
By Proposition \ref{prop: positive polarization intersect orbits trivially}, the sum
$$
\mathcal{P}|_{J^{-1}(0)}+\mfg_{J^{-1}(0)}\otimes \bC
$$
defines a vector subbundle of $T^\bC M|_{J^{-1}(0)}$. Working point-wise, we see the symplectic dual is given by
$$
(\mathcal{P}|_{J^{-1}(0)}+\mfg_{J^{-1}(0)}\otimes \bC)^\omega=\mathcal{P}\cap T^\bC J^{-1}(0)=\mathcal{P}_{J^{-1}(0)}.
$$
Hence, $\mathcal{P}_{J^{-1}(0)}$ is a subbundle and thus $\mathcal{P}$ is reducible. The reduced polarization $\mathcal{P}_0$ being positive-definite is a consequence of Corollary \ref{cor: linear reduction of positive is positive}.
\end{proof}

Applying this result together with Corollary \ref{cor: linear reduction preserves type}, allows us to conclude the following.

\begin{cor}\label{cor: reduction of K\"{a}hler structures}
If $I$ is a  K\"{a}hler structure, then it canonically induces a K\"{a}hler structure on $M_0$.
\end{cor}

\begin{egs}
    It follows from both Theorem \ref{thm: delzant's theorem} and Corollary \ref{cor: reduction of K\"{a}hler structures} that all toric manifolds are K\"{a}hler manifolds and, in fact, come equipped with an invariant K\"{a}hler structure. Let $J:(M,\omega)\to \mft^*$ be a toric $\bT$-manifold with $n=\dim(\bT)$. Then by Theorem \ref{thm: delzant's theorem} $J:(M,\omega)\to \mft^*$ is $\bT$-equivariantly diffeomorphic to the reduction of standard symplectic $\bC^N$ with respect to the action of some subtorus $K\leq (S^1)^N$. 

    \ 

    Now let $I_{can}$ be the standard complex structure on $\bC^N$. Clearly $I_{can}$ is invariant under the full $(S^1)^N$ action, hence also under the action of the subtorus $K$. Therefore by Corollary \ref{cor: reduction of K\"{a}hler structures}, $I_{can}$ induces a K\"{a}hler structure $I_\Delta$ on $M_\Delta$. Note that via Delzant's construction, the action of $\bT$ on $M$ can be identified with the action of $(S^1)^N/K$ on the reduced space. In particular, it then follows that $I_\Delta$ must be invariant under the action of $\bT$ as well.
\end{egs}

On the other extreme is of course real polarizations. These are most definitely not positive polarizations, so we will need to approach the question of when they are reducible in a slightly different fashion. Before we give the statement, let us recall a standard definition (for instance, \cite[Example 2.7, (c)]{hudson_multiplicative_2024}.)

\begin{defs}\label{def: intersect cleanyl}
    Let $M$ be a manifold and $A,B\subseteq M$ two submanifolds. We say $A$ and $B$ \textbf{intersect cleanly} if
    \begin{itemize}
        \item[(i)] $A\cap B$ is a submanifold of $M$.
        \item[(ii)] $T_x(A\cap B)=T_xA\cap T_xB$ for all $x\in A\cap B$.
    \end{itemize}
\end{defs}

\begin{prop}
Suppose $\mathcal{F}=\{L_x\}_{x\in M}$ is a Lagrangian foliation of $M$. If the leaves of $\mathcal{F}$ intersect $J^{-1}(0)$ cleanly and with constant rank, then the partition
$$
\mathcal{F}_{J^{-1}(0)}=\bigcup_{L\in \mathcal{F}}\pi_0(L\cap J^{-1}(0))
$$
is a regular foliation of $J^{-1}(0)$ and
$$
T\mathcal{F}\cap TJ^{-1}(0)=T\mathcal{F}_{J^{-1}(0)}
$$
\end{prop}

\begin{proof}
By construction, $T\mathcal{F}_{J^{-1}(0)}$ is an involutive subbundle of $TJ^{-1}(0)$. Thus, by Frobenius' Theorem (for instance, see \cite[Theorem 19.12]{lee_introduction_2013}), $T\mathcal{F}_{J^{-1}(0)}$ is the tangent bundle of a regular foliation on $J^{-1}(0)$.
\end{proof}

\begin{cor}
If $\mathcal{P}$ is a real polarization and the Lagrangian leaves of $\mathcal{P}$ intersect $J^{-1}(0)$ cleanly with constant rank, then $\mathcal{P}$ is reducible and its reduction $\mathcal{P}_0$ is also real. Furthermore, if $x\in J^{-1}(0)$, $L_x\in\mathcal{F}$ the Lagrangian leaf through $x$, and $\widetilde{L}_{[x]}\subseteq M_0$ the leaf through $[x]\in M_0$, then
$$
\pi_0^{-1}(\widetilde{L}_{[x]})=L_x\cap J^{-1}(0).
$$
\end{cor}

\section{Reduction of Polarization on Cotangent Bundles}\label{sec: reducing polarization cotangent bundles}

Our other natural family of examples we've been examining in this thesis are cotangent bundles as they are equipped with natural real polarizations. Let $G$ be a connected Lie group, $Q$ a proper $G$-space, and $\tau_Q:T^*Q\to Q$ the cotangent bundle. As we saw in Example \ref{eg: canonical polarization of cotangent bundle}, $T^*Q$ has a natural polarization $\mathcal{P}_Q\subseteq T(T^*Q)\otimes\bC$ given by the complexification of the kernel of the derivative of the bundle map $\tau_Q$. Recall that in the case that $Q$ has only one orbit-type, we showed in Theorem \ref{thm: reduced symplectomorphism} that the reduction $(T^*Q)_0$ is naturally symplectomorphic to the cotangent bundle $T^*(Q/G)$. This of course also has a natural polarization $\mathcal{P}_{Q/G}$ given by the complexification of the kernel of the bundle map $\tau_{Q/G}:T^*(Q/G)\to Q/G$. 

\begin{prop}\label{prop: nonsing reduction of cotangent polarization}
    Suppose $Q$ has only one orbit-type and let
    $$
    \phi:T^*(Q/G)\to (T^*Q)_0
    $$
    be the symplectomorphism from Theorem \ref{thm: reduced symplectomorphism}. Then $\mathcal{P}_Q$ is reducible in the sense of Definition \ref{def: reducible polarization}. Furthermore, the induced polarization $\mathcal{P}_0\subseteq T^\bC(T^*Q)_0$ satisfies
    $$
    T^\bC\widetilde{\phi}(\mathcal{P}_{Q/G})=\mathcal{P}_0.
    $$
\end{prop}

\begin{proof}
    Clearly
    $$
    \mathcal{P}_Q\cap T^\bC J^{-1}(0)=\ker(T\tau_Q|_{J^{-1}(0)})\otimes \bC
    $$
    Hence, $\mathcal{P}_Q$ is reducible and
    $$
    \mathcal{P}_0=T^\bC\pi_0(\mathcal{P}_Q\cap T^\bC J^{-1}(0))
    $$
    is a real polarization on $(T^*Q)_0$. Writing
    $$
    \overline{\tau_Q}:(T^*Q)_0\to Q/G
    $$
    for the induced projection, $(T^*Q)_0\to Q/G$ is a vector bundle over $Q/G$ and $\mathcal{P}_0$ coincides with the canonical foliation of $(T^*Q)_0$ by the fibres of $\overline{\tau_Q}$. Furthermore, the map
    $$
    \widetilde{\phi}:T^*(Q/G)\to (T^*Q)_0
    $$
    is also an isomorphism of vector bundles. Thus, $\widetilde{\phi}$ maps $\mathcal{P}_{Q/G}$ to $\mathcal{P}_0$.
\end{proof}

\cleardoublepage
\chapter{Geometric Quantization}\label{ch: geometric quantization}

Quantization is a term that refers to a wide variety of procedures and techniques that take mathematical objects corresponding to classical physics, i.e. the physics of Newton, Hooke, Lagrange, etc., and produce a quantum analogue. This is quite a vague introduction, but it must be by necessity given the great depth of different versions of ``quantization'' which tend to try and reproduce different aspects of quantum mechanical theory in vastly different fashions. We will now get more specific with a discussion of geometric quantization, at least its simplest form.

\

The general idea is that classical physical systems are to be represented by symplectic manifolds $(M,\omega)$. The manifold $M$ itself is to be thought of as a ``phase-space'' for some system, i.e. the set of all permissible positions and momenta that a system can take on. The observables are then modelled by the smooth functions $C^\infty(M)$, where a particular function $f\in C^\infty(M)$ is to be thought of as representing a some quantity that can be measured directly via experiment. The symplectic form then induces an algebraic structure of a Poisson algebra $\{\cdot,\cdot\}$ on $C^\infty(M)$. Provided we are given a distinguished function $H\in C^\infty(M)$ corresponding to the energy of the system, we can then determine how any observable $f\in C^\infty(M)$ evolves using the Hamilton-Jacobi equation:
$$
\frac{d}{dt}(f\circ \phi^H(t))=\{H,f\}(\phi^H(t)),
$$
where $\phi^H(t)$ denotes the Hamiltonian flow of $H$. 

\

On the quantum side, we make use of Hilbert spaces and self-adjoint operators instead. Here we have a Hilbert space $\mathcal{H}$ with elements of the projectivization $\mathbb{\mathcal{H}}$ corresponding to ``wave-function'' which are to be thought of as the states of a quantum mechanical system. The observables in this context are then the self-adjoint operators $S(\mathcal{H})$ (or, more generally the essentially self-adjoint operators \cite{hall_quantum_2013}) with the possible values an operator $A\in S(\mathcal{H})$ represent being given by elements of the spectrum. Similar to the case, the quantum observables $S(\mathcal{H})$ also have a natural Lie algebraic structure, that of the commutator. Given $A,B\in S(\mathcal{H})$ and $\hbar>0$, define
$$
[A,B]_\hbar:=\frac{i}{\hbar}[A,B],
$$
where $[\cdot,\cdot]$ is the commutator of operators. In this case, if we have a distinguished observable $H\in S(\mathcal{H})$ corresponding to the energy called the \textbf{Hamiltonian}, we can also track how another observable $A\in S(\mathcal{H})$ changes over time by Schrodinger's equation
$$
\frac{d}{dt}A(t)=[H,A(t)]_\hbar,
$$
where
$$
A(t)=e^{itH/\hbar}Ae^{-itH/\hbar}
$$
See Hall \cite{hall_quantum_2013} for more information about this formalism.

\

We can see there are some major parallels between the symplectic formalism of classical mechanics and the Hilbert space formalism of quantum mechanics. Namely both have a state space, a set of observables with a real Lie algebra structure, and the necessity of a choice of a distinguished operator corresponding to energy which determines how an observable changes through the Lie bracket. This leads to the following ``definition'' of quantization.

\
Let $(M,\omega)$ be a symplectic manifold. A quantization of $(M,\omega)$ consists of a Hilbert space $\mathcal{H}$ and a Lie algebra homomorphism
$$
\phi:C^\infty(M)\to S(\mathcal{H})
$$
such that $\phi(1)=\id_{\mathcal{H}}$. Furthermore, the pair $(\mathcal{H},\phi)$ should be ``minimal''

\

Now, what exactly is meant by minimal will depend on who you talk to. There is the classical definition of Dirac \cite{dirac_principles_1982} which further imposes that a complete set of functions on $(M,\omega)$ should be mapped to a complete set of self-adjoint operators on $\mathcal{H}$. As was shown by Groenewold and van Hove \cite{groenewold_principles_1946}, this is impossible for Euclidean spaces. However, Dirac quantizations do exist for tori as shown by Gotay \cite{gotay_full_1995}, so it's not an entirely impossible dream. This thesis is not concerned with Dirac quantization, so we shall leave it at that. 

\

Rather, let us impose a different version of ``minimality''.

\begin{defs}[Quantization Procedure]\label{def: quantization procedure}
    A quantization procedure for a collection $\mathcal{C}$ of symplectic manifolds consists of the following data.
    \begin{itemize}
        \item[(i)] A map 
        $$
        \mathcal{Q}:\mathcal{C}\to\text{Hilbert Spaces};\quad (M,\omega)\mapsto \mathcal{Q}(M,\omega)
        $$
        assigning a Hilbert space to each symplectic manifold in the collection $\mathcal{C}$.
        \item[(ii)] To each symplectic manifold $(M,\omega)\in \mathcal{C}$ there is a Lie subalgebra $\mathcal{A}\subseteq C^\infty(M)$ called the \textbf{quantizable functions}.
        \item[(iii)] To each symplectic manifold $(M,\omega)$ there is a Lie algebra homomorphism
        $$
        \mathcal{Q}:\mathcal{A}\to S(\mathcal{Q}(M,\omega)).
        $$
    \end{itemize}
    This data must satisfy the following axioms
    \begin{itemize}
        \item[(1)] If $1_M\in C^\infty(M)$ denotes the constant $1$ function, then $1_M\in \mathcal{A}$ and $\mathcal{Q}(1_M)=\id_{\mathcal{Q}(M,\omega)}$.
        \item[(2)] Standard symplectic $(\bR^{2n},\omega_0)\in \mathcal{C}$ and 
        $$
        \mathcal{Q}(\bR^{2n},\omega_0)=L^2(\bR^n).
        $$
    \end{itemize}
\end{defs}

I will now outline a particular quantization procedure called ``Geometric Quantization'' which ultimately goes back to the work of Kostant \cite{kostant_orbits_2009}, Kirillov \cite{kirillov_unitary_1962}, and Souriau \cite{souriau_quantification_1967}. 

\begin{remark}
    As a final note before we move on, we will be working in so-called ``natural units'' where we set $\hbar=1$ as I don't want to have to keep track of unnecessary constants as we go along. In particular, the Lie bracket on self-adjoint operators will now have the form
    $$
    [A,B]_1=i[A,B].
    $$
    I leave it as an exercise to the reader to insert $\hbar$ in the correct positions in all the definitions to follow.
\end{remark}

\section{Prequantum Line Bundles}\label{sec: prequantum}

Prequantum line bundles first emerged from the work of Kirillov \cite{kirillov_unitary_1962} as a way of obtaining unitary representations of Lie groups via their coadjoint orbits. In the subsequent work of Kostant \cite{kostant_orbits_2009} and Souriau \cite{souriau_quantification_1967}, they were extended to more general symplectic manifolds as a potential first step for achieving a systematic framework for quantization. Let us now see how examine how one could derive such objects before giving the formal definition.

\ 

Suppose for simplicity that we have a compact symplectic manifold $(M,\omega)$. We have a canonical Hilbert space associated to $(M,\omega)$, namely the space of compactly supported complex-valued functions $C^\infty(M,\bC)$ with the inner product
$$
\bra f,g\ket=\int_M f\overline{g} \frac{\omega^n}{n!},
$$
where $\overline{g}$ denotes the pointwise complex conjugate and $2n=\dim(M)$. The symplectic structure gives us a way of mapping functions to operators, namely Hamiltonian vector fields. Observe that if $f\in C^\infty(M)$ and $g_1,g_2\in C^\infty(M,\bC)$, then
\begin{align*}
    \bra V_fg_1,g_2\ket&=\int_M (V_fg_1)\overline{g_2}\frac{\omega^n}{n!}\\
    &=\int_M(V_f(g_1\overline{g_2})-g_2\overline{V_fg_2})\frac{\omega^n}{n!}\\
    &=\int_M V_f(g_1\overline{g_2})\frac{\omega^n}{n!}-\bra g_1,V_fg_2\ket
\end{align*}
Using the fact that $L_{V_f}\omega=0$ and Stokes' Theorem, we see that
$$
0=\int_MV_f\bigg(g_1\overline{g_2}\frac{\omega^n}{n!}\bigg)=\int_M V_f(g_1\overline{g_2})\frac{\omega^n}{n!}.
$$
Thus, the Hamiltonian vector field $V_f$ is a skew-adjoint operator on $C^\infty(M,\bC)$. Given that mapping a function to its Hamiltonian vector field is a Lie algebra homomorphism, we then obtain our first guess for a quantization map, namely
\begin{equation}\label{eq: first attempt at quantization}
C^\infty(M)\to S(C^\infty(M,\bC));\quad f\mapsto -iV_f,
\end{equation}

This already gets us part of the way to a quantization scheme as in Definition \ref{def: quantization procedure}, however we quickly run into a problem here. Namely, if $1_M\in C^\infty(M)$ is the constant $1$ function, then $V_{1_M}=0$, which violates Axiom (1) in Definition \ref{def: quantization procedure}. This can easily be remedied by adding another term to the map in Equation (\ref{eq: first attempt at quantization}). For any $f\in C^\infty(M)$, define the pointwise multiplication operator
$$
m_f:C^\infty(M)\to S(C^\infty(M,\bC));\quad g\mapsto m_fg,
$$
where $m_fg(x)=f(x)g(x)$ for all $x\in M$ and $g\in C^\infty(M)$. We then get a new map to the self-adjoint operators
\begin{equation}\label{eq: second attempt at quantization}
    C^\infty(M)\to S(C^\infty(M,\bC));\quad f\mapsto -iV_f+m_f
\end{equation}
which now maps $1_M$ to the identity. However, the map in Equation (\ref{eq: second attempt at quantization}) is not a Lie algebra homomorphism! Indeed, given $f,g\in C^\infty(M)$, we have
\begin{align*}
    i[-iV_f+m_f,-iV_g+m_g]&=-i[V_f,V_g]+[V_f,m_g]+[m_f,V_g]+i[m_f,m_g]\\
    &=-iV_{\{f,g\}}+[V_f,m_g]+[m_f,V_g]+i[m_f,m_g]
\end{align*}
By acting against a test function $f\in C^\infty(M,\bC)$, we obtain that 
\begin{align*}
    [V_f,m_g]&=m_{\{f,g\}}\\
    [m_f,m_g]&=0.
\end{align*}
Thus,
\begin{align*}
    i[-iV_f+m_f,-iV_g+m_g]=-iV_{\{f,g\}}+2m_{\{f,g\}}\neq -iV_{\{f,g\}}+m_{\{f,g\}}.
\end{align*}
To remedy this situation, suppose that $\omega$ is exact, that is $\omega=d\theta$ for some $1$-form $\theta\in\Omega^1(M)$. Then for any function $f\in C^\infty(M)$ we can add a correction term to Equation (\ref{eq: second attempt at quantization}), namely $\theta(V_f)$. This provides for us a third potential map 
\begin{equation}\label{eq: third attempt at quantization}
    \mathcal{Q}_{pre}:C^\infty(M)\to S(C^\infty(M,\bC));\quad f\mapsto -iV_f+m_{\theta(V_f)}+m_f
\end{equation}
can easily be verified to satisfy Axiom (1) and is a Lie algebra homomorphism. Indeed, we first see that
$$
\mathcal{Q}_{pre}(1_M)=iV_{1_M}+m_{\theta(1_M)}+m_{1_M}
$$
As we've already seen, $V_{1_M}=0$ and so
$$
\mathcal{Q}_{pre}(1_M)=i\cdot 0+m_{\theta(0)}+m_{1_M}=m_{1_M}=\id_{C^\infty(M,\bC)}
$$
As for being a Lie algebra homomorphism, for any $f,g\in C^\infty(M)$, making use of our previous computation, we obtain
\begin{align*}
    i[\mathcal{Q}_{pre}(f),\mathcal{Q}_{pre}(g)]&=i[-iV_f+m_{\theta(V_f)}+m_f,-iV_g+m_{\theta(V_g)}+m_g]\\
    &=-iV_{\{f,g\}}+2m_{\{f,g\}}+[V_f,m_{\theta(V_g)}]+[m_{\theta(V_f)},V_g].
\end{align*}
Again, acting against a test function, we obtain
$$
[V_f,m_{\theta(V_g)}]+[m_{\theta(V_f)},V_g]=m_{V_f\theta(V_g)-V_g\theta(V_f)}.
$$
Note that
$$
\omega(V_f,V_g)=V_f\theta(V_g)-V_g\theta(V_f)-\theta([V_f,V_g]).
$$
Thus, since $-\{f,g\}=\omega(V_f,V_g)$ and $[V_f,V_g]=V_{\{f,g\}}$, we conclude that
$$
[V_f,m_{\theta(V_g)}]+[m_{\theta(V_f)},V_g]=-m_{\{f,g\}}+m_{\theta(V_{\{f,g\}})}.
$$
Therefore,
\begin{align*}
    i[\mathcal{Q}_{pre}(f),\mathcal{Q}_{pre}(g)]&=-iV_{\{f,g\}}+2m_{\{f,g\}}-m_{\{f,g\}}+m_{\theta(V_{\{f,g\}})}\\
    &=-iV_{\{f,g\}}++m_{\theta(V_{\{f,g\}})}+m_{\{f,g\}}\\
    &=\mathcal{Q}_{pre}(\{f,g\}).
\end{align*}
Of course a general symplectic manifold will not be exact in the sense that the symplectic form admits a global primitive. And so, in principle, we would need to pass to an open cover of $M$ where primitives of $\omega$ exist, then glue these together to obtain a well-defined global operator. The cost of doing this is that our pre-Hilbert space is no longer $C^\infty(M,\bC)$, but rather the sections of a complex line bundle. This now leads us to the definition of a prequantum line bundle.

\begin{defs}[Prequantum Line Bundle]
Let $(M,\omega)$ be a symplectic manifold. A prequantum line bundle is a Hermitian complex line bundle over $M$ $\mathcal{L}\to M$ together with a compatible connection
$$
\nabla:\mfX(M)\otimes\Gamma(\mathcal{L})\to \Gamma(\mathcal{L});\quad X\otimes \sigma\mapsto \nabla_X\sigma
$$
such that if $V,W\in \mfX(M)$ are vector fields and $\sigma\in\Gamma(\mathcal{L})$ is a section, then
\begin{equation}\label{eq: prequantum connection condition}
\nabla_V(\nabla_W\sigma)-\nabla_Y(\nabla_V\sigma)-\nabla_{[V,W]}\sigma=i \omega(V,W)
\end{equation}
We call $\nabla$ the prequantum connection and write $(\mathcal{L},\nabla)\to(M,\omega)$ for the data of a prequantum line bundle.
\end{defs}

\begin{egs}\label{eg: prequant of exact symplectic manifold}
    Our discussion leading up to the formal definition of a prequantum line bundle showed that any exact symplectic manifold admits a prequantum line bundle. More formally, suppose $(M,\omega)$ is a symplectic manifold with $\omega=d\theta$ for some $\theta\in\Omega^1(M)$. Let
    $$
    \mathcal{L}=M\times \bC
    $$
    with the standard Hermitian structure on $\bC$. Identifying $\Gamma(\mathcal{L})=C^\infty(M,\bC)$, define
    $$
    \nabla:\mfX(M)\otimes C^\infty(M,\bC)\to C^\infty(M,\bC);\quad V\otimes f\mapsto L_Vf+i\theta(V)f,
    $$
    where $L_Vf$ is the Lie derivative of $f$ along $V$. Doing pretty well the exact same computation as we did for showing the map in Equation (\ref{eq: third attempt at quantization}), we can easily show that
    $$
    [\nabla_V,\nabla_W]-\nabla_{[V,W]}=id\theta(V,W)=i\omega(V,W).
    $$
    Hence, $(\mathcal{L},\nabla)\to (M,\omega)$ is a prequantum line bundle. 
\end{egs}

As was suggested above, the exact case corresponds to the local picture of prequantum line bundles. If $(\mathcal{L},\nabla)\to (M,\omega)$ is a prequantum line bundle, then on a trivializing neighbourhood $U\subseteq M$ so that $\mathcal{L}|_U\cong U\times \bC$ as complex line bundles, we can identify $\Gamma(\mathcal{L}|_U)$ with $C^\infty(U,\bC)$. Under this identification, we get a canonical complex-valued $1$-form $\alpha\in\Omega^1(U,\bC)$ defined by
$$
\alpha(X)=\nabla_X(1_U),
$$
where $1_U\in C^\infty(U)$ is the constant $1$-function. Using Equation (\ref{eq: prequantum connection condition}) one can verify that $\alpha$ is actually imaginary, i.e $\alpha=i\theta$ for a real-valued $1$-form $\theta\in \Omega^1(U)$ and that this real-valued $1$-form $\theta$ is a primitive of $\omega$, i.e. $d\theta=\omega$. With this discussion, we have outlined a proof of the following.

\begin{prop}\label{prop: local primitive of connection}
Let $(\mathcal{L},\nabla)\to (M,\omega)$ be a prequantum line bundle, $U\subseteq M$ an open subset, and
$$
\phi:\mathcal{L}|_U\to U\times \bC
$$
a trivialization. Then, identifying the sections of $\mathcal{L}|_U\to U$ with complex-valued smooth functions $C^\infty(U,\bC)$, there exists a $1$-form $\theta\in\Omega^1(U)$ such that $d\theta=\omega|_U$ and so that $\nabla$ can be identified with
$$
\nabla_V f=L_Vf+i\theta(V)f,
$$
where $L_V$ denotes the Lie derivative of $f$ along $V$.
\end{prop}

This then allows us to translate the the map defined by Equation (\ref{eq: third attempt at quantization}) into the language of connections to obtain the following definition.

\begin{defs}[Prequantum Operator]
    Let $(\mathcal{L},\nabla)\to (M,\omega)$ be a prequantum line bundle and $f\in C^\infty(M)$. Define the \textbf{prequantum operator} of $f$ to be the linear operator on sections
    $$
    \mathcal{Q}_{pre}(f):\Gamma(\mathcal{L})\to \Gamma(\mathcal{L})
    $$
    given by
    $$
    \mathcal{Q}_{pre}(f)=-i\nabla_{V_f}+m_f,
    $$
    where $m_f$ is the pointwise multiplication operator by $f$.
\end{defs}

As is expected from the local description given above, prequantum operators satisfy Axiom (2) of Definition \ref{def: quantization procedure}.

\begin{prop}\label{prop: prequantum map is a lie algebra hom}
Let $f,g\in C^\infty(M)$, then
$$
i[\mathcal{Q}_{pre}(f),\mathcal{Q}_{pre}(g)]=\mathcal{Q}_{pre}(\{f,g\}).
$$
Furthermore, $\mathcal{Q}_{pre}(1)=\id_{\Gamma(\mathcal{L})}$.
\end{prop}

\begin{proof}
Fix $f,g\in C^\infty(M)$. Then, as operators on $\Gamma(\mathcal{L})$
\begin{align*}
    i[\mathcal{Q}_{pre}(f),\mathcal{Q}_{pre}(g)]&=i[-i\nabla_{V_f}+m_f,-i\nabla_{V_g}+m_g]\\
    &=-i[\nabla_{V_f},\nabla_{V_g}]+[\nabla_{V_f},m_g]+[m_f,\nabla_{V_g}]+i[m_f,m_g].
\end{align*}
First note that clearly $[m_f,m_g]=0$ as both $m_f$ and $m_g$ are just pointwise multiplication by real-valued functions. Next, since $\nabla$ is a connection we have for any $\sigma\in \Gamma(\mathcal{L})$
$$
[\nabla_{V_f},m_g]\sigma=\nabla_{V_f}(g\sigma)-g\nabla_{V_f}\sigma=(V_fg)\sigma=\{f,g\}\sigma.
$$
Thus, swapping $f$ and $g$ in the above we obtain
$$
[\nabla_{V_f},m_g]+[m_f,\nabla_{V_g}]=2m_{\{f,g\}}.
$$
Finally, the prequantum condition on $\nabla$ provides
\begin{align*}
[\nabla_{V_f},\nabla_{V_g}]&=\nabla_{[V_f,V_g]}+i\omega(V_f,V_g)\\
&=\nabla_{V_{\{f,g\}}}-im_{\{f,g\}},
\end{align*}
where we used the facts that $[V_f,V_g]=V_{\{f,g\}}$ and $\omega(V_f,V_g)=-\{f,g\}$ via Proposition \ref{prop: poisson bracket on symplectic manifold}. Therefore,
\begin{align*}
    i[\mathcal{Q}_{pre}(f),\mathcal{Q}_{pre}(g)]&=-i(\nabla_{V_{\{f,g\}}}-im_{\{f,g\}})+2m_{\{f,g\}}\\
    &=-i\nabla_{V_{\{f,g\}}}+m_{\{f,g\}}\\
    &=\mathcal{Q}_{pre}(\{f,g\})
\end{align*}
as desired. $\mathcal{Q}_{pre}(1_M)=\id_{\Gamma(\mathcal{L})}$ follows from the fact that $V_{1_M}=0$. 
\end{proof}

As was the case for an exact compact symplectic manifold, if $M$ is compact then we can make $\Gamma(\mathcal{L})$ into a Hilbert space via
$$
\bra \sigma_1,\sigma_2\ket=\int_M (\sigma_1,\sigma_2)\frac{\omega^n}{n!}
$$
where $2n=\dim(M)$ and $(\cdot,\cdot)$ is the Hermitian structure on $\mathcal{L}$. In this case, $\mathcal{Q}_{pre}(f)\in S(\Gamma(\mathcal{L}))$ for all $f\in C^\infty(M)$. If $M$ is not compact, then we need to take the Hilbert space completion of the space of compactly supported section $\Gamma_0(\mathcal{L})$, in which case $\mathcal{Q}_{pre}(f)$ will only be essentially self-adjoint.

\

So prequantum line bundles seem like quite the promising direction towards quantization of symplectic manifolds. However we must now ask which symplectic manifolds admit prequantum line bundles. Given Proposition \ref{prop: local primitive of connection}, it follows that a prequantum line bundle $(\mathcal{L},\nabla)$ can be obtained by gluing together local trivializations of neighbourhoods where the symplectic form $\omega$ is exact. This procedure can only work out if $\omega$ satisfies an integrality condition. In particular, using a Cech cohomology argument, one can prove the following. 

\begin{theorem}[\cite{kostant_orbits_2009}]
A symplectic manifold $(M,\omega)$ admits a prequantum line bundle if and only if the cohomology class $\frac{1}{\hbar}[\omega]\in H^2(M)$ lies in the image of 
$$
H^2(M,\bZ)\to H^2(M).
$$
\end{theorem}

This is quite nice as it suggests the discrete properties associated to quantization, like, for instance, the discrete orbits an electron can inhabit in an atom. However, as this next example shows, prequantum line bundles cannot be the end of the story as far as quantization goes.

\begin{egs}\label{eq: trivprequant}
Let $(M,\omega) =(\bR^{2n},\omega_{can})$ be standard symplectic $\bR^{2n}$. That is, writing $(x_1,\dots,x_n,y_1,\dots,y_n)$ for coordinates on $\bR^{2n}$, we have
$$
\omega_{can}=\sum_{i=1}^n dx_i\wedge dy_i.
$$
Using the primitive 
$$
\theta=-\sum_{i=1}^n y_i dx_i
$$
of $\omega$ and making use of the construction in Example \ref{eg: prequant of exact symplectic manifold}, the Hilbert space the prequantum operators act on is the Hilbert space completion of $C^\infty_0(\bR^{2n},\bC)$. This is easily computed to be  $L^2(\bR^{2n})$, which is not the right Hilbert space!
\end{egs}

As we can see from the above example, prequantum line bundles produce a Hilbert space which is ``too big'' in the sense that the Hilbert space depends on too many variables. This is a serious issue as important quantum mechanical properties like the uncertainty principle for position and momentum will not hold in the spaces produced by prequantum line bundles. So, we need to modify our procedure and cut down on the variables. This is achieved by polarizations. 

\

Let us now record some elementary properties of prequantum line bundles which we will be using later on.

\begin{prop}\label{prop: localize connection}
Let $(\mathcal{L},\nabla)\to (M,\omega)$ be a prequantum line bundle, $\sigma\in \Gamma(\mathcal{L})$, and $x\in M$. If $V,W\in \mfX(M)$ and $V_x=W_x$, then 
$$
\nabla_V\sigma(x)=\nabla_W\sigma(x).
$$
\end{prop}

\begin{proof}
Fix $x\in M$, a section $\sigma\in\Gamma(\mathcal{L})$, and vector fields $V,W\in\mfX(M)$ with $V_x=W_x$. Let $U\subseteq M$ be a trivializing neighbourhood of $\mathcal{L}$. Then, by Proposition \ref{prop: local primitive of connection}, after identifying the sections $\Gamma(\mathcal{L}|_U)$ with complex-valued functions $C^\infty(U,\bC)$ we can identify the connection $\nabla$ with the operator
$$
\nabla=L+i\theta,
$$
where $L$ is the Lie derivative operator for a vector field and $\theta\in\Omega^1(U)$ is a local primitive of $\omega$. Via the trivialization, we can identify $\sigma|_U$ with a function $h\in C^\infty(U,\bC)$. Then,
$$
\nabla_Vh(x)=V_xh+i\theta(V_x)h(x)=W_xh+i\theta(W_x)h(x)=\nabla_Wh(x).
$$
\end{proof}

\begin{cor}\label{cor: prequantum connection applied to sections}
For each $\sigma\in \Gamma(\mathcal{L})$ and $v\in T^\bC M$, define
\begin{equation}
\nabla_v \ \sigma:=\nabla_V \ \sigma,
\end{equation}
where $V\in \mfX^\bC(M)$ is any complex vector field with $V_x=v$. Then the map
\begin{equation}\label{eq: connection map}
\nabla\sigma:T^\bC M\to L;\quad v\mapsto \nabla_v \ \sigma
\end{equation}
is well-defined and a smooth map between complex vector bundles.
\end{cor}

\begin{proof}
Due to Proposition \ref{prop: localize connection}, the map in Equation (\ref{eq: connection map}) is well-defined. To show smooth, fix $x\in M$ and let $U\subseteq M$ be an open neighbourhood which admits a frame $V_1,\dots,V_n\in \mfX^\bC(U)$ of $T^\bC U$. Then the map
$$
U\times \bC^k\to T^\bC U;\quad (y,\lambda^1,\dots,\lambda^k)\mapsto \sum_{i=1}\lambda^i (V_i)_y
$$
is an isomorphism of vector bundles. Define now
$$
\phi:U\times \bC^k\to L|_U;\quad (y,\lambda^1,\dots,\lambda^k)\mapsto \sum_{i=1}\lambda^i \nabla_{V_i}\sigma(y)
$$
Then $\nabla\sigma$ fits into the diagram

$$
\begin{tikzcd}
    T^\bC U\arrow[r,"\nabla\sigma"] & L|_U\\
    U\times \bC^k\arrow[u]\arrow[ur,"\phi", swap]&
\end{tikzcd}
$$

Hence, $\nabla\sigma$ is smooth.
\end{proof}

\section{Geometric Quantization}\label{sec: geo quant}

\begin{defs}[Quantization Data, Geometric Quantization]\label{def: non-singular quantization data}
By quantization data I will mean a triple $(M,\mathcal{P},\mathcal{L})$ where $M$ is a symplectic manifold, $\mathcal{P}$ a polarization of $M$, and $\mathcal{L}$ a prequantum line bundle over $M$. We then define the geometric quantization of the triple by
$$
\mathcal{Q}(M,\mathcal{P},\mathcal{L}):=\{\sigma\in\Gamma(\mathcal{L})  \ | \ \nabla_{\mathcal{P}}\sigma\equiv 0\}.
$$
If $\mathcal{P}$ and $\mathcal{L}$ are understood from context, we write
$$
\mathcal{Q}(M)=\mathcal{Q}(M,\mathcal{P},\mathcal{L}).
$$
\end{defs}

\begin{egs}
Going back to our example with symplectic $\bR^{2n}$. Defining 
$$
\mathcal{P}_\bR=\text{span}_\bC\bigg(\frac{\partial}{\partial y_1},\dots,\frac{\partial}{\partial y_n}\bigg)
$$
then $\mathcal{P}$ is a polarization of $\bR^{2n}$. Using the same prequantum line bundle $\mathcal{L}=\bR^{2n}\times \bC$, we obtain
$$
\mathcal{Q}(\bR^{2n})=\{f\in C^\infty(\bR^{2n},\bC) \ | \ f(x,y)=g(x)\text{ for some }g\in C^\infty(\bR^n,\bC)\}
$$
which can be identified with $C^\infty(\bR^n)$. Note that is precisely the space of polarized functions $\mathscr{O}_{\mathcal{P}_\bR}$. Of course we have another canonical polarization obtained by identifying $\bR^{2n}=\bC^n$. Letting $I_{can}$ denote the canonical complex structure, let 
$$
\mathcal{P}_\bC=T^{0,1}\bR^{2n}.
$$
In this case, if we use the primitive
$$
\theta=-\frac{i}{2}\sum\overline{z}_jdz_j
$$
Then $\mathcal{Q}(\bR^{2n})=\mathscr{O}_{\mathcal{P}_\bC}$ which is exactly the space of holomorphic functions on $\bR^{2n}$.
\end{egs}

As we see from the above, geometric quantization need not output a Hilbert space. Indeed, further procedures like half-forms or metaplectic constructions provide the extra data needed. See Sniatycki \cite{sniatycki_bohr-sommerfeld_1975} for a good overview of these kinds of constructions. However, using a very similar idea to the proof above, we can prove the following about the geometric quantization of cotangent bundles.

\begin{theorem}\label{thm: quantization of cotangent}
The quantization of $T^*Q$ is equal to the polarized functions $\mathcal{O}_{\mathcal{P}}(T^*Q)$. In particular, we have an equality
$$
\mathcal{Q}(T^*Q)= \tau^*C^\infty(Q,\bC).
$$
\end{theorem}

\begin{proof}
Let $f\in C^\infty(T^*Q,\bC)$. Then $f\in \mathcal{Q}(T^*Q)$ if and only if for any $\alpha\in T^*Q$ and $V\in T_\alpha(T^*_{\tau(\alpha)}Q)$ we have
\begin{equation}\label{cotangent covariant constant}
\nabla_V f=Vf+i \theta(V)f=0.
\end{equation}
Note that by construction, we have $\theta(V)=0$ if $V$ is tangent to the cotangent fibres. Hence, $\mathcal{Q}(T^*Q)$ coincides with the polarized functions $\mathcal{O}_\mathcal{P}(T^*Q)$. By Equation (\ref{cotangent covariant constant}), we have
$$
\mathcal{O}_\mathcal{P}(T^*Q)=\tau_Q^*C^\infty(Q,\bC).
$$
\end{proof}

\section{Quantization Commutes With Reduction}\label{sec: [Q,R]}

For all that follows, $G$ will be a connected Lie group and $J:(M,\omega)\to \mfg^*$ a proper Hamiltonian $G$-space.

\begin{defs}[Equivariant Prequantum Line Bundle]
A prequantum line bundle $(\mathcal{L},\nabla)\to (M,\omega)$ is called equivariant if $L\to M$ is equipped with the structure of a $G$-equivariant line bundle and if for each $\xi\in \mfg$ and $\sigma\in \Gamma(\mathcal{L})$ we have
$$
\xi\cdot \sigma=\nabla_{\xi_M}\sigma+i J^\xi\sigma,
$$
where $\xi\cdot\sigma\in \Gamma(\mathcal{L})$ is defined by
$$
(\xi\cdot \sigma)(x):=\frac{d}{dt}\bigg|_{t=0}\exp(t\xi)\sigma(\exp(-t\xi)\cdot x)
$$
\end{defs}

\begin{egs}\label{eg: equivariant connection exact}
Suppose $(M,\omega)$ is a symplectic manifold with exact symplectic form $\omega$ with a chosen primitive $\theta\in\Omega^1(M)$. Recall from Example \ref{eg: prequant of exact symplectic manifold} we have a canonical prequantum line bundle, the trivial bundle $\mathcal{L}=M\times \bC$ together with the connection 
$$
\nabla=L+i \theta.
$$
Now suppose a connected Lie group $G$ acts on $M$ properly and in a Hamiltonian fashion
$$
J:M\to \mfg^*.
$$
If we equip $\mathcal{L}$ with the trivial action on the second factor, then for any $\xi\in \mfg$ and $f\in C^\infty(M,\bC)$,
$$
\xi\cdot f=\xi_Mf.
$$
Hence, the equivariance condition boils down to
$$
\xi_Mf=\xi_Mf+i \theta(\xi_M)f+i J^\xi f
$$
Hence, $\nabla$ is equivariant if and only if 
\begin{equation}\label{eq: equivariance for trivial prequantum}
\theta(\xi_M)=-J^\xi
\end{equation}
for all $\xi\in \mfg$. More generally, if we also have a representation
$$
\rho:G\to \bC^\times
$$
and use this to non-trivially act on the second factor of $\mathcal{L}$, then for any $\xi\in \mfg$ and $f\in C^\infty(M,\bC)$, we have
$$
\xi\cdot f=\rho_*(\xi)f+\xi_Mf,
$$
where we are viewing $\rho_*(\xi)\in\bC$. Hence, $\nabla$ is equivariant if and only if
$$
\theta(\xi_M)=-J^\xi-i\rho_*(\xi)
$$
for all $\xi\in \mfg$. Note that this only makes sense if $\rho_*(\xi)\in i\bR$ which is guaranteed to happen if $\rho$ is a unitary representation, i.e. has image in $S^1\subseteq \bC^\times$.
\end{egs}

\begin{egs}\label{eg: cotangent is prequantum}
    We will see below examples of $1$-forms satisfying Equation (\ref{eq: equivariance for trivial prequantum}) in Example \ref{eg: equivariant connection exact}, but we can actually obtain a large number of examples from cotangent bundles. Let $G$ be a connected Lie group and $Q$ a proper $G$-space. Let $J:(T^*Q,\omega)\to \mfg^*$ be the induced Hamiltonian $G$-space structure on $T^*Q$ with respect to the canonical symplectic form where
    $$
    \bra J(\alpha),\xi\ket=\bra \alpha,\xi_Q\ket
    $$
    for $\alpha\in T^*Q$, $\xi\in \mfg$, and $\xi_Q\in \mfX(Q)$ the associated fundamental vector field. The canonical symplectic form has of course a canonical primitive, namely the negative of the Liouville $1$-form $\theta\in \Omega^1(Q)$. Observe that if $\tau:T^*Q\to Q$ is the bundle map, $\alpha\in T^*Q$, and $\xi\in \mfg$, then since $\tau$ is equivariant we have
    $$
    \theta_\alpha(\xi_{T^*Q}(\alpha))=\bra \alpha, T_\alpha\tau(\xi_{T^*Q}(\alpha))\ket=\bra \alpha, \xi_Q(\tau(\alpha))\ket=J^\xi(\alpha).
    $$
    Thus, $-\theta$ satisfies Equation (\ref{eq: equivariance for trivial prequantum}) and so the trivial prequantum line bundle $\mathcal{L}=T^*Q\times \bC$ with the connection $\nabla=L-i\theta$ is an equivariant prequantum line bundle.
\end{egs}

Due to the equivariance of the momentum map, the level set $J^{-1}(0)$ is closed under the action of $G$. Hence, so is the restriction $\mathcal{L}|_{J^{-1}(0)}$ for an equivariant prequantum line bundle. With this, we can now formally define the reduction of prequantum line bundles.

\begin{defs}[Reduction of Prequantum Line Bundle]\label{def: reduction of prequant}
    Let $G$ be a connected Lie group, $J:(M,\omega)\to \mfg^*$ a proper Hamiltonian $G$-space, and $(\mathcal{L},\nabla)\to (M,\omega)$ an equivariant prequantum line bundle. Define the \textbf{reduction of $\mathcal{L}$ along the $0$-level set} by
    \begin{equation}
        \mathcal{L}_0:=(\mathcal{L}|_{J^{-1}(0)})/G.
    \end{equation}
\end{defs}

For all that follows, we fix an equivariant prequantum line bundle $(\mathcal{L},\nabla)\to (M,\omega)$. An easy application of Theorem \ref{thm: special subbundle regular case} and Proposition \ref{prop: equivariant sections give all sections below} allows us to conclude the following.

\begin{lem}\label{lem: non-singular reduction of line bundle}
Suppose $J^{-1}(0)\subseteq M$ is an embedded submanifold and, furthermore, suppose $J^{-1}(0)$ has only one orbit-type. Write $\pi_M:J^{-1}(0)\to M_0$ and $\pi_\mathcal{L}:\mathcal{L}|_{J^{-1}(0)}\to \mathcal{L}_0$ for the quotient maps. If the map
\begin{equation}
J^{-1}(0)\times \Gamma(\mathcal{L})^G\to \mathcal{L};\quad (x,\sigma)\mapsto \sigma(x)
\end{equation}
is surjective, then the following holds.
\begin{itemize}
    \item[(i)] The induced map $\mathcal{L}_0\to M_0$ is a smooth complex line bundle over $M_0$.
    \item[(ii)] The projection $\pi_\mathcal{L}:\mathcal{L}|_{J^{-1}(0)}\to \mathcal{L}_0$ induces a bijection
    $$
    \kappa:\Gamma(\mathcal{L}|_{J^{-1}(0)})^G\to \Gamma(\mathcal{L}_0)
    $$
    \item[(iii)] The map
    $$
    \kappa:\Gamma(\mathcal{L})^G\to \Gamma(\mathcal{L}_0);\quad \sigma\mapsto \kappa(\sigma|_{J^{-1}(0)})
    $$
    is a surjective map. 
\end{itemize}
\end{lem}

\begin{remark}
If $G$ acts freely on $J^{-1}(0)$, then the assumption that the map
$$
J^{-1}(0)\times\Gamma(\mathcal{L})^G\to \mathcal{L}|_{J^{-1}(0)}
$$
is automatic. However, if the action of $G$ on $J^{-1}(0)$ is not free, then 
$$
J^{-1}(0)\times\Gamma(\mathcal{L})^G\to \mathcal{L}|_{J^{-1}(0)}
$$
need not be surjective. The main consequence of this fact is that $\mathcal{L}_0=(\mathcal{L}|_{J^{-1}(0)})/G$ need not be a manifold and hence need not be a line bundle. To see how this could be the case, consider the following example. Let $M=\bR^2\times\bR^2$ and let $\mathcal{L}=M\times \bC$ be the trivial bundle. Equip $M$ with the canonical $S^1$ action
$$
S^1\times M\to M;\quad (t,(x,y))\mapsto (t\cdot x,y),
$$
where $t\cdot x$ is the standard action of $S^1$ on $\bR^2$. Consider also the standard representation $\rho:S^1\to \bC^\times$ given by the inclusion. It's then an easy check to see that the momentum map is given by
$$
J:M\to \bR;\quad (x,y)\mapsto \frac{\|x\|^2}{2}.
$$
By Example \ref{eg: equivariant connection exact}, an equivariant connection is specified by a primitive $\theta\in\Omega^1(M)$ of the symplectic form satisfying
\begin{equation}\label{eq: equivariant connection counterexample}
\theta(\xi_M)=J^\xi-i\rho_*(\xi)
\end{equation}
for all $\xi\in \mfg=\bR$. Let $\beta$ be any primitive of the standard symplectic form on the second factor of $\bR^2\times \bR^2$ and consider the primitive
$$
\theta=\frac{-x_2dx_1+x_1dx_2}{2}+pr_2^*\beta
$$
on $\bR^2\times \bR^2$. It easily follows that $\theta$ satisfies Equation (\ref{eq: equivariant connection counterexample}). Now, since we are acting by the canonical $S^1$ action on $\bC^\times$, we see that over $J^{-1}(0)=\{0\}\times \bR^2$
$$
C^\infty(J^{-1}(0),\bC)^{S^1}=\{0\}
$$
Furthermore, we have
$$
\mathcal{L}|_{J^{-1}(0)}\cong\bR^2\times (\bC/S^1)
$$
which is not a line bundle over $J^{-1}(0)$. 

\ 

To remedy the fact that the map
$$
J^{-1}(0)\times \Gamma(\mathcal{L})^G\to \mathcal{L};\quad (x,\sigma)\mapsto \sigma(x)
$$
need not be surjective in general, one could extend this analysis to consider \textbf{orbibundles} as in done in \cite{meinrenken_singular_1999}. However, for the purposes of this thesis we will content ourselves with a more restrictive equivariance condition.
\end{remark}

As a special case of the above, we have the following Theorem of Guillemin and Sternberg.

\begin{theorem}[\cite{guillemin_geometric_1982}]
Suppose $0$ is a regular value of $J$ and $G$ acts freely on $J^{-1}(0)$. Then $L_0$ admits a unique prequantum connection $\nabla^0$ which has the property that
    $$
    \pi^*\nabla^0=\iota^*\nabla,
    $$
    where $\pi:J^{-1}(0)\to M_0$ is the quotient map and $\iota:J^{-1}(0)\into M$ is the inclusion.
\end{theorem}

This is proven by passing to $G$-invariant trivializing neighbourhoods of points in $J^{-1}(0)$ and showing that local primitives of the symplectic form are basic and hence descend to primitives of the reduced symplectic form. This approach requires that $G$ act freely, and so we need to make an adjustment to extend to the more general (but still non-singular!) setting of Lemma \ref{lem: non-singular reduction of line bundle}. To set up this more general approach, recall that by Corollary \ref{cor: pushforwards of equivariant vector fields, regular} if $J^{-1}(0)$ is an embedded submanifold and has only one orbit-type, then the reduction map $\pi_0:J^{-1}(0)\to J^{-1}(0)/G=M_0$ induces a surjection
$$
(\pi_0)_*: \mfX(J^{-1}(0))^G\to \mfX(M_0).
$$
Using this map, we can construct $\nabla^0$ as follows. Letting $V\in \mfX(M_0)$ and $\sigma\in \Gamma(\mathcal{L}_0)$, find lifts to $W\in \mfX(M,J^{-1}(0))^G$ and $\tau\in\Gamma(\mathcal{L})^G$ with 
$$
(\pi_0)_*(W|_{J^{-1}(0)})=V\quad\text{and}\quad\kappa(\tau)=\sigma.
$$ 
Then define
\begin{equation}
\nabla^0_V\sigma:=\kappa(\nabla_W\tau).
\end{equation}
By Corollary \ref{cor: pushforwards of equivariant vector fields, regular} and Lemma \ref{lem: non-singular reduction of line bundle}, both $W$ and $\tau$ are guaranteed to exist. The issue now is to show that $\nabla^0$ is well-defined, is smooth, and is a prequantum connection.

\begin{theorem}\label{thm: can reduce non-singular}
Suppose $J^{-1}(0)\subseteq M$ is an embedded submanifold with only one orbit type and suppose
$$
J^{-1}(0)\times \Gamma(\mathcal{L})^G\to \mathcal{L};\quad (x,\sigma)\mapsto \sigma(x)
$$
is surjective. Then the following holds.
\begin{itemize}
    \item[(i)] Suppose $\sigma_1,\sigma_2\in\Gamma(\mathcal{L})^G$ and $V_1,V_2\in\mfX(M,J^{-1}(0))^G$ satisfy 
    $$
    \sigma_1|_{J^{-1}(0)}=\sigma_2|_{J^{-1}(0)}
    $$
    and 
    $$
    (\pi_0)_*(V_1|_{J^{-1}(0)})=(\pi_0)_*(V_2|_{J^{-1}(0)}).
    $$
    Then,
    \begin{equation}
    \nabla_{V_1}\sigma_1|_{J^{-1}(0)}=\nabla_{V_2}\sigma_2|_{J^{-1}(0)}.
    \end{equation}
    \item[(ii)] $\mathcal{L}_0=(\mathcal{L}|_{J^{-1}(0)})/G$ is canonically a prequantum line bundle over $M_0=J^{-1}(0)/G$ with prequantum connection $\nabla^0$ given by
    $$
    \nabla^0_V\sigma=\kappa(\nabla_W\tau),
    $$
    where $W\in \mfX(M,J^{-1}(0))^G$ and $\tau\in\Gamma(\mathcal{L})^G$ satisfy
    $$
    (\pi_0)_*W|_{J^{-1}(0)}=V\quad\text{and}\quad \kappa(\tau)=\sigma.
    $$
\end{itemize}
\end{theorem}

\begin{proof}
\begin{itemize}
    \item[(i)] Fix equivariant sections $\sigma_1,\sigma_2\in\Gamma(\mathcal{L})^G$ and restrictable vector fields $V_1,V_2\in\mfX(M,J^{-1}(0))^G$ satisfying $\sigma_1|_{J^{-1}(0)}=\sigma_2|_{J^{-1}(0)}$ and $(\pi_0)_*(V_1|_{J^{-1}(0)})=(\pi_0)_*(V_2|_{J^{-1}(0)})$. Observe that
    $$
    \nabla_{V_1}\sigma_1-\nabla_{V_2}\sigma_2=\nabla_{V_1}(\sigma_1-\sigma_2)+\nabla_{V_1-V_2}\sigma_2
    $$
    I claim both terms vanish after restricting to $J^{-1}(0)$. 

    \

    To show the first term vanishes, fix $x\in J^{-1}(0)$ and let $U\subseteq M$ be a trivialization neighbourhood for $\mathcal{L}$. Then, after choosing an isomorphism $\mathcal{L}|_U\to U\times \bC$, we can identify $\Gamma(\mathcal{L}|_U)$ with $C^\infty(U,\bC)$. Thus, since $\sigma_1=\sigma_2$ over $J^{-1}(0)$, we can identify $\sigma_1-\sigma_2$ with a function $f\in C^\infty(U,\bC)$ with $f|_{J^{-1}(0)}=0$. Furthermore, we can find a local primitive $\alpha\in\Omega^1(U)$ of $\omega|_U$ so that
    $$
    \nabla=d+i \alpha.
    $$
    Hence, 
    $$
    \nabla_{V_1}(f)=V_1f+i \alpha(V_1)f.
    $$
    Using the fact that $f(x)=0$, we get
    $$
    \nabla_{V_1}f(x)=(V_1)_xf.
    $$
    Since $V_1$ is tangent to $J^{-1}(0)$ and $f$ vanishes on $J^{-1}(0)$, we have $(V_1)_xf=0$. Since this holds in any trivializing neighbourhood intersecting $J^{-1}(0)$, we obtain $\nabla_{V_1}(\sigma_1-\sigma_2)|_{J^{-1}(0)}=0$.

    \

    Now for the second term. Since $(\pi_0)_*V_1|_{J^{-1}(0)}=(\pi_0)_*V_2|_{J^{-1}(0)}$, we have
    $$
    (\pi_0)_*((V_1-V_2)|_{J^{-1}(0)})=0.
    $$
    Since $\pi_0:J^{-1}(0)\to M_0$ is the quotient map by the $G$-action and the action of $G$ on $J^{-1}(0)$ is regular, there exists functions $h_1,\dots,h_M\in C^\infty(J^{-1}(0))$ and Lie algebra elements $\xi_1,\dots,\xi_M\in\mfg$ so that
    $$
    (V_1-V_2)|_{J^{-1}(0)}=\sum_{i=1}^M h_i (\xi_i)_{J^{-1}(0)},
    $$
    where $(\xi_i)_{J^{-1}(0)}\in \mfX(J^{-1}(0))$ is the fundamental vector field associated to $\xi_i$. Thus,
    $$
    \nabla_{V_1-V_2}\sigma_2|_{J^{-1}(0)}=\sum_{i=1}^Mh_i\nabla_{(\xi_i)_M}\sigma_2|_{J^{-1}(0)}.
    $$
    By the equivariance of $\nabla$ and $\sigma_2$, we have
    $$
    \nabla_{(\xi_i)_M}\sigma_2=\xi_i\cdot \sigma+i J^{\xi_i}\sigma_2=+i J^{\xi_i}\sigma_2
    $$
    Thus, restricting to $J^{-1}(0)$ we have
    $$
    \nabla_{V_1-V_2}\sigma_2|_{J^{-1}(0)}=0.
    $$

    \item[(ii)] $\mathcal{L}_0$ is a complex line bundle is statement (i) from Lemma \ref{lem: non-singular reduction of line bundle}. Now let us show $\nabla^0$ is well-defined. Fix $V\in\mfX(M_0)$ and $\sigma\in\Gamma(\mathcal{L}_0)$ and suppose $W_1,W_2\in\mfX(M,J^{-1}(0))^G$ are two lifts of $V$ and $\tau_1,\tau_2\in \Gamma(\mathcal{L})^G$ are two lifts of $\sigma$. Then, by construction we have
    $$
    (\pi_0)_*W_1|_{J^{-1}(0)}=(\pi_0)_*W_2|_{J^{-1}(0)}=V
    $$
    and
    $$
    \kappa(\tau_1)=\kappa(\tau_2)=\sigma.
    $$
    Hence, by (i)
    $$
    \nabla_{W_1}\tau_1|_{J^{-1}(0)}=\nabla_{W_2}\tau_2|_{J^{-1}(0)}.
    $$
    Since both $\nabla_{W_1}\tau_1$ and $\nabla_{W_2}\tau_2$ are equivariant sections of $\mathcal{L}$, we conclude by (iii) of Lemma \ref{lem: non-singular reduction of line bundle} that
    $$
    \kappa(\nabla_{W_1}\tau_1)=\kappa(\nabla_{W_2}\tau_2).
    $$
    Hence, $\nabla^0$ is well-defined. $\nabla^0$ is also clearly smooth since the image of $\kappa$ is the smooth sections of $\Gamma(\mathcal{L}_0)$. Finally, we show $\nabla^0$ is prequantum. Fix $V_1,V_2\in\mfX(M_0)$ with lifts $W_1,W_2\in\mfX(M,J^{-1}(0))^G$ and a section $\sigma\in\Gamma(\mathcal{L}_0)$ with a lift $\tau\in\Gamma(\mathcal{L})^G$. Unfolding definitions shows
    $$
    [\nabla^0_{V_1},\nabla^0_{V_2}]\sigma=\kappa([\nabla_{W_1},\nabla_{W_2}]\tau)
    $$
    and since the map $(\pi_0)_*:\mfX(M,J^{-1}(0))^G\to \mfX(M_0)$ is a Lie algebra homomorphism,
    $$
    \nabla^0_{[V_1,V_2]}\sigma=\kappa(\nabla_{[W_1,W_2]}\tau).
    $$
    Hence,
    \begin{align*}
    [\nabla^0_{V_1},\nabla^0_{V_2}]\sigma-\nabla^0_{[V_1,V_2]}\sigma&=\kappa([\nabla_{W_1},\nabla_{W_2}]\tau-\nabla_{[W_1,W_2]}\tau)\\
    &=\kappa(-i \omega(W_1,W_2)\tau)\\
    &=-i \omega_0(V_1,V_2)\sigma.
    \end{align*}
    where we used the fact that $\nabla$ is prequantum and that $\omega|_{J^{-1}(0)}=\pi_0^*\omega_0$. Thus, $\nabla^0$ is prequantum.
\end{itemize}
\end{proof}

Now let us fix a $G$-invariant polarization $\mathcal{P}\subseteq T^\bC M$. If $\mathcal{P}$ is reducible as in Definition \ref{def: reducible polarization}, then we can reduce the polarization to get a new polarization $\mathcal{P}_0\subseteq T^\bC M_0$. The question now is what is the relationship between the quantization of $M$ with respect to $\mathcal{P}$ and $(\mathcal{L},\nabla)\to (M,\omega)$ and that of the reduced space? Guillemin and Sternberg showed that provided a strong relationship in the compact and free case.

\begin{theorem}[\cite{guillemin_geometric_1982}]\label{thm: guillemin-sternberg non-singular [Q,R]=0}
Let $(\mathcal{L},\nabla)\to (M,\omega)$ be an equivariant prequantum line bundle and $\mathcal{P}\subseteq T^\bC M$ a $G$-invariant positive polarization. Suppose further both $M$ and $G$ are compact and $G$ acts freely on $M$. Then the map $\kappa:\Gamma(\mathcal{L})^G\to \Gamma(L_0)$ from Lemma \ref{lem: non-singular reduction of line bundle} restricts to a bijection
$$
\kappa:\mathcal{Q}(M)^G\to \mathcal{Q}(M_0).
$$
\end{theorem}

This Theorem is the archetypal $[Q,R]=0$ result. We can provide a partial generalization for arbitrary proper actions and polarizations.

\begin{theorem}\label{thm: non-singular [Q,R]=0}
Let $(\mathcal{L},\nabla)\to (M,\omega)$ be an equivariant prequantum line bundle and $\mathcal{P}\subseteq T^\bC M$ a $G$-invariant reducible polarization. 
\begin{itemize}
    \item[(1)] The quantization $Q(M)\subseteq \Gamma(\mathcal{L})$ is a linear subrepresentation with respect to the canonical action
    $$
    G\times \Gamma(\mathcal{L})\to \Gamma(\mathcal{L});\quad (g,\sigma)\mapsto g\cdot \sigma,
    $$
    where 
    $$
    (g\cdot \sigma) (x)=g\cdot \sigma(g^{-1}\cdot x).
    $$
    \item[(2)] The map $\kappa:\Gamma(\mathcal{L})^G\to \Gamma(L_0)$ from Lemma \ref{lem: non-singular reduction of line bundle} restricts to a map
    \begin{equation}\label{non-singular [Q,R] map}
    \kappa:Q(M)^G\to Q(M_0).
    \end{equation}
\end{itemize}
\end{theorem}

\begin{proof}
    \begin{itemize}
        \item[(1)] Since $G$ is connected, it suffices to show $\xi\cdot \sigma\in\mathcal{Q}(M)$ for all $\sigma\in\mathcal{Q}(M)$ and $\xi\in \mfg$. To that end, fix $\sigma\in \mathcal{Q}(M)$, $\xi\in \mfg$, and $V\in \Gamma(\mathcal{P})$. Then, using the equivariance property of $\nabla$, we obtain
        \begin{align*}
        \nabla_V(\xi\cdot \sigma)&=\nabla_V(\nabla_{\xi_M}\sigma+iJ^\xi\sigma)\\
        &=\nabla_V\nabla_{\xi_M}\sigma-i(VJ^\xi)\sigma,
        \end{align*}
        where we used the derivation property of the connection and the fact that $\nabla_V\sigma=0$. Observe that by definition
        $$
        \nabla_V\nabla_{\xi_M}\sigma=\nabla_{\xi_M}\nabla_V\sigma+\nabla_{[V,\xi_M]}\sigma+i\omega(V,\xi_M)\sigma
        $$
        Since $\mathcal{P}$ is invariant under the $G$-action, we have $[V,\xi_M]\in\Gamma(\mathcal{P})$ and hence the terms $\nabla_{\xi_M}\nabla_V\sigma$ and $\nabla_{[V,\xi_M]}\sigma$ vanish. Furthermore, by definition of the momentum map, we have
        $$
        \omega(V,\xi_M)=-VJ^\xi\sigma
        $$
        Thus,
        $$
        \nabla_V(\xi\cdot \sigma)=-iVJ^\xi\sigma+iVJ^\xi\sigma=0.
        $$
        Thus, $\xi\cdot\sigma\in \mathcal{Q}(M)$.

        \item[(2)] Suppose $\sigma\in \mathcal{Q}(M)^G$ and let $\kappa(\sigma)\in \Gamma(\mathcal{L}_0)$ be the induced section on the reduced prequantum line bundle. To show $\kappa(\sigma)\in \mathcal{Q}(M_0)$, let $x\in J^{-1}(0)$ and $v\in (\mathcal{P}_0)_{[x]}$ be an element of the reduced polarization and choosing a lift $\widetilde{v}\in \mathcal{P}_x$. Then, by the definition of $\nabla^0$, we have
        $$
        \nabla^0_v \kappa(\sigma)=\kappa(\nabla_{\widetilde{v}}\sigma)=\kappa(0)=0.
        $$
    \end{itemize}
\end{proof}

\section{Non-Singular [Q,R]=0 Examples}\label{sec: [Q,R] examples}

\subsection{Cotangent Bundles}\label{sub: [Q,R] cotangent}
With this, we can now show that quantization commutes with reduction for cotangent bundles. Let $G$ be a connected Lie group , $Q$ a proper $G$-space, and $J:T^*Q\to \mfg^*$ the Hamiltonian lift as in Definition \ref{def: hamiltonian lift}. Writing $\theta_Q\in\Omega^1(Q)$ for the Liouville $1$-form, we get a canonical equivariant prequantum line bundle given by $\mathcal{L}_Q=T^*Q\times \bC$ and prequantum connection
$$
\nabla^Q=L-i\theta_Q,
$$
where $L$ denotes the Lie derivative operator. Via Theorem \ref{thm: reduced symplectomorphism}, the reduction $(T^*Q)_0=J^{-1}(0)/G$ is canonically symplectomorphic to $T^*(Q/G)$. Letting $\pi:Q\to Q/G$ be the quotient map, recall that the dual of the projection $T\pi:TQ\to \pi^{-1}T(Q/G)$, 
$$
\phi:\pi^{-1}T^*(Q/G)\to J^{-1}(0),
$$
has image in $J^{-1}(0)$ and induces a symplectomorphism $\widetilde{\phi}:T^*(Q/G)\to (T^*Q)_0$. Not only that, but as we saw, the Liouville $1$-form $\theta_Q$ induces a $1$-form $\theta_0$ on $(T^*Q)_0$ which pulls back to the Liouville $1$-form $\theta_{Q/G}\in\Omega^1(T^*(Q/G))$. Finally, writing $\mathcal{P}_Q\subseteq T^\bC(T^*Q)$ and $\mathcal{P}_{Q/G}\subseteq T^\bC(T^*(Q/G))$ for the canonical real polarizations, we saw in Proposition \ref{prop: nonsing reduction of cotangent polarization} that $\mathcal{P}_Q$ is reducible and induces a polarization $\mathcal{P}_0$ on $T^*(Q/G)$ with the property that $T\widetilde{\phi}(\mathcal{P}_{Q/G})=\mathcal{P}_0$. With this, we are now ready to prove the following.

\begin{theorem}\label{theorem: non-singular [Q,R] cotangent bundles}
Suppose $Q$ has only one orbit-type. 
\begin{itemize}
    \item[(i)] $(\mathcal{L}_Q,\nabla^Q)$ is reducible. In particular,
    $$
    (\mathcal{L}_Q|_{J^{-1}(0)})/G\cong (T^*Q)_0\times \bC
    $$
    and the induced prequantum connection $\nabla^0$ has the form
    $$
    \nabla^0=L-i \theta_0,
    $$
    where $\theta_0\in \Omega^1((T^*Q)_0)$ is as in Lemma \ref{lem: 0 level set nice cotangent}. 
    \item[(ii)] The natural map
    $$
    \kappa:\mathcal{Q}(T^*Q)^G\to \mathcal{Q}((T^*Q)_0)
    $$
    is a bijection. In particular, there are two canonical linear isomorphisms
    $$
    \mathcal{Q}((T^*Q)_0)\to \mathcal{Q}(T^*(Q/G))
    $$
    and
    $$
    \mathcal{Q}(T^*Q)^G\to \mathcal{Q}(T^*(Q/G))
    $$
    which make the diagram commute
    $$
    \begin{tikzcd}
    \mathcal{Q}(T^*Q)^G\arrow[r,"\kappa"]\arrow[dr] & \mathcal{Q}((T^*Q)_0)\arrow[d]\\
    & \mathcal{Q}(T^*(Q/G))
    \end{tikzcd}
    $$
\end{itemize}
\end{theorem}

\begin{proof}
\begin{itemize}
    \item[(i)] Have $\mathcal{L}_Q=T^*Q\times \bC$. Hence,
    $$
    \mathcal{L}_Q|_{J^{-1}(0)}=J^{-1}(0)\times \bC.
    $$
    Since $G$ only acts on the first factor, the statement that $\mathcal{L}_Q|_{J^{-1}(0)}$ is generated by $G$-invariant sections is equivalent to the statement that
    $$
    Q\times C^\infty(J^{-1}(0),\bC)^G\to Q\times \bC;\quad (q,f)\mapsto (q,f(q))
    $$
    is surjective. Clearly this holds. Thus, $\mathcal{L}_Q$ induces a prequantum line bundle $(\mathcal{L}_0,\nabla^0)\to (T^*Q)_0$. Clearly,
    $$
    \mathcal{L}_0=(T^*Q)_0\times \bC.
    $$
    Now, let $\pi_0:J^{-1}(0)\to (T^*Q)_0$ be the quotient map. The map
    $$
    \pi_0^*:C^\infty((T^*Q)_0,\bC)\to C^\infty(J^{-1}(0),\bC)^G
    $$
    is a bijection. Hence, the canonical map
    $$
    \kappa:\Gamma(\mathcal{L}_Q)^G\to \Gamma(\mathcal{L}_0)
    $$
    can be identified with the map
    $$
    \kappa:C^\infty(T^*Q,\bC)^G\to C^\infty((T^*Q)_0,\bC);\quad f\mapsto (\pi_0^*)^{-1}(f|_{J^{-1}(0)}).
    $$
    Now, suppose $V\in \mfX((T^*Q)_0)$ and $f\in C^\infty((T^*Q)_0,\bC)$. Choose lifts $W\in \mfX(T^*Q,J^{-1}(0))^G$ and $h\in C^\infty(T^*Q,\bC)^G$. Then,
    $$
    \nabla^0_Vf=\kappa(\nabla_Wh)=\kappa(Wh+i \theta_Q(W)h)=Vf+i \theta_0(V)f
    $$
    Hence the claim.

    \item[(ii)] For all three quantizations, the prequantum connections have the form
    $$
    \nabla=d+i \theta,
    $$
    where $\theta$ vanishes on the vector bundle fibres. Furthermore, the polarizations are given by vectors tangent to the vector bundle fibres. Hence, the quantizations all coincide with the polarized functions. Since
    $$
    \widetilde{\phi}:T^*(Q/G)\to (T^*Q)_0
    $$
    is a symplectomorphism which preserves the polarizations, the pullback defines a bijection
    $$
    \widetilde{\phi}^*:\mathcal{O}_{\mathcal{P}_0}((T^*Q)_0)\to \tau^*_{Q/G}C^\infty(Q/G,\bC)
    $$
    This gives us our first bijection. As for the second, note that $\mathcal{Q}(T^*Q,\bC)=\tau^*_QC^\infty(Q,\bC)$ and $\mathcal{Q}(T^*(Q/G))=\tau_{Q/G}^*C^\infty(Q/G,\bC)$. Since $\tau_Q$ is equivariant, we have
    $$
    (\tau^*_QC^\infty(Q,\bC))^G=\tau_Q^*(C^\infty(Q,\bC)^G)
    $$
    The quotient map $\pi:Q\to Q/G$ defines a bijection
    $$
    \pi^*:C^\infty(Q/G,\bC)\to C^\infty(Q,\bC)^G
    $$
    Furthermore, $\tau_Q^*:C^\infty(Q,\bC)\to \tau_Q^*C^\infty(Q,\bC)$ and $\tau_{Q/G}^*:C^\infty(Q/G,\bC)\to \tau_{Q/G}^*C^\infty(Q/G,\bC)$ are also bijections since $\tau_Q$ and $\tau_{Q/G}$ are surjective submersions. Thus, chaining this all together, we get a chain of bijections
    $$
    \lambda=\tau_{Q/G}^*\circ (\pi^*)^{-1}\circ (\tau_Q^*)^{-1}:\tau_Q^*C^\infty(Q,\bC)^G\to \tau_{Q/G}^*C^\infty(Q/G,\bC).
    $$
    It is then a simple matter of a diagram chase to show
    $$
    \widetilde{\phi}^*\circ\kappa=\lambda.
    $$
\end{itemize}
\end{proof}

We will discuss the case where $G$ does not have a single orbit type when we get to singular spaces.

\subsection{Geometric Quantization of Toric Manifolds}\label{section: non-singular quantization toric manifold}

Recall Delzant's construction of a toric manifold. Let 
$$
\Delta=\{x\in \bR^n \ | \ \bra x,v_i\ket \leq \lambda_i, \ i=1,\dots,N\}
$$
be a Delzant polytope with inward pointing normals $v_i$. For all that follows, we let $\lambda=(\lambda_1,\dots,\lambda_N)$. Recall that as part of Delzant's construction, we obtain a toric $T^n$ manifold
$$
J_\Delta:(M_\Delta,\omega_\Delta)\to \bR^n
$$
via the symplectic reduction of standard symplectic $\bC^N$ with respect to the action of a subtorus $K\leq (S^1)^N$. Specifically, letting $k\substeq \bR^N$ be the kernel of the map
$$
\bR^N\to \bR^n;\quad e_i\mapsto v_i
$$
and $K=\exp(k)\substeq (S^1)^N$, then 
$$
M_\Delta=(\iota^*\circ J)^{-1}(0)/K,
$$
where 
$$
J:\bC^N\to \bR^N;\quad (z_1,\dots,z_N)\mapsto \frac{1}{2}(|z_1|^2,\dots,|z_N|^2)-\lambda
$$
and $\iota^*:\bR^N\to \mfk^*$ is the dual of the inclusion $\iota:\mfk\into \bR^N$. We can then identify $(S^1)^N/K=T^n$ a $n$-dimensional torus.

\begin{theorem}
In the setup above, $J_\Delta:(M_\Delta,\omega_\Delta)\to \bR^n$ admits a $T^n$-equivariant prequantum line bundle if both $\lambda\in \bZ^N$ and $\iota^*(\lambda)$ lies in the integral lattice of $k$.
\end{theorem}

\begin{proof}
Assume $\lambda\in \bZ^N$ and $\iota^*(\lambda)\in\mfk^*$ lie in the integral lattices. Writing $(z_1,\dots,z_n)$ for the complex coordinates of $\bC^N$ and $(\theta_1,\dots,\theta_N)$ for the real coordinates on $\bR^N$, we have a canonical primitive for the canonical symplectic form
$$
\alpha=\sum_{j=1}^N \frac{r_j^2}{2} d\theta_j.
$$
With this, the trivial line bundle $\mathcal{L}_{can}=\bC^N\times \bC$ with connection
$$
\nabla^{can}=d+i \alpha
$$
is a prequantum line bundle over $(\bC^N,\omega_{can})$. Since $\lambda\in \bZ^N$, we get a representation
$$
\rho:T^N\to \bC^\times;\quad (e^{i\theta_1},\dots,e^{i\theta_N})\mapsto e^{-i\bra \lambda,(\theta_1,\dots,\theta_N)\ket}.
$$
It is then a straightforward matter to show for all $f\in C^\infty(\bC^N,\bC)$ and $j=1,\dots,N$ that with respect to this representation,
$$
e_j\cdot f=\partial_{\theta_j}f-i \lambda_j f
$$
where $e_1,\dots,e_N\in \bR^N$ are the standard basis. Hence, letting 
$$
J:\bC^N\to \bR^N;\quad (z_1,\dots,z_N)\mapsto \frac{1}{2}(|z_1|^2,\dots,|z_N|^2)-\lambda
$$
we obtain 
$$
\xi\cdot f=\nabla_{\xi_{\bC^N}}f+i J^{\xi}f
$$
hence $(\mathcal{L}_{can},\nabla^{can})\to (\bC^N,\omega_{can})$ is a $T^N$-equivariant prequantum line bundle. Since $\iota^*(\lambda)$ is also assumed to be integral, we automatically obtain that $(\mathcal{L}_{can},\nabla^{can})\to (\bC^N,\omega_{can})$ is also a $K$-equivariant prequantum line bundle. Since $K$ acts freely on $J^{-1}(0)$, by Lemma \ref{lem: non-singular reduction of line bundle}
$$
\mathcal{L}_\Delta:=(\mathcal{L}_{can}|_{J^{-1}(0)})/K=(J^{-1}(0)\times \bC)/K
$$
is a complex line bundle over $M_\Delta=J^{-1}(0)/K$ and comes equipped with a canonical prequantum connection $\nabla^\Delta$. Note that since the action of $K$ is free on $J^{-1}(0)\times \bC$, there is a natural action of $T^n\cong T^N/K$ on $\mathcal{L}_\Delta$ making
$$
\mathcal{L}_\Delta \to M_\Delta
$$
into an equivariant complex line bundle. I claim this is an equivariant prequantum connection. Indeed, recall the definition of $\nabla^\Delta$. Take $V\in \mfX(M_\Delta)$ and $\sigma\in \Gamma(\mathcal{L}_\Delta)$. Then we can find lifts $\widetilde{V}\in \mfX(\bC^N)^K$ and $\sigma\in \Gamma(\mathcal{L}_{can})^K$. We then set
$$
\nabla^\Delta_V\sigma=\kappa(\nabla^{can}_{\widetilde{V}}\widetilde{\sigma}),
$$
where 
$$
\kappa:\Gamma(\mathcal{L}_{can})^K\to \Gamma(\mathcal{L}_\Delta)
$$
is the surjection as in Lemma \ref{lem: non-singular reduction of line bundle}. In particular, taking $\xi\in \bR^n$ and $\sigma\in\Gamma(\mathcal{L}_\Delta)$, choose $\eta\in \bR^N$ and $f\in C^\infty(\bC^N,\bC)^K$ such that $\pi(\eta)=\xi$ and so that $\kappa(f)=\sigma$. It's then a simple matter to show
$$
\kappa(\eta\cdot f)=\xi\cdot \sigma.
$$
From here, it then follows that
\begin{align*}
\xi\cdot \sigma&=\kappa(\eta\cdot f)\\
&=\kappa(\nabla^{can}_{\eta_{\bC^N}}f+i J^\eta f)\\
&=\nabla^\Delta_{\xi_M}f+i J_\Delta^\xi\sigma.
\end{align*}
Hence the conclusion.
\end{proof}

Now, let $I_{can}$ be the canonical complex structure on $\bC^N$. Since $I_{can}$ is $K$-equivariant and the action of $K$ on $J^{-1}(0)$ is free, it induces a K\"{a}hler structure $I_\Delta$ on $M_\Delta$. This induced K\"{a}hler structure is also $T^n$ invariant. We will write $\mathcal{P}_\Delta=T^{0,1}_{I_\Delta}M_\Delta$.

\begin{theorem}[\cite{hamilton_quantization_2008}]
For any element $a=(a_1,\dots,a_n)\in \bZ^n_{\geq 0}$, write
$$
\psi_a:\bC^n\to \bC;\quad (z_1,\dots,z_n)\mapsto z_1^{a_1}\cdots z_n^{a_n}.
$$
Then there is a canonical isomorphism of $T^n$ representations
$$
\mathcal{Q}(M_\Delta,\mathcal{P}_\Delta,\mathcal{L}_\Delta)\cong \text{span}_\bC\{\psi_a \ | \ a\in \Delta\cap \bZ^n\}.
$$
\end{theorem}

\begin{proof}
Let $T^{0,1}\bC^N$ denote the antiholomorphic tangent bundle of $\bC^N$ with respect to the canonical complex structure. Since $T^{0,1}\bC^N$ has a canonical trivialization, being spanned by the anti-holomorphic vector fields $\partial/\partial \overline{z_i}$, we see that a function $f\in C^\infty(\bC^N,\bC)$ lies in the quantization $\mathcal{Q}(\bC^N)$ if and only if
$$
\nabla_{\frac{\partial}{\partial\overline{z_j}}}f=0
$$
for all $j=1,\dots,N$. Making use of the fact hat
$$
\nabla^{can}=d+i \sum_{j=1}^N\frac{r_j^2}{2}d\theta_j
$$
we easily deduce that $f\in \mathcal{Q}(\bC^N)$ if and only if there exists holomorphic function $h\in \mathcal{O}(\bC^N)$ such that
$$
f(z)=h(z)e^{-2\pi \|z\|^2},
$$
where $\|\cdot \|$ is the canonical $L^2$ norm on $\bC^N$. Note that the factor $e^{-2\pi \|z\|^2}$ is invariant under the $T^N$ action, hence as $T^N$ spaces we have a canonical isomorphism between $\mathcal{Q}(\bC^N)$ and the space of holomorphic functions $h:\bC^N\to \bC$ where we have equipped $\bC$ with the action induced by $\lambda$. In particular, we have
$$
\mathcal{Q}(\bC^N)^K\cong \{h\in \mathcal{O}(\bC^N) \ | \ h(t\cdot z)=\rho(t)h(z)\text{ for all }t\in K\}.
$$
For any element $a=(a_1,\dots,a_N)\in\bZ^N$, write
$$
z^a:=z_1^{a_1}\cdots z_N^{a_N}.
$$
Then, as is shown in \cite{hamilton_quantization_2008}, the space of equivariant holomorphic functions has basis
$$
\{z^a \ | \ a\in \bZ^N\cap \Delta\},
$$

Since $K$ acts freely on $J^{-1}(0)$ and $\mathcal{P}_\Delta$ is a K\"{a}hler structure, it follows then from Theorem \ref{thm: non-singular [Q,R]=0} that
$$
\mathcal{Q}(M_\Delta)\cong \mathcal{Q}(\bC^N)^K\cong\{\psi_a \ | \ a\in \Delta\cap \bZ^n\}.
$$
\end{proof}

\part{Differentiable Stratified Spaces}\label{pt: strat}
\cleardoublepage
\chapter{Subcartesian Spaces}\label{ch: subcartesian}

Subcartesian structures are, along with diffeologies (e.g. \cite{iglesias-zemmour_diffeology_2013}) and stacks (e.g. \cite{vistoli_grothendieck_2005}) one the main ways to studying the differential geometry of singular spaces, i.e. spaces which are not manifolds. This approach dates back to the work of Aronszajn and Smith \cite{aronszajn_theory_1961} in their study of Bessel potentials, and was subsequently generalized in the works of authors like Marshall \cite{marshall_calculus_1975} to a extension of the theory of smooth manifolds. 

\ 

The idea of subcartesian spaces is quite simple. Take a topological space $X$, locally embed it into some Euclidean space $\bR^N$, then restrict differential geometric constructions on $\bR^N$ to $X$ suitably. This simple idea gets us a lot of mileage as we can get vector fields, tangent bundles, and even flows on singular spaces in this fashion (see \cite{sniatycki_differential_2013} or \cite{karshon_vector_2023} for more information on flows). Of course it is not quite as simple as I have presented it, as we need an a priori notion of a what a ``smooth function'' means on the topological space $X$ in order to make the local embeddings coherent. This is where the notion of a differential (also known as a Sikorski) structure comes into play. 

\ 

Differential structures, as their alternate name suggests, are due to Sikorski \cite{sikorski_abstract_1967}. They provide an abstract framework for discussing ``differentiable functions'' on a topological space which allows for an easy generalization of constructions in smooth geometry which only depend on the smooth functions. For instance, vector fields on a smooth manifold can be characterized as derivations of the algebra of smooth functions, and so we can use the exact same definition for differential structures to obtain so-called Zariski vector fields. In this set-up, a subcartesian structure is just a differential structure which locally looks like the restriction of a smooth structure on a euclidean space. 

\ 

It is worth noting that in older literature on subcartesian spaces, there is no mention of differential structures. Rather there is another formalism based on compatible charts, mimicking the definition of a smooth structure on a smooth manifold. As we show in Theorem  \ref{thm: charts of subcartesian spaces are locally diffeomorphic} the two formalisms are in fact equivalent.

\section{Differential Spaces}

We now begin our discussion on differential spaces. Much of the following can be found in either \cite{sniatycki_differential_2013} or \cite{karshon_vector_2023}.

\begin{defs}[Differential Space]\label{def: differential space}
Let $(X,\tau)$ be a topological space. A collection $\mathcal{F}\subseteq C^0(X)$ of continuous functions is called a differential or Sikorski structure if the following axioms hold.
\begin{itemize}
	\item[(1)] The weak topology induced by $\mathcal{F}$, denoted $\tau(\mathcal{F})$ coincides with the given topology $\tau$ on $X$.
	\item[(2)] For any $f_1,\dots,f_n\in \mathcal{F}$ and any $G\in C^\infty(\bR^n)$, the composition $G(f_1,\dots,f_n)\in\mathcal{F}$.
	\item[(3)] If $f\in C^0(X)$ has the property that around each point $x\in X$ there exists an open neighbourhood $U\subseteq X$ and a function $g\in \mathcal{F}$ satisfying $g|_U=f|_U$, then $f\in \mathcal{F}$.
\end{itemize}
Call the pair $(X,\mathcal{F})$ a differential space. If the context is clear, we will write $C^\infty(X)$ for a differential structure. If the topology on $X$ is given, we will call a differential structure $\mathcal{F}$ which induces the same topology a differential structure on $X$.
\end{defs}

\begin{egs}
\begin{itemize}
    \item[(1)] Let $(X,\tau)$ be a completely regular topological space and $C^0(X)$ the continuous functions on $X$. Complete regularity is equivalent to the weak topology induced by $C^0(X)$, $\tau(C^0(X))$, being equal to the topology $\tau$ on $X$. Composition of continuous functions remaining continuous gives the second axiom, and the third amounts to the fact that the association
    $$
    \{\text{open in }X\}\to \text{Ab};\quad U\mapsto C^0(U)
    $$
    is a sheaf. 
    \item[(2)] Let $M$ be a smooth manifold and consider the ring of smooth functions $C^\infty(M)$. Write $\tau$ for the given topology on $M$. Just as with continuous functions, the second and third axioms of a differential structure hold immediately. The only thing to show is that the weak topology induced by $C^\infty(M)$, $\tau(C^\infty(M))$ is equal to the given topology $\tau$ on $M$. Since $M$ is Hausdorff and second countable, the $\tau$ is determined by convergent sequences. So let $\{x_k\}\subset M$ be a sequence and $x_0\in M$ such that for all $f\in C^\infty(M)$, we have
    $$
    \lim_{k\to \infty}f(x_k)=f(x_0).
    $$
    We show $\lim_{k\to \infty}x_k=x_0$ in the topology of $M$. To do this, let $\phi:U\to \bR^n$ be a chart with $x_0\in U$. Letting $B\subseteq U$ be a relatively compact ball centered on $x_0$, we can find a compact set $K\subseteq U$ and a smooth function $\psi:M\to \bR$ such that
    $$
    \psi|_B\equiv 1\text{ and }\psi|_{M\setminus K}\equiv 0.
    $$
    So now define $F=\psi \phi:M\to \bR^n$. Writing $F(x)=(F_1(x),\dots,F_n(x))$ we see each $F_i\in C^\infty(M)$ and hence $F(x_k)$ converges to $F(x_0)$. Since $F|_B:B\to \phi(B)$ is a homeomorphism and $x_0\in B$, we thus get $x_k$ converges to $x_0$. Since $\{x_k\}$ and $x_0$ we arbitrary, we conclude $\tau(C^\infty(M))=\tau$. Hence, $C^\infty(M)$ is a differential structure on $M$.
\end{itemize}
\end{egs}

\begin{remark}
Note that all differential spaces are automatically completely regular. 
\end{remark}

\subsection{Induced Differential Structures}\label{sub: induced differential}

Given that differential structures are a kind of $C^\infty$-algebraic space, it is only natural that we should consider generators of such a structure. Unlike commutative rings where generators are simply elements whose linear combinations span the whole ring, generators for differential structures need to be ``$C^\infty$ combinations''. We shall make this notion more precise below.

\begin{theorem}\label{thm: differential structure generated by functions}
Let $(X,\tau)$ be a topological space and $\{\mathcal{F}_\alpha\}_{\alpha\in I}$ an arbitrary collection of differential structures on $X$. Define
$$
\mathcal{F}:=\bigcap_{\alpha\in I}\mathcal{F}_\alpha.
$$
If $\mathcal{F}$ is non-empty and $\tau(\mathcal{F})=\tau$, then $\mathcal{F}$ is a differential structure on $X$.
\end{theorem}

\begin{proof}
Just as with the previous examples, we only need to show the second and third axioms hold. First, suppose $G\in C^\infty(\bR^n)$ and $f_1,\dots,f_n\in \mathcal{F}$. Fixing $\alpha\in I$, we have $f_1,\dots,f_n\mathcal{F}_\alpha$ and hence $G(f_1,\dots,f_n)\in \mathcal{F}_\alpha$. Since $\alpha$ was arbitrary, we conclude that $G(f_1,\dots,f_n)\in \mathcal{F}$. 

\

Now suppose $f:X\to \bR$ is a real-valued function, $\{U_i\}$ is an open cover of $X$, and $\{g_i\}\subseteq \mathcal{F}$ satisfies $f|_{U_i}=g|_{U_i}$. Fixing $\alpha$ once again, we have $\{g_i\}\subseteq \mathcal{F}_\alpha$ and hence $f\in \mathcal{F}_\alpha$. The arbitrariness of $\alpha$ implies $f\in \mathcal{F}$. 
\end{proof}

\begin{cor}\label{cor: generator of differential structure}
Let $X$ be any set and $A$ an arbitrary (non-empty) collection of real-valued functions on $X$. If $X$ is endowed with the weak topology $\tau(A)$ from $A$, then there is a smallest differential structure, denoted $\bra A\ket$, with $\tau(\bra A\ket)=\tau(A)$ and $A\subseteq \bra A\ket$.
\end{cor}

\begin{proof}
Let $\tau(A)$ denote  the weak topology induced by $A$. Also write $C_A(X)$ for the set of all real-valued continuous functions with respect to $\tau(A)$. Since $C_A(X)$ is a differential structure on $X$ with $A\subseteq C_A(X)$, this implies the set
$$
I=\{\mathcal{F}\subseteq C_A(X) \ | \ \mathcal{F}\text{ is a differential structure on }X\text{ with }A\subseteq\mathcal{F}\}
$$
is non-empty. Set
$$
\bra A\ket:=\bigcap_{\mathcal{F}\in I}\mathcal{F}.
$$
Since the weak topology induced by $C_A^0(X)$ coincides with $\tau(A)$ and $A\subseteq \bra A\ket\subseteq C_A^0(X)$, we conclude that $\tau(\bra A\ket)=\tau(A)$ and hence by Theorem \ref{thm: differential structure generated by functions}, $\bra A\ket$ is a differential structure on $X$.
\end{proof}

\begin{egs}\label{eg: product differential structure}
    Using Corollary \ref{cor: generator of differential structure} we can obtain a differential structure on product spaces. Let $X_1$ and $X_2$ be two differential spaces and let $pr_1:X_1\times X_2\to X_1$ and $pr_2:X_1\times X_2\to X_2$ be the projections onto the first and second factors, respectively. Write for $i=1,2$ write
    $$
    pr_i^*C^\infty(X_i):=\{f\circ pr_i:X_1\times X_2\to \bR \ | \ f\in C^\infty(X_i)\}.
    $$
    Define now
    $$
    C^\infty(X\times Y)=\bra pr_1^*C^\infty(X)\cup pr_2^*C^\infty(Y)\ket
    $$
    We call $C^\infty(X\times Y)$ the product differential structure.
\end{egs}

\begin{prop}\label{prop: differential structures are Cinfty rings}
Let $X$ be a differential space and suppose we have a subset $A\subseteq C^\infty(X)$ with $\bra A\ket=C^\infty(X)$ as in Corollary \ref{cor: generator of differential structure}. Then $f\in C^\infty(X)$ $\iff$ $f$ satisfies property (P) stated below.
\begin{itemize}
    \item[(P)] For each $x\in X$ there exists $G\in C^\infty(\bR^n)$ for some $n$, $h_1,\dots,h_n\in A$, and an open neighbourhood $U\subseteq X$ of $x$ such that
    $$
    f|_U=G(h_1,\dots,h_n)|_U.
    $$
\end{itemize}
\end{prop}

\begin{proof}
Let $\mathcal{F}\subseteq C^0(Y)$ denote all continuous functions satisfying property (P). First note that since $A\subseteq \bra A\ket$ and $\bra A\ket$ is a differential structure, if $h_1,\dots,h_n\in A$ and $G\in C^\infty(\bR^n)$, then $G(h_1,\dots,h_n)\in \bra A\ket$. The sheaf property of differential structures then implies that $\mathcal{F}\subseteq \bra A\ket=C^\infty(Y)$. Note that since $A\subseteq \mathcal{F}$ we have the weak topology induced by $\mathcal{F}$ is equal to the one induced by $A$. 

\

Using the minimality of $\bra A\ket$, it suffices to now show that $\mathcal{F}$ satisfies the $C^\infty$-ring property and the sheaf property of a differential space. The sheaf property is inherent from the definition of $\mathcal{F}$. The $C^\infty$-ring property holds automatically since all elements of $\mathcal{F}$ are locally compositions of elements of $A$ and smooth functions on Euclidean spaces. Composing with another function from Euclidean space preserves this. Hence $\mathcal{F}$ is a differential structure and so $\mathcal{F}=\bra A\ket$ by minimality.
\end{proof}

\begin{defs}[Sheaf of Differentiable Functions]
Let $X$ be a differentiable space. For every open subset $U\subseteq X$, define
$$
C^\infty(U):=\bra C^\infty(X)|_U\ket.
$$
Write $C^\infty_X$ for the resulting sheaf of differentiable functions.
\end{defs}

\begin{prop}
$C^\infty_X$ being a sheaf is equivalent to condition (3) in the definition of differential structure.
\end{prop}

\subsection{Differentiable Maps}\label{sub: diff maps}

Now that we have our objects, we need our morphisms. As can be seen in sources like \cite{nestruev_smooth_2003}, a continuous map $F:M\to N$ between two smooth manifolds is smooth if and only if $F$ pulls back smooth functions on $N$ to smooth functions on $M$. In symbols, $F^*C^\infty(N)\substeq C^\infty(M)$. Since we now have a characterization of a differential geometric property in terms of the ring of smooth functions, we can now generalize to differential spaces.

\begin{defs}[Differentiable Map]
Let $X$ and $Y$ be differential spaces. A map $F:X\to Y$ is called differentiable if $F^*C^\infty(Y)\subseteq C^\infty(X)$, where
$$
F^*C^\infty(Y)=\{f\circ F:X\to \bR \ | \ f\in C^\infty(Y)\}.
$$
If $F$ is invertible and the inverse $F^{-1}$ is also differentiable, call $F$ a diffeomorphism. Write $C^\infty(X,Y)$ for the set of all differentiable maps from $X$ to $Y$.
\end{defs}

\begin{prop}\label{prop: differentiable maps are continuous}
Let $X$ and $Y$ be differential spaces and $F:X\to Y$ a map of sets.
\begin{itemize}
    \item[(1)] If $F$ is differentiable, then $F$ is continuous.
    \item[(2)] If $C^\infty(Y)=\bra A\ket$, then $F$ is differentiable $\iff$ $F^*A\subseteq C^\infty(X)$.
\end{itemize}
\end{prop}

\begin{proof}
\begin{itemize}
    \item[(1)] Let $U\subseteq Y$ be open and suppose $x\in F^{-1}(U)$. Since $Y$ has the weak topology induced by $C^\infty(Y)$, there exists $f_1,\dots,f_m\in C^\infty(Y)$ and open intervals $I_1,\dots,I_m\subseteq \bR$ with
    $$
    F(x)\in \bigcap_{i=1}^m f_i^{-1}(I_i)\subseteq U.
    $$
    Hence,
    $$
    x\in F^{-1}(F(x))\subseteq \bigcap_{i=1}^m (f\circ F)^{-1}(I_i)\subseteq F^{-1}(U).
    $$
    Since $f_i\circ F$ is continuous for each $i$, we conclude that $U$ must be open.

    \item[(2)] Clearly if $f:X\to Y$ is differentiable, then $f^*A\subseteq C^\infty(X)$. So we just need to show $f^*A\subseteq C^\infty(X)$ implies $f$ is differentiable. In this case, we make use of Proposition \ref{prop: differential structures are Cinfty rings} and the sheaf property of differential spaces. In particular, if $f\in C^\infty(Y)$ we show around each $x\in M$ there exists open neighbourhood $U\subseteq X$ of $x$ and $g\in C^\infty(X)$ such that $g|_U=F^*f|_U$. So fix $f\in C^\infty(Y)$ and $x\in X$. By Proposition \ref{prop: differential structures are Cinfty rings}, we can find an open neighbourhood $V\subseteq Y$ of $F(x)$, $G\in C^\infty(\bR^n)$ for some $n$, and $h_1,\dots,h_n\in A$ so that
    $$
    f|_V=G(h_1,\dots,h_n)|_V.
    $$
    Note that 
    $$
    F^*G(h_1,\dots,h_n)=G(F^*h_1,\dots,F^*h_n)
    $$
    and so $g:=F^*G(h_1,\dots,h_n)\in C^\infty(X)$ by the $C^\infty$-ring property. In particular, $F^*f|_{F^{-1}(V)}=g|_{F^{-1}(V)}$ and $x\in F^{-1}(V)$. Hence, $F^*f\in C^\infty(X)$.
\end{itemize}
\end{proof}

\begin{prop}
Let $X$, $Y$, and $Z$ be differential spaces and $F:X\to Y$ and $G:Y\to Z$ differentiable maps. Then the composition
$$
G\circ F:X\to Z
$$
is also differentiable.
\end{prop}

\begin{egs}
Let $X$ and $Y$ be completely regular topological spaces. Then the set of continuous functions from $X$ to $Y$ coincides with the set of differentiable maps. This is a consequence of (1) from Proposition \ref{prop: differential structures are Cinfty rings}.
\end{egs}

\begin{egs}
Let $M$ and $N$ be two smooth manifolds, then the differentiable maps between $M$ and $N$ are precisely equal to the smooth maps between $M$ and $N$. That is, a map $F:M\to N$ is differentiable if and only if for each $x\in M$ there exists charts $\phi:U\subseteq M\to \bR^N$ about $x$ and $\psi:V\subseteq N\to \bR^M$ such that $F(U)\subseteq V$ and so that
$$
\psi\circ F\circ \phi^{-1}:\phi(U)\to \psi(V)
$$
is a $C^\infty$-map between open subsets of Euclidean spaces. 

\

Since smooth maps are closed under composition, a map $F:M\to N$ being smooth in the usual sense clearly implies $F$ is differentiable. Conversely, if $F:M\to N$ is differentiable and $x\in M$, then by continuity alone we can find charts $\phi:U\subseteq M\to \bR^N$ and $\psi:V\subseteq N\to \bR^M$ so that $F(U)\subseteq V$. Clearly the resulting map
$$
G:=\psi\circ F\circ \phi^{-1}:\phi(U)\to \psi(V)
$$
is differentiable. Hence, if we write $x_1,\dots,x_M$ for the coordinate functions on $\bR^M$ $G^*x_1,\dots,G^*x_M\in C^\infty(\phi(U))$ and hence the coordinate functions of $G$ are all $C^\infty$-differentiable, hence a smooth map. 
\end{egs}

\subsection{Subspace Differential Structures}\label{sec: subspaces}

Our goal is to give a definition of subcartesian spaces which are differential spaces which locally ``look like'' subsets of Euclidean spaces. In order to make this precise, we need to discuss how to endow a subset of a differential space with a differential structure. 

\begin{defs}[Subspace Differential Structure]
Let $X$ be a differential space and $Y\subseteq X$ a subspace. Define the subspace differential structure on $Y$, denoted $C^\infty(Y)$ by
$$
C^\infty(Y):=\bra C^\infty(X)|_Y\ket,
$$
where
$$
C^\infty(X)|_Y=\{f|_Y:Y\to \bR \ | \ f\in C^\infty(X)\}.
$$
\end{defs}

\begin{prop}\label{prop: subspace differential structure}
If $X$ is a differential space and $Y\subseteq X$ is a subspace.
\begin{itemize}
    \item[(1)] $C^\infty(Y)$ is the smallest differential structure on $Y$ such the inclusion $\iota:Y\into X$ is differentiable.
    \item[(2)] The weak topology induced by $C^\infty(Y)$ coincides with the subspace topology on $Y$.
    \item[(3)] $f\in C^\infty(Y)$ if and only if for each $y\in Y$ there exists a neighbourhood $U\subseteq X$ of $x$ and $g\in C^\infty(X)$ such that
    $$
    f|_{U\cap Y}=g|_{U\cap Y}.
    $$
\end{itemize}
\end{prop}

\begin{proof}
\begin{itemize}
    \item[(1)] Since $C^\infty(Y)$ is generated by $C^\infty(X)|_Y=\iota^*C^\infty(X)$, it automatically follows that $\iota$ is differentiable, hence continuous. Thus, the subspace topology $\iota^*\tau$ is contained in $\tau(C^\infty(Y))$, the weak topology induced by $C^\infty(Y)$. Conversely, if $U\subseteq Y$ is a $\tau(C^\infty(Y))$ open set, then since this is the same as the topology induced by $C^\infty(X)|_Y$, then for every $y\in U$ we can find $f_1,\dots,f_\ell\in C^\infty(X)$ and open intervals $I_1,\dots,I_\ell\subseteq \bR$ so that
    $$
    y\in \bigcap_{i=1}^\ell (f|_Y)^{-1}(I_i)\subseteq U.
    $$
    Note that
    $$
    \bigcap_{i=1}^\ell (f|_Y)^{-1}(I_i)=\bigg(\bigcap_{i=1}^\ell f_i^{-1}(I_i)\bigg)\cap Y.
    $$
    Thus, for each $y\in U$ we can find an open set $V_y\subseteq X$ containing $y$ so that $V_y\cap Y\subseteq U$ which implies
    $$
    U=\bigg(\bigcup_{y\in U}V_y\bigg)\cap Y
    $$
    and hence $U$ is $\iota^*\tau$ open. Thus, the weak topology induced by $C^\infty(Y)$ coincides with the subspace topology.
    \item[(2)] As we say already, $\iota$ is differentiable. Since any other differential structure on $Y$ with $\iota$ differentiable must also contain $\iota^*C^\infty(X)=C^\infty(X)|_Y$, it follows from minimality that $C^\infty(Y)$ is the smallest.
    \item[(3)] This is a straightforward usage of Proposition \ref{prop: differential structures are Cinfty rings}. Let $y\in Y$ and $f\in C^\infty(Y)$. Then we can find a neighbourhood $W\subseteq Y$ of $y$, $G\in C^\infty(\bR^m)$ for some $m$, and $h_1,\dots,h_m\in C^\infty(X)$ so that 
    $$
    f|_W=G(h_1|_Y,\dots,h_m|_Y)|_W.
    $$
    Since the topology on $Y$ is the subspace topology, there exists open $U\subseteq X$ with $W=U\cap Y$ and and $G(h_1,\dots,h_m)\in C^\infty(X)$, we have
    $$
    f|_{U\cap X}=G(h_1,\dots,h_m)|_{U\cap X}
    $$
    hence the claim.
\end{itemize}
\end{proof}

\begin{prop}\label{prop: closed subspace has restriction differential structure}
    Suppose $X$ is a differential space which admits partitions of unity. That is, for any collection of open sets $\{U_i\}_{i\in I}$ we can find functions $\{\rho_i\}_{i\in I}\subseteq C^\infty(X)$ so that
    \begin{itemize}
        \item[(1)] $\text{supp}(\rho_i)\subseteq U_i$ for all $i\in I$.
        \item[(2)] The collection $\{\rho_i\}_{i\in I}$ is locally finite.
        \item[(3)] $\displaystyle\sum_{i\in I}\rho_i=1$.
    \end{itemize}
    Then, for any closed subset $Y\subseteq X$, we have
    $$
    C^\infty(Y)=\bra C^\infty(X)|_Y\ket=C^\infty(X)|_Y.
    $$
\end{prop}

\begin{proof}
    Fix $f\in C^\infty(Y)$. We show there exists $g\in C^\infty(X)$ so that $g|_Y=f$. Due to Proposition \ref{prop: subspace differential structure}, we can find an open cover $\{U_i\}_{i\in I}$ of $Y$ and functions $h_i\in C^\infty(X)$ so that
    $$
    f|_{Y\cap U_i}=h_i|_{Y\cap U_i}
    $$
    for all $i\in I$. Consider now
    $$
    U_\infty:=X\setminus Y
    $$
    and write $h_\infty$ for the zero function on $X$. Choose a partition of unity $\{\rho_i\}_{i\in I\cup\{\infty\}}$ subordinate to the cover $\{U_\infty\}\cup \{U_i\}_{i\in I}$ and define
    $$
    g:=\sum_{i\in I\cup\{\infty\}}\rho_i h_i.
    $$
    Clearly $g\in C^\infty(X)$ and $g|_Y=f$. 
\end{proof}

\begin{prop}
Let $X$ be a differential space and $R\subseteq S\subseteq X$ a chain of subsets. Then
$$
\bra C^\infty(X)|_{R}\ket=\bra C^\infty(S)|_R\ket.
$$
\end{prop}

\begin{proof}
This is an easy application of minimality.
\end{proof}

\begin{defs}[Differential Embedding]
Let $X$ and $Y$ be differential spaces and $f:X\to Y$ a differentiable map. Say $f$ is a differentiable embedding if the induced map 
$$
f:(X,C^\infty(X))\to (f(X),C^\infty(f(X)))
$$
is a diffeomorphism.
\end{defs}

\section{Subcartesian Spaces}\label{sec: subcartesian}

We now can give a formal definition of a subcartesian space. As we shall see, these are strict generalizations of smooth manifolds.

\begin{defs}[Subcartesian Space]
A subcartesian space is a differential space $(X,C^\infty(X))$ such that
\begin{itemize}
	\item $X$ is Hausdorff, second countable, and paracompact.
	\item $C^\infty(X)$ admits \textbf{singular charts}. That is, for every $x\in X$ there exists an open neighbourhood $U\subseteq X$, a positive integer $N\in \bN$, and a differentiable embedding
	$$
	\phi:U\into \bR^N
	$$
\end{itemize}
We will just write $X$ for a subcartesian space if there will be no confusion. If all charts can be chosen so that the image is locally closed, we call $X$ a locally closed subcartesian space.
\end{defs}

\begin{egs}
Let $M$ be a smooth manifold. Then the usual charts $\phi:U\subseteq M\to \bR^n$ with images being open subsets are also singular charts. Hence $M$ is a subcartesian space.

\

Conversely, if $X$ is a subcartesian space which admits a covering $\{(U_i,\phi_i)\}_{i\in I}$ by singular charts with the property that 
\begin{itemize}
    \item[(1)] There exists $n\in \bN$ such that each $\phi_i$ maps into $\bR^n$.
    \item[(2)] $\phi_i(U_i)\subseteq \bR^n$ is an open subset,
\end{itemize}
then $\{(U_i,\phi_i)\}_{i\in I}$ is a smooth atlas in the usual sense and the smooth functions are precisely the subcartesian structure $C^\infty(X)$. 
\end{egs}

\begin{egs}
    Given two subcartesian spaces $X$ and $Y$, the product differential structure $C^\infty(X\times Y)$ is clearly subcartesian once again.
\end{egs}

\begin{egs}
    Let $N\geq 1$ be an integer and $X\subseteq \bR^N$ a subset. Then $X$ together with its subspace differential structure $C^\infty(X)=\bra C^\infty(\bR^N)|_X\ket$ is subcartesian. Thus, objects which are very far away from being manifolds (like the Cantor set) are subcartesian.
\end{egs}

\begin{egs}
    Even ``infinite dimensional'' spaces can be subcartesian. As a rather trivial example, consider the disjoint union
    $$
    X=\bigsqcup_{n\in \bZ_{>0}}\{n\}\times \bR^n.
    $$
    Endow $X$ with the disjoint union topology and a differential structure $C^\infty(X)$ given by
    $$
    f\in C^\infty(X) \ \iff \ f|_{\{n\}\times \bR^n}\in C^\infty(\bR^n)
    $$
    for all $n\in \bZ_{>0}$. It is easy to see that $C^\infty(X)$ is a differential structure. Furthermore, since each of the $\bR^n$'s making up $X$ are open subsets, $X$ is subcartesian.
\end{egs}

\begin{egs}
    Not all differential spaces are subcartesian. Recall that for each $n\in \bZ_{>0}$ we have a natural embedding
    $$
    \mathbb{CP}^n\into \mathbb{CP}^{n+1};\quad [z_0,\dots,z_{n+1}]\mapsto [z_0,\dots,z_{n+1},0]
    $$
    Define then
    $$
    \mathbb{CP}^\infty=\mathrm{colim}_{n} \mathbb{CP}^n
    $$
    endowed with the weakest topology so that the inclusions
    \begin{equation}
    \iota_n:\mathbb{CP}^n\into \mathbb{CP}^\infty
    \end{equation}
    is continuous. Define
    $$
    C^\infty(\mathbb{CP}^\infty):=\{f:\mathbb{CP}^\infty \to\bR \ | \ \iota_n^*f\in C^\infty(\mathbb{CP}^n) \ \forall n\in \bZ_{>0}\}.
    $$
    Making use of the natural CW complex structure on $\mathbb{CP}^\infty$, we deduce that the topology is sequential. Using then the compactness of $\mathbb{CP}^n$ for all $n$ then shows that the weak topology induced by $C^\infty(\mathbb{CP}^\infty)$ is the same as the colimit topology. Thus, making use of Theorem \ref{thm: differential structure generated by functions} shows that $(\mathbb{CP}^\infty,C^\infty(\mathbb{CP}^\infty))$ is a differential space. However, it is not subcartesian.

    \ 

    Indeed, if $\phi:U\subseteq \mathbb{CP}^\infty\into \bR^N$ were some singular chart for some $N$, it is easy to see that this would provide local embeddings of $\mathbb{CP}^n$ into $\bR^N$ for arbitrary $n$, which is clearly absurd.
\end{egs}

\begin{prop}\label{prop: subsets of subcartesian spaces are subcartesian}
Let $X$ be a subcartesian space and $Y\subseteq X$ a subspace. Then $C^\infty(Y)=\bra C^\infty(X)|_Y\ket$ is also a subcartesian space. 
\end{prop}

\begin{proof}
Subsets of a Hausdorff, second countable, and paracompact space are also Hausdorff, second countable, and paracompact. So we just need to show the existence of singular charts. To this end, let $y\in Y$ and let $\phi:U\into \bR^N$ be a singular chart about $y$. Then I claim
$$
\phi|_{U\cap Y}:U\cap Y\into \bR^N
$$
is also a singular chart. Indeed, we just need to show the induced map
$$
\phi|_{U\cap Y}:(U\cap Y,C^\infty(U\cap Y))\to (\phi(U\cap Y),C^\infty(\phi(U\cap Y)))
$$
is a diffeomorphism of differential spaces. This is a trivial consequence of the fact that $\phi|_{Y\cap U}$ and its inverse are both compositions of differentiable maps and hence both $\phi|_{U\cap Y}$ and $(\phi_{U\cap Y})^{-1}$ are differentiable, hence the conclusion.
\end{proof}

\begin{cor}
Let $M$ be a smooth manifold and $N\subseteq M$ a subset which is a topological manifold in the subspace topology. Then the induced subcartesian structure $C^\infty(N)=\bra C^\infty(M)|_N\ket$ defines a smooth manifold structure on $N$ $\iff$ $N$ is an embedded submanifold. That is, around any $x\in N$ there exists a submanifold chart.
\end{cor}

\begin{proof}
    If $N$ is an embedded submanifold, then clearly $C^\infty(N)$ defines a smooth manifold structure on $N$. Conversely, if $C^\infty(N)$ defines a manifold structure, then passing to local charts as in the proof of Proposition \ref{prop: subsets of subcartesian spaces are subcartesian} provides the desired submanifold coordinates around any point $n\in N$.
\end{proof}

\begin{lem}\label{lem: locally differentiable maps are restrictions of smooth maps}
Let $X$ and $Y$ be two subcartesian spaces, $F:X\to Y$ a differentiable map, and $x\in X$. Then there exists
\begin{itemize}
    \item[(1)] singular charts $\phi:U\subseteq X\into \bR^N$ about $x$ and $\psi:V\subseteq Y\into \bR^M$ about $F(x)$ with $F(U)\subseteq V$;
    \item[(2)] open neighbourhood $W\subseteq \bR^N$ of $\phi(U)$; and
    \item[(3)] a smooth map $H:W\to \bR^M$
\end{itemize}
such that
$$
H|_{\phi(U)}=\psi\circ F\circ \phi^{-1}
$$
\end{lem}

\begin{proof}
Fixing $x\in X$, by continuity of $F$ we can find charts $\phi:U\subseteq X\into \bR^N$ about $x$ and $\psi:V\subseteq Y\into \bR^M$ with $F(U)\subseteq V$. Write
$$
G:=\psi\circ F\circ\phi^{-1}:\phi(U)\to \psi(V).
$$
Since $G$ is a differentiable map, if we write $G=(G_1,\dots,G_M)$, then $G_i\in C^\infty(\phi(U))$ for each $i$. By Proposition \ref{prop: subspace differential structure}, for each $i$ we can find an open neighbourhood $W_i\subseteq \bR^N$ of $\phi(U)$ and smooth functions $h_i\in C^\infty(W_i)$ such that $h_i|_{\phi(U)}=G_i$. Letting
$$
W:=\bigcap_{i=1}^M W_i
$$
and letting 
$$
H:=(h_1|_W,\dots,h_M|_W)
$$
we get the desired result.
\end{proof}

\subsection{Quotients by Proper Group Actions}\label{sub: quotients of compact groups}

As a special, but very important case, let's examine the case of representations of compact group actions. For all that follows, we let $K$ be a compact Lie group and $V$ a finite-dimensional real $K$-representation. 

\begin{defs}
Let $\pi:V\to V/K$ be the canonical quotient map. Define
$$
C^\infty(V/K):=\{f:V/K\to \bR \ | \ \pi^*f\in C^\infty(V)^K\}.
$$
\end{defs}
Our goal is to show $C^\infty(V/K)$ is a subcartesian structure on $V/K$. Write $\bR[V]^K$ for the ring of $K$ invariant polynomials on $V$.

\begin{theorem}[Hilbert {\cite{weyl_classical_2016}}]
The ring $\bR[V]^K$ is finitely generated.
\end{theorem}

This allows us to make the following definition following \cite{mol_stratification_2024}.

\begin{defs}[Hilbert Map]\label{def: hilbert map}
    Let $p_1,\dots,p_N\in \bR[V]^K$ be a finite list of generators. Call the map
    $$
    p:V\to \bR^N;\quad v\mapsto (p_1(v),\dots,p_N(v))
    $$
    a \textbf{Hilbert map}.
\end{defs}

Note that different choices of generators will give us different Hilbert maps.

\begin{theorem}[{\cite{schwarz_smooth_1975, mather_differentiable_1977}}]\label{thm: schwarz embedding}
The Hilbert map $p:V\to \bR^N$ induces a continuous embedding
$$
\overline{p}:V/K\into \bR^N.
$$
Furthermore, $\overline{p}^*C^\infty(\bR^N)=C^\infty(V/K)$.
\end{theorem}

This means that quotients of compact group actions have natural global singular charts. 

\begin{cor}\label{cor: fixed point set of representation is embedded nicely by Hilbert maps}
Let $V$ be a finite dimensional $K$-representation and $V^K$ the fixed point set of $V$.
\begin{itemize}
    \item[(1)] $C^\infty(V^K)=C^\infty(V)|_{V^K}=C^\infty(V)^K|_{V^K}$.
    \item[(2)] A choice of a Hilbert map $p:V\to \bR^N$ defines an embedding of smooth manifolds
    $$
    p|_{V^K}:V^K\into \bR^N.
    $$
\end{itemize}
\end{cor}

\begin{proof}
\begin{itemize}
    \item[(1)] The identity $C^\infty(V^K)=C^\infty(V)|_{V^K}$ is a consequence of $V^K$ being a closed subset. For the equality $C^\infty(V)|_{V^K}=C^\infty(V)^K|_{V^K}$ choose some $f\in C^\infty(V)$. Choosing a $K$ invariant inner-product on $K$ such that $\text{vol}(K)<\infty$, we define
    $$
    f^K(v):=\frac{1}{\text{vol}(K)}\int_{k\in K}f(k\cdot v).
    $$
    Then $f^K$ is invariant and if $v\in V^K$ then $f^K(v)=f(v)$. Hence, $f^K|_{V^K}=f|_{V^K}$.

    \item[(2)] This amounts to unfolding definitions.
\end{itemize}
\end{proof}

Thanks to the slice theorem (Theorem \ref{thm: slice theorem}), this whole discussion can naturally be extended to that of proper group actions. Now let $G$ be a connected Lie group and $M$ a smooth proper $G$-space.

\begin{prop}\label{prop: subcartesian on quotient space}
$C^\infty(M/G)$ is a subcartesian structure on $M/G$.
\end{prop}

\begin{proof}
This is a simple application of Theorem \ref{thm: slice theorem}. Around any $x\in M$ we can find a $G$-invariant neighbourhood $U\subseteq M$ of $x$ and a $G$-equivariant diffeomorphism $\phi:U\to G\times_{G_x}\nu_x(M,G)$. Hence, we have
$$
U/G\cong (G\times_{G_x}\nu_x(M,G))/G\cong \nu_x(M,G)/G_x
$$
Equipping $\nu_x(M,G)$ with a Hilbert map as in  Theorem \ref{thm: schwarz embedding} $p:\nu_x(M,G)\to \bR^N$ provides the desired singular charts.
\end{proof}

This result can be generalized to distinguished closed subsets of $M$. 

\begin{defs}[Sliceable Subset]\label{def: sliceable subset}
    Let $A\substeq M$ be a subset which is closed under the $G$-action. Say $A$ is \textbf{sliceable} if for any $a\in A$, we can find a slice neighbourhood $\phi:U\subseteq M\to G\times_{G_a}\nu_a(M,G)$ about $a$ as in Theorem \ref{thm: slice theorem} so that
    $$
    \phi(U\cap A)=G\times_{G_a} S,
    $$
    where $S\substeq \nu_a(M,G)$ is a $G_a$-invariant subset.
\end{defs}

\begin{prop}\label{prop: subcartesian on quotient of closed subset}
    Suppose $A\substeq M$ is a closed sliceable subset of $M$ and write $\pi_A:A\to A/G$ for the quotient map of the $G$-action on $A$. Then
    $$
    C^\infty(A/G)=\{f:A/G\to \bR \ | \pi_A^*f\in C^\infty(A)^G\}
    $$
    is a subcartesian structure on $A$.
\end{prop}

\begin{proof}
    Since $\pi_A$ is an open map, it readily follows that $C^\infty(A/G)$ is a differential structure. So we only need to produce singular charts, which is a local matter. Thus, since $A$ is sliceable, we may assume $M=G\times_K V$ for a compact subgroup $K\leq G$ and $K$-representation $V$, and that $A=G\times_K S$ for a closed $K$ invariant subset $S\substeq V$. Letting $p:V\to \bR^N$ be a Hilbert map, we have have by Schwarz,
    $$
    p^*C^\infty(\bR^N)=C^\infty(V)^K.
    $$
    Now making use of Proposition \ref{prop: closed subspace has restriction differential structure}, since $S$ is closed, we have $C^\infty(S)=C^\infty(V)|_S$ and thus $C^\infty(S)^K=C^\infty(V)^K|_S$. Therefore,
    $$
    (p|_S)^*C^\infty(\bR^N)=C^\infty(S)^K
    $$
    Thus, the induced map
    $$
    A/G\cong S/K\into \bR^N
    $$
    is a singular chart for $A/G$.
\end{proof}

\begin{egs}\label{eg: subcartesian structure on symplectic reduction}
    To see an example of a sliceable subset, consider now a proper Hamiltonian $G$-space $J:(M,\omega)\to\mfg^*$ for connected Lie group $G$. If non-empty, I claim $J^{-1}(0)$ is a closed sliceable subset. From the continuity and equivariance of $J$, clearly $J^{-1}(0)$ is closed and $G$-equivariant. For sliceability, recall from Lemma \ref{lem: local normal form for Hamiltonian spaces} that for any $x\in J^{-1}(0)$ we can find a $G$-invariant open neighbourhood $U\substeq M$ of $x$ and a isomorphism of Hamiltonian $G$-spaces
    $$
    \phi:U\to V\subseteq G\times_{G_x}(\mfg_x^\circ\times S\nu_x(M,G)),
    $$
    where $V$ is a $G$-invariant neighbourhood of $[e,0,0]$. In this local model, the momentum map $J$ has the form
    $$
    J:G\times_{G_x}(\mfg_x^\circ\times S\nu_x(M,G))\to \mfg^*;\quad [g,q,v]\mapsto \text{Ad}_g^*(q+\mfp\circ J_x(v)),
    $$
    where $J_x:S\nu_x(M,G)\to \mfg_x^*$ is the quadratic momentum map and $\mfp:\mfg_x^*\to \mfg^*$ is a $G_x$-equivariant splitting. Here we see then that
    $$
    J^{-1}(0)=G\times_{G_x}(\{0\}\times J_x^{-1}(0))
    $$
    and clearly $J_x^{-1}(0)\substeq S\nu_x(M,G)$ is closed and $G_x$-invariant. Therefore, by Proposition \ref{prop: subcartesian on quotient of closed subset} we can equip the reduced space $M_0=J^{-1}(0)/G$ with a subcartesian structure $C^\infty(M_0)$ given by
    $$
    C^\infty(M_0)=\{f:M_0\to \bR \ | \ \pi_0^*f\in C^\infty(M)^G|_{J^{-1}(0)}\},
    $$
    where $\pi_0:J^{-1}(0)\to M_0$ is the quotient map.
\end{egs}

\section{Zariski Tangent Vectors}\label{sec: zariski tangent vectors}

Now that we have set up subcartesian spaces, we can begin a discussion on where subcartesian spaces really shine as a generalization of differential geometry in comparison to diffeologies or stacks: tangent structures. In some sense, subcartesian spaces are extremely well-suited to generalizing covariant geometric structures (like tangent spaces), but are very ill-suited to contravariant structures (like cotangent spaces). 

\ 

As with any good discussion on vector fields, we first need to give a point-wise description.

\begin{defs}[Zariski Tangent]
Let $X$ be a subcartesian space and fix $x\in X$. A Zariski tangent vector at $x$ is a linear map
$$
V:C^\infty(X)\to \bR
$$
such that for any $f,g\in C^\infty(X)$ we have
$$
V(fg)=(Vf)g(x)+f(x)(Vg).
$$
Write $T_xX$ for the set of all Zariski tangent vectors at $x$.
\end{defs}

\begin{egs}
Let $M$ be a smooth manifold, then this notion perfectly recovers the definition of a usual tangent vector.
\end{egs}

\begin{egs}\label{eg: union of coordinate axes tangent spaces}
Let $X=\{(x_1,x_2)\in \bR^2 \ | \ x_1=0\text{ or }x_2=0\}$ be the union of the coordinate axes in $\bR^2$. Then we see
$$
T_{(x_1,x_2)}X=
\begin{cases}
\text{span}_\bR(\partial_1),\quad &x_1\neq 0\\
\text{span}_\bR(\partial_2),\quad &x_2\neq 0\\
T_{(0,0)}\bR^2,\quad &(x_1,x_2)=(0,0)
\end{cases}
$$
In particular, the dimension of Zariski tangent spaces can vary over a general subcartesian space.
\end{egs}

\begin{lem}\label{lem: Zariski tangent vectors annihilate locally constant functions}
Let $X$ be a subcartesian space, $x\in X$, and $U\subseteq X$ an open neighbourhood of $x$. Suppose $f\in C^\infty(X)$ satisfies $f|_U=0$. Then $Vf=0$ for all $V\in T_x X$.
\end{lem}

\begin{proof}
Choose a neighbourhood $W\subseteq U$ of $x$ and a function $\phi\in C^\infty(X)$ such that $\phi|_W=1$ and $\phi|_{X\setminus U}=0$. Then $\phi f=0$. Hence for any $V\in T_xX$ we have
$$
V(\phi f)=0.
$$
Using the derivation property, we then get
$$
0=V(\phi f)=V(\phi)f(x)+\phi(x)V(f)=V(f).
$$
\end{proof}

\begin{cor}
If $X$ is a subcartesian space, $x\in X$, and $f,g\in C^\infty(X)$ satisfy $f|_U=g|_U$ on some open neighbourhood $U\subset X$ of $x$, then $Vf=Vg$ for all $V\in T_xX$.
\end{cor}

\begin{defs}[Derivative at a Point]\label{def: derivative at a point}
If $X$ and $Y$ are two subcartesian spaces, $F:X\to Y$ a differentiable map, and $x\in X$. Define \textbf{the derivative of $F$ at $x$} to be the linear map
$$
T_x F:T_xX\to T_{F(x)}Y
$$
defined as follows. Choosing $V\in T_x X$ and $f\in C^\infty(Y)$, we define
$$
T_xF(V)(f):=V(f\circ F).
$$
\end{defs}

\begin{cor}\label{cor: description of tangent space to subspace}
Let $X$ be a subcartesian space, $Y\subseteq X$ a subspace, and $y\in Y$. If $\iota:Y\into X$ is the inclusion, then
\begin{itemize}
    \item[(1)] $T_y\iota:T_yY\to T_yX$ is injective.
    \item[(2)] Let $I_Y=\{f\in C^\infty(X) \ | \ f|_Y=0\}$. Then,
    $$
    T_y\iota(T_yY)=\{V\in T_yX \ | \ Vf=0\text{ for all }f\in I_Y\}.
    $$
\end{itemize}
\end{cor}

\begin{proof}
\begin{itemize}
    \item[(1)] Let $V\in T_yY$ and suppose $T_y\iota(V)=0$. We show $V=0$. To do this, fix $f\in C^\infty(Y)$. Then we can find a neighbourhood $U\subseteq X$ of $y$ and $g\in C^\infty(X)$ so that $g|_{U\cap Y}=f|_{U\cap Y}$. By Lemma \ref{lem: Zariski tangent vectors annihilate locally constant functions}, this implies $V(g|_Y)=V(f)$. Notice that, by definition, $V(g|_Y)=T_y\iota(V)g$. Hence $Vf=0$. Since $f$ was arbitrary, we conclude that $V=0$ and thus $T_y\iota$ is injective.

    \item[(2)] Suppose $V\in T_y Y$ and $f\in I_Y$. Then clearly
    $$
    T_y\iota(V)f=V(f|_Y)=V(0)=0.
    $$
    Now suppose $V\in T_yX$ satisfies $Vf=0$ for all $f\in I_Y$. We define $\overline{V}\in T_yY$ as follows. Choose $f\in C^\infty(Y)$ and let $U\subseteq X$ be a neighbourhood of $y$ and $g\in C^\infty(X)$ satisfy $g|_{U\cap Y}=f|_{U\cap Y}$. Then define
    \begin{equation}\label{eq: extending Zariski tangent vector from subspace}
    \overline{V}f:=Vg.
    \end{equation}
    
    This is well-defined. To see this, let $g_1,g_2\in C^\infty(X)$ be two functions and $U_1,U_2\subseteq X$ two neighbourhoods of $y$ so that $g_i|_{U_i\cap Y}=f|_{U_i\cap Y}$. We show $Vg_1=Vg_2$. Letting $U=U_1\cap U_2$, we see that $g_1-g_2\in I_{U\cap Y}$. Choose a neighbourhood $W\subseteq X$ of $x$ and a non-negative function $\phi\in C^\infty(X)$ with
    $$
    \phi|_W=1\quad\text{ and }\quad \phi|_{X\setminus U}=0.
    $$
    Then we see $\phi(g_1-g_2)\in I_Y$. In particular, we have
    $$
    0=V(\phi(g_1-g_2))=V(\phi)(g_1(y)-g_2(y))+\phi(y)V(g_1-g_2)=V(g_1-g_2).
    $$
    Hence $Vg_1=Vg_2$. Therefore, the formula in Equation (\ref{eq: extending Zariski tangent vector from subspace}) is well-defined. Clearly $\overline{V}\in T_y Y$ and clearly $T_y\iota(\overline{V})=V$. 
\end{itemize}
\end{proof}

\begin{lem}\label{lem: local diffeos on subcartesian}
Suppose $X,Y\subseteq \bR^N$ are two subsets and suppose $F:X\to Y$ is a diffeomorphism of subcartesian spaces. Then around each $x\in X$, we can find neighbourhoods $U,V\subseteq \bR^N$ of $x$ and $F(x)$, respectively, and a diffeomorphism $H:U\to V$ so that $H|_{U\cap X}=F|_{U\cap X}$.
\end{lem}

\begin{proof}
After potentially applying affine transformations, we may assume that $x=F(x)=0$. Due to Lemma \ref{lem: locally differentiable maps are restrictions of smooth maps}, we can find a neighbourhoods $U,V\subseteq \bR^N$ around $0$ and a smooth map $G:U\to V$ so that
$$
G|_{U\cap X}=F|_{U\cap X}
$$
An argument of Mol \cite[Proposition 2.33]{mol_stratification_2024} shows that after possibly shrinking $U$, we can find a map $A:U\to V$ such that $A|_{U\cap X}=0$ and so that $H=G+A$ and $T_0H$ is invertible. Thus, after possibly shrinking $U$ once again, we get the desired map $H$.
\end{proof}

\begin{theorem}\label{thm: charts of subcartesian spaces are locally diffeomorphic}
Let $X$ be a subcartesian space and $\phi:U\into \bR^N$ and $\psi:V\into \bR^M$ two singular charts with $U\cap V\neq\emptyset$. Write $Q:=\text{max}(N,M)$ and let $\iota_N:\bR^N\into \bR^Q$ and $\iota_M:\bR^M\into \bR^Q$ be the inclusions of the first coordinates. Then, for each $x\in U\cap V$ there exists 
\begin{itemize}
    \item[(1)] an open neighbourhood $x\in W\subseteq U\cap V$;
    \item[(2)] open neighbourhoods $O_\phi,O_\psi\subseteq \bR^Q$ about $\phi(W)$ and $\psi(W)$, respectively;
    \item[(3)] a $C^\infty$-diffeomorphism $H:O_\phi\to O_\psi$ 
\end{itemize}
all so that the diagram below commutes.
\begin{equation}
    \begin{tikzcd}
        O_\phi\arrow[rr,"H"] && O_\psi \\
        & W\arrow[ul,"\phi"]\arrow[ur,"\psi",swap] &
    \end{tikzcd}
\end{equation}

\end{theorem}

\begin{proof}
Without loss of generality, we may assume $N=M=Q$. Fix charts $\phi:U\into \bR^N$ and $\psi:V\into \bR^N$ with $x\in U\cap V$. Then its a triviality to see that the chain of maps
$$
\begin{tikzcd}
    \phi(U\cap V)\arrow[r,"\phi^{-1}"] & U\cap V\arrow[r,"\psi"] & \psi(U\cap V)
\end{tikzcd}
$$
defines a diffeomorphism between two subcartesian subsets of $\bR^N$. Hence, by Lemma \ref{lem: local diffeos on subcartesian} we can find open neighbourhoods $A,B\subseteq \bR^N$ about $\phi(x)$ and $\psi(x)$, respectively, and a diffeomorphism $H:A\to B$ such that $H|_{A\cap \phi(U\cap V)}=\psi\circ\phi^{-1}|_{A\cap \phi(U\cap V)}$. Letting $W=\phi^{-1}(\phi(U\cap V)\cap A)$ gives the desired diagram.
\end{proof}

\section{Zariski Vector Fields}\label{sec: zariski vfields}

Now that we have Zariski tangent vectors, it is natural to consider ``smoothly varying'' tangent vectors, i.e vector fields. These can be characterized either abstractly as derivations of the ring of differentiable functions, or as differentiable sections of the Zariski tangent bundle (a notion to be made more precise later on). A powerful aspect of Zariski vector fields is that they have flows (see for instance \cite{karshon_vector_2023}), but we will not be discussing flows in this thesis as they are not relevant to the later discussions on quantization.

\begin{defs}[Zariski Vector Field]
Let $X$ be a subcartesian space. A Zariski vector field on $X$ is a derivation
$$
V:C^\infty(X)\to C^\infty(X).
$$
Write $\mfX(X)$ for the set of all Zariski vector fields on $X$.
\end{defs}

\begin{remark}
I will insist upon using the full phrase ``Zariski vector field'' rather than simply ``vector field'' because Sniatycki \cite{sniatycki_differential_2013} defines vector fields to be Zariski vector fields with flows by local diffeomorphisms. Many of the Zariski vector fields we will be discussing later on will also be vector fields in the sense of Sniatycki, but as flows are not an important consideration in this thesis, I will not be remarking further on this.
\end{remark}

Since Zariski vector fields are just derivations, we get a canonical Lie bracket on Zariski vector fields. Indeed, if $V,W\in \mfX(X)$ and $f\in C^\infty(X)$, then since $V(f),W(f)\in C^\infty(X)$, we can define
$$
V(W(f))-W(V(f))
$$
It is straightforward to see that this actually defines a new derivation, hence Zariski vector field.

\begin{defs}[Lie Bracket on Zariski Vector Fields]\label{def: Lie bracket of Zariski vf}
    Let $X$ be a subcartesian space and $V,W\in \mfX(X)$. Define the Lie bracket of $V$ and $W$, denoted $[V,W]$, by
    $$
    [V,W](f)=V(W(f))-W(V(f)).
    $$
\end{defs}

It is then an exercise in linear algebra to see that
$$
[\cdot,\cdot]:\mfX(X)\times \mfX(X)\to \mfX(X);\quad (V,W)\mapsto [V,W]
$$
is bilinear, antisymmetric, and satisfies the Jacobi identity. Now, before we get to key properties of Zariski vector fields in regards to subspaces, I would first like to give an alternative interpretation in terms of a version of the tangent bundle for subcartesian spaces.

\begin{defs}
Let $X$ be a subcartesian space. Define the Zariski tangent bundle by
$$
TX:=\bigcup_{x\in X}T_xX
$$
\end{defs}

Let $\pi:TX\to X$ be the natural projection defined by
$$
\pi|_{T_xX}:T_xX\to X;\quad V\mapsto x,
$$
for $x\in X$. This allows us to formally define a collection of real-valued functions on $TX$ by
$$
\pi^*C^\infty(X):=\{f\circ \pi \ | \ f\in C^\infty(X)\}.
$$
Furthermore, given $f\in C^\infty(X)$ we can define
$$
df:TX\to \bR;\quad V\mapsto V(f).
$$
Hence, allowing us to define another formal collection of functions on $TX$:
$$
dC^\infty(X):=\{df \ | \ f\in C^\infty(X)\}.
$$

\begin{defs}[Tangent Differential Structure, {\cite{sniatycki_differential_2013}}]\label{def: tangent differential structure}
Let $X$ be a subcartesian space and endow $TX$ with the initial topology induced by $\pi^*C^\infty(X)\cup dC^\infty(X)$ and the differential structure
$$
C^\infty(TX)=\bra \pi^*C^\infty(X)\cup dC^\infty(X)\ket.
$$
\end{defs}

Our aim now is to show that $C^\infty(TX)$ defines a subcartesian structure on $TX$. To do this, let us also introduce the (for now) formal derivative of a differentiable map between subcartesian spaces.

\begin{defs}[Tangent Map]
Let $F:X\to Y$ be a differentiable map between subcartesian spaces. Define the tangent map $TF:TX\to TY$ by
$$
TF|_{T_xX}=T_xF:T_xX\to T_{F(x)}Y,
$$
where $T_xF$ is the derivative map from Definition \ref{def: derivative at a point}.
\end{defs}

As one might imagine, the derivative map is a differentiable map.

\begin{prop}
Let $F:X\to Y$ and $G:Y\to Z$ be a differentiable maps between subcartesian spaces. 
\begin{itemize}
    \item[(1)] The diagram commutes
    \begin{equation}\label{eq: derivatives are bundle maps}
    \begin{tikzcd}
        TX\arrow[r,"TF"]\arrow[d] & TY\arrow[d]\\
        X\arrow[r,"F"] & Y
    \end{tikzcd}
    \end{equation}
    where the vertical arrows are the canonical projection maps $TX\to X$ and $TY\to Y$.
    \item[(2)] $TF:TX\to TY$ is differentiable.
    \item[(3)] $T(G\circ F)=TF\circ TG$
    \item[(4)] If $F$ is a diffeomorphism, so is $TF$.
\end{itemize}
\end{prop}

\begin{proof}
    \begin{itemize}
        \item[(1)] This is an immediate consequence of the fact that the image of $T_xF$ lies in $T_{F(x)}Y$ for all $x\in X$.
        \item[(2)] Write $\pi_X:TX\to X$ and $\pi_Y:TY\to Y$ for the projection maps. By Proposition \ref{prop: differentiable maps are continuous} it suffices to show for all $f\in C^\infty(Y)$ that 
        \begin{itemize}
            \item $(TF)^*\pi_Y^*f\in C^\infty(X)$ and
            \item $(TF)^*df\in C^\infty(X)$.
        \end{itemize}
        For definiteness, let us fix $f\in C^\infty(Y)$. First, we have from Equation (\ref{eq: derivatives are bundle maps}) that
        $$
        (TF)^*\pi_Y^*f=(\pi_Y\circ TF)^*f=(F\circ \pi_X)^*f.
        $$
        The map $F\circ \pi_X$ is differentiable, hence $(TF)^*\pi_Y^*f\in C^\infty(TX)$. Now, to show that $(TF)^*df\in C^\infty(TX)$ observe that from unwinding definitions that
        $$
        (TF)^*df=d(F^*f)
        $$
        Once again, $F^*f\in C^\infty(X)$, and thus $d(F^*f)\in C^\infty(TX)$.
        \item[(3)] Trivially follows from the definitions of the point-wise derivative.
        \item[(4)] Apply (3) to $F$ and $F^{-1}$ and the result follows.
    \end{itemize}
\end{proof}

\begin{lem}\label{lem: zariski vector bundle is the same as usual tangent bundle for smooth manifold}
Let $M$ be a smooth manifold. Then the natural smooth structure on $TM$ coincides with the differential structure $C^\infty(TM)$ defined in Definition \ref{def: tangent differential structure}.
\end{lem}

\begin{proof}
Recall that $TM$ is constructed using differentials of local charts. In particular, $f:TM\to \bR$ is smooth if and only if for each chart 
$$
\phi:U\subseteq M\to \bR^n,
$$ 
the map
$$
f\circ (T\phi)^{-1}:T\phi(U)\to \bR
$$
is smooth. Let $x_1,\dots,x_n$ be the coordinate function on $\bR^n$ and $v_1,\dots,v_n$ the corresponding linear coordinates on $T\bR^n=\bR^n\times \bR^n$. Letting $F=f\circ (T\phi)^{-1}$, we see that since $T\phi(U)\subseteq T\bR^n$ is an open subset that $F$ is a function of the form
$$
F(x_1,\dots,x_n,v_1,\dots,v_n).
$$
Now, writing $\phi=(\phi_1,\dots,\phi_n)$ for the coordinates of $\phi$ and $\pi:TM\to M$ the projection map, we have 
$$
f=F\circ T\phi=F(\pi^*\phi_1,\dots,\pi^*\phi_n,d\phi_1,\dots,d\phi_n)
$$
This precisely means that $f\in C^\infty(TM)$ from Definition \ref{def: tangent differential structure}. The converse inclusion is clear.
\end{proof}

\begin{theorem}\label{thm: zariski tangent bundle is a pseudo-bundle}
Let $X$ be a subcartesian space and endow $TX$ with the differential structure
$$
C^\infty(TX)=\bra \pi^*C^\infty(X)\cup dC^\infty(X)\ket.
$$
Then the following holds.
\begin{itemize}
    \item[(1)] The projection $\pi:TX\to X$ and the zero section
    $$
    0_X:X\to TX
    $$
    are differentiable.
    \item[(2)] $C^\infty(X)$ is a subcartesian structure. In particular, if $\phi:U\subseteq X\into \bR^N$ is a singular chart of $X$, then
    $$
    T\phi:\pi^{-1}(U)=TU\into T\bR^N
    $$
    is a singular chart of $TX$.
    \item[(3)] Both scalar multiplication
    $$
    \bR\times TX\to TX
    $$
    and fibre-wise addition
    $$
    TX\times_\pi TX\to TX
    $$
    are differentiable.
\end{itemize}
\end{theorem}

\begin{proof}
\begin{itemize}
    \item[(1)] By definition, $\pi^*C^\infty(X)\subseteq C^\infty(TX)$, and so $\pi$ is differentiable. Let
    $$
    0_X:X\to TX
    $$
    be the map which sends each $x$ to the zero tangent vector. Note that
    $$
    0_X^*\pi^*C^\infty(X)=C^\infty(X)
    $$
    and
    $$
    0_X^*dC^\infty(X)=\{0\}\subseteq C^\infty(X).
    $$
    By Proposition \ref{prop: differentiable maps are continuous}, $0_X$ is differentiable.

    \item[(2)] Let $\phi:U\subseteq X\into \bR^N$ be a singular chart. Since $\phi$ is a diffeomorphism onto its image, we have $T\phi:TU\to T\phi(U)\subseteq T\bR^n$ is also a diffeomorphism.

    \item[(3)] In local charts, both scalar multiplication and addition can be realized as restrictions of the usual scalar multiplication and additions on $T\bR^n$, hence are smooth.
\end{itemize}
\end{proof}

\begin{prop}\label{prop: global to infinitesimal for zariski vector fields}
Let $X$ be a subcartesian space.
\begin{itemize}
    \item[(1)] Let $V\in\mfX(X)$ be a Zariski vector field and $x\in X$. Then the map
    \begin{equation}\label{eq: global derivation defines tangent vector}
        V_x:C^\infty(X)\to \bR;\quad f\mapsto V(f)(x)
        \end{equation}
    is a Zariski tangent vector.
    \item[(2)] The map in Equation (\ref{eq: global derivation defines tangent vector}) defines a canonical bijection between the Zariski vector fields $\mfX(X)$ and the set of all differentiable sections of the map $\pi:TX\to X$.
\end{itemize}
\end{prop}

\section{Restricting Vector Fields}\label{sec: restricting vfields}

To close out this chapter, let us now discuss the matter of restricting Zariski vector fields to subsets. As is the case with submanifolds, we cannot restrict any arbitrary vector field to a subset and hope that it remains tangent. Rather, only special vector fields can be restricted. As is the case with submanifolds, these can be characterized in terms of preserving the vanishing ideal of the subset. They key importance here for the following discussion is generalizing the construction of a reduced prequantum connection from Section \ref{sec: [Q,R]}, where we restricted, then pushed vector fields tangent to the zero level-set of the momentum map to the reduced space. As it turns out, we can use almost precisely the same construction in the singular setting. But first, we need our set-up.

\begin{defs}[Restrictable Vector Fields]\label{def: restrictable vf}
Let $X$ be a subcartesian space and $Y\subseteq X$ a subset. Write
$$
\mfX(X,Y):=\{V\in \mfX(X) \ | \ V_y\in T_yY\text{ for all }y\in Y\}.
$$
for the set of all Zariski vector fields which are tangent to $Y$.
\end{defs}

As promised, restrictable vector fields can be characterized in terms of preserving vanishing ideals.

\begin{prop}\label{prop: characterization of restrictable vector fields}
Let $Y\subseteq X$ is a subset of a subcartesian space. 
$$
\mfX(X,Y)=\{V\in \mfX(X) \ | \ VI_Y\subseteq I_Y\}.
$$
Furthermore, for any $V\in \mfX(X,Y)$ and $f\in C^\infty(X)$ we have
$$
V|_Y(f|_Y)=V(f)|_Y.
$$
\end{prop}

\begin{proof}
This is a straightforward application of Proposition \ref{prop: global to infinitesimal for zariski vector fields} together with the point-wise characterization of tangent spaces to subsets given in Corollary \ref{cor: description of tangent space to subspace}. 
\end{proof}

\begin{cor}\label{cor: Lie bracket commutes with restriction}
    Suppose $V,W\in \mfX(X,Y)$. Then,
    \begin{itemize}
        \item[(1)] The Lie bracket $[V,W]\in \mfX(X,Y)$
        \item[(2)] The Lie bracket commutes with restriction. That is, $[V|_Y,W|_Y]=[V,W]|_Y$.
    \end{itemize}
\end{cor}

\begin{proof}
    \begin{itemize}
        \item[(1)] Suppose $f\in I_Y$. By Proposition \ref{prop: characterization of restrictable vector fields} we have $V(f),W(f)\in I_Y$. Applying the Proposition once again, we then obtain that $V(W(f)),W(V(f))\in I_Y$ as well. Hence,
        $$
        [V,W](f)=V(W(f))-W(V(f))\in I_Y.
        $$
        Thus, $[V,W]\in \mfX(X,Y)$.

        \item[(2)] To show $[V|_Y,W|_Y]=[V,W]|_Y$ it suffices to show for any $f\in C^\infty(X)$ that
        $$
        [V|_Y,W|_Y](f|_Y)=[V,W]|_Y(f|_Y).
        $$
        Fixing $f\in C^\infty(X)$, we first see that
        $$
        [V|_Y,W|_Y](f|_Y)=V|_Y(W|_Y(f|_Y))-W|_Y(V|_Y(f|_Y))
        $$
        Using Proposition \ref{prop: characterization of restrictable vector fields}, we have 
        $$
        V|_Y(W|_Y(f|_Y))=V(W(f)|_Y)=V(W(f))|_Y
        $$
        By symmetry, we obtain
        $$
        [V|_Y,W|_Y](f|_Y)=V(W(f))|_Y-W(V(f))|_Y=[V,W](f)|_Y=[V,W](f|_Y).
        $$
    \end{itemize}
\end{proof}

Given a subcartesian space $X$ and a subset $Y\substeq X$, we now have a notion restrictable vector fields $\mfX(X,Y)$ and a natural map
$$
\mfX(X,Y)\to \mfX(Y);\quad V\mapsto V|_Y.
$$
The question is now, when is this map surjective? This is a non-trivial question as it will determine the viability of defining prequantum connections on reduced spaces. To begin to study this question, we then turn to a result of Karshon and Lerman.

\begin{lem}[{\cite{karshon_vector_2023}}]\label{lem: local extension of vector fields in a chart}
Suppose $X\subseteq \bR^N$ is a subset and $V\in\mfX(X)$. Then for any open neighbourhood $W$ of $X$ and $h\in C^\infty(W)$ we have
$$
V(h)=\sum_{i=1}^NV(x_i|_X)\frac{\partial h}{\partial x_i}\bigg|_X.
$$
\end{lem}

\begin{cor}\label{cor: local extension by vector fields}
Let $X\subseteq \bR^N$ be a subset and $V\in\mfX(X)$. Then there is a neighbourhood $W\subseteq \bR^N$ of $X$ and a vector field $\widetilde{V}\in\mfX(W,X)$ with $\widetilde{V}|_X=V$.
\end{cor}

\begin{proof}
Let $x_1,\dots,x_N:\bR^N\to \bR$ be the coordinate functions on $\bR^N$. Since there are only finitely many functions, we can find an open neighbourhood $W\subseteq \bR^N$ of $X$ and $g_1,\dots,g_N\in C^\infty(W)$ so that 
$$
V(x_i|_X)=g_i|_X.
$$
Define $\widetilde{V}\in\mfX(W)$ by
$$
\widetilde{V}:=\sum_{i=1}^Ng_i\frac{\partial}{\partial x_i}
$$
Observe that if $U\subseteq W$ is an open set and $f\in I_{U\cap X}$, then making use of Lemma \ref{lem: local extension of vector fields in a chart}, we get
$$
V(f|_{U\cap X})=\sum_{i=1}^N V(x_i|_{U\cap X})\frac{\partial f}{\partial x_i}\bigg|_{U\cap X}=0.
$$
In particular, notice that if $x_0\in U\cap X$ and $(f)_{x_0}\in (I_X)_{x_0}$, then
$$
\widetilde{V}_{x_0}(f)_{x_0}=0.
$$
Hence $\widetilde{V}|_X=V$.
\end{proof}

\begin{cor}\label{cor: can extend vector fields on closed subset}
If $M$ is a smooth manifold and $X\subseteq M$ is a closed set, then the canonical map
$$
\mfX(M,X)\to \mfX(X);\quad V\mapsto V|_X
$$
is a surjection.
\end{cor}

\begin{proof}
Let $V\in\mfX(X)$. Using Corollary \ref{cor: local extension by vector fields} chart-wise, we can find a covering $\{U_i\}_{i\in I}$ of $X$ and vector fields $\widetilde{V}_i\in\mfX(U_i,U_i\cap X)$ such that $\widetilde{V}_i|_{U_i\cap X}=V|_{U_i\cap X}$. Let $U_\infty=M\setminus X$ and $V_\infty=0$. Then we can find a partition of unity $\{\phi_i\}_{i\in I\cap \{\infty\}}$ and we define
$$
\widetilde{V}:=\sum_{i\in I\cup \{\infty\}} \phi_i\widetilde{V}_i.
$$
It is then straightforward to see that $\widetilde{V}|_X=V$. Hence the result.
\end{proof}

\begin{egs}\label{eg: vector fields need not span the zariski tangent bundle}
Using the above result, we can show that extending is not possible in general.  That is, we can find an example of a subcartesian space $X$ and a closed set $A\subseteq X$ so that
$$
\mfX(X,A)\to \mfX(A)
$$
is not surjective. This is an example due to Sniatycki \cite{sniatycki_differential_2013}. Indeed, consider
$$
X=\{(x,y)\in \bR^2 \ | \ x=0\text{ or }y=0\}
$$
and let
$$
A=\{(x,y)\in \bR^2 \ | \ x\geq 0\text{ and }y=0\}.
$$
Observe that $A\subset X\subset \bR^2$ are both closed subsets of $\bR^2$, hence the maps
$$
\mfX(\bR^2,X)\to \mfX(X)\quad\text{ and }\quad \mfX(\bR^2,A)\to \mfX(A)
$$
are surjective. In particular, observe that if $V\in \mfX(\bR^2,X)$ then $V_0=0$. Indeed, this is a consequence of the fact that $V(xy)\equiv 0$. On the other hand, we can easily see that $\partial/\partial x\in \mfX(\bR^2,A)$ and this vector field clearly does not vanish at $0$. In particular, we see that $\partial/\partial x|_A$ admits no extension to $\mfX(X)$. 
\end{egs}

So, in general we have no reason to presume that a vector field on a closed set should admit an extension to the whole space. This is not a problem for us because the kinds of spaces we will be dealing with later on all have contained within them a distinguished connected open dense set which carries a manifold structure. For spaces like this, extension results do hold.

\begin{cor}\label{cor: sufficient for restriction to be surjective}
Let $X$ be a locally closed subcartesian space and $A\subseteq X$ a closed set. Suppose $X$ has a connected open dense set $U\subseteq X$ such that
\begin{itemize}
    \item[(1)] $A\cap U$ is dense in $A$.
    \item[(2)] For any point $a\in A$, we can find a chart $\phi:W\subseteq X\into \bR^N$ so that
    $$
    \phi(U\cap W)=\bR^k\times \{0\}
    $$
    for some $k$.
\end{itemize}
Then
$$
\mfX(X,A)\to \mfX(A);\quad V\mapsto V|_A
$$
is surjective.
\end{cor}

\begin{proof}
Fix a Zariski vector field $V\in\mfX(A)$. It suffices to show for any $a\in A$ we can find a neighbourhood $W_a\subseteq X$ and a vector field $\widetilde{V}_a\in\mfX(W_a,A\cap W_a)$ so that
$$
\widetilde{V}_a|_{A\cap W_a}=V|_{A\cap W_a}.
$$
Indeed, if this is the case, then we can find a cover $\{W_i\}$ of $A$ and partitions of unit $\{\phi_i\}$ subordinate to $\{W_i\}\cup\{X\setminus A\}$ and patch these local extensions together to obtain a global extension that vanishes outside an open neighbourhood of $A$. 

\

Thus, passing to charts, it suffices to assume that $X\subseteq \bR^N$ is a closed subset and that the distinguished dense subset $U$ is an open subset of $U=\bR^k\times\{0\}$. Now let $x_1,\dots,x_N$ be the coordinate functions on $\bR^N$. 

\

By construction, $x_{k+1}|_U=\cdots=x_{N}|_U=0$. Thus, since $U$ is dense in $X$, we have $x_{k+1}|_X=\cdots =x_N|_X=0$. Thus, $V(x_{k+1}|_A)=\cdots=V(x_N|_A)=0$. Thus, by Corollary \ref{cor: local extension by vector fields}, if we let $\alpha_1,\dots,\alpha_k\in C^\infty(\bR^N)$ be extensions of $V(x_1|_A),\dots,V(x_k|_A)$, then
$$
\widetilde{V}=\sum_{i=1}^k \alpha_i \frac{\partial}{\partial x_i}
$$
is an extension of $V$ to $\mfX(\bR^N)$. Furthermore, by construction $\widetilde{V}$ is tangent to $U$ and hence to $X$. Indeed, if $f\in I_X$ is a function which vanishes on $X$ then $f\in I_U$ as well. By definition, $\widetilde{V}(I_U)\subseteq I_U$ and hence $\widetilde{V}(f)\in I_U$. Hence, $\widetilde{V}(f)|_X$ is a continuous function which vanishes along an open dense set and thus $\widetilde{V}(f)|_X=0$. Thus, $\widetilde{V}\in \mfX(\bR^N,X)$ and $\widetilde{V}|_X\in \mfX(X,A)$ with $\widetilde{V}|_A=V$. 
\end{proof}
\cleardoublepage
\chapter{Differentiable Stratified Spaces}\label{ch: diff strat spaces}

Another approach to studying singular spaces, would be to take the singular space, partition it into smooth manifolds, then do differential geometry on each of the pieces of the partition and hope we can glue the results together coherently. This is roughly the idea of a stratified space. These spaces were first investigated by Whitney \cite{whitney_local_2015}, Thom \cite{thom_ensembles_1969}, and Mather \cite{mather_stratifications_1973} as a way to extend smooth geometry to algebraic varieties, as well as to studying the stability of mappings. Later on, Bierstone \cite{bierstone_lifting_1975}, Schwarz \cite{schwarz_smooth_1975}, and Mather \cite{mather_differentiable_1977} would extend this theory to quotients by proper group actions, an area we will be discussing in great detail in Section \ref{sec: strats and group actions}. Molino \cite{molino_riemannian_1988} also extended stratification theory into the realm of singular foliations with his definition of the ``dimension stratification''. This also is discussed (and actually proven to be a stratification!) in Subsection \ref{sub: SRFs}.

\ 

Now, there are unfortunately many inequivalent approaches to the study of stratified spaces, and I shall be adding to this confusion. In order to avoid introducing more concepts on top of the mountain of ideas covered by this thesis, I shall be relying on the theory of stratified spaces presented in \cite{crainic_orbispaces_2018}. The curious reader is encouraged to consult that source as it gives an excellent overview of the various approaches. I also discuss the different approaches to stratified spaces in my paper \cite{ross_stratified_2024}.

\section{Topological Stratified Spaces}\label{sec: top strat spaces}

To begin, let us discuss topological stratified spaces. These are topological spaces together with a partition into topological manifolds, which fit together in a manner similar to the open cells of a CW complex. We will introduce differential geometry into the equation later on via a compatibility condition with subcartesian structures.

\begin{defs}[Topological Stratified Space]\label{def: topological stratified space}
A topological stratified space is a pair $(X,\Sigma)$ where
\begin{itemize}
	\item[(i)] $X$ is a Hausdorff, second countable, paracompact topological space; and
	\item[(ii)] $\Sigma$ is a locally finite partition of $X$ into connected and locally closed subsets
\end{itemize}
such that the following axioms hold.
\begin{itemize}
	\item[(1)] Each element $S\in\Sigma$ is a topological manifold in the subspace topology.
	\item[(2)] (Axiom of the Frontier) If $S,R\in\Sigma$ are two pieces with $S\cap\overline{R}\neq\emptyset$ and $S\neq R$, then $S\subseteq \overline{R}$ and $\dim S<\dim R$. Write $S\leq R$ if $S\subseteq \overline{R}$.
\end{itemize}
We call $\Sigma$ a stratification of $X$ and the elements of $\Sigma$ are known as strata. We will simply write $X$ for $(X,\Sigma)$ if the stratification is understood from context. 
\end{defs}

\begin{remark}
    The symbol $\leq$ used in the axiom of the frontier in Definition \ref{def: topological stratified space} is apt as the axiom of frontier makes $\Sigma$ into a partially ordered set. This ties into one definition of stratified spaces which can be found in, for instance the paper of Ayala et. al. \cite{ayala_local_2017}. In this, a stratified space is a topological space $X$, a partially ordered set $\mathcal{P}$, and a continuous map
    $$
    \Sigma:X\to \mathcal{P};\quad x\mapsto \Sigma_x
    $$
    with respect to (some kind of topology). The idea here being that the ``strata'' are simply the pre-images $\Sigma^{-1}(S)$. Of course this is a vast generalization of Definition \ref{def: topological stratified space} in the sense that the only assumption really being made of the strata is that they satisfy the axiom of the frontier and nothing more. 
\end{remark}

\begin{egs}\label{eg: coordinate axes}
    Consider the union of the coordinate axes in $\bR^2$:
    $$
    X=\{(x,y)\in \bR^2  \ | \ x=0\text{ or }y=0\}.
    $$
    \begin{figure}[h]
        \centering
        \begin{tikzpicture}[scale=2]
        \draw[-,thick] (-1,0) -- (1,0);
        \draw[-,thick] (0,-1) -- (0,1);
        \draw (1,0.9) node {$X$};
        \end{tikzpicture}
    \end{figure}
    
    There are many ways to partition $X$ into manifolds, but not all of them will be stratifications. For instance, consider the partition
    \begin{equation}\label{eq: first partition of coordinate axes}
    X=H\cup V_+\cup V_-,
    \end{equation}
    where
    $$
    H=\{(x,0)\in \bR^2 \ | \ x=0\}
    $$
    and
    $$
    V_{\pm}=\{(0,y)\in \bR^2  \ | \ \pm y >0\}
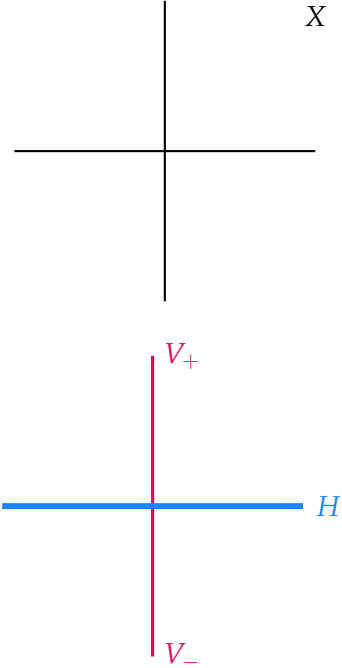
    $$

    \begin{figure}[h]
        \centering
        \begin{tikzpicture}[scale=2]
        \draw[-, very thick, myred] (0,1) node[right] {$V_+$} -- (0,0.01);
        \draw[-,very thick,myred] (0,-0.01) -- (0,-1) node[right] {$V_-$};
        \draw[-,line width=0.08cm,myblue] (-1,0) -- (1,0) node[right] {$H$} ;
        \end{tikzpicture}
    \end{figure}

    This defines a partition of $X$ into topological manifolds (each piece being homeomorphic to $\bR$), but this partition is not a stratification. Indeed, observe that 
    $$
    \overline{V_+}=\{(0,y) \ | \ y\in [0,\infty)\}
    $$
    and hence $H\cap \overline{V_+}\neq \emptyset$. However, clearly $H\not\subseteq \overline{V_+}$ as, for instance, $(1,0)\not\in \overline{V_+}$. Thus, the axiom of the frontier fails to hold for this partition.

    \ 

    Nonetheless, a refinement of the partition in Equation (\ref{eq: first partition of coordinate axes}) does give a stratification of $X$. Consider now the partition
    \begin{equation}
        X=\{(0,0)\}\cup H_+\cup H_-\cup V_+\cup V_-,
    \end{equation}
    where
    $$
    H_{\pm}=\{(x,0)\in \bR^2  \ | \ \pm x >0\}.
    $$
    \begin{figure}[h]
        \centering
        \begin{tikzpicture}[scale=2]
        \draw[-,very thick,myblue] (-1,0) node[left] {$H_-$} -- (1,0) node[right] {$H_+$} ;
        \draw[-,very thick,myred] (0,-1) node[right] {$V_-$}-- (0,1) node[right] {$V_+$};
        \filldraw (0,0) circle(0.7pt) node[anchor=south west] {$\{(0,0)\}$};
        \end{tikzpicture}
    \end{figure}
    
    Now we have four one-dimensional pieces, $H_\pm, V_\pm$ and one zero-dimensional piece $\{(0,0)\}$. In this case, we see that the axiom of the frontier holds as, for instance $\{(0,0)\}\subseteq \overline{H_+}$. 
\end{egs}

\begin{egs}\label{eg: trivial strat of manifold}
Let $M$ be a topological manifold. Since $M$ may be disconnected, the partition $\{M\}$ is not a stratification. Rather, the set of connected components, $\pi_0(M)$, canonically forms a stratification.
\end{egs}

\begin{egs}\label{eg: stratification of polytope}
Let $\Delta\subseteq \bR^n$ be a subset carved out by hyperplanes. That is, there exists dual vectors $\alpha_1,\dots,\alpha_m\in (\bR^n)^*$ such that
$$
\Delta=\bigcap_{i=1}^m\{x\in \bR^n  \ | \ \alpha_i(x)\geq 0\}.
$$
$\Delta$ has a canonical stratification, described as follows. For each $x\in\Delta$, define a set of indices $I(x)\subseteq \{1,\dots, m\}$ by
$$
I(x)=\{i\in \{1,\dots,m\} \ | \ \alpha_i(x)=0\}.
$$
Now, given any subset $I\subseteq\{1,\dots,m\}$, define
$$
\Sigma_I:=\{x\in \Delta  \ | \ I(x)=I\}.
$$
Then define
\begin{equation}
\mathcal{S}(\Delta)=\{\Sigma_I \ | \ I\subseteq \{1,\dots,m\}, \ \Sigma_I\neq\emptyset\}.
\end{equation}
Observe that for any $\Sigma_I\in \mathcal{S}(\Delta)$, we have 
$$
\overline{\Sigma_I}=\{x\in \bR^n \ | \ I\subseteq I_x\}
$$
This immediately shows $\mathcal{S}(\Delta)$ satisfies the axiom of the frontier. 
\end{egs}

\begin{egs}
    Given a stratified space $(X,\Sigma)$, we can produce a new stratified space by taking the open cone of $X$. That is, let
    $$
    CX=(X\times [0,1))/\sim
    $$
    where we declare $(x,0)\sim (x',0)$ for all $x,x'\in X$. Write $\infty$ for the equivalence class $[(x,0)]$. Then the partition
    $$
    C\Sigma=\{\infty\}\cup \{S\times (0,1) \ | \ S\in \Sigma\}
    $$
    defines a stratification of $CX$. Taking cones of stratified spaces adds a $0$ dimensional stratum and increases the dimension of all strata by $1$. Also note that the stratum $\{\infty\}$ is the unique minimal element with respect to the natural partial order on the strata. This is a useful property and gets much use in the theory of intersection homology (see for instance \cite{maxim_intersection_2019}). 
\end{egs}

As a last elementary definition for stratified spaces, let us introduce the notion of a stratified morphism. 

\begin{defs}[Stratified Morphisms]
Let $(X_1,\Sigma_1)$ and $(X_2,\Sigma_2)$ be two topological stratified spaces. A stratified morphism between them is a continuous map $F:X_1\to X_2$ such that for each $S_1\in\Sigma_1$ there exists $S_2\in\Sigma_2$ such that $F(S_1)\subseteq S_2$.
\end{defs}

\begin{egs}
    A continuous map between topological manifolds $f:M\to N$ is stratified with respect to the canonical stratifications in Example \ref{eg: trivial strat of manifold}.
\end{egs}

\begin{egs}
    We will mostly be considering ourselves with stratifications arising from group actions and foliations in the rest of the thesis, but as a last elementary example let us return back to the cone of a stratified space. Let $(X,\Sigma)$ be a stratified space and fix $t\in [0,1)$. Then,
    $$
    f_t:X\to CX;\quad x\mapsto [x,t]
    $$
    defines a stratified morphism for all $t$. Note that for $t=0$, the map $f_0$ is constant and sends all points to the unique minimal stratum $\{\infty\}$. 
\end{egs}

\section{Differentiable Stratified Spaces}\label{sec: diff strat spaces}

Stratifications are usually structures we place on singular spaces (i.e. things that aren't manifolds) in the hopes of being able to transfer over ideas and techniques from smooth geometry to settings where it normally wouldn't apply. This of course also ties into the ideas of Sikorski spaces and Subcartesian spaces. Here we give a union of these two ideas in the definition of a differentiable stratified space. First, let us give the elementary definitions and examples of differentiable stratified spaces before we begin a more in depth study of examples arising from the theory of singular foliations and proper group actions. 

\begin{defs}{Differentiable Stratified Space}\label{def: diff stratified space}
Let $(X,C^\infty(X))$ be a subcartesian space and $\Sigma$ a stratification of $X$.
\begin{itemize}
	\item[(1)] We say $\Sigma$ and $C^\infty(X)$ are \textbf{compatible} if for each stratum $S\in\Sigma$, the restricted subcartesian structure $C^\infty(S)$ is a smooth structure on $S$. 
	\item[(2)] If $C^\infty(X)$ and $\Sigma$ are compatible and $X$ is a locally closed subcartesian structure, we call the triple $(X,\Sigma,C^\infty(X))$ a \textbf{differentiable stratified space}.
	\item[(3)] If $C^\infty(X)$ and $\Sigma$ are compatible and $X$ is not assumed to be locally closed, we say $(X,\Sigma,C^\infty(X))$ is a \textbf{weakly differentiable stratified space}.
\end{itemize}
We will just write $X$ if $\Sigma$ and $C^\infty(X)$ are understood from context.
\end{defs}

\begin{remark}
    There are roughly three different approaches to equipping stratified spaces with ``smooth structure''. There is the subcartesian approach we are taking in this thesis (and most notably can be found in \cite{sniatycki_differential_2013}), the smooth atlas approach found in \cite{pflaum_analytic_2001} or \cite{zimhony_commutative_2024}, and the reduced differentiable structure approach found in \cite{crainic_orbispaces_2018} or \cite{mol_stratification_2024}. These approaches are not exactly the same, but thanks to results like Theorem \ref{thm: charts of subcartesian spaces are locally diffeomorphic} many of the ideas from one approach can be ported over the other directly. This will be most useful when we discuss the Whitney conditions later on in Section \ref{sub: Whitney}.
\end{remark}

\begin{egs}
Let $M$ be a smooth manifold and let $C^\infty(M)$ denote the set of smooth functions. Then the triple $(M,\pi_0(M),C^\infty(M))$ is naturally a differentiable stratified space.
\end{egs}

\begin{egs}
The polytopes considered in Example \ref{eg: stratification of polytope} are clearly differentiable stratified spaces when equipped with the subspace subcartesian structure. 
\end{egs}

\begin{prop}
Let $(X_1,\Sigma_1)$ and $(X_2,\Sigma_2)$ be two weakly differentiable stratified spaces, and let $F:X_1\to X_2$ be a continuous map. If $F$ is both differentiable and a stratified morphism then for any strata $S_1\in\Sigma_1$ and $S_2\in\Sigma_2$ with $F(S_1)\subseteq S_2$, the restriction map $F|_{S_1}:S_1\to S_2$ is a smooth map between smooth manifolds.
\end{prop}

\begin{proof}
    Fix two strata $S_1\in\Sigma_1$ and $S_2\in \Sigma_2$ with $F(S_1)\subseteq S_2$. If $f\in C^\infty(X_2)$, then clearly 
    $$
    (F|_{S_1})^*(f|_{S_2})=(F^*f)|_{S_1}.
    $$
    Thus, by Proposition \ref{prop: differentiable maps are continuous}, we have $F^*C^\infty(S_2)\subseteq C^\infty(S_1)$. Thus, $F|_{S_1}$ is a smooth map.
\end{proof}

\subsection{Singular Foliations and Stratifications}\label{sub: foliations and strats}

A somewhat similar appearing structure to a stratification is that of a singular foliation. These are also partitions of a space into manifolds, but with no assumption of being locally finite.

\begin{defs}[(Stefan-Sussmann) Singular Foliation]\label{def: stefan-sussmann sf}
Let $M$ be a manifold and $\mathscr{F}$ a partition of $M$ into connected immersed submanifolds. For each $x\in M$ write $L_x$ for the unique element of $\mathscr{F}$ containing $x$. We call $\mathscr{F}$ a singular foliation of $M$ and the elements of $\mathscr{F}$ leaves if the following smoothness condition holds.
\begin{itemize}
    \item[(S)] If $x\in L$, and $v\in T_xL_x$, then in a neighbourhood $U\subseteq M$ of $x$ there exists a vector field $V\in\mfX(U)$ such that $V_x=v$ and $V_y\in T_yL_y$ for all $y\in U$.
\end{itemize}
Call the pair $(M,\mathscr{F})$ a foliated space.
\end{defs}

\begin{egs}\label{eg: regular foliation}
A first example would of course be the usual notion of a foliation, here called regular foliations. A regular foliation of a manifold $M$ is a partition $\mathscr{F}$ into immersed submanifolds which admits foliation charts. That is, around each point $x\in M$ there exists a chart
$$
\phi:U\to \Omega\subseteq \bR^k\times \bR^m
$$
for open $\Omega$, with the property that for all $y\in U$, we have
$$
\phi(U\cap L_y)=\Omega\cap (\bR^k\times \{v\})
$$
for some $v\in \bR^m$. By a result of Stefan \cite{stefan_accessible_1974}, a Stefan-Sussmann foliation $\mathscr{F}$ is regular if, and only if, the map
$$
M\to \bZ;\quad x\mapsto \dim L_x
$$
is constant.
\end{egs}

\begin{egs}\label{eg: first example of singular foliation}
    To see a truly singular example of a foliation, consider the partition $\mathscr{F}$ of $\bR^2$ given by
    $$
    L_x=\{y\in \bR^2  \ | \ \|y\|^2=\|x\|^2\},
    $$
    where $\|\cdot\|$ is the usual norm on $\bR^2$. This is a partition of $\bR^2$ into two kinds of leaves: one-dimensional concentric circles centred on the origin and a single zero-dimensional leaf consisting of the origin $\{0\}$. To see that $\mathscr{F}$ is a singular foliation, let $x_1,x_2$ be the standard coordinate functions on $\bR^2$. Now consider
    $$
    V_{(x_1,x_2)}=x_2\frac{\partial}{\partial x_1}-x_1\frac{\partial}{\partial x_2}.
    $$
    It's easy to see that $V$ is tangent to all the one-dimensional leaves and is non-vanishing. Furthermore, $V_{(0,0)}=0$. Thus, through $V$ we can implement the smoothness condition (S) in Definition \ref{def: stefan-sussmann sf}.
\end{egs}

\begin{egs}\label{eg: three foliations associated to proper group action}
    Let $G$ be a connected Lie group and $M$ a proper $G$-space. From the $G$-action, we will get three, possibly four distinct singular foliations. 
    \begin{itemize}
        \item[(1)] \textbf{Orbit Foliation}.

        \ 

        Due to Proposition \ref{prop: fundamental facts about proper G-spaces}, the set of all orbits
        $$
        \mathscr{F}_G=\{G\cdot x \ | \ x\in M\}
        $$
        is a singular foliation. Indeed, each orbit is an embedded submanifold and the map
        $$
        \mfg\times M\to T\mathcal{F}_G;\quad (\xi,m)\mapsto \xi_M(m)
        $$
        is a surjection, where $\xi_M\in \mfX(M)$ is the associated fundamental vector field. In particular, $T\mathscr{F}_G=\mfg_M$. Note that equipping $\bR^2$ with its canonical $S^1$ action given by rotations about the origin, returns the foliation in Example \ref{eg: first example of singular foliation}.

        \item[(2)] \textbf{Foliating by Manifolds of Symmetry}.

        \ 

        Recall that the manifolds of symmetry are submanifolds of $M$ of the form
        $$
        M_K=\{x\in M \ | \ G_x=K\}
        $$
        for a compact subgroup $K\leq G$. We saw from Proposition \ref{prop: tangent space to manifold of symmetry and orbit-type} that if $x\in M_K$, then $T_xM_K=(T_xM)^K$. Furthermore, recall that since $TM$ is a $G$-equivariant vector bundle, we may define a subset
        $$
        \widetilde{TM}=\bigcup_{x\in M} (T_xM)^{G_x}.
        $$
        Recall from Proposition \ref{prop: equivariant sections give all sections below} that the map
        $$
        M\times \mfX(M)^G\to \widetilde{TM};\quad (x,V)\mapsto V_x
        $$
        is a surjection. Thus, this implies the partition
        $$
        \mathscr{F}=\bigcup_{K\leq G}\pi_0(M_K)
        $$
        is a singular foliation of $M$.

        \item[(3)] \textbf{Foliating by Orbit-Type Pieces}.

        \ 

        Recall that the orbit-types of $M$ are subsets of the form
        $$
        M_{(K)}=\{x\in M \ | \ G_x\text{ is conjugate to }K\}
        $$
        for subgroups $K\leq G$. As we also saw from Corollary \ref{cor: manifold of symmetry and orbit-type are manifolds}, the connected components of the orbit-types are embedded submanifolds. Furthermore, if $S\subseteq M_{(K)}$ is an orbit-type piece, then 
        $$
        T_xS=(T_xM)^{G_x}+T_x(G\cdot x).
        $$
        Thus, letting $\mathscr{F}$ denote the partition of $M$ by connected components of the orbit-types, then the module $\mathcal{M}\subseteq \mfX(M)$ given by
        $$
        \mathcal{M}=\{V+\xi_M \ | \ V\in \mfX(M)^G\text{ and }\xi\in \mfg\}
        $$
        spans $T\mathscr{F}$. Hence, $\mathscr{F}$ is a singular foliation. 
    \end{itemize}
    There is potentially one further foliation by \textbf{reduced orbit-types}, but we will discuss this partition in more detail in Subsection \ref{sub: infnitesimal stratification}.
\end{egs}

\begin{egs}\label{egs: symplectic foliation}
    Let $(M,\pi)$ be a Poisson manifold. Then $\pi$ induces a singular foliation by symplectic submanifolds, called the symplectic foliation of $\pi$. Recall due to Theorem \ref{thm: existence of symplectic leaves} each point $x\in M$ admits a unique connected symplectic submanifold $S\substeq M$ such that
    $$
    T_xS=\{V_f(x) \ | \ f\in C^\infty(M)\}.
    $$
    In particular, letting $\mathscr{F}_\pi$ denote the partition into symplectic leaves, we have the map
    $$
    M\times C^\infty(M)\to T\mathscr{F}_\pi;\quad (x,f)\mapsto V_f(x)
    $$
    is a surjection and hence $\mathscr{F}_\pi$ is a singular foliation.
\end{egs}

\begin{egs}\label{eg: cantor foliation}
    As one last example, it is worth noting that singular foliations can be truly singular in general. Indeed, let $C\subset[0,1]$ be the Cantor set and consider the partition $\mathscr{F}$ of $\mathbb{R}$ consisting of
\begin{itemize}
	\item the points of $C$
	\item the connected components of $\mathbb{R}\setminus C$, a countable collection of disjoint open intervals.
\end{itemize}
This partition is indeed a foliation. Since $C$ is closed, we may find a smooth function $f:\mathbb{R}\to \mathbb{R}$ such that $f^{-1}(0)=C$. Define a vector field $X\in\mfX(\mathbb{R})$ by
$$
X_x:=f(x)\frac{d}{dx}.
$$
Clearly $X_x\in T_x\mathscr{F}$ for every $x\in \mathbb{R}$. Furthermore, if $v\in T_{x_0}\mathscr{F}$ for some $x_0$, then $v=c\displaystyle\frac{d}{dx}\bigg|_{x=x_0}$. If $x_0\in C$, then $v=X_{x_0}$. Else, $v=\displaystyle\frac{c}{f(x_0)}X_{x_0}$.

\ 

In general, given any manifold $M$ and any closed set $C\subseteq M$, we can construct a singular foliation $\mathscr{F}_C$ where the zero leaves are precisely given by the points of $C$. Hence, singular foliations can be arbitrarily singular. 
\end{egs}

One of the most celebrated results in the theory of regular foliations is Frobenius' Theorem (see for instance \cite[Theorem 19.12]{lee_introduction_2013}) which provides an equivalence between regular foliations on a manifold and subbundles of the tangent bundle $TM$ which are closed under Lie brackets. This way of viewing foliations still bears fruit in the setting of singular foliations, although we will in general be dealing with pseudobundles instead of subbundles (see Chapter \ref{ch: stratpseud} for a precise definition).

\begin{defs}[Tangent Bundle of SF]
Let $M$ be a manifold and $\mathscr{F}$ a partition of $M$ into immersed submanifolds. Define the tangent bundle of the partition by
$$
T\mathscr{F}:=\bigcup_{x\in M}T_xL_x\subseteq TM.
$$
Equip $T\mathscr{F}$ with the subspace topology. Write $\mfX(\mathscr{F})$ for smooth vector fields $V\in\mfX(M)$ with $V_x\in T_xL_x$ for all $x\in M$.
\end{defs}

For our own convenience, we will now collect a bunch of significant results about singular foliations into one statement.

\begin{theorem}[{\cite{stefan_accessible_1974,sussmann_orbits_1973,miyamoto_basic_2023}}]\label{thm: fundamental theorem of sf}
Let $\mathscr{F}$ be a partition of a smooth manifold $M$ into connected immersed submanifolds. The following are equivalent.
\begin{itemize}
    \item[(1)] $\mathscr{F}$ is a singular foliation of $M$.
    \item[(2)] The map
    $$
    \mfX(\mathscr{F})\times M\to T\mathscr{F};\quad (V,x)\mapsto V_x
    $$
    is surjective.
    \item[(3)] $x,y\in M$ lie on the same leaf $\iff$ there exists $V_1,\dots,V_n\in \mfX(\mathscr{F})$ and $t_1,\dots,t_n\in \bR$ such that
    $$
    \phi_{t_1}^{V_1}\circ\dots\circ \phi_{t_n}^{V_n}(x)=y,
    $$
    where $\phi_{t_i}^{V_i}$ is the flow of $V_i$ by time $t_i$.
\end{itemize}
\end{theorem}

\begin{cor}\label{cor: sfs satisfy the axiom of the frontier}
Singular foliations satisfy the axiom of the frontier. That is, if $(M,\mathscr{F})$ is a foliated space and $L_x,L_y\in \mathscr{F}$ are two leaves with $L_x\cap \overline{L_y}\neq\emptyset$, then $L_x\subseteq \overline{L_y}$.
\end{cor}

\begin{proof}
Without loss of generality, we may assume $x\in \overline{L_y}$. Let $z\in L_x$ be some other point. Then, by (3) in Theorem \ref{thm: fundamental theorem of sf}, we can find vector fields $V_1,\dots,V_n\in \mfX(\mathscr{F})$ and times $t_1,\dots,t_n$ so that
$$
\phi_{t_1}^{V_1}\circ\cdots \circ \phi_{t_n}^{V_n}(x)=z.
$$
Let $U\subseteq M$ be an open neighbourhood so that
$$
\psi:=\phi_{t_1}^{V_1}\circ\cdots \circ \phi_{t_n}^{V_n}
$$
is defined. Note that $\psi$ is a diffeomorphism which preserves the leaves. Now let $\{y_k\}\subset L_y$ be a sequence converging to $x$. After possibly shrinking $U$ and passing to a subsequence, we may assume $\{y_k\}\subseteq U$. Hence, by continuity, we have
$$
\lim_{k\to \infty}\psi(y_k)=\psi(x)=z
$$
However, $\psi$ preserves the foliation and hence $\psi(y_k)\in L_y$ for all $k$. Thus, $z\in \overline{L_y}$. 
\end{proof}

From this we achieve one technique for proving a partition of a manifold is a stratification, namely showing that the candidate partition is a singular foliation!

\begin{cor}\label{cor: locally finite foliations are stratifications}
If $(M,\mathscr{F})$ is a singular foliation which is locally finite and such that the leaves are embedded submanifolds of $M$, then $\mathscr{F}$ is a differentiable stratification of $M$.
\end{cor}

\begin{remark}
\begin{itemize}
    \item[(1)] Note that given a regular foliation $\mathscr{F}$ of a manifold $M$, if the dimensions of the leaves is not $\dim(M)$, then $\mathscr{F}$ has no chance of being a stratification. This is purely for the reason that a non-trivial regular foliation will never be locally finite, as is self-evident from the definition of foliation charts in Definition \ref{eg: regular foliation}.
    \item[(2)] In a similar vein of giving words of caution, a stratification $\Sigma$ of a manifold $M$ need not define a singular foliation. Indeed, consider the following example due to Whitney \cite{whitney_local_2015}. Let $X\subseteq\bR^3$ be the variety
    $$
    X=\{(x,y,z)\in \bR^3 \ | \ xy(x-y)(y-xz)=0\}
    $$
    This variety consists of 4 ``sheets'' which all meet in the $z$-axis.
    \begin{figure}[h]
    \centering

\begin{tikzpicture}[x={(0.9cm,-0.1cm)},y={(0.5cm,0.2cm)},z={(0cm,1cm)}] 
\draw[dashed] (3,3,3) -- (3,-3,3) -- (-3,-3,3) -- (-3,3,3) -- cycle;
\draw[dashed] (3,3,-3) -- (3,-3,-3) -- (-3,-3,-3) -- (-3,3,-3) -- cycle;
\draw[dashed] (3,3,3) -- (3,3,-3);
\draw[dashed] (-3,3,3) -- (-3,3,-3);
\draw[dashed] (-3,-3,3) -- (-3,-3,-3);

  \shade[top color=myred,bottom color=black,middle color=black!20!myred,shading angle=-40,smooth,draw=none,even odd rule] (3,-3,-1) -- (3,3,1) 
  -- (2.8,3,3/2.8) -- (2.6,3,3/2.6) -- (2.4,3,3/2.4) -- (2.2,3,3/2.2) -- (2,3,3/2) -- (1.8,3,3/1.8) -- (1.6,3,3/1.6) -- (1.4,3,3/1.4) --(1.2,3,3/1.2) --
   (1,3,3) -- (-1,-3,3)
   -- (-1.2,-3,3/1.2) -- (-1.4,-3,3/1.4) -- (-1.6,-3,3/1.6) -- (-1.8,-3,3/1.8) -- (-2,-3,3/2) --  (-2.2,-3,3/2.2) -- (-2.4,-3,3/2.4) -- (-2.4,-3,3/2.4) -- (-2.6,-3,3/2.6) -- (-2.8,-3,3/2.8) --
   (-3,-3,1) -- (-3,3,-1)
   -- (-2.8,3,-3/2.8) -- (-2.6,3,-3/2.6) -- (-2.4,3,-3/2.4) -- (-2.2,3,-3/2.2) -- (-2,3,-3/2) -- (-1.8,3,-3/1.8) -- (-1.6,3,-3/1.6) -- (-1.4,3,-3/1.4) -- (-1.2,3,-3/1.2) -- 
    (-1,3,-3) -- (1,-3,-3)
    -- (1.2, -3, -3/1.2) --(1.4, -3, -3/1.4) --(1.6, -3, -3/1.6) --(1.8, -3, -3/1.8) -- (2, -3, -3/2) -- (2.2, -3, -3/2.2) -- (2.4, -3, -3/2.4) -- (2.6, -3, -3/2.6) -- (2.8, -3, -3/2.8) -- 
    cycle;
    \shade[top color=red!40!myred,bottom color=black,middle color=black!5!myred,shading angle=-65,smooth,draw=none,even odd rule]
    (-3,-3,1) ..controls (-0.9,-.4) and (1,0.3)..  (1,-3,-3)
        -- (1.2, -3, -3/1.2) --(1.4, -3, -3/1.4) --(1.6, -3, -3/1.6) --(1.8, -3, -3/1.8) -- (2, -3, -3/2) -- (2.2, -3, -3/2.2) -- (2.4, -3, -3/2.4) -- (2.6, -3, -3/2.6) -- (2.8, -3, -3/2.8) -- (3,-3,-1) -- (3,3,1) --  (2.8,3,3/2.8) -- (2.6,3,3/2.6) -- (2.4,3,3/2.4) -- (2.2,3,3/2.2) -- (2,3,3/2) -- (1.8,3,3/1.8) -- (1.6,3,3/1.6) -- (1.4,3,3/1.4) --(1.2,3,3/1.2) --
   (1,3,3) -- (-1,-3,3)
   -- (-1.2,-3,3/1.2) -- (-1.4,-3,3/1.4) -- (-1.6,-3,3/1.6) -- (-1.8,-3,3/1.8) -- (-2,-3,3/2) --  (-2.2,-3,3/2.2) -- (-2.4,-3,3/2.4) -- (-2.4,-3,3/2.4) -- (-2.6,-3,3/2.6) -- (-2.8,-3,3/2.8) --
   (-3,-3,1) -- (-3,3,-1)-- cycle;
   \shade[inner color=white, outer color=myred, fill opacity=0.3] (-0.5,0,0) -- (1.75,0,0) -- (1.25,0.5,2) -- (-0.5,0.5,2);
       \filldraw[fill opacity=0.1,draw=none] (0,3,3) -- (0,3,-3) -- (0,-3,-3) -- (0,-3,3) --cycle;
    \filldraw[fill opacity=0.1,draw=none] (3,0,3) -- (3,0,-3) -- (-3,0,-3) -- (-3,0,3) --cycle;
    \filldraw[fill opacity=0.1,draw=none] (3,3,3) -- (3,3,-3) -- (-3,-3,-3) -- (-3,-3,3) --cycle;
\draw[->] (-3,0,0) -- (3,0,0) node[right] {$x$};
\draw[->] (0,-3,0) -- (0,3,0) node[right] {$y$};
\draw[->] (0,0,-1) -- (0,0,3.1) node[above] {$z$};
\draw[ultra thick] (0,0,-1.3) -- (0,0,3);
\draw[ultra thick,opacity=0.33,dotted] (0,0,-1.3) -- (0,0,-2.6);
\draw[ultra thick] (0,0,-3) -- (0,0,-2.6);

\draw[dashed] (3,3,3) -- (3,-3,3) -- (-3,-3,3) -- (-3,3,3) -- cycle;
\draw[dashed] (3,-3,3) -- (3,-3,-3);

\end{tikzpicture}
\caption{The variety $xy(x-y)(y-xz)=0$}
\label{fig: not foliation stratification}
    \end{figure}
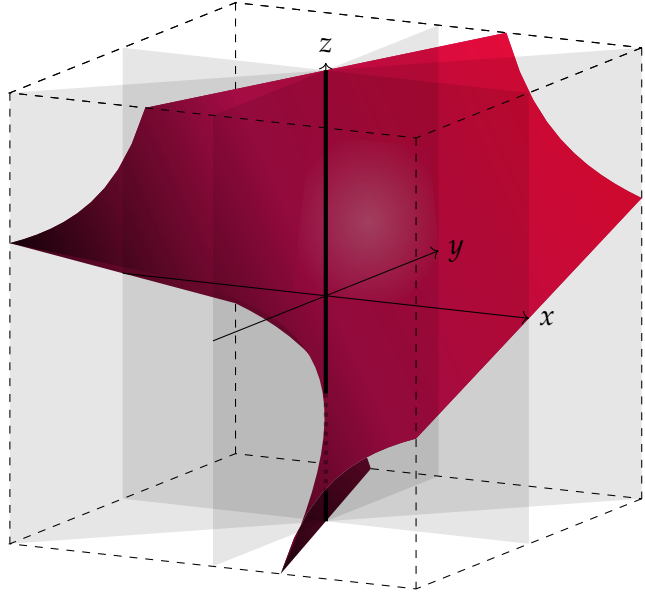
    
    Using $X$, we obtain a stratification of $\bR^3$ as follows. Let 
    $$
    Z=\{(0,0,z)\in \bR^3  \ | \ z\in \bR\}
    $$
    denote the $z$-axis, and let
    $$
    \Sigma = \pi_0(\bR^3\setminus X)\cup \pi_0(X\setminus Z)\cup \{Z\}.
    $$
    Then $\Sigma$ is a stratification of $\bR^3$ with one $1$-dimensional stratum given by the $z$-axis, the $2$-dimensional strata given by the connected components of $X\setminus Z$, and the $3$-dimensional strata given by the connected components of the complement of $X$. However, $\Sigma$ is not a foliation. Indeed, an argument from Whitney involving cross-ratios demonstrates that we cannot find a non-trivial vector field whose flow remains tangent to the sheets and the $z$-axis, violating (3) from Theorem \ref{thm: fundamental theorem of sf}. 
\end{itemize}
\end{remark}

\begin{defs}[Foliate Map]
Let $(M_1,\mathscr{F}_1)$ and $(M_2,\mathscr{F}_2)$ be two foliated spaces and $F:M_1\to M_2$ a smooth map. Then we call $F$ foliate if for each $x\in M$ we have
$$
F(L^1_x)\subseteq L^2_{F(x)},
$$
where $L^1_x\in\mathscr{F}_1$ and $L^2_{F(x)}\in \mathscr{F}_2$. Call $F$ a foliate diffeomorphism if $F$ is a diffeomorphism and its inverse $F^{-1}$ is also foliate.
\end{defs}

\begin{egs}
    As an elementary example, consider $\bR^2$ with the orbit foliation from the canonical $S^1$ action by rotations about the origin. Fixing $t\in \bR\setminus\{0\}$, the scalar multiplication map
    $$
    \mu_t:\bR^2\to \bR^2;\quad (x,y)\mapsto (tx,ty)
    $$
    is foliate as it fixes the origin and maps circles to circles. In particular, foliate diffeomorphisms need not fix the leaves.
\end{egs}

\begin{cor}\label{cor: foliate iff morphism of pseudo-bundles}
A smooth map $F:M_1\to M_2$ between two foliated spaces $(M_1,\mathscr{F}_1)$ and $(M_2,\mathscr{F}_2)$ is foliate $\iff$
$$
TF(T\mathscr{F}_1)\subseteq T\mathscr{F}_2.
$$
\end{cor}

\subsection{Singular Riemannian Foliations}\label{sub: SRFs}
As we saw in Example \ref{eg: cantor foliation}, singular foliations can be arbitrarily singular in general. So it is natural to impose regularity and compatibility conditions in order to get more workable objects. One profoundly rich (and rigid) class of well-behaved singular foliations are Singular Riemannian Foliations (SRF) due to Molino \cite{molino_riemannian_1988}. We will now walk through the basic definitions and structural results of singular foliations before we show that underlying every SRF is a very well-behaved stratification called the ``dimension-type'' stratification.

\begin{defs}[Singular Riemannian Foliation]\label{def: srf}
Let $M$ be a manifold, $g$ a Riemannian metric on $M$, and $\mathscr{F}$ a singular foliation on $M$. Say the triple $(M,g,\mathscr{F})$ is a Singular Riemannian Foliation (SRF) if the following orthogonality condition holds.
\begin{itemize}
    \item[(O)] Suppose $\gamma:(a,b)\to M$ is a geodesic such that $\dot{\gamma}(t_0)\perp T_{\gamma(t_0)}\mathscr{F}$ for some $t_0\in (a,b)$. Then $\dot{\gamma}(t)\perp T_{\gamma(t)}\mathscr{F}$ for all $t\in (a,b)$.
\end{itemize}
\end{defs}

\begin{egs}
Any proper group action can be made into a SRF \cite{palais_existence_1961}.
\end{egs}

Now SRFs come equipped with a canonical stratification by the dimension of the leaves. To set this up, let $(M,g,\mathscr{F})$ be an SRF and $0\leq k\leq n$ an integer. Define
$$
M_{(k)}:=\{x\in M \ | \ \dim L_x=k\}.
$$

\begin{defs}[Dimension Partition]\label{def: dimension-type}
Let $(M,g,\mathscr{F})$ be an SRF. Define the dimension-type partition associated to $\mathscr{F}$ by
$$
\mathcal{S}_{dim}(M,\mathscr{F}):=\bigcup_{k=0}^n \pi_0(M_{(k)}).
$$
\end{defs}

Our goal now is to show that $\mathcal{S}_{dim}(M,\mathscr{F})$ is a stratification for any singular Riemannian foliation $\mathscr{F}$. First, let us discuss one of the foundational structural results for SRFs. For these purposes, we fix $x\in M$ and $L\in \mathscr{F}$ the leaf containing $x$. Since leaves are immersed, we can find an open neighbourhood $P\subseteq L$ containing $x$ which is embedded in $M$ as a submanifold. We call such a subset a \textbf{plaque}. Using geodesics, we can define the exponential tubular neighbourhood embedding
$$
\exp_x:\nu_x(M,P)\to M
$$
which maps the zero section to $P$. As $\nu_x(M,P)$ is a vector bundle over $P$, it has equipped a scalar multiplication $\mu_t$. Call the induced monoid action on the image \textbf{homotheties}. That is, a homothety $h_t$ for $t\in \bR$ is a diffeomorphism on the image of $\exp_x$ making the diagram commute
$$
\begin{tikzcd}
    \nu_x(M,P)\arrow[r,"\exp_x"]\arrow[d,"\mu_t"] & \exp_x(\nu_x(M,P))\arrow[d,"h_t"]\\
    \nu_x(M,P)\arrow[r,"\exp_x"] & \exp_x(\nu_x(M,P))
\end{tikzcd}
$$

\begin{lem}[Homothety Lemma {\cite{molino_riemannian_1988}}]\label{lem: homothety lemma}
    Let $\exp_x:\nu_x(M,P)\to M$ be an exponential tubular neighbourhood embedding of a plaque as above. Then the homotheties $h_t$ for $t\in [0,1]$ are foliate diffeomorphisms.
\end{lem}

The upshot of the Homothety Lemma is that locally an SRF can be linearized. 

\begin{theorem}[{\cite{lytchak_curvature_2010}}]\label{thm: local normal form SRF}
Let $(M,g,\mathscr{F})$ be a singular Riemannian foliation, $S\in \mathcal{S}_{dim}(M,\mathscr{F})$ a dimension-type piece, and $x\in S$. 
\begin{itemize}
	\item[(1)] There is a canonical singular Riemannian foliation $\nu_x\mathscr{F}$ on the normal space $\nu_x(M,S)$ with respect to the induced linear metric on $\nu_x(M,S)$ with the property that $\{0\}$ is an isolated $0$-dimensional leaf and such that scalar multiplication by non-zero scalars defines a foliate diffeomorphism.
	\item[(2)] Equip $T_xS$ with the regular foliation $\mathscr{F}_x^{reg}$ given by translates of the tangent space to the leaf $T_xL_x$ through $x$. Then there exists open neighbourhoods $U\subseteq M$ of $x$, $V\subseteq T_xS\times \nu_x(M,S)$ of $(0,0)$, and a foliate diffeomorphism
	$$
	\phi:U\subseteq M\to V\subseteq (T_xS\times\nu_x(M,S),\mathscr{F}_x^{reg}\times \nu_x\mathscr{F})
	$$
	with respect to the restricted foliations.
\end{itemize}
\end{theorem}

\begin{defs}[Infinitesimal SRF]
We call singular foliations like (1) infinitesimal singular Riemannian foliations.
\end{defs}

Infinitesimal singular Riemannian foliations are quite well-behaved as far as singular foliations go. They play very nicely with the underlying linear structure of the Euclidean space they lie on and can be understood by restricting to the sphere.

\begin{lem}[{\cite{radeschi_low_2012}}]
Let $(V,g_{Euc},\mathscr{F})$ be an infinitesimal singular Riemannian foliation and let $S(V)$ denote the unit sphere in $V$. Then the following hold.
\begin{itemize}
    \item[(1)] For every $x\in V\setminus\{0\}$ we have $L_x\cap S(V)\neq \emptyset$.
    \item[(2)] The partition $\widehat{\mathscr{F}}$ of $S(V)$ obtained by passing to the connected components of the partition
    $$
    \{L_x\cap S(V)\}_{x\in V}
    $$
    defines a singular Riemannian foliation with respect to the induced metric $\hat{g}$ on $S(V)$.
    \item[(3)] For each $S\in \mathcal{S}_{dim}(V,\mathscr{F})$, the connected components of $S\cap S(V)$ lie in $\mathcal{S}_{dim}(S(V),\widehat{\mathscr{F}})$.
\end{itemize} 
\end{lem}

With this result in hand, we can now prove that the dimension-type partition truly is a locally finite singular foliation with embedded leaves, hence a stratification.

\begin{theorem}[{\cite{miyamoto_srfs_2025}}]\label{theorem: dimension-type parititon is a locally finite sf}
The dimension-type partition $\mathcal{S}_{dim}(M,\mathscr{F})$ is a locally finite singular foliation with embedded leaves.
\end{theorem}

\begin{proof}
Let's first prove $\mathcal{S}_{dim}(M,\mathscr{F})$  is locally finite. This is an inductive argument based on the local normal form. First, it suffices to show the dimension-type partition of an infinitesimal foliation is finite. Indeed, if this is the case then around any point $x\in M$ contained in dimension-type piece $S\in\mathcal{S}_{dim}(M,\mathscr{F})$ there is a local foliate diffeomorphism between a neighbourhood $U\subseteq M$ of $x$ and $V\subseteq T_xM$
\begin{equation}\label{eq: srf local normal form}
    \phi:U\subseteq M\to V\subseteq T_xS\oplus\nu_x(M,S)
\end{equation}
Then since $\mathcal{S}_{\nu_x\mathscr{F}}(\nu_x(M,S))$ is finite, it follows that only finitely many elements of $\mathcal{S}_{dim}(M,\mathscr{F})$ intersect $U$ non-trivially.

\

Now let us fix an infinitesimal foliation $(V,\mathscr{F})$. We prove by induction on $\dim V$ that $\mathcal{S}_{\mathscr{F}}(V)$ is finite. First, if $\dim V=0$ then $\mathcal{S}_{\mathscr{F}}(\{0\})$ is a singleton, hence finite. Let us now assume the claim is true for infinitesimal foliations on $\bR^k$ for all $k<n$. Choosing a point $x$ in the sphere $S(V)$ of $V$ and applying the local normal form in Theorem \ref{thm: local normal form SRF}, we get $(V,\mathscr{F})$ is locally diffeomorphic to a regular foliation times an infinitesimal singular Riemannian foliation of strictly less dimension than $V$, hence is finite there. Since $S(V)$ is compact, we can cover it by finitely many neighbourhoods where the intersection is finite. Using the homothety property (see Lemma \ref{lem: homothety lemma}), we deduce that $\mathcal{S}_{\mathscr{F}}(V)$ is finite.

\

Now to show that $\mathcal{S}_{dim}(M,\mathscr{F})$ is a singular foliation. Due to the local normal form in Equation (\ref{eq: srf local normal form}), it suffices to assume 
\begin{equation*}
(M,\mathscr{F})=(V_1\times V_2,\mathscr{F}_1\times\mathscr{F}_2)
\end{equation*}
where $V_1,V_2$ are vector spaces, $\mathscr{F}_1$ is a regular foliation on $V_1$, and $\mathscr{F}_2$ a singular foliation on $V_2$ with $\{0\}$ being the only leaf. 

\ 

\noindent \textbf{Claim 1:} If $W\in \mfX(V_1\times V_2)$ is foliate, i.e. the flow $\phi_t^W$ is a foliate diffeomorphism for all $t$, then $W\in \mfX(\mathscr{F})$.

\ 

Fix a foliate vector field $W$. It suffices to show
\begin{equation*}
W|_{V_1\times\{0\}}\in \mfX(V_1\times\{0\}).
\end{equation*}
Indeed, fix a point $(x_0,0)\in V_1\times \{0\}$ and let $(x_t,y_t)$ be an integral curve of $W$ through $(x_0,0)$. Note that since $W$ is foliate, we have
\begin{equation*}
\dim L_{(x_0,0)}=\dim L_{(x_t,y_t)}
\end{equation*}
for all $t$. In particular, since $\{0\}$ is the only $0$-dimensional leaf of $\mathscr{F}_2$ and $\mathscr{F}_1$ is regular, it follows that $y_t=0$ for all $t$. In particular, 
\begin{equation*}
W_{(x_0,0)}=(\dot{x}_0,\dot{y}_0)=(\dot{x}_0,0)\in T_{(x_0,0)}(V_1\times \{0\}).
\end{equation*}

\ 

\noindent\textbf{Claim 2}: For all $v\in T_{(0,0)}V_1\times \{0\}$, we can find a foliate vector field $W\in\mfX(\mathscr{F})$ so that $W_{(0,0)}=(v,0)$.

\

Fix $v\in V_1$. Observe that the directional derivative along $v$, denoted $\frac{\partial}{\partial v}$, is foliate. Indeed, if $L'_{v_2}$ denotes the leaf of $v_2\in V_2$, then for every $(v_1,v_2)\in V_1\times V_2$,
\begin{equation*}
L_{(v_1,v_2)}=(L+v_1)\times L'_{v_2}.
\end{equation*}
Then, for any $t$, if $\phi_t$ denotes the flow of $\partial/\partial v$, we have
\begin{equation*}
\phi_t(L_{(v_1,v_2)})=(L+v_1+tv)\times L'_{v_2}=L_{\phi_t(v_1,v_2)}.
\end{equation*}
Clearly $\displaystyle\frac{\partial}{\partial v}\bigg|_{(0,0)}=(v,0)$. 

\ 

Using Claim 1 and Claim 2 together show that condition (S) in Definition \ref{def: stefan-sussmann sf} is satisfied locally by foliate vector fields. Using partitions of unity, we deduce that $\mathcal{S}_{dim}(M)$ is a singular foliation.
\end{proof}

\begin{cor}
The dimension-type partition $\mathcal{S}_{dim}(M,\mathscr{F})$ is a stratification by totally geodesic submanifolds.
\end{cor}

\begin{proof}
    Since $\mathcal{S}_{dim}(M,\mathscr{F})$ is locally finite and the leaves are embedded submanifolds, the result follow immediately from Corollary \ref{cor: locally finite foliations are stratifications}.
\end{proof}

\begin{remark}
    \begin{itemize}
        \item[(1)] Henceforth, we shall be calling the partition $\mathcal{S}_{dim}(M)$ of an SRF the \textbf{dimension-type stratification}.
        \item[(2)] Given any singular foliation $\mathscr{F}$ of a manifold $M$, we can define a dimension-type partition $\mathcal{S}_{dim}(M,\mathscr{F})$ as in Definition \ref{def: dimension-type}. However, in general these will not be stratifications. Indeed, consider the singular foliation $\mathscr{F}$ on $\bR$ from Example \ref{eg: cantor foliation} where the $0$-dimensional leaves are given by the points in the cantor set $C$ and the $1$-dimensional leaves are given by the connected components of $\bR\setminus C$. I claim $\mathcal{S}_{dim}(M,\mathscr{F})$ is not locally finite. Indeed, observe that since $C$ is totally disconnected, we have
        $$
        \mathcal{S}_{dim}(M,\mathscr{F})=\mathscr{F}
        $$
        In which case, we see that if $x\in C$ and $U$ is any open subset of $\bR$ containing $x$, then $U\cap C$ is infinite. Thus, $\mathscr{F}$ is not locally finite which implies the dimension-type partition is not locally finite.

        \ 

        Although the dimension-type partition may not be a stratification, it is still a singular foliation! Perhaps a more general result could be proven that the dimension-type partition is always a singular foliation which could make use of foliate vector fields as in the proof of Theorem \ref{theorem: dimension-type parititon is a locally finite sf}.

        \item[(3)] As a final remark on the dimension-type stratification of a singular Riemannian foliation, it is also worth noting that while the dimension-type stratification $\mathcal{S}_{dim}(M,\mathscr{F})$ is a singular foliation, it will almost never be singular Riemannian. Indeed, consider the standard infinitesimal SRF $\mathscr{F}$ on $\bR^2$ with leaves given by the origin and concentric circles about the origin. The dimension-type stratification is quite easy to compute in this case, given by
        $$
        \mathcal{S}_{dim}(M,\mathscr{F})=\{\{(0,0)\},\bR^2\setminus\{(0,0)\}\}.
        $$
        Consider now the geodesic 
        $$
        \gamma(t)=(t,0),\quad t\in \bR
        $$
        Trivially, $\gamma$ is orthogonal to the zero-dimensional leaf $\{(0,0)\}$, but it is not orthogonal to the dense leaf $\bR^2\setminus\{(0,0)\}$. Thus, $\mathcal{S}_{dim}(M,\mathscr{F})$ fails condition (O) from Definition \ref{def: srf}.
    \end{itemize}
\end{remark}

\section{Stratifications Arising From Proper Group Actions}\label{sec: strats and group actions}

Let us now discuss an interesting class of stratifications arising from proper group actions. Namely, the orbit-type, infinitesimal-type, and reduced orbit-type. To set up our analysis, we will need to first delve into the foliation by manifolds of symmetry.

\ 

For all that follows, we will be letting $G$ be a connected Lie group, $\mfg=\text{Lie}(G)$ its Lie algebra, and $M$ a proper $G$-space.

\subsection{Orbit-Type Stratifications From Proper Group Actions}\label{sub: orbit-type strat}

Let $G$ be a connected Lie group and $M$ a proper $G$-space. Although the quotient space $M/G$ need not be a manifold, it does have a canonical decomposition into manifolds called the \textbf{canonical stratification}. The main ingredient for this decomposition is the orbit-type partition discussed in Example \ref{eg: three foliations associated to proper group action}. Let us now give a formal definition to set up some notation.

\begin{defs}[Orbit-Type and Canonical Partitions]\label{def: orbit-type partition}
Define the orbit-type partition of $M$ by
$$
\mathcal{S}_G(M):=\bigcup_{H\leq G}\pi_0(M_{(H)}).
$$
We then define the canonical partition of $M/G$ by
$$
\mathcal{S}_G(M/G):=\{S/G \ | \ S\in\mathcal{S}_G(M)\}.
$$
\end{defs}

We will now show that both $\mathcal{S}_G(M)$ and $\mathcal{S}_G(M/G)$ are differentiable stratifications of their respective subcartesian spaces.

\begin{lem}\label{lem: pieces of orbit-type and canonical strats are manifolds}
The pieces in the orbit-type $\mathcal{S}_G(M)$ and canonical stratifications $\mathcal{S}_G(M/G)$ are topological manifolds in the subspace topology. Furthermore, the induced subcartesian structures on the pieces are smooth manifold structures.
\end{lem}

\begin{proof}
As we've seen in Corollary \ref{cor: manifold of symmetry and orbit-type are manifolds}, the orbit-type pieces are embedded submanifolds. Furthermore, given a point $x\in M$ we have a maximal slice neighbourhood $\phi:U\to G\times_{G_x}\nu_x(M,G)$ in which the orbit-type piece $S\in \mathcal{S}_G(M)$ containing $x$ will have the form
$$
\phi(U\cap S)=G\times_{G_x}\nu_x(M,G)^{G_x}\cong (G/G_x)\times\nu_x(M,G).
$$
Thus, in the quotient $M/G$, the corresponding canonical piece $S/G$ will locally have the form
$$
(U\cap S)/G\cong ((G/G_x)\times \nu_x(M,G)^{G_x})/G\cong \nu_x(M,G)
$$
and so the canonical pieces of $M/G$ are also topological manifolds. Note that due to Corollary \ref{cor: fixed point set of representation is embedded nicely by Hilbert maps}, if we choose a Hilbert map $p:\nu_x(M,G)\to \bR^N$, then the restriction defines an embedding $p:\nu_x(M,G)^{G_x}\into \bR^N$ of smooth manifolds. In particular, we see that if $S\in\mathcal{S}_G(M/G)$ is a canonical piece then $S$ is a topological manifold in the subspace topology and the induced subcartesian structure $C^\infty(S)=\bra C^\infty(M/G)|_S\ket$ is a smooth manifold structure on $S$.
\end{proof}

To continue on our proof that both $\mathcal{S}_G(M)$ and $\mathcal{S}_G(M/G)$ are differentiable stratifications, we now need to show that they satisfy the following conditions.
\begin{itemize}
    \item Both are locally finite
    \item Both satisfy the axiom of the frontier.
\end{itemize}
Since $\pi:M\to M/G$ is a quotient map from a group action, it will suffice to simply show the orbit-type stratification $\mathcal{S}_G(M)$ satisfies these axioms.

\begin{lem}\label{lem: orbit-type locally fintite}
The orbit-type stratification $\mathcal{S}_G(M)$ is locally finite.
\end{lem}

\begin{proof}
Passing to slice coordinates as in the Slice Theorem (Theorem \ref{thm: slice theorem}), it suffices to show that if $K$ is a compact group and $V$ is a finite dimensional $K$-representation, then
$$
\mathcal{S}_K(V)=\bigcup_{H\leq K}\pi_0(V_{(H)})
$$
is finite. By Lemma 4.3.6 in \cite{pflaum_analytic_2001}, there is a finite list $K_1,\dots,K_N\leq K$ of subgroups so that
$$
V=\bigsqcup_{i=1}^N V_{(K_i)}.
$$
Hence, it suffices to show $\pi_0(V_{(H)})$ is finite for any compact subgroup $H\leq K$. This in turn was proven in Theorem 1.5 from \cite{schwarz_lifting_1980}.
\end{proof}

Now that local finiteness is established, we need to show the axiom of the frontier holds. For this, we will use our discsussion of the orbit-type decomposition from Example \ref{eg: three foliations associated to proper group action}. Namely, we gave an alternate proof to the following result of of Jotz-Ratiu-Sniatycki.

\begin{theorem}[{\cite{jotz_singular_2011}}]\label{thm: orbit-type is sf}
$\mathcal{S}_G(M)$ is a singular foliation.
\end{theorem}

Now we are set up to show that the two partitions from Definition \ref{def: orbit-type partition} are indeed stratifications.

\begin{theorem}
Both $(M,\mathcal{S}_G(M))$ and $(M/G,\mathcal{S}_G(M/G))$ are differentiable stratified spaces and the quotient map $\pi:M\to M/G$ is a differentiable stratified morphism.
\end{theorem}

\begin{proof}
Applying Lemmas \ref{lem: pieces of orbit-type and canonical strats are manifolds} and \ref{lem: orbit-type locally fintite}, together with Theorem \ref{thm: orbit-type is sf}, we obtain that the orbit-type partition $\mathcal{S}_G(M)$ is a locally finite singular foliation with embedded leaves, hence by Corollary \ref{cor: locally finite foliations are stratifications} is a differentiable stratification with respect to the smooth structure $C^\infty(M)$. Now to discuss the canonical partition $\mathcal{S}_G(M/G)$ of the quotient space $M/G$.

\ 

Using Lemma \ref{lem: pieces of orbit-type and canonical strats are manifolds} once again, we know that the pieces of $\mathcal{S}_G(M/G)$ are topological manifolds. Since $\pi:M\to M/G$ is an open quotient map, we easily deduce that $\mathcal{S}_G(M/G)$ inherits local finiteness and the frontier condition from $\mathcal{S}_G(M)$. Hence, a final application of Lemma \ref{lem: pieces of orbit-type and canonical strats are manifolds} demonstrates that $(M/G,\mathcal{S}_G(M/G),C^\infty(M/G))$ is a differentiable stratified space, where
$$
C^\infty(M/G)=\{f:M/G\to \bR \ | \ \pi^*f\in C^\infty(M)^G\}.
$$
\end{proof}

Before we end this discussion on orbit-type stratifications, I will record a final local normal form for proper $G$-spaces that will come in handy later when we discuss regularity conditions. 

\begin{lem}
Let $G$ be a connected Lie group, $K\leq G$ a compact subgroup, and $V$ a finite-dimensional $K$-representation. Choose a $K$-invariant metric on $\mfg$ and let $\mfk^\perp\subseteq \mfg$ be the orthogonal complement to $\mfk=\text{Lie}(K)$. Then the map
$$
\phi:\mfk^\perp\times V\to G\times_KV;\quad (\xi,v)\mapsto [\exp(\xi),v]
$$
has the following properties.
\begin{itemize}
    \item[(1)] $\phi$ is a $K$-equivariant smooth map which defines a diffeomorphism in a neighbourhood of $(0,0)$.
    \item[(2)] If $U\subseteq \mfk^\perp\times V$ is a neighbourhood of $(0,0)$ so that $\phi$ is a diffeomorphism onto open $W=\phi(U)$, then for any closed subgroup $H\leq K$ we have
    $$
    \phi(U\cap (\mfk^\perp\oplus V)_{(H)})=W\cap (G\times_K V)_{(H)}.
    $$
\end{itemize}
\end{lem}

\begin{proof}
\begin{itemize}
    \item[(1)] Let $\xi\in\mfk^\perp$ and $v\in V$. Then for any $k\in K$ we see that
    $$
    k\cdot \phi(\xi,v)=\phi(\text{Ad}_k(\xi),k\cdot v)=[k\exp(\xi)k^{-1},k\cdot v].
    $$
    Next, we see that if we identify $T_{[e,0]}G\times_K V=\mfg/\mfk\oplus V$, then $T_{(0,0)}\phi$ can be identified with the map
    $$
    T_{(0,0)}\phi:\mfk^\perp\oplus V\to (\mfg/\mfk)\oplus V;\quad (\xi,v)\mapsto (\xi+\mfk,v)
    $$
    which is clearly an isomorphism. Hence, $\phi$ defines a diffeomorphism in a neighbourhood of $(0,0)$.

    \item[(2)] Follows from equivariance.
\end{itemize}
\end{proof}

\subsection{Reduced Orbit-Type and Infinitesimal Stratifications}\label{sub: infnitesimal stratification}

We get another stratification from a proper group action, namely the one arising from the induced Lie algebra action. We call this the infinitesimal stratification. Let $G$ be a connected Lie group, $\mfg$ its Lie algebra, and $M$ a proper $G$-space. For any Lie subalgebra $\mfh\subseteq \mfg$, define
$$
M_{(\mfh)}:=\{x\in M \ | \ \mfg_x\text{ is conjugate to }\mfh\}.
$$
Using this, we define the infinitesimal-type partition of $M$ in an analogous fashion as we did for the orbit-type stratification.

\begin{defs}[Infinitesimal Partition]
Define the infinitesimal-type partiton of $M$ by
$$
\mathcal{S}_\mfg(M):=\bigcup_{x\in M}\pi_0(M_{(\mfg_x)}).
$$
\end{defs}

\begin{remark}
    Note that $\mathcal{S}_\mfg(M)$ only depends on the induced Lie algebra action
    $$
    \mfg\to \mfX(M);\quad \xi\mapsto \xi_M
    $$
    Since two distinct Lie groups $G_1,G_2$ could induce the same Lie algebra action, we should expect in general that the infinitesimal-type partition should be coarser than the orbit-type partition.
\end{remark}

Before we can show that the infinitesimal-type partition is a stratification, we will need a local linear model for the infinitesimal-type pieces.

\begin{prop}\label{prop: infinitesimal-type local description}
Let $G$ be a connected Lie group, $K\leq G$ a compact subgroup, $\mfk=\text{Lie}(K)$ its Lie algebra, and $V$ a finite-dimensional $K$-representation. Let
$$
V^\mfk:=\bigcap_{\xi\in\mfk}\ker(\xi_V).
$$
Then,
$$
(G\times_K V)_{(\mfk)}=G\times_K V^\mfk.
$$
\end{prop}

\begin{proof}
Let $[g,v]\in (G\times_K V)_{(\mfk)}$. Since $G_{[g,v]}$ is conjugate to $G_{[e,v]}$ and hence $\mfg_{[g,v]}$ is conjugate to $\mfg_{[e,v]}$, we may assume $g=e$. Observe that $\mfg_{[e,v]}=\mfk_v\subseteq \mfk$. Since $\mfk_v$ and $\mfk$ are conjugate and hence of the same dimension, it follows that $\mfk_v=\mfk$ and hence $v\in V^\mfk$. 
\end{proof}

\begin{prop}
Let $M$ be a proper $G$-space.
\begin{itemize}
    \item[(1)] Then the elements of $\mathcal{S}_\mfg(M)$ are embedded $G$-invariant submanifolds. 
    \item[(2)] The orbit-type stratification $\mathcal{S}_G(M)$ is a refinement of $\mathcal{S}_\mfg(M)$. In general, a strict refinement.
\end{itemize}
\end{prop}

\begin{proof}
Applying the Slice Theorem and Proposition \ref{prop: infinitesimal-type local description}, we recover (1) and the first statement in (2). To see that $\mathcal{S}_G(M)$ is in general a strict refinement of $\mathcal{S}_\mfg(M)$, let $G=S^1$ and $M=\bR^2\times \bR^2$. Define an action of $S^1$ on $M$ by
$$
S^1\times M\to M;\quad (z,x_1,x_2)\mapsto (zx_1,z^2x_2).
$$
Then we see $\mathcal{S}_{S^1}(M)$ has three pieces, $\{(0,0)\}$, $\{(x_1,x_2) \ | \ x_1\neq 0\}$, and $\{(0,x_2) \ | \ x_2\neq 0\}$. Conversely, if we let $i\bR$ denote the Lie algebra of $S^1$, then $\mathcal{S}_{i\bR}(M)$ consists of only two pieces, $\{(0,0)\}$ and $M\setminus\{(0,0)\}$. 
\end{proof}

Before we prove that $\mathcal{S}_\mfg(M)$ is a stratification of $M$, let us discuss another partition which we get for free from the $G$-action. For any Lie group $G$, write $G^\circ$ for the connected component containing the identity element. Equivalently, this is the maximal connected subgroup. Given a connected subgroup $K\leq G$, define
$$
M_{(K)}^\circ:=\{x\in M \ | \ G_x^\circ\text{ is conjugate to }K\}.
$$

\begin{defs}[Reduced Orbit-Type]
Let $M$ be a proper $G$-space. Define
$$
\mathcal{S}_G(M)^\circ:=\bigcup_{x\in M}\pi_0(M_{(G_x^\circ)}^\circ).
$$
\end{defs}

\begin{prop}\label{prop: local form for reduced orbit-types}
Let $G$ be a connected Lie group, $K\leq G$ a compact subgroup, and $V$ a finite-dimensional $K$-representation. Then,
$$
(G\times_K V)_{(K^\circ)}=G\times_K V^{K^\circ}.
$$
\end{prop}

\begin{proof}
First let us show that $G\times_K V^{K^\circ}\subseteq (G\times_K V)_{(K^\circ)}$. Fixing $[g,v]\in G\times_K V^{K^\circ}$, we have
$$
G_{[g,v]}=gK_vg^{-1}.
$$
By definition, $K_v\subseteq K$ and hence $K_v^\circ \subseteq K^\circ$. However, since we are assuming that $v\in V^{K^\circ}$, we also have $K^\circ\subseteq K_v$. Hence, $G_{[g,v]}^\circ =gK^\circ g^{-1}$ and thus $[g,v]\in (G\times_K V)_{(K^\circ)}$. 

\ 

The reverse containment easily follows.
\end{proof}

Making use of slice coordinates from Theorem \ref{thm: slice theorem} together with Proposition \ref{prop: local form for reduced orbit-types}, we easily deduce the following Corollary.

\begin{cor}
    The reduced orbit-type partition $\mathcal{S}_G(M)^\circ$ is a partition of $M$ into embedded $G$-invariant submanifolds.
\end{cor}

So we now have two auxiliary partitions associated to a proper group action, the infinitesimal-type partition $\mathcal{S}_\mfg(M)$ and the reduced orbit-type partition $\mathcal{S}_G(M)^\circ$. In turns out that we can use our results about singular Riemannian foliations to show that these partitions are the same and define a stratification of $M$.

\begin{prop}[{\cite{posthuma_resolutions_2017}}]\label{prop: infinitesimal, reduced orbit, and dimension stratifications are the same}
For a proper $G$-space $M$, let $\mathscr{F}_G$ denote the singular foliation by orbits. Then the infinitesimal partition $\mathcal{S}_\mfg(M)$, the reduced orbit-type partition $\mathcal{S}_G(M)^\circ$, and the dimension-type partition $\mathcal{S}_{dim}(M,\mathscr{F}_G)$ are all equal. That is,
$$
\mathcal{S}_\mfg(M)=\mathcal{S}_G(M)^\circ=\mathcal{S}_{dim}(M,\mathscr{F}_G).
$$
\end{prop}

\begin{proof}
We first show $\mathcal{S}_\mfg(M)=\mathcal{S}_G(M)^\circ$. Indeed let $x,y\in M$ and suppose $G_x^\circ$ and $G_y^\circ$ are conjugate. Since $\mfg_x$ and $\mfg_y$ are the Lie algebras of $G_x^\circ$ and $G_y^\circ$, it follows that $\mfg_x$ and $\mfg_y$ are conjugate. Conversely, suppose $\mfg_x$ and $\mfg_y$ are conjugate by $g\in G$. That is, $\text{Ad}_g(\mfg_x)=\mfg_y$. Write
$$
c_g:G\to G;\quad a\mapsto gag^{-1}
$$
for conjugation by $g$. Since the exponential maps $\exp:\mfg_x\to G_x^\circ$ and $\exp:\mfg_y\to G_y^\circ$ have dense image and for all $\xi\in \mfg_x$ we have
$$
\exp(\text{Ad}_g(\xi))=c_g(\exp(\xi)),
$$
it follows that the diagram commutes
$$
\begin{tikzcd}
\mfg_x\arrow[r, "\exp"]\arrow[d,"\text{Ad}_g",swap] & G_x^\circ\arrow[d,"c_g"]\\
\mfg_y\arrow[r,"\exp"] & G_y^\circ
\end{tikzcd}
$$
Hence $G_x^\circ$ and $G_y^\circ$ are conjugate.
\

Now, we show $\mathcal{S}_{dim}(M,\mathscr{F}_G)=\mathcal{S}_G(M)^\circ$ using the argument presented in the paper \cite{posthuma_resolutions_2017}. To do this, observe that for any $x\in M$ we have
\begin{align}\label{eq: dimension of orbit and connected component of stabilizer}
\begin{split}
    \dim(G\cdot x)&=\dim(G/G_x)\\
    &=\dim(G)-\dim(G_x)\\
    &=\dim(G)-\dim(G_x^\circ).
\end{split}
\end{align}

In particular, if $S\in \mathcal{S}_{G}(M)^\circ$ is the reduced orbit-type piece containing $x$ and $R\in\mathcal{S}_{dim}(M,\mathscr{F}_G)$ is the dimension-type piece containing $x$, then $S\subseteq R$. To show $S=R$ we show $S$ is both closed and open in $R$. Since $R$ is assumed to be connected, this will show $S=R$.

\

Indeed, let $y\in \overline{S}\cap R$ lie in the relative closure of $S$ in $R$. We choose a slice neighbourhood $U\subseteq M$ around $y$ which is equivariantly diffeomorphic to $G\times_{G_y} \nu_y(M,G)$ and let $[g,v]\in G\times_{G_y}\nu_y(M,G)$ correspond to $x\in S\cap U$. Then $G_x=G_{[g,v]}=g(G_y)_vg^{-1}$. Due to Equation (\ref{eq: dimension of orbit and connected component of stabilizer}), we see that the dimension of the orbit through $x$ is $\dim(G)-\dim((G_y)_v)$ and the dimension of the orbit through $y$ is $\dim(G)-\dim(G_y)$. Since $x,y$ lie in the same dimension-type this implies
$$
\dim((G_y)_v)=\dim(G_y).
$$
Since $(G_y)_v^\circ\subseteq G_y^\circ$ and $(G_y)_v^\circ$ is closed, we obtain $(G_y)_v^\circ=G_y^\circ$ and thus $G_x^\circ$ is conjugate to $G_y^\circ$. Therefore, $y\in S$. Hence $\overline{S}\cap R=S\cap R$ and $S$ is closed in $R$.

\

To show $S$ is open, we use the same argument. 
\end{proof}

\begin{cor}\label{cor: infinitesimal partition is an sf and a stratification}
The infinitesimal partition $\mathcal{S}_\mfg(M)$ is a singular foliation, hence a differentiable stratification of $M$.
\end{cor}

\begin{proof}
Since the action of $G$ on $M$ is proper, by a Theorem of Palais \cite{palais_existence_1961}, there exists a $G$-invariant Riemannian metric on $M$ making the orbit foliation $\mathscr{F}_G$ into an SRF. By Theorem \ref{theorem: dimension-type parititon is a locally finite sf}, the dimension partition $\mathcal{S}_{dim}(M,\mathscr{F}_G)$ is a singular foliation hence a stratification. Applying Proposition \ref{prop: infinitesimal, reduced orbit, and dimension stratifications are the same}, we obtain the desired conclusion.
\end{proof}

\section{Stratified Tangent Vectors and Vector Fields}\label{sec: strat vfs}

Just as we had tangent bundles for subcartesian spaces and tangent bundles for foliations, we can also obtain a notion of a tangent bundle and hence vector fields for differentiable stratified spaces. These will inherit many of the structures from Zariski tangent bundles. Let $(X,\Sigma)$ be a weakly differentiable stratified space and $S\in\Sigma$ a stratum. Differentiating the inclusion $S\into X$, we get a natural injection $TS\into T X$. This allows us to define the following.

\begin{defs}[Stratified Tangent Bundle]
Let $(X,\Sigma)$ be a weakly differentiable stratified space. Define the stratified tangent bundle by
$$
T\Sigma:=\bigcup_{S\in\Sigma}TS
$$
with canonical subcartesian structure $C^\infty(T\Sigma):=\bra C^\infty(T X)|_{T\Sigma}\ket$
\end{defs}

\begin{egs}
Given a smooth manifold $M$ together with its canonical stratification $\pi_0(M)$ by connected components, then the stratified tangent bundle
$$
T\pi_0(M)=\bigcup_{S\in \pi_0(M)}TS
$$
is exactly the same thing as the usual tangent bundle $TM$ by Lemma \ref{lem: zariski vector bundle is the same as usual tangent bundle for smooth manifold}.
\end{egs}

\begin{egs}
If $(M,\mathscr{F})$ is a foliated manifold with $\mathscr{F}$ satisfying the hypotheses of Corollary \ref{cor: locally finite foliations are stratifications}, then the foliation tangent bundle is the same as the stratified tangent bundle.
\end{egs}

\begin{defs}
Let $(X,\Sigma)$ be a weakly differentiable stratified space.
\begin{itemize}
    \item[(1)] Fix $S\in\Sigma$ and $x\in S$. Then a stratified tangent vector at $x$ is an element of $T_xS$. Write $T_x\Sigma$ for the stratified tangent vectors at $x$.
    \item[(2)] A stratified vector field is a smooth section $V:X\to T\Sigma$ of the canonical projection $\pi:T\Sigma\to X$. Write $\mfX(X,\Sigma)$ or, if context is clear, $\mfX(\Sigma)$ for the set of all stratified vector fields.
\end{itemize}
\end{defs}

\begin{prop}
Let $(X,\Sigma)$ be a weakly differentiable stratified space. Then $\mfX(\Sigma)\subseteq \mfX(X)$ with the inclusion being strict in general. Furthermore,
\begin{equation}\label{eq: stratified vector fields restrict}
\mfX(\Sigma)=\bigcap_{S\in \Sigma}\mfX(X,S),
\end{equation}
where $\mfX(X,S)$ is as in Definition \ref{def: restrictable vf}.
\end{prop}

\begin{proof}
    Since the differentiable structure on $T\Sigma$ is the restriction of the differentiable structure on $TX$, it follows from Proposition \ref{prop: global to infinitesimal for zariski vector fields} that $\mfX(\Sigma)\subseteq \mfX(X)$. 

    \ 

    In general, the above containment is strict. Indeed, consider $X=\bR^2$ together with the stratification
    $$
    \Sigma=\{\{(0,0)\}, \bR^2\setminus\{(0,0)\}\}.
    $$
    Any stratified vector field $V\in \mfX(\Sigma)$ must vanish at the origin, which clearly is not true for a general vector field on $\bR^2$.

    \ 

    The identity in Equation (\ref{eq: stratified vector fields restrict}) is a straightforward consequence of the definition of $\mfX(X,S)$.
\end{proof}

Recall from Definition \ref{def: Lie bracket of Zariski vf} that we have a natural Lie algebra structure on the set of Zariski vector fields, namely for $V,W\in \mfX(X)$ we define $[V,W]\in \mfX(X)$ by
$$
[V,W](f)=V(W(f))-W(V(f)),\quad f\in C^\infty(X).
$$
Using now Corollary \ref{cor: Lie bracket commutes with restriction}, we obtain the following.

\begin{cor}\label{cor: stratified vf and lie bracket}
    Let $(X,\Sigma)$ be a weakly differentiable stratified space and $V,W\in \mfX(\Sigma)$. Then for every stratum $S\in \Sigma$, we have
    $$
    [V|_S,W|_S]=[V,W]|_S
    $$
\end{cor}

As a final result, we record a hard Theorem of Schwarz that uses techniques that go beyond this thesis.

\begin{theorem}[{\cite{schwarz_lifting_1980}}]\label{theorem: equivariant vector fields project to stratified vector fields}
Let $G$ be a connected Lie group and $M$ a proper $G$-space. Let $\pi:M\to M/G$ be the quotient map.
\begin{itemize}
    \item[(1)] If $V\in\mfX(M)^G$ is a $G$-invariant vector, then the fibre-wise application of the derivative map $T\pi:T\mathcal{S}_G(M)\to T\mathcal{S}_G(M/G)$ defines a stratified vector field $\pi_*V\in\mfX(\mathcal{S}_G(M/G))$. 
    \item[(2)] The map
    $$
    \pi_*:\mfX(M)^G\to \mfX(\mathcal{S}_G(M/G));\quad V\mapsto \pi_*V
    $$
    is a surjection.
\end{itemize}
\end{theorem}

\section{Regularity Conditions}\label{sec: regularity conds}

Differentiable Stratified spaces are somewhat pathological in general and the kinds of spaces that appear usually are tamed. For this reason, we impose various regularity conditions on stratified spaces. The most famous of which being the Whitney conditions (see Definition \ref{def: whitney conditions}). We will also discuss a more useful kind of regularity assumption called Smoothly Locally Trivial With Conical Fibres introduced in \cite{zimhony_commutative_2024}. 

\subsection{Whitney Conditions}\label{sub: Whitney}

The Whitney conditions are local conditions that restrict how the strata in a stratification can fit together. They use the topology of Grassmannians in a non-trivial way, so let us first review this basic notion.

\begin{defs}[Rank $k$ Grassmannian]
    Let $V$ be a finite-dimensional vector space and $k\leq \dim(V)$. Define the rank $k$ Grassmannian on $V$, denoted $\text{Gr}_k(V)$, to be the set of all $k$-dimensional subspaces of $V$. In symbols,
    $$
    \text{Gr}_k(V) = \{ W\substeq V \  | \ W\text{ linear subspace and }\dim(W)=k\}.
    $$
\end{defs}

Given a vector space $V$ and an integer $k\leq \dim(V)$, we can endow $\text{Gr}_k(V)$ with a topology by equipping $V$ with an inner product. Then, given a subspace $W\in \text{Gr}_k(V)$ let $P_W:V\to V$ denote the orthogonal projection onto $W$. This defines an injective map
\begin{equation}\label{eq: embedding of grassmannian}
    \text{Gr}_k(V)\into \text{End}(V);\quad W\mapsto P_W
\end{equation}
and thus we can endow $\text{Gr}_k(V)$ with the subspace topology inherited from the vector space $\text{End}(V)$. Now let
$$
O(V)=\{T:V\to V \ | \ T\text{ is an isometry}\}
$$
Then we have a natural action of $\text{O}(V)$ on $\End(V)$ by conjugation
$$
\text{O}(V)\times \End(V)\to \End(V);\quad (g,T)\mapsto g\circ T\circ g^{-1}.
$$
One can easily show that if $W\in \text{Gr}_k(V)$ and $g\in\text{O}(V)$ is an isometry, then
$$
g\circ P_W\circ g^{-1}=P_{g(W)}.
$$
Clearly for any two $k$-dimensional subspaces $W,W'\in\text{Gr}_k(V)$ we can find an isometry $g$ so that $g(W)=W'$. Hence, $\text{Gr}_k(V)$ is the orbit of the linear $\text{O}(V)$ action and thus is naturally a compact connected smooth manifold. With this, we are now ready to discuss the Whitney conditions.

\begin{defs}
Let $R,S\subseteq \bR^n$ be two embedded submanifolds and $x\in S$. Say $(R,S)$ is Whitney regular at $x$ if the following two conditions hold.
\begin{itemize}
    \item[(A)] If $\{x_k\}\subseteq R$ is a sequence in $R$ converging to $x$ and $\tau\subseteq T_x\bR^n$ is a subspace such that $\{T_{x_k}R\}$ converges to $\tau$ in the $\dim(R)$ Grassmannian of $T\bR^n$, then $T_x S\subseteq \tau$.
    \item[(B)] Suppose $\{y_k\}\subseteq S$ is another sequence converging to $x$ with $x_k\neq y_k$ for all $k$ and write $\ell_k$ for the line connecting $x_k$ to $y_k$. If the sequence $\{\ell_k\}$ converges to a line $\ell$ in the $1$-dimensional Grassmannian of $T\bR^n$, then  $\ell\subseteq \tau$.
\end{itemize}
If $(R,S)$ only satisfy (A) at $x$, say they are Whitney (A) at $x$. 
\end{defs}

\begin{defs}[Whitney Regular Stratified Spaces]\label{def: whitney conditions}
Let $(X,\Sigma)$ be a differentiable stratified space. Say $(X,\Sigma)$ is Whitney (A) (resp. (B)) if for each pair of strata $(R,S)$ with $S\subseteq\overline{R}$, any $x\in S$, and any chart $\phi:U\into \bR^N$ the pair $(\phi(S\cap U),\phi(R\cap U))$ is Whitney (A) (resp. (B)) at $x$. If $(X,\Sigma)$ satisfies both (A) and (B) at $x$, say $(X,\Sigma)$ is Whitney regular.
\end{defs}

\begin{egs}
    As we will see, most spaces we've discussed so far are Whitney regular. Non-examples of Whitney regular spaces which could be found in \cite{whitney_local_2015}, \cite{mather_notes_2012}, or \cite{pflaum_analytic_2001} will fail in our context because we demand that if $S,R\in \Sigma$ are strata with $S\neq R$ and $S\subseteq \overline{R}$ then $\dim S<\dim R$. 
\end{egs}

In the subcartesian setting in which we are working, the Whitney conditions are somewhat unnatural as they have been formulated above. We will now end this section with the discussion of a stronger condition that resembles condition (S) in the definition of a singular foliation (see Definition \ref{def: stefan-sussmann sf}).

\begin{prop}\label{prop: enough sections implies Whitney A}
Let $(X,\Sigma)$ be a differentiable stratified space. If the map
$$
\mfX(\Sigma)\times X\to T\Sigma;\quad (V,x)\mapsto V_x
$$
is surjective, then $(X,\Sigma)$ is Whitney (A).
\end{prop}

\begin{proof}
For all that follows, fix a differentiable stratified space $(X,\Sigma)$. First note that if 
$$
\mfX(\Sigma)\times X\to T\Sigma;\quad (V,x)\mapsto V_x
$$
is surjective, then for any open subset $U\subseteq X$ so is
$$
\mfX(\Sigma_U)\times U\to T\Sigma_U;\quad (V,x)\mapsto V_x.
$$
Thus, passing to charts, we may assume that $X\subseteq \bR^N$ is a subset. Now suppose $S,R\in\Sigma$ are two strata, $x\in S$, and $\{y_n\}\subseteq R$ is a sequence and $\tau\subseteq T_x\bR^N$ is a subspace so that
$$
\lim_{n}y_n=x
$$
and 
$$
\lim_nT_{y_n}R=\tau
$$
in the $\dim(R)$ Grassmannian of $\bR^N$. We want to show that $T_xS\subseteq \tau$. To do this, choose $v\in T_xS$ and let $V\in \mfX(\Sigma)$ be a stratified vector field with $V_x=v$. Identify $T_z\bR^N=\bR^N$ for all $z\in \bR^N$ and for any subspace $W\subseteq \bR^N$, write $P_W:\bR^N\to \bR^N$ for the orthogonal projection onto $W$ with respect to the standard metric on $\bR^N$. By definition of the topology on the $\dim(R)$ Grassmannian of $\bR^N$, we have
$$
\lim_{n\to \infty}\|P_{T_{y_n}R}-P_\tau\|=0,
$$
where $\|\cdot\|$ is the standard operator norm. By continuity of $V\in\mfX(\Sigma)$, we have
$$
\lim_{n\to \infty}V_{y_n}=v
$$
In particular, we conclude
$$
\lim_{n\to\infty}(V_{y_n}-P_W(V_{y_n}))=V_x-P(V_x)=v-P(v)
$$
Furthermore, since $P_{T_{y_n}R}(V_{y_n})=V_{y_n}$ we have for any $n$
$$
\|V_{y_n}-P_{\tau}(V_{y_n})\|\leq \|V_{y_n}-P_{T_{y_n}R}(V_{y_n})\|+\|P_{T_{y_n}R}(V_{y_n})-P_{\tau}(V_{y_n})\|\leq \|P_{T_{y_n}R}-P_\tau\|\cdot \|V_{y_n}\|.
$$
Hence,
$$
\lim_{n\to\infty}\|V_{y_n}-P_{\tau}(V_{y_n})\|=\|v-P_\tau(v)\|\leq 0\cdot \|v\|=0.
$$
Hence, $P_\tau(v)=v$ and thus $v\in \tau$.
\end{proof}

\begin{egs}
    Let us now return to the union of the coordinate axes $X$ in $\bR^2$ along with the stratification $\Sigma$ given by 
    $$
    \Sigma=\{\{(0,0)\},H_-,H_+,V_-,V_+\}
    $$
    from Example \ref{eg: coordinate axes}. We saw in Example \ref{eg: vector fields need not span the zariski tangent bundle} that the Zariski vector fields do not span the Zariski tangent bundle. That is, the map
    $$
    X\times \mfX(X)\to TX;\quad (x,V)\mapsto V_x
    $$
    is not surjective. As we saw, this is due to the fact that $T_{(0,0)}X=T_{(0,0)}\bR^2$, but if $V\in \mfX(X)$, then $V_{(0,0)}=0$. However, it's now easy to see that the stratified vector fields do span the stratified tangent bundle, i.e. the map
    $$
    X\times \mfX(\Sigma)\to T\Sigma
    $$
    is surjective. Hence, $(X,\Sigma)$ is Whitney (A).
\end{egs}

As a consequence, any stratification of a manifold which is also a singular foliation is Whitney (A). In fact, this also holds more generally.

\begin{cor}
The following are all Whitney (A).
\begin{itemize}
    \item[(1)] Let $(M,g,\mathscr{F})$ be an SRF. Then the dimension-type stratification $\mathcal{S}_{dim}(M,\mathscr{F})$ is Whitney (A).
    \item[(2)] If $G$ is a connected Lie group and $M$ a proper $G$-space, then both the orbit-type stratification $\mathcal{S}_G(M)$ and the infinitesimal stratification $\mathcal{S}_\mfg(M)$ are Whitney (A). 
\end{itemize}
\end{cor}

\begin{proof}
\begin{itemize}
    \item[(1)] By Theorem \ref{theorem: dimension-type parititon is a locally finite sf}, the dimension-type stratification $\mathcal{S}_{dim}(M,\mathscr{F})$ is a singular foliation, hence satisfies Proposition \ref{prop: enough sections implies Whitney A} and so is Whitney (A).
    \item[(2)] Same argument holds for the orbit-type stratification $\mathcal{S}_G(M)$. For the canonical stratification, let $\pi:M\to M/G$ be the quotient map, $x\in M$, and $v\in T_[x]\mathcal{S}_G(M/G)$ a stratified tangent vector. In particular, if $S\in\mathcal{S}_G(M)$ is the orbit-type stratum containing $x$ then $v\in T_[x](S/G)$ hence admits a lift to 
    $$
    \widetilde{v}\in T_x S
    $$
    since the map $S\to S/G$ is a surjective submersion. As we saw in Example \ref{eg: three foliations associated to proper group action}, there exists equivariant vector field $\widetilde{V}\in \mfX(M)^G$ and Lie algebra element $\xi\in \mfg$ so that
    $$
    \widetilde{V}_x+\xi_M(x)=\widetilde{v}.
    $$
    By Theorem \ref{theorem: equivariant vector fields project to stratified vector fields}, $\widetilde{V}$ projects to an element of $\mfX(\mathcal{S}_G(M/G))$ and clearly
    $$
    T_x\pi(\widetilde{v})=T_x\pi(\widetilde{V}_x+\xi_M(x))=T_x\pi(\widetilde{V}_x)=V_{[x]}.
    $$
    In particular, we have just shown that the map
    $$
    M/G\times\mfX(\mathcal{S}_G(M/G))\to T\mathcal{S}_G(M/G);\quad ([x],V)\mapsto V_{[x]}
    $$
    is surjective and thus $(M/G,\mathcal{S}_G(M/G))$ is Whitney (A) by Proposition \ref{prop: enough sections implies Whitney A}.
\end{itemize}
\end{proof}

As for Whitney (B), we will also be attempting to dispose of it by making another analogy with singular foliations, namely that they are ``locally trivial'' in some sense. A necessary result that we will need to make this leap is due to Pflaum.

\begin{prop}[{\cite{pflaum_analytic_2001}}]\label{prop: suffices to check Whitney B in one chart}
Let $(X,\Sigma)$ be a differentiable stratified space and $S,R\in\Sigma$ two strata with $S\subseteq \overline{R}$ and $x\in S$. If there exists a chart $\phi:U\subseteq X\into \bR^N$ about $x$ so that $(\phi(S\cap U),\phi(R\cap U))$ is Whitney (B) at $x$, then for any chart $\psi:W\into \bR^M$ about $x$, $(\phi(S\cap W),\phi(R\cap W))$ satisfy Whitney (B).
\end{prop}

\subsection{Local Triviality}\label{sub: local trivial}

One of the main reasons we care about Whitney (B) is mostly due to its consequences, in particular the existence of control data \cite{mather_stratifications_1973}. This is not an aspect of this thesis which we will be discussing. Thankfully, there is a stronger (and in my opinion, more natural) condition than Whitney (B) called Smoothly Locally Trivial With Conical Fibres due to Zimhony \cite{zimhony_commutative_2024}, which we will be discussing below. The nice consequence of this condition is that it not only implies the existence of control data like Whitney (B), but this control data will in fact be smooth! See \cite{zimhony_commutative_2024} for more information about smooth control data.

\ 

To begin, let's discuss a weaker version of smoothly locally trivial with conical fibres, namely locally trivial.

\begin{defs}[Local Triviality]\label{def: locally trivial}
Let $(X,\Sigma)$ be a differentiable stratified space. Say $X$ is (smoothly) locally trivial if for each stratum $S\in \Sigma$ and point $x\in S$, we can find a chart
$$
\phi:U\subseteq X\into \bR^k\times \bR^M
$$
with $\phi(x)=(0,0)$ such that
\begin{itemize}
    \item[(1)] there exists open neighbourhood $V\subseteq \bR^k$ of $0$ so that
    $$
    \phi(S\cap U)=V\times \{0\};
    $$
    \item[(2)] if $R\in \Sigma$ satisfies $S\subseteq \overline{R}$, then there exists submanifold $C\subseteq \bR^M$ so that
    $$
    \phi(U\cap R)=V\times C.
    $$
\end{itemize}
We will call the singular chart $\phi:U\into \bR^k\times \bR^M$ a \textbf{splitting chart} centred at $x$ and $S$.
\end{defs}

\begin{egs}
    Suppose $(M,\mathscr{F})$ is a singular foliation of a manifold so that $\mathscr{F}$ is also a stratification. By a result of Stefan \cite{stefan_accessible_1974}, $\mathscr{F}$ satisfies local triviality.
\end{egs}

\begin{remark}
    Shortly, we will be introducing an even stronger condition of local triviality, so I will leave the less trivial examples for later.
\end{remark}

\begin{prop}\label{prop: locally trivial strat implies enough sections}
    Suppose $(X,\Sigma)$ is a locally trivial differentiable stratified space which admits differentiable partitions of unity. Then the map
    $$
    X\times \mfX(\Sigma)\to T\Sigma
    $$
    is surjective. 
\end{prop}

\begin{proof}
    Fix a stratum $S\in \Sigma$, $x\in S$, and $v\in T_x S$. Let $\phi:U\subseteq X\into \bR^k\times \bR^M$ be a splitting chart centred on $x$ and $S$ as in Definition \ref{def: locally trivial}. Write $\phi(S\cap U)=\Omega\times\{0\}$ for open neighbourhood $\Omega\subseteq \bR^k$ about $0$. 

    \ 

    Since $S$ is  a smooth manifold, we can find a vector field $V\in\mfX(S\cap U)$ so that $V_x=v$. Push $V$ forward via $\phi$ to a vector field $W\in\mfX(\Omega)$. We can extend $W$ to a vector field on $\Omega\times\bR^M$ by
    $$
    W_{(a,b)}=(W_a,0)\in T_{(a,b)}(\Omega\times \bR^M),
    $$
    for $(a,b)\in\Omega \times \bR^M$. Trivially, $W$ is tangent to the images of the strata $\phi(R\cap U)$ for $R\in \Sigma$. Hence, we can push $W$ forward by $\phi^{-1}$ to get a stratified vector field $\widetilde{V}\in\mfX(U)$. Using bump functions, we can then extend $\widetilde{V}$ to a global stratified vector field with $\widetilde{V}_x=v$.
\end{proof}

\begin{cor}
    Locally trivial differentiable stratified spaces are Whitney (A).
\end{cor}

\begin{proof}
    An easy consequence of Proposition \ref{prop: enough sections implies Whitney A}.
\end{proof}

\subsection{Locally Conical and Quasi-Homegeneity}\label{sub: quasi-homo}

Local triviality is a powerful property for a stratified space to have as it implies Whitney (A), but it does not necessarily imply Whitney (B) holds. To ensure this stronger condition holds, we need to impose that the local trivializations are homogeneous as well.

\begin{defs}\label{def: quasi homo}
    Let $(X,\Sigma)$ be a locally trivial differentiable stratified space, $x\in X$, $S\in \Sigma$, and $\phi:U\into \bR^k\times \bR^M$ a splitting chart centred about $x$ and $S$. Write $\phi(U\cap S)=V\times \{0\}$ for an open neighbourhood $V\subseteq \bR^k$ of $0$.

    \ 

    Say $\phi:U\to \bR^k\times \bR^M$ is quasi-homogeneous of weight $\mathbf{d}=(d_1,\dots,d_M)$ if for any stratum $R\in\Sigma$, there exists a submanifold $C\subseteq \bR^M\setminus\{0\}$ which is closed under the $\bR_{>0}$-action
    $$
    \bR_{>0}\times \bR^M\to \bR^M;\quad (t,(x_1,\dots,x_M))\mapsto (t^{d_1}x_1,\dots,t^{d_M}x_M)
    $$
    such that 
    $$
    \phi(R\cap U)=V\times C.
    $$
    If for any $S\in \Sigma$ and $x\in S$ we can find a quasi-homogeneous chart of weight $\mathbf{d}$ for some $\mathbf{d}$, then say $(X,\Sigma)$ is quasi-homogeneous. If we can always choose $\mathbf{d}$ so that $\mathbf{d}=(1,\dots,1)$, then say $(X,\Sigma)$ is locally conical.
    \end{defs}

\begin{egs}
Consider the coordinate axes
$$
X=\{(x,y) \in \bR^2 \ | \ xy=0\}
$$
with the stratification
$$
\Sigma = \{\{(0,0)\}\}\cup \pi_0(X\setminus \{(0,0)\})
$$
that was discussed in Example \ref{eg: coordinate axes}. The inclusion $\iota:X\into \bR^2$ defines a locally conical chart centred on $(0,0)$ as the stratum containing $(0,0)$ is the singleton $\{(0,0)\}$ and any of the other strata, which are open rays emanating from $(0,0)$ are all closed under positive scalar multiplication.
\end{egs}

To show that the quasi-homogeneous condition from Definition \ref{def: quasi homo} is stronger than being Whitney regular, we will first need a Lemma.

\begin{lem}\label{lem: intersection of quasi homogeneous submanifold and the sphere is a submanifold}
Let $\mathbf{d}=(d_1,\dots,d_n)\in\bZ_{>0}^n$ and consider the $\bR_{>0}$-action
\begin{equation}
    \mu:\bR_{>0}\times \bR^n\to \bR^n;\quad(t,x)\mapsto \mu_t(x)=(t^{d_1}x_1,\dots,t^{d_n}x_n).
\end{equation}
Suppose $C\subseteq \bR^n\setminus\{0\}$ is an embedded $\bR_{>0}$-invariant submanifold. Then $C\cap S^{n-1}$ is an embedded submanifold.
\end{lem}

\begin{proof}
First let us consider the weighted Euler-like vector field $E\in\mfX(\bR^n)$ defined by the following action on functions. If $f\in C^\infty(\bR^n)$ and $x_0\in \bR^n$, define
$$
E_{x_0}f:=\frac{d}{dt}\bigg|_{t=0} f\circ \mu_{e^t}(x_0).
$$
In terms of coordinates $(x_1,\dots,x_n)$ we have
$$
E=\sum_{i=1}^nd_ix_i\frac{\partial}{\partial x_i}.
$$
In particular, notice that $E=0$ only at $\{0\}$. Furthermore, since $C$ is closed under the action of $\bR_{>0}$, it follows that $E$ is tangent to $C$.

\

Now consider the map 
$$
R:C\to \bR;\quad c\mapsto \|c\|^2,
$$
where $\|\cdot\|$ is the standard norm on $\bR^n$. I claim $1$ is a regular value of $R$. Indeed, if $c=(c_1,\dots,c_n)\in R^{-1}(1)$ then we can show
$$
T_cR(E_c)=2\bigg(\sum_{i=1}^n d_ic_i^2\bigg)\frac{d}{d\lambda}\bigg|_{\lambda=1}.
$$
Observe that since $d_1,\dots,d_n\geq 1$, we have
$$
\sum_{i=1}^n d_ic_i^2\geq \sum_{i=1}^n c_i^2=R(c)=1>0.
$$
In particular, $T_cR\neq 0$ and hence $1$ is a regular value. Therefore,
$$
R^{-1}(1)=C\cap S^{n-1}
$$
is an embedded submanifold.
\end{proof}

We are now in the position to show the main result of this section.

\begin{theorem}\label{thm: quasi-homogeneous implies Whitney regular}
If $(X,\Sigma)$ is quasi-homogeneous in the sense of Definition \ref{def: quasi homo}, then $(X,\Sigma)$ is Whitney regular.
\end{theorem}

\begin{proof}
This is an adaption of an argument due to David Miyamoto.

\

Suppose $(X,\Sigma)$ is quasi-homogeneous and fix two strata $S,R\in\Sigma$ and $x\in S$. Due to Proposition \ref{prop: suffices to check Whitney B in one chart}, we may work in a quasi-homogeneous chart with weight $\mathbf{d}=(d_1,\dots,d_M)\in \bZ_{>0}^M$ centred at $x$ and $S$. That is, we may assume
$$
S=\bR^k\times \{0\},
$$
and 
$$
R=\bR^k\times C
$$
for some submanifold $C\subseteq \bR^{M}\setminus\{0\}$ closed under the action
$$
\mu:\bR_{>0}\times \bR^M\to \bR^M;\quad y\mapsto \mu_t(y)=(t^{d_1}y_1,\dots,t^{d_M}y_M).
$$
Suppose $\{(x_n,y_n)\}\subseteq R$ and $\{(x_n',0)\}\subseteq S$ are two sequences such that
\begin{itemize}
    \item[(1)] $x_n\neq x_n'$ for all $n$.
    \item[(2)] Both sequences converge to $(0,0)$
    \item[(3)] The tangent spaces $T_{(x_n,y_n)}R$ converge to some $\tau\subseteq T_{(0,0)}\bR^N$.
    \item[(4)] The lines $\ell_n$ connecting $(x_n,y_n)$ to $(x_n',0)$ converge to some line $\ell\subseteq T_0\bR^k\oplus T_0\bR^M$.
\end{itemize}
First I claim that there exists $\overline{y}\in S^{M-1}\subseteq \bR^{M}$ and a subspace $\tau'\subseteq T_0\bR^{M}$ so that
\begin{equation}
    \tau=T_0\bR^k\oplus (T_0\bR\cdot\{\overline{y}\}+\tau').
\end{equation}
To see this, first notice that since
$$
T_{(x_n,y_n)}R=T_{x_n}\bR^k\oplus T_{y_n}C
$$
we must have
$$
\tau=T_0\bR^k\oplus \tau'',
$$
where $\tau''\subseteq T_0\bR^{M}$. Now observe that since none of the $y_n$'s are zero and $C$ is closed under the action of $\bR_{>0}$, a quick continuity argument shows there exists $t_n\in \bR_{>0}$ for all $n$ so that 
$$
\mu_{t_n}(y_n)\in C\cap S^{M-1}
$$
for all $n$.  Since $S^{M-1}$ is compact, there exists a subsequence of $\{\mu_{t_n}(y_n)\}$ which converges to some $\overline{y}\in S^{M-1}$. Furthermore, using the compactness of Grassmannians and Lemma \ref{lem: intersection of quasi homogeneous submanifold and the sphere is a submanifold}, we may assume that the sequence of subspaces $T_{\mu_{t_n}(y_n)}C\cap S^{M-1}$ converges to some $\tau'$. In now claim that
\begin{equation}\label{eq: proving Whitney}
    \tau''=T_0(\bR\cdot \overline{y})+\tau'.
\end{equation}
Indeed, $\tau'$ is tangent to $S^{M-1}$ while the line $\bR\cdot \overline{y}$ is perpendicular. Hence, $T_0(\bR\cdot y)$ and $\tau'$ intersect trivially. Hence, the dimensions of both sides of Equation (\ref{eq: proving Whitney}) are equal to the dimension of $S$. Furthermore, clearly both $T_0(\bR\cdot \overline{y}),\tau'\subseteq \tau''$. Thus, equality follows.

\

Two consequences:
\begin{itemize}
    \item[(1)] $T_{(0,0)}S\subseteq \tau$.

    \

    Indeed, this follows from the fact that $T_{(0,0)}S=T_0\bR^k\oplus\{0\}$ and $\tau=T_0\bR^k\oplus\tau''$.

    \item[(2)] $\ell\subseteq \tau$.

    \

    Observe that each line $\ell_n$ has direction vector
    $$
    (x_n,y_n)-(x_n',0)=(x_n-x_n',y_n).
    $$
    Clearly $x_n-x_n'\in \bR^k$ still and each $y_n\in T_0\bR\cdot y_n/\|y_n\|$ which converges to $T_0\bR\cdot \overline{y}$. Hence, they all converge.
\end{itemize}
\end{proof}

Every stratified space we've considered thus far is quasi-homogeneous.

\begin{theorem}[{\cite{miyamoto_srfs_2025}}]
Let $(M,g,\mathscr{F})$ be an SRF and $\mathcal{S}_{dim}(M,\mathscr{F})$ the dimension-type stratification. Then $(M,\mathcal{S}_{dim}(M,\mathscr{F}))$ is locally conical.
\end{theorem}

\begin{proof}
Let $S\in \mathcal{S}_{dim}(M,\mathscr{F})$ be a dimension stratum and $x\in S$. Let 
$$
\phi:U\subseteq M\to V\subseteq T_xS\oplus \nu_x(M,S)
$$
be a local normal form neighbourhood centred on $x$. Then,
$$
\phi(U\cap S)=V\cap (T_xS\oplus \{0\}).
$$
Now suppose that $R\in\mathcal{S}_{dim}(M,\mathscr{F})$ is another dimension-type stratum with $S\subseteq \overline{R}$. If $\dim(R)=\dim(S)+d$, then
$$
\phi(U\cap R)=V\cap (T_xS\oplus \Sigma_d),
$$
where $\Sigma_d\subseteq \nu_x(M,S)$ is the dimension-type stratum containing leaves of dimension $d$ in $\nu_x(M,S)$. By the Homothety Lemma (Lemma \ref{lem: homothety lemma}), $\Sigma_d$ is a conical submanifold. Therefore $\phi:U\to V$ is a conical neighbourhood.
\end{proof}

\begin{cor}
Let $G$ be a connected Lie group and $M$ a proper $G$-space. Then the infinitesimal stratification $\mathcal{S}_\mfg(M)$ is locally conical.
\end{cor}

\begin{proof}
Just apply Corollary \ref{cor: infinitesimal partition is an sf and a stratification}.
\end{proof}

\begin{theorem}[{\cite{zimhony_commutative_2024}}]
Let $G$ be a connected Lie group and $M$ a proper $G$-space. Write $\mathcal{S}_G(M)$ for the orbit-type stratification of $M$. Then $(M,\mathcal{S}_G(M))$ is locally conical.
\end{theorem}

\begin{proof}
Let $S\in\mathcal{S}_G(M)$ be an orbit-type stratum, $x\in S$, and $U\subseteq M$ a maximal slice chart about $x$, i.e. a $G$-equivariant diffeomorphism
$$
\phi:U\to G\times_{G_x}\nu_x(M,G).
$$
mapping $x$ to $[e,0]$. For convenience, write $K=G_x$ and $V=\nu_x(M,G)$. Choosing a $K$-invariant metric on $V$, let $V':=(V^K)^\perp$ be the orthogonal complement to the fixed point set $V^K$. Then, there is a canonical $G$-equivariant diffeomorphism
$$
G\times_K V=(G\times_K V')\times V^K.
$$
Letting $A_1=(G/K)\times V^K$ and $A_2=G\times_K V')$ we see that we have a diffeomorphism
$$
\phi:U\to A_1\oplus A_2
$$
from $U$ to a sum of two vector bundles over $G/K$. Observe that 
$$
\phi(U\cap S)=G\times_K V^K=(G\times_K V^K)\oplus(G\times_K \{0\})=A_1\oplus \{0\}
$$
Next, if $R>S$ and say $R\in \pi_0(M_{(H)})$ for some compact subgroup $H\leq K$, then first
$$
\phi(U\cap M_{(H)})= (G\times_K V'_{(H)})\times V^K
$$
and so there exists a connected component $R'\in \pi_0(V'_{(H)})$ so that
$$
\phi(U\cap R)=(G\times_K R')\times V^K=A_1\oplus C,
$$
where $C=G\times_K R'$. $C$ is clearly a submanifold of $A_2$ and since the action of $K$ on $V$ is linear, the orbit-type strata are all conical.
\end{proof}

\begin{theorem}[{\cite{zimhony_commutative_2024}}]\label{thm: quasi-hom quotient space}
Let $G$ be a connected Lie group, $M$ a proper $G$-space, and let $\mathcal{S}_G(M)$ denote the canonical stratification of $M/G$. Then $(M/G,\mathcal{S}_G(M/G))$ is quasi-homogeneous.
\end{theorem}

\begin{proof}
Suffices to assume $M=G\times_K V$ where $K\leq G$ a compact subgroup and $V$ is a finite-dimensional linear $K$-representation. Choose a $K$-invariant metric on $V$ and let $F=(V^K)^\perp$. Observe that
$$
G\times_K V=(G\times_K F)\times V^K.
$$
Thus,
$$
(G\times_K V)/G\cong V^K\times (F/K)
$$
Now, as is shown by Duistermaat \cite[pg. 16]{duistermaat_dynamical_nodate}, we can find a generating set $q_1,\dots,q_M\in \bR[F]^K$ so that
\begin{itemize}
    \item[(1)] each of the $q_i$ are homogeneous
    \item[(2)] If $x_1,\dots,x_n$ are coordinates on $V^K$, then
    $$
    x_1,\dots,x_n,q_1,\dots,q_M
    $$
    is a basis for $\bR[V^K\times F]^K$.
\end{itemize}
Thus, we get a Hilbert map (see Definition \ref{def: hilbert map})
$$
p:V^K\times F\to V\times \bR^M;\quad (v,w)\mapsto (v,q_1(w),\dots,q_M(w))
$$
which induces an embedding 
$$
p:V^k\times F/K\to V^K\into \bR^M
$$
For each $i=1,\dots,M$, let $d_i=\deg(q_i)$ then and write 
$\mathbf{d}=(d_1,\dots,d_M)$. Letting $M=V^K\times \bR^M$, we can define a monoid action $\mu:\bR\times M\to M$ by
$$
\mu_t(v,y_1,\dots,y_M)=(v,t^{d_1}y_1,\dots,t^{d_M}y_M).
$$
I now claim that $p:V^K\times F/K\into V^K\times \bR^M$ is a quasi-homogeneous chart about the stratum $V^K\times\{0\}$ with weight $\mathbf{d}$. We first observe that
$$
p(V^K\times \{0\})=V^K\times \{0\}
$$
which is the first condition from Definition \ref{def: quasi homo}. For the second, first let $R\subseteq V^K\times F/K$ be any other stratum. Then $R$ is the connected component of the orbit-type $V^K\times F_{(H)}/K$ for some compact subgroup $K\leq H$. $F_{(H)}$ is closed under scalar multiplication by $\bR_{>0}$, so it immediately follows that so is $R$. Hence, since $p$ is made up of homogeneous polynomials, we have $p(R)$ is quasi-homogeneous. 
\end{proof}
\cleardoublepage
\chapter{Stratified Pseudobundles}\label{ch: stratpseud}

Pseudobundles are a natural generalization of smooth vector bundles to the subcartesian setting. A pseudobundle $E$ over a subcartesian space $X$ can be thought of as almost a vector bundle, except the rank of the fibres is allowed to vary. Their utility for our later discussion on singular quantization is that they provide a convenient language for discussing the many ``singular vector bundles'' we will need to consider, like stratified tangent bundles and stratified polarizations. 

\ 

The study of pseudobundles was first initiated by Marshall \cite{marshall_calculus_1975} and, as is characteristic for older subcartesian literature, are defined in terms of atlases of local trivializations. I followed this approach in my paper \cite{ross_stratified_2024} where I defined ``stratified vector bundles''. I have subsequently realized that this terminology of pseudobundles exists and so we will using this term instead. I also have given a much more general definition of a pseudobundle in terms of the differentiable functions in analogy with definining a subcartesian structure in terms of its underlying differential structure. 

\section{Pseudobundles}\label{sec: pseudo}

\begin{defs}[Pseudobundle]
    Let $X$ be a subcartesian space and $\mathbb{K}\in \{\bR,\bC\}$. A $\mathbb{K}$-pseudobundle over $X$ consists of the following data
    \begin{itemize}
        \item[(i)] a subcartesian space $E$ and a surjective differentiable map $\pi:E\to X$; and 
        \item[(ii)] a $\bK$-vector space structure equipped to each fibre $E_x=\pi^{-1}(x)$, $x\in X$;
    \end{itemize}
    subject to the following conditions.
    \begin{itemize}
        \item[(1)] Writing
        $$
        E\times_\pi E:=\{(e_1,e_2)\in E\times E \ | \ \pi(e_1)=\pi(e_2)\}
        $$
        and equipping $E\times_\pi E$ with the subset subcartesian structure inherited from $E\times E$, then the induced global scalar multiplication and addition maps
        \begin{align*}
            \mu:\bK\times E\to E;\quad &(t,e)\mapsto \mu_t(e)=t\cdot e\\
            +:E\times_\pi E\to E;&\quad (e_1,e_2)\mapsto e_1+e_2
        \end{align*}
        are differentiable.
        \item[(ii)] The map $\pi:E\to M$ is a topological quotient map.
        \item[(iii)] The subcartesian structure $C^\infty(E)$ is generated by the pullbacks of functions on $X$, $\pi^*C^\infty(X)$, and the real linear functions $C_{(1)}^\infty(E)$, where 
        $$
        C^\infty_{(1)}(E):=\{f\in C^\infty(E) \ | \ f(t\cdot e)=t\cdot f(e)\text{ for all }(t,e)\in \bK\times E\}.
        $$
        That is, for any function $f\in C^\infty(E)$ and any point $e\in E$, there exists an open neighbourhood $U\substeq E$ of $e$, functions $g_1,\dots,g_n \in C^\infty(X)$, linear functions $\alpha_1,\dots,\alpha_m\in C^\infty(E)_{(1)}(E)$, and a smooth function $F\in C^\infty(\bR^n\times \bR^m)$ so that
        $$
        f|_U=F(\pi^*g_1,\dots,\pi^*g_n,\alpha_1,\dots,\alpha_m)
        $$
    \end{itemize}
    We will refer to $\bR$-pseudobundles simply as pseudobundles and $\bC$-pseudobundles as complex pseudobundles.
\end{defs}

Pseudobundles are a straightforward generalization of vector bundles. The last condition that the smooth structure of a vector bundle is generated by pullbacks of smooth functions on the base and linear functions amounts to passing to trivializing neighbourhoods and using the same techniques as in the proof of Lemma \ref{lem: zariski vector bundle is the same as usual tangent bundle for smooth manifold}. For more general pseudobundles, let us now turn to Zariski tangent bundles.

\begin{egs}\label{eg: Zariski tangent bundles are pseudobundles}
    Let $X$ be a subcartesian space. By Theorem \ref{thm: zariski tangent bundle is a pseudo-bundle}, the Zariski tangent bundle $\pi:TX\to X$ satisfies  most of the axioms of a pseudobundle, except by definition the subcartesian structure on $TX$ is defined to be
    $$
    C^\infty(TX)=\bra \pi^*C^\infty(X)\cup dC^\infty(X)\ket
    $$
    where $dC^\infty(X)$ is the set of functions of the form
    $$
    df:TX\to \bR;\quad v\mapsto v(f)
    $$
    for $f\in C^\infty(X)$. Clearly $dC^\infty(X)\substeq C^\infty_{(1)}(TX)\substeq C^\infty(TX)$ and hence
    $$
    C^\infty(TX)=\bra \pi^*C^\infty(X)\cup dC^\infty(X)\ket\substeq \bra \pi^*C^\infty(X)\cup C^\infty_{(1)}(TX)\ket\substeq C^\infty(TX).
    $$
    Thus, $\pi:TX\to X$ defines a pseudobundle.
\end{egs}

\begin{egs}
    Thanks to the previous example, we can now use Zariski tangent bundles to show that pseudobundles need not be vector bundles. Indeed, consider the union of the coordinate axes in $\bR^2$,
    $$
    X=\{(x,y)\in \bR^2  \ | \ xy=0\}.
    $$
    As we saw in Example \ref{eg: union of coordinate axes tangent spaces}, 
    $$
    \dim(T_{(x,y)}X)=
    \begin{cases}
        1,\quad &(x,y)\neq (0,0)\\
        2,\quad &(x,y)=(0,0)
    \end{cases}
    $$
\end{egs}

Many of the other ``singular vector bundles'' we've considered so far have been subsets of either vector bundles or Zariski tangent bundles. And so, to see that they are all indeed pseudobundles, we need the following definition.

\begin{defs}[Subbundles]
    Let $\pi:E\to X$ be a pseudobundle. A subbundle if a pair of subsets $A\substeq E$ and $Y\substeq X$ such that $\pi(A)=Y$ and so that the restriction $\pi|_A:A\to Y$ is a pseudobundle where $A$ and $Y$ are equipped with the subset subcartesian structures. 
\end{defs}

\begin{prop}
    Let $\pi:E\to X$ be a pseudobundle and $A\substeq E$ a subset with $Y=\pi(A)$. If $A$ is closed under scalar multiplication and fibre-wise addition, then $\pi|_A:A\to Y$ is a subbundle.
\end{prop}

\begin{proof}
    Fix a subset $A\substeq E$ which is closed under scalar multiplication and fibre-wise addition. Endowing $A$ and its image $Y=\pi(A)\substeq X$ with the subspace subcartesian structures, we see that $\pi|_A:A\to Y$ already satisifies most of the conditions of a subbundle, except we need to show
    $$
    C^\infty(A)=\bra (\pi|_A)^*C^\infty(Y)\cup C^\infty_{(1)}(A)\ket.
    $$
    Clearly, $\bra (\pi|_A)^*C^\infty(Y)\cup C^\infty_{(1)}(A)\ket\substeq C^\infty(A)$, so we only need to show the reverse containment. To do this, fix $f\in C^\infty(A)$ and $a\in A$. Then in a neighbourhood $U\subseteq E$ of $a$, we can find $g_1,\dots,g_k\in C^\infty(X)$, $h_1,\dots,h_m\in C^\infty_{(1)}(E)$, and a function $G\in C^\infty(\bR^k\times \bR^m$ so that
    $$
    f|_{U\cap A}=G(\pi^*g_1,\dots,\pi^*g_k, h_1,\dots,h_m)|_{U\cap A}
    $$
    By possibly shrinking $U$, we may assume that $\pi(U)\subseteq X$ is an open subset. In which case, we observe for all $i=1,\dots,k$ and $j=1,\dots, m$,
    $$
    \pi^*g_i|_{A}=(\pi|_A)^*(g_i|_Y)\in (\pi|_A)^*C^\infty(Y)
    $$
    and
    $$
    h_j|_{A}\in C^\infty_{(1)}(A)
    $$
    In particular, $f\in \bra (\pi|_A)^*C^\infty(Y)\cup C^\infty_{(1)}(A)\ket$.
\end{proof}

\begin{egs}\label{eg: singular foliations give pseudobundles}
    Given any singular foliation $\mathscr{F}$ of a manifold $M$, the tangent bundle to the foliation $T\mathscr{F}\substeq TM$ is endowed with the structure of a pseudobundle via the subspace differential structure and the restrictions of the vector bundle operations on $TM$ to $T\mathscr{F}$.
\end{egs}

\begin{egs}\label{eg: stratifications give pseudobundles}
    In a similar vein to Example \ref{eg: singular foliations give pseudobundles}, given a differentiable stratified space $(X,\Sigma)$, the stratified tangent bundle $T\Sigma\subseteq TX$ inherits the structure of a pseudobundle from the Zariski tangent bundle $TX$.
\end{egs}

\begin{defs}[Pseudobundle Morphism]
    Let $\pi_1:E^1\to X_1$ and $\pi_2:E^2\to X_2$ be two pseudobundles. A morphism is a pair of differentiable maps $f:X_1\to X_2$ and $\phi:E^1\to E^2$ such that
    \begin{itemize}
        \item[(1)] the diagram commutes
        $$
        \begin{tikzcd}
        E^1\arrow[r,"\phi"]\arrow[d] & E^2\arrow[d]\\
        X_1\arrow[r,"f"] & X_2
        \end{tikzcd}
        $$
        \item[(2)] For all $x_1\in X_1$, the induced map
        $$
        \phi_x=\phi|_{E^1_x}:E^1_x\to E^2_{f(x_1)}
        $$
        is linear.
    \end{itemize}
\end{defs}

\begin{egs}
    In this language, another way of interpreting Corollary \ref{cor: foliate iff morphism of pseudo-bundles} is that a smooth map $F:M_1\to M_2$ between two foliated spaces $(M_1,\mathscr{F}_1)$ and $(M_2,\mathscr{F}_2)$ is foliate $\iff$ the derivative $TF:TM_1\to TM_2$ restricts to a pseudobundle morphism $TF:T\mathscr{F}_1\to T\mathscr{F}_2$.
\end{egs}

\begin{defs}[Restriction of pseudobundles]
    Let $\pi:E\to X$ be a pseudobundle and $Y\substeq X$ a subset. Define the restricted pseudobundle $\pi_Y:E|_Y\to Y$ by
    $$
    E|_Y:=\pi^{-1}(Y)\quad\quad \pi_Y:=\pi|_{E|_Y}.
    $$
\end{defs}

As one may easily verify, restricting a pseudobundle to a subset returns a pseudobundle once again. 

\begin{defs}[Subtrivial pseudobundle]
    Let $\pi:E\to X$ be a pseudobundle. Say $\pi:E\to X$ is locally subtrivial if for any $x\in X$, we can find a neighbourhood $U\substeq X$ of $x$ and a morphism of pseudobundles
    $$
    \begin{tikzcd}
        E|_U\arrow[rr,"\phi"]\arrow[dr,"\pi_U"] && U\times \bC^N\arrow[dl,"pr_1"]\\
         & U&
    \end{tikzcd}
    $$
    for some $N$ such that $\phi:E_U\to \phi(E_U)$ is an isomorphism of pseudobundles over $U$. Call the pair $(U,\phi)$ a \textbf{local trivialization}. If around any point $x\in X$ we can find local trivializations so that $\phi(E|_U)=U\times \bR^N$, then call $\pi:E\to X$ a \textbf{differentiable vector bundle} over $X$. 
\end{defs}

\begin{egs}
    Recall that for a subcartesian space $X$, any local chart $\phi:U\into \bR^N$ differentiates to a local chart $T\phi:TU\into T\bR^N=\bR^N\times \bR^N$. By construction these charts for $TX$ define a subtrivial structure on $TX$. 
\end{egs}

\begin{prop}
    Let $\pi:E\to X$ be a subtrivial pseudobundle and $\pi|_A:A\to Y$ a subbundle. Then $\pi|_A:A\to Y$ is also subtrivial.
\end{prop}

\begin{proof}
    Let $y\in Y$ and let $(\phi,U)$ be a local trivialization of $\pi:E\to X$ about $y$. Then clearly the composition
    $$
    A|_{U\cap Y}\into E|_U\into U\times \bR^N
    $$
    defines a local trivialization 
    $$
    \begin{tikzcd}
        A|_{U\cap Y}\arrow[rr,"\phi"]\arrow[dr,"\pi|_A"] && (U\cap Y)\times \bR^N\arrow[dl,"pr_1"]\\
        & U\cap Y &
    \end{tikzcd}
    $$
\end{proof}

\begin{cor}
    All pseudobundles in Examples \ref{eg: Zariski tangent bundles are pseudobundles}, \ref{eg: singular foliations give pseudobundles}, and \ref{eg: stratifications give pseudobundles} are subtrivial.
\end{cor}

As a final elementary definition, we define sections of pseudobundles.

\begin{defs}[Sections of Pseudobundle]
    Let $\pi:E\to X$ be a pseudobundle. A \textbf{section} is a differentiable map $\sigma:E\to X$ such that $\pi\circ \sigma=\id_X$. Write $\Gamma(E)$ for the space of all sections.
\end{defs}

\begin{egs}
    Let $X$ be a subcartesian space, then by Proposition \ref{prop: global to infinitesimal for zariski vector fields} the set of Zariski vector fields $\mfX(X)$ is equal to the sections $\Gamma(TX)$ of the Zariski tangent bundle. 
\end{egs}

Note that due to Example \ref{eg: vector fields need not span the zariski tangent bundle}, we see that in general, given a pseudobundle $\pi:E\to X$, the map
    $$
    X\times \Gamma(E)\to E;\quad (x,\sigma)\mapsto \sigma(x)
    $$
need not be surjective. For this reason, we give one final definition.

\begin{defs}[Enough Sections]
    Let $\pi:E\to X$ be a pseudobundle. We say $\pi:E\to X$ \textbf{has enough sections} if the map
    $$
    X\times \Gamma(E)\to E;\quad (x,\sigma)\mapsto \sigma(x)
    $$
    is surjective.
\end{defs}

\begin{egs}
    Clearly differentiable vector bundles and tangent bundles to singular foliations have enough sections. Similarly, if $(X,\Sigma)$ is a locally trivial stratified space, then by Proposition \ref{prop: locally trivial strat implies enough sections}, $T\Sigma\to X$ has enough sections as well.
\end{egs}

\section{Quotients of Equivariant Vector Bundles}\label{sec: pseudo quotient}

Now we have the setup to return back to our discussion on equivariant vector bundles from Section \ref{sec: equivariant vector bundles}. For all that follows, let $G$ be a connected Lie group and $\pi:E\to M$ a proper equivariant vector bundle. In Definition \ref{def: special subbundle}, we defined a distinguished pseudobundle inside $E$ given by
    $$
    \widetilde{E}:=\bigcup_{x\in M}E_x^{G_x}
    $$
    As we saw in Theorem \ref{thm: special subbundle regular case}, if $M$ has only one orbit-type, then $\widetilde{E}$ is a $G$-equivariant vector bundle over $M$. In general, $\widetilde{E}$ will only be a pseudobundle. 
    
    \begin{egs}
    Indeed, consider the $S^1$-equivariant vector bundle
    $$
    pr_1:\bR^2\times \bR^2\to \bR^2,
    $$
    where each copy of $\bR^2$ carries with it the natural $S^1$ action by rotations about the origin. Since the origin $0$ is a fixed point of the $S^1$ action, we have
    \begin{align*}
    \widetilde{\bR^2\times \bR^2}&=\bigcup_{x\in \bR^2} \{x\}\times (\bR^2)^{S^1_x}\\
    &=(\{0\}\times (\bR^2)^{S^1})\bigsqcup_{x\in \bR^2\setminus\{0\}}\{x\}\times \bR^2\\
    &=(\{0\}\times \{0\})\sqcup ((\bR^2\setminus\{0\})\times \bR^2)
    \end{align*}
    Thus, the rank of $\widetilde{\bR^2\times \bR^2}$ is $0$ over the origin and $2$ elsewhere. Hence, the restriction
    $$
    \widetilde{\bR^2\times \bR^2}\to \bR^2
    $$
    is not a vector bundle.
    \end{egs}
    
    Nonetheless, we saw in Proposition \ref{prop: equivariant sections give all sections below} that $\widetilde{E}$ is spanned by the equivariant sections. That is,
    $$
    M\times \Gamma(E)^G\to \widetilde{E};\quad (x,\sigma)\mapsto \sigma(x)
    $$
    is surjective and hence $\widetilde{E}$ has enough sections. Furthermore, quotient $\widetilde{E}/G\to M/G$ is a vector bundle and there is a canonical surjection
    $$
    \Gamma(E)^G\to \Gamma(\widetilde{E}/G).
    $$

    We would now like to generalize these results to the singular case. For that, we need the following definition.

\begin{defs}[Equivariantly Generated]
    Let $G$ be a connected Lie group, $\pi:E\to M$ a proper equivariant vector bundle, and $A\substeq M$ a subset closed under the $G$-action. Say $E$ is equivariantly generated over $A$, if the map
    $$
    A\times \Gamma(E|_A)^G\to E|_A;\quad (a,\sigma)\mapsto \sigma(a)
    $$
    is surjective.
\end{defs}

Suppose now that $A\substeq M$ is a sliceable subset in the sense of Definition \ref{def: sliceable subset}. Writing $\rho:M\to M/G$ for the quotient map, recall from Proposition \ref{prop: subcartesian on quotient of closed subset} that 
$$
C^\infty(A/G)=\{f:A/G\to \bR \ | \ (\rho|_A)^*f\in C^\infty(M)^G|_A\}
$$
is a subcartesian structure on $A/G$. With this, we can now prove the following.

\begin{theorem}\label{thm: equivariantly generated implies quotient is vector bundle}
    Let $G$ be a connected Lie group, $\pi:E\to M$ a proper equivariant vector bundle, and $A\substeq M$ a sliceable subset. Suppose $E$ is equivariantly generated over $A$. Then the following holds. 
    \begin{itemize}
        \item[(1)] The induced map $\overline{\pi}:(E|_A)/G\to A/G$ is a differentiable vector bundle over $A/G$.
        \item[(2)] Write $\rho_M:M\to M/G$ and $\rho_E:E\to E/G$ for the quotient maps. Then there is an induced linear isomorphism
        $$
        (\rho_E)_*:\Gamma(E|_A)^G\to \Gamma((E|_A)/G).
        $$
    \end{itemize}
\end{theorem}

\begin{proof}
\begin{itemize}
    \item[(1)] First observe that since $A$ is closed and sliceable, so is $E|_A$. Indeed, around any point $a\in A$ we can find a neighbourhood of $a$ so that $\pi:E\to M$ is locally $G$-equivariantly isomorphic to the vector bundle
    $$
    \pi:G\times_K(V\times W)\to G\times_KV;\quad [g,(v,w)]\mapsto [g,v],
    $$
    where $K\leq G$ is a compact subgroup and $V,W$ are linear $K$-representations and so that $A$ is locally diffeomorphic to $G\times_K S$, where $S\substeq V$ is a closed $G$-invariant subspace. Then, under this local isomorphism, $E|_A$ becomes identified with
    \begin{equation}\label{eq: slicing restriction}
    E|_A\cong G\times_K(S\times W)
    \end{equation}
    and hence $E|_A$ is sliceable. 

    \ 

    Now, let $A_0:=A/G$ and $E_0:=(E|_A)/G$. Since both $A$ and $E|_A$ are sliceable, we get canonical subcartesian structures $C^\infty(A_0)$ and $C^\infty(E_0)$ given by
    \begin{align*}
        C^\infty(A_0)&=\{f:A_0\to \bR \ | \ (\rho_M|_A)^*f\in C^\infty(M)|_A^G\}\\
        C^\infty(E_0)&=\{f:E_0\to \bR \ | \ (\rho_E|_{E|_A})^*f\in C^\infty(E)|_{E|_A}^G\}\\
    \end{align*}
    Clearly the induced map $\overline{\pi}:E_0\to A_0$ is differentiable and a topological projection. Now to endow the fibres of $\overline{\pi}:E_0\to A_0$ with vector space structures. Fix $a\in A$, and consider the map
    \begin{equation}\label{eg: fibre bijection}
    E_a\to (E_0)_{[a]};\quad e\mapsto \rho_E(e).
    \end{equation}
    I claim $E_a$ is a bijection. Clearly the map in Equation (\ref{eg: fibre bijection}) is surjective so we only have to show injective. Suppose now that $e,e'\in E_a$ so that $\rho_E(e)=\rho_E(e')$. By definition, this implies there exists $g\in G$ so that $g\cdot e=e'$. By equivariance of the map $\pi:E\to M$, we obtain $g\cdot a=a$ and hence $g\in G_a$. Since $E$ is equivariantly generated over $A$, we have by Theorem \ref{thm: special subbundle regular case} that $E_a^{G_a}=E_a$. Thus, $g\cdot e=e$ and hence $e=e'$. 

    \ 

    Thus, let us endow $(E_0)_{[a]}$ with the vector spaces structure from $E_a$. Making use of the normal form in Equation (\ref{eq: slicing restriction}), we see that around every point $a\in A$, we have a $G$-equivariant isomorphism of differentiable vector bundles
    $$
    \begin{tikzcd}
    E|_{U\cap A}\arrow[r]\arrow[d] & G\times_K (S\times W)\arrow[d]\\
    U\cap A\arrow[r] & G\times_K S
    \end{tikzcd}
    $$
    As we already saw, since $E$ is equivariantly generated over $A$, $E_a^{G_a}=E_a$ for all $a\in A$. Thus, since $W$ corresponds to the fibre over $[e,0]$ we must have $W^K=W$ and thus $G\times_K(S\times W)=(G\times_K S)\times W$. Hence, we obtain diffeomorphisms of subcartesian spaces
    $$
    \begin{tikzcd}
    E_0|_{U\cap A/G}\arrow[r]\arrow[d] & (G/K)\times W\arrow[d]\\
    (U\cap A)/G\arrow[r] & G/K
    \end{tikzcd}
    $$
    This at once proves that the scalar multiplication and addition on $E_0$ are differentiable, as well as providing charts showing that $E_0$ is a vector bundle.

    \item[(2)] The proof from here is pretty well identical to the proof of Proposition \ref{prop: equivariant sections give all sections below}.

    \end{itemize}
\end{proof}

\section{Stratified pseudobundles}\label{sec: strat pseud}

Given that many of the singular spaces that have been of interest to us have come equipped with natural stratifications, it is only natural to ask when is the pseudobundle compatible with this stratification. This is where the idea of a stratified pseudobundle arises.

\begin{defs}[Stratified Pseudobundle]\label{def: strat pseud}
    Let $(X,\Sigma)$ be a differentiable stratified space and $\bK\in \{\bR,\bC\}$. A stratified $\bK$-pseudobundle over $(X,\Sigma)$ is a $\bK$-pseudobundle $\pi:E\to X$ such that the induced partition
    $$
    \pi^{-1}\Sigma=\{\pi^{-1}(S) \ | \ S\in \Sigma\}
    $$
    satisfies the following.
    \begin{itemize}
        \item[(1)] $(E,\pi^{-1}\Sigma,C^\infty(E))$ is a weakly differentiable stratified space in the sense of Definition \ref{def: diff stratified space}.
        \item[(2)] For each $S\in \Sigma$, the map
        $$
        \pi_S:=\pi|_{\pi^{-1}(S)}:\pi^{-1}(S)\to S
        $$
        is a smooth $\bK$-vector bundle.
    \end{itemize}
\end{defs}

\begin{egs}
    Let $(X,\Sigma)$ be a differentiable stratified space and $\pi:E\to X$ a differentiable vector bundle of rank $k$. Then for each stratum $S\in \Sigma$, the restriction $E|_S\to S$ is clearly a smooth vector bundle as locally $E|_S$ is diffeomorphic to $U\times \bR^k$ which provides smooth charts. This also shows that $\pi^{-1}\Sigma$ is a stratification as around any point $x\in X$, we have a trivialization $E|_U\cong U\times \bR^k$. In this trivialization, the pieces of $\pi^{-1}\Sigma$ have the form $\{S\times \bR^k\}$ which satisfies the axiom of the frontier.
\end{egs}

\begin{egs}
    Let $X$ be the union of coordinate axes in $\bR^2$,
    i.e.
    $$
    X=\{(x,y)\in\bR^2 \ | \ xy=0\}.
    $$
    Endow $X$ with the stratification
    $$
    \Sigma=\{\{(0,0)\}, H_+, H_-, V_+, V_-\}
    $$
    from Example \ref{eg: coordinate axes}
    \begin{figure}[h]
        \centering
        \begin{tikzpicture}[scale=2]
        \draw[-,very thick,myblue] (-1,0) node[left] {$H_-$} -- (1,0) node[right] {$H_+$} ;
        \draw[-,very thick,myred] (0,-1) node[right] {$V_-$}-- (0,1) node[right] {$V_+$};
        \filldraw (0,0) circle(0.7pt) node[anchor=south west] {$\{(0,0)\}$};
        \end{tikzpicture}
    \end{figure}
    Explicitly,
    \begin{align*}
        H_+&=(0,\infty)\times \{0\}\\
        H_-&=(-\infty,0)\times \{0\}\\
        V_+&=\{0\}\times(0,\infty)\\
        V_-&=\{0\}\times (-\infty,0)\\
    \end{align*}
    Then the stratified tangent bundle $T\Sigma$ has a natural partition into smooth vector bundles
    $$
    T\Sigma=T\{(0,0)\}\sqcup TH_+\sqcup TH_-\sqcup TV_+\sqcup TV_-.
    $$
    It's easy to see that $T\Sigma$ satisfies the axioms of a being a stratified pseudobundle except, perhaps, the frontier condition. For that observe that as subsets of $T\bR^2$,
    \begin{align*}
        \overline{T\{(0,0)\}}&=\{(0,0,0,0)\}\\
        \overline{TH_+}&=([0,\infty)\times\{0\})\times (\bR\times \{0\})\\
        \overline{TH_-}&=((-\infty,0]\times\{0\})\times (\bR\times \{0\})\\
        \overline{TV_+}&=(\{0\}\times [0,\infty))\times \{0\}\times\bR)\\
        \overline{TV_-}&=(\{0\}\times (-\infty,0])\times \{0\}\times\bR)\\
    \end{align*}
    And so, we can manually check that the partition $\pi^{-1}\Sigma$ satisfies the frontier condition as the only stratum that intersects the closure of any other stratum in $\{(0,0,0,0)\}$ and clearly it is contained in the closures of the other strata.
\end{egs}

More generally, suppose $(X,\Sigma)$ is a differentiable stratified space. Due to Example \ref{eg: stratifications give pseudobundles}, the stratified tangent bundle $\pi:T\Sigma\to X$ is a pseudobundle. Furthermore, the partition
$$
\pi^{-1}\Sigma=\{TS \ | \ S\in \Sigma\}
$$
is clearly a partition of $T\Sigma$ into smooth vector bundles over each of the strata of $X$. So the only possible way $T\Sigma$ could fail to be a stratified pseudobundle is if the partition $\pi^{-1}\Sigma$ failed to be a stratification. In this case, it's straightforward to see that on that front the only condition that could fail would be the frontier condition. 

\begin{theorem}[{\cite{pflaum_analytic_2001}}]
    Let $(X,\Sigma)$ be a differentiable stratified space and let $\pi:T\Sigma\to X$ be the stratified tangent bundle. Then the partition 
    $$
    \pi^{-1}\Sigma=\{TS \ | \ S\in \Sigma\}
    $$
    makes $(T\Sigma,\pi^{-1}\Sigma,C^\infty(T\Sigma))$ into a weakly differentiable stratified space $\iff$ $(X,\Sigma)$ is Whitney A.
\end{theorem}

As we discussed in Subsection \ref{sub: Whitney}, examples of non Whitney A stratified spaces are rather pathological and will not make an appearance in this thesis. Thus, all stratified tangent bundles henceforth will be stratified spaces. Thus, the following result will hold.

\begin{cor}
    Let $(X,\Sigma)$ be a differentiable stratified space. Then $\pi:T\Sigma\to (X,\Sigma)$ is a stratified pseudobundle $\iff$ $(X,\Sigma)$ is Whitney (A).
\end{cor}

Whitney condition (A) can very naturally be extended to more general stratified pseudobundles. 

\begin{defs}[Linear Whitney A]\label{def: linear whitney}
    Let $(X,\Sigma)$ be a differentiable stratified space and $\pi:E\to X$ a pseduobundle such that for all $S\in \Sigma$, the restriction $\pi_S:E|_S\to S$ is a smooth vector bundle. Say $\pi:E\to X$ is \textbf{Linear Whitney A} if for any two strata $S,R\in \Sigma$, $x\in S$, and local trivialization
    $$
    \begin{tikzcd}
        E|_U\arrow[rr,"\phi"]\arrow[dr,"\pi_U"]&& U\times \bR^N\arrow[dl,"pr_1"]\\
        & U&
    \end{tikzcd}
    $$
    the following condition holds. Write $r$ for the rank of $E|_R$ and $s$ for the rank of $E|_S$.
    \begin{itemize}
        \item[(A)] If $\{x_k\}\substeq R$ is a sequence such that $\{\phi(E_{x_k})\}$ converges to $\tau\substeq \{x\}\times \bR^N$ in $\text{Gr}_r(\bR^N)$, then $E_x\substeq \tau$.
    \end{itemize}
\end{defs}

This definition is a straight replacement of the stratified tangent bundle with a more general pseudobundle. Thus, it is no surprise that a differentiable stratified space is Whitney (A) $\iff$ its stratified tangent bundle is Linear Whitney (A). Also as should be expected, Pflaum's Theorem can be easily ported to this more general context.

\begin{prop}
    Let $\pi:E\to (X,\Sigma)$ be a Linear Whitney A pseudobundle over $(X,\Sigma)$. Then $\pi:E\to (X,\Sigma)$ is a stratified pseudobundle.
\end{prop}

\begin{proof}
    The only thing that could possibly go wrong here is that the induced partition $(E,\pi^{-1}(\Sigma))$ from Definition \ref{def: strat pseud} is not a stratification. On that front, it's easy to see that the only failure could be the frontier condition. And so to show that $\pi^{-1}\Sigma$ satisifies the frontier condition, let $S,R\in \Sigma$ be two strata with $E|_S\cap \overline{E|_R}\neq \emptyset$. We show $E|_S\subseteq \overline{E|_R}$. To do so, fix $e_0\in E|_S$.
    
    \ 
    
    By continuity of the bundle map $\pi:E\to X$, we obtain that $S\cap \overline{R}\neq \emptyset$ and hence $S\substeq \overline{R}$. Thus, if we write $x_0=\pi(e_0)$, then there exists a sequence $\{x_k\}\substeq \overline{R}$ such that 
    $$
    \lim_{k\to \infty} x_k=x_0.
    $$
    Passing now to a local trivialization, centred on $x$
    $$
    \phi:E|_U\to U\times \bR^N
    $$
    consider the sequence $\{\phi(E_{x_k})\}\substeq \text{Gr}_r(\bR^N)$ where $r$ is the rank of the smooth vector bundle $E|_R\to R$. As we saw in the discussion at the beginning of Section \ref{sub: Whitney}, $\text{Gr}_r(\bR^N)$ is compact and hence after passing to a subsequence, we may assume that $\{\phi(E_{x_k})\}$ converges to some $\tau\in \text{Gr}_r(\bR^N)$. By Whitney (A), we have $E_{x_0}\substeq \tau$ and thus $e_0\in \tau$.

    \ 

    Recall that the topology on $\text{Gr}_k(\bR^N)$ is the subspace topology from $\text{End}(V)$. Using the canonical inner-product on $\bR^N$, the topology on $\text{End}(V)$ is generated by the operator norm
    $$
    \|T\|:=\text{inf}_{v\in \bR^N\setminus\{0\}}\frac{\|T(v)\|}{\|v\|}.
    $$
    In particular, 
    $$
    \lim_{k}E_{x_k}=\tau
    $$
    if and only if
    \begin{equation}\label{eq: convergence in terms of operator norm}
    \lim_k\|P_{E_{x_k}}-P_\tau\|=0
    \end{equation}
    where for any subspace $W\substeq \bR^N$, we write $P_W$ for the orthogonal projection onto $W$. In particular, observe that $\{P_{E_{x_k}}(e_0)\}\substeq E|_R$ is a sequence on $E|_R$ and by Equation (\ref{eq: convergence in terms of operator norm}) we have
    $$
    0=\lim_{k}\|P_{E_{x_k}}(e_0)-P_\tau(e_0)\|=\lim_{k}\|P_{E_{x_k}}(e_0)-e_0\|
    $$
    since $e\in \tau$. Thus, $\{P_{E_{x_k}}(e_0)\}$ converges to $e_0$ and thus $e_0\in \overline{E|_S}$. Therefore, $\pi^{-1}\Sigma$ satisfies the frontier condition.
\end{proof}

As we discussed in Subsection \ref{sub: Whitney}, the Whitney conditions may seem somewhat unnatural for a smooth geometer. However, thankfully, the same foliation-theoretic upgrade that we discussed for the usual Whitney A still holds for Linear Whitney A.

\begin{theorem}
    Let $(X,\Sigma)$ be a differentiable stratified space and $\pi:E\to X$ a pseudobundle. Suppose the following two conditions hold.
    \begin{itemize}
        \item[(1)] For each stratum $S\in \Sigma$, the restriction $\pi_S:E|_S\to S$ is a smooth vector bundle.
        \item[(2)] $\pi:E\to X$ has enough sections, i.e the map
        $$
        X\times\Gamma(E)\to E;\quad (x,\sigma)\mapsto \sigma(x)
        $$
        is surjective.
    \end{itemize}
    Then $\pi:E\to X$ is Linear Whitney A.
\end{theorem}

\begin{proof}
    Near identical proof to Proposition \ref{prop: enough sections implies Whitney A}.
\end{proof}

\begin{egs}\label{eg: special subbundle is a stratified subbundle}
    Let $G$ be a connected Lie group and $\pi:E\to M$ a proper equivariant vector bundle. Recall that 
    $$
    \widetilde{E}=\bigcup_{x\in M}E_x^{G_x}
    $$
    has enough sections. Thus, if we show $\widetilde{E}|_S\to S$ is a smooth vector bundle for all orbit-type strata $S\in \mathcal{S}_G(M)$, then we will have shown that $\widetilde{E}\to (M,\mathcal{S}_G(M))$ is a Whitney (A) stratified pseudobundle. Indeed fixing $x\in M$, we may pass to the local normal form of proper vector bundles and assume that $E=G\times_K(V\times W)$, $M=G\times_K V$, and $\pi:E\to M$ is the map $\pi([g,v,w])=[g,v]$ for compact $K\leq G$ and linear $K$-representations $V$ and $W$. In this case, the stratum $S$ containing $x$ has the local form
    $$
    S=G\times_K V^K\cong (G/K)\times V^K.
    $$
    Note that since stabilizers of points in $S$ are conjugate to $K$, it follows that
    $$
    \widetilde{E}|_S=G\times_K(V^K\times W^K)=(G/K)\times V^K\times W^K.
    $$
    This provides a local trivialization for $\widetilde{E}|_S$ and hence $\widetilde{E}|_S\to S$ is a smooth vector bundle.
\end{egs}

\begin{egs}
    Continuing on with the last example, let $G$ be a connected Lie group, $K\leq G$ a compact subgroup, and $V$ a $K$-representation. Write $M=G\times_K V$. As we saw in Example \ref{eg: normal form for tangent bundle of proper space},
    $$
    TM=G\times_K(V\times(\mfg/\mfk)\times V).
    $$
    Now observe that if we let $S=G\times_K V^K$, then on one hand by Proposition \ref{prop: fixed point set of quotient of normal lie algebra},
    $$
    \widetilde{TM}|_S=G\times_K(V^K\times(\mfg/\mfk)^K\times V^K)\cong (G/K)\times V^K\times (\mfn_G(K)/\mfk)\times V^K
    $$
    while
    $$
    TS\cong (G/K)\times V^K\times (\mfg/\mfk)\times V^K.
    $$
    Thus, $\widetilde{TM}|_S\substeq TS$. From here, we easily deduce that for any proper $G$-space $M$ that $\widetilde{TM}\substeq T\mathcal{S}_G(M)$, where $\mathcal{S}_G(M)$ is the orbit-type stratification, with equality if and only if all the subgroups appearing as stabilizers are normal. In particular, we have shown the following.
\end{egs}

\begin{cor}\label{cor: abelian implies special is orbit-type}
    Let $G$ be a connected abelian Lie group and $M$ a proper $G$-space. Then $\widetilde{TM}=T\mathcal{S}_G(M)$.
\end{cor}

We have already shown that many of the pseudobundles we've encountered so far have enough sections. So, provided the pseudobundle is compatible with a stratification on the base space, the result will be a stratified pseudobundle.

\begin{cor}\label{cor: sing fols can be stratified pseuds}
    Suppose $(M,\Sigma)$ is a stratified space and $\mathscr{F}$ is a singular foliation on $M$. For each $x\in M$, write $L_x\in \mathscr{F}$ for the leaf through $x$. Further, suppose each stratum $S\in \Sigma$ satisfies the following two conditions
    \begin{itemize}
        \item[(1)] $S=\bigcup_{x\in S}L_x$
        \item[(2)] The induced foliation $\mathscr{F}_S$ on $S$ is regular.
    \end{itemize}
    Then the tangent bundle $T\mathscr{F}\to M$ is a stratified pseudobundle over $(M,\Sigma)$.
\end{cor}

This opens the door to many more examples of stratified pseudobundles.

\begin{egs}
    Let $(M,g,\mathscr{F})$ be a singular Riemannian foliation. Then the dimension-type stratification $\mathcal{S}_{dim}(M,\mathscr{F})$ satisfies the conditions of Corollary \ref{cor: sing fols can be stratified pseuds}, and hence gives a stratified pseudobundle.
\end{egs}

\begin{egs}
    Let $G$ be a connected Lie group and $M$ a proper $G$-space. Then the orbit foliation $\mathscr{F}_G=\{G\cdot x \ | \ x\in M\}$ saturates the orbit-type stratification $\mathcal{S}_G(M)$. That is any orbit-type stratum $S\in \mathcal{S}_G(M)$ is a union of orbits which all have the same dimension, hence by Frobenius defines a regular foliation.
\end{egs}

\section{Complexification and Stratified Distributions}\label{sec: strat dist}

In our discussion so far, pseudobundles were allowed to have base field either $\bR$ or $\bC$. In this section, we will show how to produce a complex pseudobundle from a real one. This will then allow us to discuss complex distributions over subcartesian spaces, an important idea as we move towards stratified polarizations.

\begin{defs}[Complexification]
    Let $\pi:E\to X$ be a real pseudobundle. Define the complexification, denoted $\pi_\bC:E\otimes\bC\to X$ to be the disjoint union
    $$
    E\otimes\bC:=\bigcup_{x\in X} E_x\otimes_\bR \bC,
    $$
    with $\pi_\bC:E\otimes\bC\to X$ the natural projection.
\end{defs}

We wish to endow $\pi_\bC:E\otimes\bC\to X$ with the structure of a complex pseudobundle. To do this, we first need to give $E\otimes\bC$ a subcartesian structure. For this note that for all $x\in X$ we have
$$
E_x\otimes \bC=E_x\oplus iE_x
$$
where $i\in \bC$ is the square root of $-1$. Hence, we may view $E\otimes\bC$ as $E\times_\pi E\substeq E\times E$. Under this identification, $\pi_\bC$ may be viewed as
$$
\pi_\bC:E\otimes \bC\to X;\quad (e_1,e_2)\mapsto \pi(e_1).
$$
Recall from Example \ref{eg: product differential structure} that if we write $pr_1,pr_2:E\times E\to E$ for the projections onto the first and second factor, respectively, then
$$
C^\infty(E\times E)=\bra pr_1^*C^\infty(E)\cup pr_2^*C^\infty(E)\ket.
$$
By definition of $\pi:E\to X$ being a pseudobundle, we have
$$
C^\infty(E)=\bra \pi^*C^\infty(X)\cup C^\infty_{(1)}(E)\ket.
$$
Hence, using the fact that $E\times_\pi E\substeq E\times E$ is a closed subset and that $\pi\circ pr_1|_{E\times_\pi E}=\pi\circ pr_2|_{E\times_\pi E}$, we deduce 
\begin{equation}\label{eq: subcartesian structure on complexification}
C^\infty(E\otimes \bC)=\bra \pi_\bC^*C^\infty(X)\cup pr_1^*C^\infty_{(1)}(E)\cup pr_2^*C^\infty_{(1)}(E)\ket
\end{equation}
Next, define complex scalar multiplication by 
\begin{equation}\label{eq: complex scalar multiplication}
    \mu^\bC:\bC\times E_\bC\to E_\bC;\quad (x+iy,e_1+ie_2)\mapsto (xe_1-ye_2+i(xe_2+ye_1))
\end{equation}
This is a combination of addition and scalar multiplication on $E$ and hence is differentiable. Finally, define addition
\begin{equation}\label{eq: addition on complexification}
(E\otimes \bC)\times_{\pi_\bC} (E\otimes \bC)\to E\otimes \bC;\quad (e_1+ie_2, u_1+iu_2)\mapsto e_1+u_1+i(e_2+u_2).
\end{equation}

\begin{theorem}\label{thm: complexification of pseudobundles}
    Let $\pi:E\to X$ be a real pseudobundle. Then $\pi_\bC:E\otimes \bC\to X$ together with the subcartesian structure in Equation (\ref{eq: subcartesian structure on complexification}), complex scalar multiplication in Equation (\ref{eq: complex scalar multiplication}), and addition in Equation (\ref{eq: addition on complexification}) is a complex pseudobundle over $X$. Furthermore, if $\pi:E\to X$ satisfies any of the following properties, so does $\pi_\bC:E\otimes \bC\to X$.
    \begin{itemize}
        \item Subtrivial
        \item Vector bundle
        \item Enough sections
        \item Linear Whitney (A)
    \end{itemize}
\end{theorem}

\begin{proof}
    We've already done most of the proof already, we just need to check that 
    $$
    C^\infty(E\otimes \bC)=\bra \pi_\bC^*C^\infty(X)\cup C^\infty_{(1)}(E\otimes \bC)\ket,
    $$
    where $C^\infty_{(1)}(E\otimes \bC)$ are linear with respect to the induced real scalar multiplication. Clearly,
    $$
    \bra \pi_\bC^*C^\infty(X)\cup C^\infty_{(1)}(E\otimes \bC)\ket\substeq C^\infty(E\otimes \bC)
    $$
    so we only need to show the reverse containment. This is a triviality since by the definition of $C^\infty(E\otimes \bC)$ in Equation (\ref{eq: subcartesian structure on complexification}), clearly $pr_1^*C^\infty_{(1)}(E),pr_2^*C^\infty(E)\substeq C^\infty_{(1)}(E\otimes \bC)$. 

    \ 

    The other conditions follows trivially. 
\end{proof}

A consequence of the fact that $E\otimes \bC$ is identified with a subset of $E\times E$ is the following.

\begin{cor}\label{cor: direct sum of complex sections}
    Let $\pi:E\to X$ be a pseudobundle. Then any section $\sigma\in\Gamma(E\otimes \bC)$ can uniquely be written in the form
    $$
    \sigma=\sigma_1+i\sigma_2.
    $$
    That is, as a real vector space, $\Gamma(E\otimes \bC)=\Gamma(E)\otimes \bC$.
\end{cor}

Our main examples thus far of pseudobundles have been Zariski and stratified tangent bundles, thus we can now apply Theorem \ref{thm: complexification of pseudobundles} to complexify them.

\begin{defs}[Complex Zariski/Stratified Tangent Bundle]
    \begin{itemize}
        \item[(1)] Let $X$ be a subcartesian space. Call $TX\otimes \bC$ the \textbf{complexified Zariski tangent bundle} of $X$. Write $\mfX_\bC(X)$ for sections of $TX\otimes \bC$ and call the elements of $\mfX_\bC(X)$ \textbf{complex Zariski vector fields}.
        \item[(2)] Let $(X,\Sigma)$ be Whitney regular. Call  $T\Sigma\otimes \bC$ the \textbf{complexified stratified tangent bundle} of $(X,\Sigma)$. Write $\mfX_\bC(\Sigma)$ for sections of $T\Sigma\otimes \bC$ and call the elements of $\mfX_\bC(\Sigma)$ complex stratified vector fields.
    \end{itemize}
\end{defs}

Let $X$ be a subcartesian space. Recall that the elements of the Zariski tangent bundle $TX$ have an interpretation as derivations of $C^\infty(X)$. We can give an analagous interpretation of $TX\otimes \bC$ in terms of differentiable complex-valued functions $C^\infty(X,\bC)$.

\begin{prop}\label{prop: complexified zariski is complex derivations}
    Let $X$ be a subcartesian space and $x\in X$. Then we have a canonical identification of $T_x\otimes \bC$ with the space of derivations $v:C^\infty(X,\bC)\to \bC$ such that
    \begin{equation}\label{eq: complex derivation}
    v(fg)=v(f)g(x)+f(x)v(g).
    \end{equation}
\end{prop}

\begin{proof}
    First, given $v\in T_xX\otimes \bC$, define derivation
    $$
    v:C^\infty(X,\bC)\to \bC
    $$
    as follows. Using Corollary \ref{cor: direct sum of complex sections}, write $v=v_1+iv_2$ for $v_1,v_2\in \mfX(X)$. Then, given a function $f\in C^\infty(X,\bC)$ we can also write $f=f_1+if_2$ for $f_1,f_2\in C^\infty(X)$. Now define
    $$
    v(f):=v_1(f_1)-v_2(f_2)+i(v_1(f_2)+v_2(f_1))
    $$
    It is then a straightforward matter to check that $v$ satisfies Equation (\ref{eq: complex derivation}). For the converse, given a complex derivation $v:C^\infty(X,\bC)\to \bC$ satisfying Equation (\ref{eq: complex derivation}), one easily verifies that $v_1,v_2:C^\infty(X)\to \bR$ defined by
    $$
    v_1(f)=\frac{v(f)+\overline{v(f)}}{2}\quad\quad v_2(f)=\frac{v(f)-\overline{v(f)}}{2i}
    $$
    for $f\in C^\infty(X)$ define elements of $T_xX$ which satisfy $v=v_1+iv_2$.
\end{proof}

This then allows us to conclude the following.

\begin{cor}\label{cor: complex zariski vector fields are the complex derivations}
    Let $X$ be a subcartesian space. Then the complex Zariski vector fields $\mfX_\bC(X)$ can canonically be identified with the set of complex derivations of $C^\infty(X,\bC)$.
\end{cor}

\begin{proof}
    Apply Proposition \ref{prop: complexified zariski is complex derivations} fibre-wise.
\end{proof}

Let us now move to the stratified setting. Given a Whitney regular stratified space, the complexified stratified tangent bundle $T\Sigma\otimes \bC$ has a natural stratification, namely
$$
T\Sigma\otimes \bC=\bigcup_{S\in \Sigma}TS\otimes \bC.
$$
In the next chapters, we will be studying generalizations of polarizations from Chapter \ref{ch: polarizations} which will be complex subbundles of stratified tangent bundles. To set this up, let us define distributions.

\begin{defs}[Stratified Complex Distribution]
    Let $(X,\Sigma)$ be a Whitney regular stratified space. A \textbf{stratified complex distribution} is a subbundle $D\substeq TX\otimes \bC$ over $X$ such that $D\to (X,\Sigma)$ is a complex stratified pseudobundle. That is, $D\substeq T\Sigma\otimes\bC$ is a pseudobundle, for each stratum $S\in \Sigma$ $D|_S\subseteq TS\otimes\bC$ is a smooth complex subbundle, and the partition
    $$
    D=\bigcup_{S\in\Sigma}D|_S
    $$
    is a weakly differentiable stratification of $D$.
\end{defs}

\begin{egs}
    Let $(X,\Sigma)$ be a Whitney stratified space and $E\substeq T\Sigma$ a subbundle such that $E\to (X,\Sigma)$ is a stratified pseudobundle. Then clearly the complexification $E\otimes \bC\substeq T\Sigma\otimes \bC$ is a  stratified complex distribution.  We call distributions of this form \textbf{real} distributions.
\end{egs}

\begin{egs}\label{eg: stratified complex structures}
    Let $G$ be a connected abelian Lie group and $(M,I)$ a complex manifold such that $M$ is a proper $G$-space and $I$ is a $G$-invariant complex structure. Then we get two complex $G$-invariant distributions on $M$, namely the holomorphic and antiholomorphic tangent bundles $T^{1,0}M$ and $T^{0,1}M$. It then follows from Example \ref{eg: special subbundle is a stratified subbundle} that
    $$
    \widetilde{T^{1,0}M}=\bigcup_{x\in M} (T^{1,0}_xM)^{G_x}
    $$
    and
    $$
    \widetilde{T^{0,1}M}=\bigcup_{x\in M} (T^{0,1}_xM)^{G_x}
    $$
    are stratified distributions over $(M,\mathcal{S}_G(M))$. As we saw in Corollary \ref{cor: abelian implies special is orbit-type}, $\widetilde{TM}=T\mathcal{S}_G(M)$. Thus, we have for each orbit-type stratum $S\in \mathcal{S}_G(M)$ 
    $$
    TS\otimes \bC=\widetilde{T^{1,0}M}|_S\oplus \widetilde{T^{0,1}M}|_S.
    $$
    From here, we see that the above decomposition defines complex structures $\{I_S\}$ on each stratum. 
\end{egs}

Note for a Whitney regular stratified space $(X,\Sigma)$, both $\mfX(X)$ and $\mfX(\Sigma)$ have natural Lie bracket structures given by
$$
[V,W]f=V(W(f))-W(V(f))
$$
for $f\in C^\infty(X)$. By Corollary \ref{cor: direct sum of complex sections} we have $\mfX_\bC(X)=\mfX(X)\otimes\bC$ and $\mfX_\bC(\Sigma)=\mfX(\Sigma)\otimes \bC$, so we can extend the Lie bracket complex-linearly. This bracket can also be realized by the identification of $\mfX_\bC(X)$ with complex derivations of $C^\infty(X,\bC)$.

\begin{defs}[Involutive]
    Let $D\substeq T\Sigma\otimes \bC$ be a complex stratified distribution. Say $D$ is involutive if for all strata $S\in \Sigma$, $D|_S\subseteq TS\otimes \bC$ is involutive.
\end{defs}

\begin{egs}
    Returning now to the case of a connected Lie group $G$ and a proper complex $G$-space $(M,I)$ from Example \ref{eg: stratified complex structures}, we see that $\widetilde{T^{1,0}M}$ and $\widetilde{T^{0,1}M}$ are examples of involutive stratified distributions since over each stratum, they are given by the holomorphic and antiholomorphic tangent bundles of induced complex structures.
\end{egs}

Now to get examples of involutive distributions, we will first need a Lemma.

\begin{lem}\label{lem: complex stratified and lie bracket}
    Let $(X,\Sigma)$ be a Whitney regular stratified space and $V,W\in \mfX_\bC(\Sigma)$. Then for every stratum $S\in\Sigma$ we have
    $$
    [V|_S,W|_S]=[V,W]|_S\in \mfX_\bC(S).
    $$
\end{lem}

\begin{proof}
    Follows immediately from the fact that the complex Lie bracket is simply the complex-linear extension of the usual Lie bracket and from Corollary \ref{cor: stratified vf and lie bracket}.
\end{proof}

\begin{prop}
    Let $(X,\Sigma)$ be a Whitney stratified space and $D\substeq T\Sigma\otimes \bC$ an involutive complex stratified distribution. Then, for any $V,W\in \Gamma(D)$, we have $[V,W]\in \Gamma(D)$.
\end{prop}

\begin{proof}
    Apply Lemma \ref{lem: complex stratified and lie bracket} stratum-wise.
\end{proof}

Finally, I would like to give a characterization of when a complex distribution on a manifold induces a stratified distribution on a subset.

\begin{lem}\label{lem: a stratified distribution can be reconstructed from the strata}
Let $X\subseteq \bR^N$ be a locally closed subspace equipped with a stratification $\Sigma$ by embedded submanifolds of $\bR^N$ such that $(X,\Sigma)$ is Whitney regular. Furthermore, assume let $D\subseteq T^\bC \bR^N$ be a complex distribution on $\bR^N$ which is also closed as a subspace, and for each $S\in \Sigma$, let
$$
D_S:=D\cap T^\bC S.
$$
Suppose the following properties hold
\begin{itemize}
    \item[(1)] $D_S\subseteq T^\bC S$ is a complex distribution.
    \item[(2)] If $S\subseteq \overline{R}$, then $\text{rank}(D_S)\leq \text{rank}(D_R)$.
\end{itemize}
Then the union
$$
D_X:=\bigcup_{S\in \Sigma}D_S\subseteq T^\bC \Sigma
$$
is a stratified complex distribution on $(X,\Sigma)$.
\end{lem}

\begin{proof}
We only need to show that $D_X$ satisfies the frontier condition. Since this is a local condition, we may assume $X\subseteq \bR^N$ for some $N$. 

\

To show frontier, let $S,R\in \Sigma$ be two strata and suppose $D_S\cap \overline{D_R}\neq \emptyset$. We show $D_S\subseteq D_R$. By continuity, we have $S\cap \overline{R}\neq\emptyset$ and hence $S\subseteq \overline{R}$. 

\

Let $x\in S$ and let $\{x_k\}\subseteq R$ be a sequence converging to $x$. Since Grassmannians are compact, after possibly passing to a subsequence, we can find $W\subseteq T_x \bR^N$ so that $T_{x_k} R$ converges to $W$ in the $\dim(R)$ Grassmannian of $\bR^N$. Making use of the canonical metric on $\bR^N$, we can canonically find orthogonal projections
$$
\Pi_k:T_{x_k}R\to (D_R)_{x_k}
$$
for each $k$. Extend by $0$ to obtain orthogonal projections $\Pi_k:T_{x_k}\bR^N\to (D_R)_{x_k}$. By assumption, $D|_R$ has constant rank along $R$, denoted by $r$. It then follows that the sequence of projections $\{\Pi_k\}$ lies in the Grassmannian of $\bR^N$ of dimension $r$ subspaces. Since this is compact, after potentially passing to a subsequence, we may assume $\{\Pi_k\}$ converges to some orthogonal projection $\Pi_\infty$ onto a rank $r$ subspace $A$ of $T_x\bR^N$. By continuity, it follows that $A\subseteq W$.  

\

Now, observe for any $w\in \bR^N$ that $\{\Pi_k(w)\}$ is a sequence in $D_R$ and that
$$
\Pi_\infty(w)=\lim_{k\to \infty}\Pi_k(w).
$$
Since $D$ is assumed to be closed, it follows that $\Pi_\infty(v)\in D$. Again, using continuity, we also have $\Pi_\infty(v)\in D_x$. By the dimension assumption, it follows that $\Pi_\infty$ can be viewed as a projection onto $W\cap D_x$ and hence contains $(D_S)_x$ since $T_x S\subseteq W$. 

\

In particular, for any $v\in (D_S)_x$ let $v_k=\Pi_k(v)\in (D_R)_x$. Then 
$$
\lim_{k\to \infty}v_k=\Pi_\infty(v)
$$
Since $v$ lies in the image of $\Pi_\infty$ and $\Pi_\infty$ is a projection, it follows that $\Pi_\infty(v)=v$ and hence $v\in \overline{D_R}$. Therefore, $D_S\subseteq \overline{D_R}$.
\end{proof}
\cleardoublepage
\part{Singular Quantization}\label{pt: quant}
\chapter{Symplectic Stratified Spaces}\label{ch: symp strat}

Now that we have laid the foundations for the study of singular spaces (at least from a subcartesian and stratified perspective), we can now begin our study of symplectic stratified spaces. These are spaces which arose in the work of Sjamaar and Lerman \cite{sjamaar_stratified_1991} in their study of singular reduced spaces. Their construction will be examined in great detail in Sections \ref{sec: orbit symp strat}, \ref{sec: reduce mfld symmetry}, and \ref{sec: poisson on reduced space}. Mol \cite{mol_stratification_2024} has refined their ideas greatly into the setting of symplectic groupoids with his ``Hamiltonian stratifications'', which we will not be discussing here. 

\ 

Now, given that is was mentioned in Chapter \ref{ch: subcartesian} that subcartesian spaces are not particularly well-suited to a theory of differential forms, it might seem surprising that we can even discuss symplectic geometry--an area defined by distinguished $2$-forms. The key here is that we are studying stratified spaces, spaces which are partitioned into manifolds where differential form theory makes sense. Furthermore, Poisson structures, being ``covariant'' in nature, are readily definable in a subcartesian setting. So, taking a page out of the theory of symplectic foliations of Poisson structures, we define a symplectic stratified space to be a collection of symplectic forms on each of the strata and a Poisson structure on the total space, such that the inclusion is Poisson. In this way, symplectic stratified spaces are like locally finite symplectic foliations. This idea can be refined even further (and made literal) through Hamiltonian stratifications as studied by Mol \cite{mol_stratification_2024}.

\section{Symplectic Stratified Spaces}\label{sec: symp strat}

\begin{defs}[Stratified Symplectic Space]
Let $(X,\Sigma,C^\infty(X))$ be a Whitney stratified space. A symplectic structure on $X$ consists of the following data.
\begin{itemize}
	\item[(1)] A symplectic form $\omega_S$ on each stratum $S\in\Sigma$.
	\item[(2)] A Poisson bracket $\{\cdot,\cdot\}$ on $C^\infty(X)$.
\end{itemize}
Subject to the compatibility condition that for any stratum $S\in\Sigma$, the inclusion $S\into X$ is Poisson. We will write $(X,\{\cdot,\cdot\})$ for a symplectic stratified space if the stratification and subcartesian structure are understood from context.
\end{defs}

\begin{egs}
Let $(M,\pi)$ be a Poisson manifold and recall from Example \ref{egs: symplectic foliation} that we have a canonical foliation $\mathscr{F}_\pi$ of $M$ by symplectic submanifolds. Due to Proposition \ref{prop: enough sections implies Whitney A}, if $\mathscr{F}_\pi$ is locally finite and has embedded leaves, then $(M,\mathscr{F}_\pi)$ is a Whitney (A) stratified space. If we suppose further that $\mathscr{F}_\pi$ is Whitney regular, then $(M,\mathscr{F}_\pi,\{\cdot,\cdot\})$ is a symplectic stratified space. 
\end{egs}

\begin{egs}\label{eg: first non-trivial symplectic stratified space}
    Consider the Poisson structure
    $$
    \pi=r^2\frac{\partial}{\partial x}\wedge \frac{\partial}{\partial y}
    $$ 
    on $\bR^2$ from Example \ref{eg: not symplectic Poisson}. As we saw, the symplectic foliation $\mathscr{F}_\pi$ has two symplectic leaves, namely the origin $\{0\}$ and its complement $\bR^2\setminus\{0\}$. This is equal to the orbit-type stratification of $\bR^2$ by the canonical $S^1$-action, hence is a Whitney stratification. Therefore, $(\bR^2,\mathscr{F}_\pi,\{\cdot,\cdot\})$ is a symplectic stratified space.
\end{egs}

\begin{egs}\label{eg: stratification by symplectic manifolds not stratsymp}
    A stratification of a Poisson manifold by submanifolds which are also symplectic, need not be a symplectic stratified space. Indeed, write
    $$
    \bR^4=\{(x_1,y_1,x_2,y_2) \ | \ x_i,y_j\in \bR\}
    $$
    and let $\{\cdot,\cdot\}_{can}$ be the canonical Poisson bracket induced by the canonical symplectic form
    $$
    \omega_{can}=dx_1\wedge dy_1+dx_2\wedge dy_2.
    $$
    Furthermore, let 
    $$
    S=\{(x_1,y_1,0,0) \in \bR^4\}
    $$
    be an embedded copy of symplectic $\bR^2$. Indeed, we can see that
    $$
    \omega_S=\omega_0|_S=dx_1\wedge dy_1
    $$
    and hence $(S,\omega_S)$ is symplectic. However, $S\into \bR^4$ is not Poisson. Indeed, consider $f=x_1+x_2$ and $g=y_1+y_2$. Then, we can easily compute that
    $$
    \{f,g\}_{can}=-2
    $$
    whereas
    $$
    \{f|_S,g|_S\}_S=-1,
    $$
    where $\{\cdot,\cdot\}_S$ is the induced Poisson structure on $S$ from $\omega_S$. In particular, $\{f,g\}|_S\neq \{f|_S,g|_S\}_S$ and hence $S\into \bR^4$ is not Poisson. Therefore, the stratification
    $$
    \Sigma=\{S,\bR^2\setminus S\}
    $$
    is not a symplectic stratification even though each stratum has a canonical symplectic form.
\end{egs}

Just as was the case for Poisson manifolds, symplectic stratified spaces have a notion of a Hamiltonian vector field.

\begin{defs}
Let $X$ be a symplectic stratified space and $f\in C^\infty(X)$ a smooth map. The (Zariski) Hamiltonian vector field defined by $f$, denote $V_f$ is the derivation
$$
V_f:C^\infty(X)\to C^\infty(X);\quad g\mapsto \{f,g\}.
$$
\end{defs}

\begin{lem}\label{lem: restricting hamiltonian vf to strata}
    Let $X$ be a symplectic stratified space and $f\in C^\infty(X)$. Then $V_f\in \mfX(\Sigma)$ and for any stratum $S\in \Sigma$, if $V^S_{f|_S}\in \mfX(S)$ denotes the Hamiltonian vector field induced by the symplectic form $\omega_S$, then
    $$
    V^S|_{f|_S}=V_f|_S.
    $$
\end{lem}

\begin{proof}
    Let $f\in C^\infty(X)$. To first show that $V_f\in \mfX(\Sigma)$, fix a stratum $S\in \Sigma$. Letting $g\in I_S$ be any function which vanishes on $S$, we have 
    \begin{equation}\label{eq: Hamiltonian vf tangent to strata}
    V_f(g)|_S=\{f,g\}|_S=\{f|_S,g|_S\}_S=\{f|_S,0\}_S=0,
    \end{equation}
    where $\{\cdot,\cdot\}_S$ is the Poisson bracket induced by the symplectic form $\omega_S\in \Omega^2(S)$. By Proposition \ref{prop: characterization of restrictable vector fields}, it then follows that $V_f\in \mfX(X,S)$. Since $S$ was an arbitrary stratum, it follows that $V_f\in \mfX(\Sigma)$. 

    \ 

    Fixing a stratum $S\in \Sigma$ once again, to show $V^S|_{f|_S}=V_f|_S$ we use Equation (\ref{eq: Hamiltonian vf tangent to strata}). Indeed, since $C^\infty(S)=\bra C^\infty(X)|_S\ket$, it suffices to show that 
    $$
    V^S_{f|_S}(g|_S)=V_f|_S(g|_S)
    $$
    for any $g\in C^\infty(X)$. To that end, note that since $V_f\in \mfX(X,S)$ we have from another application of Proposition \ref{prop: characterization of restrictable vector fields} that
    $$
    V_f|_S(g|_S)=V_f(g)|_S=\{f,g\}|_S
    $$
    On the other hand,
    $$
    V^S_{f|_S}(g|_S)=\{f|_S,g|_S\}_S=\{f,g\}|_S.
    $$
    Hence the result.
\end{proof}

\begin{prop}
    Let $(X,\Sigma,\{\cdot,\cdot\})$ be a symplectic stratified space. Then the map
    $$
    X\times C^\infty(X)\to T\Sigma;\quad (x,f)\mapsto V_f(x)
    $$
    is surjective.
\end{prop}

\begin{proof}
    Let $x\in X$, $S\in \Sigma$ the stratum containing $x$, and $v\in T_xS$. Since $S$ is symplectic, we can find a function $f\in C^\infty(S)$ so that $V^S_f(x)=v$, where $V^S_f\in \mfX(S)$ denotes the Hamiltonian vector field with respect to $S$. Since $C^\infty(S)=\bra C^\infty(X)|_S\ket$ as a Sikorski structure, we can find an open neighbourhood $U\subseteq X$ of $x$ and a function $F\in C^\infty(X)$ so that
    $$
    f|_{S\cap U}=F|_{S\cap U}.
    $$
    By Lemma \ref{lem: restricting hamiltonian vf to strata}, $V_F|_{S\cap U}=V^S_f|_{S\cap U}$ and thus $V_F(x)=v$.
\end{proof}

In this way, we see that symplectic stratified spaces are a generalization of Poisson manifolds with locally finite symplectic foliations.

\section{The Orbit-Type and Canonical Stratifications of the momentum map}\label{sec: orbit symp strat}

So far, we have only discussed symplectic stratified spaces where the underlying differentiable space is a manifold. Now that we are armed with the Hamiltonian slice theorem as in Theorem \ref{thm: local normal form for Hamiltonian spaces}, we can discuss now some truly singular examples. For all that follows, we will let $J:(M,\omega)\to \mfg^*$ be a proper $G$-space for a connected Lie group $G$. Suppose now that $J^{-1}(0)$ is nonempty, let $M_0=J^{-1}(0)/G$, write $\pi_0:J^{-1}(0)\to M_0$ for the quotient map.

\begin{lem}\label{lem: subcartesian structures on level set and reduction}
     Both $J^{-1}(0)$ and $M_0$ are canonically subcartesian spaces, where 
    $$
    C^\infty(J^{-1}(0))=C^\infty(M)|_{J^{-1}(0)}
    $$
    and
    $$
    C^\infty(M_0)=\{f:M_0\to \bR \ | \ \pi_0^*f\in C^\infty(J^{-1}(0))^G\}.
    $$
\end{lem}

\begin{proof}
    $C^\infty(J^{-1}(0))=C^\infty(M)|_{J^{-1}(0)}$ being a subcartesian structure is a consequence of Proposition \ref{prop: closed subspace has restriction differential structure} due to $J^{-1}(0)$ being closed. As for $C^\infty(M_0)$ being subcartesian, this is precisely the statement of Proposition \ref{prop: subcartesian on quotient of closed subset}.
\end{proof}

We want to show now that $M_0$ is canonically a symplectic stratified space. As it turns out, thanks to the local normal form from proper Hamiltonian spaces (see Theorem \ref{thm: local normal form for Hamiltonian spaces}), the procedure for defining the stratification is nearly identical to obtaining the canonical stratification on $M/G$.

\begin{defs}
Define the \textbf{orbit-type} stratification of $J^{-1}(0)$, denoted $\mathcal{S}_G(J^{-1}(0))$ to be the partition
$$
\mathcal{S}_G(J^{-1}(0))=\bigcup_{S\in \mathcal{S}_G(M)}\pi_0(S\cap J^{-1}(0)).
$$
Furthermore, define the \textbf{canonical stratification} of $M_0$ by
$$
\mathcal{S}_G(M_0)=\{\pi_0(S) \ | \ S\in \mathcal{S}_G(J^{-1}(0))\}.
$$
\end{defs}

Although we called the two partitions defined above stratifications, we still need to show that they are, for instance, partitions into submanifolds with respect to the subcartesian structures in Lemma \ref{lem: subcartesian structures on level set and reduction}. 

\begin{lem}\label{lem: tangent space of strata}
    The pieces of the orbit-type partition $\mathcal{S}_G(J^{-1}(0))$ are $G$-invariant embedded submanifolds. Furthermore, for any $S\in \mathcal{S}_G(J^{-1}(0))$ and $x\in S$, there exists $G_x$-equivariant linear symplectomorphism of linear $G_x$ symplectic representations
    $$
    \phi_x:(T_xM,\omega_x)\to ((\mfg/\mfg_x)\times (\mfg/\mfg_x)^*\oplus S\nu_x(M,G))
    $$
    under which,
    $$
    T_x\phi(T_xS)=(\mfg/\mfg_x)\times \{0\}\times S\nu_x(M,G)
    $$
\end{lem}

\begin{proof}
    Let us fix $S\in \mathcal{S}_G(J^{-1}(0))$ and $x\in X$. Since being an embedded submanifold and the structure of tangent spaces is a local question, we pass to a symplectic slice centred on $x$. That is, we may assume $M=G\times_{G_x}(\mfg_x^\circ\times S\nu_x(M,G))$, $x=[e,0,0]$, with the symplectic form given by Lemma \ref{lem: local normal form for Hamiltonian spaces} and the momentum map given by
    $$
    J:G\times_{G_x}(\mfg_x^\circ\times S\nu_x(M,G))\to \mfg^*;\quad [g,p,v]\mapsto \text{Ad}_g^*(p+\mfp\circ J_x(v)),
    $$
    where $\mfp:\mfg_x^*\to \mfg^*$ is a $K$-equivariant splitting of the short exact sequence
    $$
    0\to \mfg_x^\circ\to \mfg^*\to \mfg_x^*\to 0
    $$
    and $J_x:S\nu_x(M,G)\to \mfg_x^*$ is the quadratic momentum map. Observe that
    $$
    J^{-1}(0)=G\times_{G_x}(\{0\}\times J_x^{-1}(0))
    $$
    Hence, by Lemma \ref{lem: foundational properties of symplectic representations},
    \begin{equation}\label{eq: stratum of momentum map description}
    S=J^{-1}(0)_{(G_x)}=G\times_{G_x} (\{0\}\times (J_x^{-1}(0))^{G_x})\cong (G/K)\times S\nu_x(M,G)^{G_x}
    \end{equation}
    Hence, $S$ is an embedded submanifold. Now, observe that
    $$
    T_{[e,0,0]}G\times_K(\mfg_x^\circ\times S\nu_x(M,G))=(\mfg/\mfg_x)\times \mfg_x^\circ\times S\nu_x(M,G)
    $$
    Note that we can identify $\mfg_x^\circ$ with $(\mfg/\mfg_x)^*$ and that by construction of the symplectic form on normal form neighbourhoods, the symplectic form $\omega$ restricts to the canonical symplectic form on $(\mfg/\mfg_x)\times(\mfg/\mfg_x)^*=T^*(\mfg/\mfg_x)$. Hence, 
    $$
    T_{[e,0,0]}M=(\mfg/\mfg_x)\times(\mfg/\mfg_x)^*\times S\nu_x(M,G).
    $$
    Next, from our description of $S$ in Equation (\ref{eq: stratum of momentum map description}) we have
    $$
    T_{[e,0,0]}S=(\mfg/\mfg_x)\times \{0\}\times S\nu_x(M,G).
    $$
\end{proof}

\begin{cor}
    Let $S\in \mathcal{S}_G(J^{-1}(0))$. Then $\omega|_S$ is basic and the induced $2$-form $\omega_{S}\in \Omega^2(S/G)$ is symplectic.
\end{cor}

\begin{proof}
    Since $\omega$ is invariant, it follows that $\omega|_{G\cdot x}=0$ for all $x\in M$. In particular, due to Lemma \ref{lem: tangent space of strata}, it follows for any $x\in S$ that
    $$
    \ker((\omega|_S)_x^\flat)=T_x(G\cdot x).
    $$
    Therefore, $\omega|_S$ is basic and hence there exists a unique $2$-form $\omega_{S}\in \Omega^2(S/G)$ satisfying
    $$
    \pi_S^*\omega_{S}=\omega|_S.
    $$
    This form being symplectic is a consequence of linear symplectic reduction (Proposition \ref{prop: linear reduction}).
\end{proof}

Making use of some results, we now have shown that $\mathcal{S}_G(J^{-1}(0))$ and $\mathcal{S}_G(M_0)$ are partitions into connected smooth manifolds and that each piece of $\mathcal{S}_G(M_0)$ is symplectic. To finish showing each partition is a stratification, we just need to show that they are both locally finite and satisfy the axiom of the frontier. To do this, we first show it's true for linear symplectic representations.

\begin{lem}
    Let $K$ be a compact Lie group and $(V,\omega)$ a symplectic $K$-representation. Write $J_V:V\to \mfk^*$ for the quadratic momentum map. Then the partition
    $$
    \mathcal{S}_K(J^{-1}(0))=\bigcup_{H\leq K}\pi_0(J_V^{-1}(0)\cap V_{(H)}) 
    $$
    is finite stratification of $J^{-1}(0)$. Furthermore, each piece of $\mathcal{S}_K(J^{-1}(0))$ is invariant under scalar multiplication by $t\in (0,1)$.
\end{lem}

\begin{proof}
    See \cite{mol_stratification_2024} for a proof that $\mathcal{S}_K(J^{-1}(0))$ is finite.
    \ 

    We will now show that $\mathcal{S}_K(J^{-1}(0))$ is a stratification with strata closed under scalar multiplication by elements of $(0,1)$ at the same time. First, recall that by Lemma \ref{lem: foundational properties of symplectic representations} $V^K$ is a symplectic subspace and its symplectic complement $W=(V^K)^\omega$ is a symplectic subrepresentation. Furthermore, letting $J_W:W\to \mfk^*$ be the quadratic momentum map, we have
    $$
    J_V^{-1}(0)=V^K+J_W^{-1}(0).
    $$
    In particular, for any subgroup $H\leq K$ we have
    $$
    J_V^{-1}(0)_{(H)}=V^K+J_W^{-1}(0)_{(H)}
    $$
    and 
    $$
    J_V^{-1}(0)_{(K)}=V^K.
    $$
    Now, since $J_W$ is an invariant quadratic map (i.e. $J_W(tw)=t^2J_W(w)$ for all $t\in \bR$ and $w\in W$), it follows that the orbit-type pieces $J_W^{-1}(0)_{(H)}$ are closed under scalar multiplication by non-zero scalars. Hence, so are their connected components. Using this, we can now show for any $R\in \mathcal{S}_K(J^{-1}(0))$ that
    $$
    J_V^{-1}(0)_{(K)}\substeq \overline{R}.
    $$
    Indeed, supposing $R\subseteq V_{(H)}$ for some $H\leq K$, there exists a connected component $R'\substeq J_W^{-1}(0)_{(H)}$ so that
    $$
    R=V^K+R'.
    $$
    Fixing any $v\in V^K$ and $w\in R'$, consider now the sequence $\{v+w/n\}_{n> 0}$. Since $R'$ is closed under scalar multiplication by non-zero scalars, this sequence lies in $R$. Hence, the limit
    $$
    \lim_{n\to \infty} \bigg(v+\frac{w}{n}\bigg)=v
    $$
    lies in $\overline{R}$. Therefore, $J_V^{-1}(0)_{(K)}\substeq \overline{R}$. 

    \ 

    Now for any two arbitrary strata $S,R\in \mathcal{S}_K(J^{-1}(0))$ with $S\cap\overline{R}\neq \emptyset$, pass to a symplectic slice neighbourhood around $x\in S\cap \overline{R}$. Then repeating the same argument in the symplectic normal space $S\nu_x(V,K)$ shows $S\substeq \overline{R}$.
\end{proof}

\begin{theorem}[{\cite{sjamaar_stratified_1991,zimhony_commutative_2024}}]\label{thm: quasi-homo reduced spaces}
    Both $\mathcal{S}_G(J^{-1}(0))$ and $\mathcal{S}_G(M_0)$ are quasi-homogeneous stratifications. Furthermore, for each stratum $S\in \mathcal{S}_G(J^{-1}(0))$, there exists a unique symplectic form $\omega_{S/G}\in \Omega^2(S/G)$ such that
    $$
    (\pi_0|_S)^*\omega_{S/G}=\omega|_S.
    $$
\end{theorem}

\begin{proof}
    First to show $\mathcal{S}_G(J^{-1}(0))$ is a quasi-homogeneous stratification let $x\in J^{-1}(0)$ and $S\in \mathcal{S}_G(J^{-1}(0))$ the stratum containing $x$. Passing to a symplectic slice neighbourhood, we may assume
    $$
    M=G\times_{G_x}(\mfg_x^\circ\times S\nu_x(M,G))
    $$
    with 
    $$
    J^{-1}(0)=G\times_{G_x}(\{0\}\times J_x^{-1}(0)),
    $$
    where $J_x:S\nu_x(M,G)\to \mfg_x^*$ is the quadratic momentum map. Thus, since $\mathcal{S}_{G_x}(J_x^{-1}(0))$ is a finite stratification by invariant subsets, it follows immediately that $\mathcal{S}_G(J^{-1}(0))$ is a quasi-homogeneous stratification. 

    \ 

    Since the quotient map $\pi_0:J^{-1}(0)\to M_0$ is an open quotient, it follows that $\mathcal{S}_G(M_0)$ is also a stratification. Now using an argument near identical to the one from Theorem \ref{thm: quasi-hom quotient space}, we see that $\mathcal{S}_G(M_0)$ is also quasi-homogeneous.
\end{proof}

\section{Manifolds of Symmetry and Singular Reduced Space}\label{sec: reduce mfld symmetry}

Let $J:(M,\omega)\to \mfg^*$ be a proper Hamiltonian $G$-space. The strata of the reduced space $M_0=J^{-1}(0)/G$ can also be realized through usual Marsden-Weinstein-Meyer reduction, but we need to pass to the manifolds of symmetry first. Indeed, fix $x\in J^{-1}(0)$ and write $K=G_x$. Recall from Corollary \ref{cor: manifold of symmetry and orbit-type are manifolds} and Proposition \ref{prop: tangent space to manifold of symmetry and orbit-type} that the manifold of symmetry $M_K=\{y\in M \ | \ G_y=K\}$ is a submanifold of $M$ and, furthermore, 
$$
T_yM_K=(T_yM)^K
$$
for any $y\in M_K$. As we have seen, the fixed point set of a symplectic representation is once again symplectic, hence $M_K$ is a symplectic submanifold of $M$. Now, by Corollary \ref{cor: manifold of symmetry and orbit-type are manifolds}, $M_K$ is a free $N_G(K)/K$-space where $N_G(K)$ is the normalizer group of $K$. Since the $G$-action preserves the symplectic form, it follows that $N_G(K)/K$ preserves the form on $M_K$ as well. 

\ 

We will now show that this is a Hamiltonian action. To set this up, let $\xi\in \mfk$. Since $K$ acts trivially on $M_K$, it easily follows that $\xi_M|_{M_K}=0$. In particular, for any $y\in M_K$ and any $\xi\in \mfk$ we have
$$
\bra J(y),\xi\ket=0
$$
and thus $J(M|_K)\substeq \mfk^\circ$. Now let $\iota:\mfn_G(K)\into \mfg$ be the inclusion, then since $\iota$ is $K$-equivariant, it follows that 
$$
\iota^*\circ J(M|_K)\substeq \mfk^\circ \substeq \mfn_G(K)^*
$$
Using the identification $\mfk^\circ\cong (\mfn_G(K)/\mfk)^*$ we obtain a canonical $K$-equivariant map
$$
J_K:M_K\to (\mfn_G(K)/\mfk)^*.
$$
Since the action of $G$ on $M$ is Hamiltonian, it follows that so is the action of $N_G(K)$ on $M_K$ and hence $N_G(K)/K$. Thus, we obtain the following.

\begin{prop}\label{prop: manifold of symmetry is a hamiltonian space}
    Let $K\leq G$ be a compact subgroup with $M_K\neq \emptyset$. Then $J_K:M_K\to (\mfn_G(K)/\mfk)^*$ is a momentum map for the $N_G(K)/K$ action on $M_K$.
\end{prop}

Since the action of $N_G(K)/K$ on $M_K$ is free and proper, usual Marsden-Weinstein-Meyer reduction applies and thus $J_K^{-1}(0)/(N_G(K)/K)$ is canonically a symplectic manifold. As it turns out, this symplectic manifold is precisely one of the canonical strata of the singular reduced space!

\begin{theorem}\label{thm: manifold of symmetry reduction gives strata}
    Let $x\in J^{-1}(0)$ and write $K=G_x$. 
    \begin{itemize}
        \item[(1)] $J^{-1}(0)\cap M_K=J_K^{-1}(0)$.
        \item[(2)] Write $(M_K)_0$ for the reduced space of $J_K:M_K\to (\mfn_G(K)/\mfk)^*$. If $x\in J_K^{-1}(0)$ and $S\in \mathcal{S}_G(J^{-1}(0))$ is the orbit-type stratum through containing $x$, then there is a canonical symplectomorphism between $S/G$ and the connected component of $(M_K)_0$ containing $[x]$. 
    \end{itemize}
\end{theorem}

\begin{proof}
    \begin{itemize}
        \item[(1)] This follows immediately from the fact that $J_K$ is obtained by restricting $J$.
        \item[(2)] Recall from Corollary \ref{cor: manifold of symmetry and orbit-type are manifolds} that $G\cdot M_K=M_{(K)}$. Thus, we obtain a $G$-equivariant bijection
        $$
        G\times_{N_G(K)}J_K^{-1}(0)\to J^{-1}(0)_{(K)}
        $$
        Furthermore, this bijection must preserve the restricted symplectic form on each connected component of $J^{-1}(0)_{(K)}$. Hence, the result then follows.
    \end{itemize}
\end{proof}

\section{The Poisson Bracket on Reduced Space}\label{sec: poisson on reduced space}

Let $J:(M,\omega)\to \mfg^*$ be a proper Hamiltonian $G$-space. We have nearly shown that $(M_0,\mathcal{S}_G(M_0))$ is a symplectic stratified space. Indeed, by Theorem \ref{thm: quasi-homo reduced spaces}, $M_0$ is quasi-homogeneous and each stratum is equipped with a canonical symplectic form. Furthermore, as we saw in Theorem \ref{thm: quasi-homogeneous implies Whitney regular} quasi-homogeneous implies Whitney regular. So the only last piece of the puzzle is the Poisson bracket. To do this, we need to first endow $C^\infty(J^{-1}(0))^G=C^\infty(M)^G|_{J^{-1}(0)}$ with a Poisson bracket.

\ 

Recall now if $M$ has a single orbit-type, then $J^{-1}(0)$ is a coisotropic submanifold of $M$. In particular, this means that if $f\in C^\infty(M)$, then its Hamiltonian vector field $V_f\in \mfX(M)$ is tangent to $J^{-1}(0)$. In particular, for any $f,g\in C^\infty(M)$ we can define
$$
\{f|_{J^{-1}(0)},g|_{J^{-1}(0)}\}_{J^{-1}(0)}:=V_f|_{J^{-1}(0)}(g|_{J^{-1}(0)}),
$$
which endows $C^\infty(J^{-1}(0))$ with a Poisson structure. In the singular case, a similar construction will not necessarily work for the entire subcartesian structure $C^\infty(J^{-1}(0))=C^\infty(M)|_{J^{-1}(0)}$, but we can provide a Poisson bracket to $C^\infty(J^{-1}(0))^G$. First, we need to establish a relationship between the manifolds of symmetry and $J^{-1}(0)$.

\

\begin{lem}
    Let $f\in C^\infty(M)^G$ and let $S\in \mathcal{S}_G(J^{-1}(0))$. Then the Hamiltonian vector field $V_f$ is tangent to $S$. In symbols, $V_f\in \mfX(M,S)^G$.
\end{lem}

\begin{proof}
    Fix $f\in C^\infty(M)^G$ and $S\in \mathcal{S}_G(J^{-1}(0))$. From the defining equation of $V_f$,
    $$
    df=\omega(V_f,\cdot)
    $$
    and the fact that $\omega$ is invariant, it follows that $V_f\in \mfX(M)^G$. Now fix $x\in S$, we will show $V_f(x)\in T_xS$.
    
    \ 
    
    As we showed in Proposition \ref{prop: equivariant sections give all sections below}, equivariant vector fields are tangent to the manifolds of symmetry. Thus, if we write $K=G_x$, then $V_f(x)\in T_x M_K$. Furthermore, by Theorem \ref{thm: manifold of symmetry reduction gives strata}, we have 
    $$
    J^{-1}(0)\cap M_K=J_K^{-1}(0)
    $$
    Observe that for any $\xi\in \mfn_G(K)$, we have by the equivariance of $f$ that $\xi_M(f)=0$. This then implies that $V_f|_{M_K}(J_K^\xi)=0$ for all $\xi\in \mfn_G(K)$. Thus, since $0$ is a regular value of $J_K$, we obtain that $V_f$ is tangent to $J_K^{-1}(0)$. The result then follows.
\end{proof}

\begin{cor}\label{cor: equivariant ham vector fields can be restricted}
    If $f\in C^\infty(M)^G$, the Hamiltonian vector field of $f$ is tangent to $J^{-1}(0)$ as a Zariski vector field and the restriction is an equivariant stratified vector field. In symbols, $V_f\in \mfX(M, J^{-1}(0))^G$ and $V_f|_{J^{-1}(0)}\in \mfX(\mathcal{S}_G(J^{-1}(0)))^G$.
\end{cor}

Recall from Proposition \ref{prop: characterization of restrictable vector fields} that if $X$ is a subcartesian space, $Y\subseteq X$ is a subspace, and $V\in \mfX(X,Y)$, then for any $f\in C^\infty(X)$ we have
$$
V|_Y(f|_Y)=V(f)|_Y.
$$
Using this, we can now define a Poisson bracket on $C^\infty(M)^G|_{J^{-1}(0)}$

\begin{defs}
    Define the Poisson bracket $\{\cdot,\cdot\}_{J^{-1}(0)}$ on $C^\infty(M)^G|_{J^{-1}(0)}$ by
    $$
    \{f|_{J^{-1}(0)},g|_{J^{-1}(0)}\}_{J^{-1}(0)}:=V_f(g)|_{J^{-1}(0)},
    $$
    where $f,g\in C^\infty(M)^G$. 
\end{defs}

And now, we can finally endow the reduced space with a Poisson bracket.

\begin{defs}[Poisson Bracket on Reduced Space]
    Let $f,g\in C^\infty(M_0)$. Define their Poisson bracket $\{f,g\}_0\in C^\infty(M_0)$ to be the unique function so that
    $$
    \pi_0^*\{f,g\}_0=\{\pi_0^*f,\pi_0^*g\}_{J^{-1}(0)}
    $$
\end{defs}

Before we can prove the main result of this section, we need to discuss the Poisson brackets on the strata of $M_0$.

\begin{prop}\label{prop: pushforward of hamiltonian vf}
    Let $S\in \mathcal{S}_G(J^{-1}(0))$ be an orbit-type stratum, $\pi_S:S\to S/G$ the quotient map, and $\omega_S\in \Omega^2(S/G)$ the unique symplectic form satisfying $\pi_S^*\omega_S=\omega|_S$. Then, given $f\in C^\infty(M)^G$, we have
    $$
    (\pi_S)_*(V_f|_S)=V_{f_0},
    $$
    where $f_0\in C^\infty(S/G)$ satisfies $\pi_S^*f_0=f|_S$ and $V_{f_0}\in \mfX(S/G)$ is the Hamiltonian vector field of $f_0$ with respect to $\omega_S$.
\end{prop}

\begin{proof}
    Fix $f\in C^\infty(M)^G$ and let $f_0\in C^\infty(S/G)$ satisfy $\pi_S^*f_0=f|_S$. First note that since $S$ is assumed to have only one orbit-type, by Corollary \ref{cor: pushforwards of equivariant vector fields, regular}, the map
    $$
    (\pi_S)^*:\mfX(S)^G\to \mfX(S/G)
    $$
    is surjective. In particular, there exists vector field $\widetilde{V}\in \mfX(S)^G$ so that $(\pi_S)_*\widetilde{V}=V_{f_0}$. From this, we can easily deduce that
    $$
    \pi_S^*(\omega_S(V_{f_0},\cdot))=(\pi_S^*\omega_S)(\widetilde{V},\cdot)=\omega|_S(\widetilde{V},\cdot).
    $$
    On the other hand, by definition, $df_0=\omega_S(V_{f_0},\cdot)$ and thus
    $$
    \pi_S^*(\omega_S(V_{f_0},\cdot))=\pi_S^*(df_0)=df|_S.
    $$
    Using the fact that $V_f$ is tangent to $S$, we obtain $df|_S=\omega|_S(V_f|_S,\cdot)$ and hence we have
    $$
    \omega|_S(\widetilde{V},\cdot)=\omega|_S(V_f|_S,\cdot)
    $$
    Locally extending $\widetilde{V}$ and using the non-degeneracy of $\omega$, we obtain $\widetilde{V}=V_f|_S$ and thus the result follows.
\end{proof}

\begin{theorem}\label{thm: singular reduction of symplectic manifolds}
    The reduced space $(M_0,\mathcal{S}_G(M_0),\{\cdot,\cdot\}_0)$ is a symplectic stratified space.
\end{theorem}

\begin{proof}
    We need to show the inclusions of the strata are Poisson. To do so, let $S\in \mathcal{S}_G(J^{-1}(0))$ be a orbit-type stratum, $\omega_S\in\Omega^2(S/G)$ the induced symplectic form, and $f,g\in C^\infty(M_0)$. Writing $\{\cdot,\cdot\}_S$ for the Poisson bracket on $C^\infty(S/G)$, we show
    $$
    \{f,g\}_0|_{S/G}=\{f|_{S/G},g|_{S/G}\}_S
    $$
    To do this, let $F,G\in C^\infty(M)^G$ satisfy $\pi_0^*f=F|_{J^{-1}(0)}$ and $\pi_0^*g=G|_{J^{-1}(0)}$. If we write $\pi_S:=\pi_0|_S$, then by Proposition \ref{prop: pushforward of hamiltonian vf}, we have
    $$
    \pi_S^*\{f|_{S/G},g|_{S/G}\}_S=\pi_S^*(V_{f|_{S/G}}(g|_{S/G}))=V_{F|_S}(G|_S)=\{F,G\}|_S
    $$
    Conversely, by the  definition of the Poisson bracket on $M_0$,
    $$
    \pi_S^*\{f,g\}_0|_{S/G}=(\pi_0^*\{f,g\}_0)|_{S/G}=\{\pi_0^*f,\pi_0^*g\}_{J^{-1}(0)}|_S=\{F,G\}|_S
    $$
    Hence, 
    $$
    \pi_S^*\{f|_{S/G},g|_{S/G}\}_S=\pi_S^*(\{f,g\}_0|_{S/G})
    $$
    and so, since $\pi_S$ is a surjective submersion, we obtain
    $$
    \{f|_{S/G},g|_{S/G}\}_S=\{f,g\}_0|_{S/G}
    $$
    which completes the proof.
\end{proof}

\section{Hamiltonian Modules}\label{sec: Hamiltonian module}
Let us now introduce an important class of vector fields on a symplectic stratified space, called a Hamiltonian Module. These are a distinguished class of vector fields which contain all the Hamiltonian vector fields, together with a generating set for the stratified tangent bundle. As differential forms are not a well-behaved topic in the subcartesian setting, these will allow us to give a ``smoothness'' criterion to connections of line bundles later on. 

\begin{defs}[Hamiltonian Module]\label{def: hamiltonian module}
Let $(X,\Sigma,\{\cdot,\cdot\})$ be a symplectic stratified space. A Hamiltonian module is a sub Lie algebra $\mathcal{H}\subseteq \mfX(\Sigma)$ such that
\begin{itemize}
    \item[(1)] $\mathcal{H}$ contains all Hamiltonian vector fields.
    \item[(2)] For any stratum $S\in \Sigma$, $x\in S$, and smooth vector field $V\in \mfX(S)$, we can find a vector field $W\in \mathcal{H}$ and a neighbourhood $U\subseteq X$ of $x$ so that
    $$
    V|_{U\cap S}=W|_{U\cap S}.
    $$
\end{itemize}
\end{defs}

\begin{egs}
    Suppose $(M,\pi)$ is a Poisson manifold so that the symplectic foliation $\mathscr{F}_\pi$ is a Whitney regular stratification of $M$. Then, the set of all smooth vector fields tangent to the leaves $\mfX(\mathscr{F}_\pi)$ is a Hamiltonian module of $(M,\mathscr{F}_\pi,\{\cdot,\cdot\})$. 
\end{egs}

\begin{egs}\label{eg: non-trivial Ham module}
    Returning now to the Poisson structure
    $$
    \pi=r^2\frac{\partial}{\partial x}\wedge \frac{\partial}{\partial y}
    $$
    from Example \ref{eg: not symplectic Poisson}. The symplectic foliation has two leaves, $\{0\}$ and $\bR^2\setminus\{0\}$ and defines a symplectic stratification. And so, one possible Hamiltonian module would be
    $$
    \mfX(\mathscr{F}_\pi)=\{V\in \mfX(\bR^2) \ | \ V_0=0\}.
    $$
    However, we can produce an even smaller Hamiltonian module by controling the order of vanishing at the origin.
    
    \ 
    
    Let $V\in \mfX(\bR^2)$. We will say $V$ vanishes to order $2$ at $0$ in a neighbourhood $U\substeq \bR^2$ of the origin, we can find a vector field $W\in \mfX(U)$ so that
    $$
    V|_U=r^2W.
    $$
    Define then
    $$
    \mathcal{H}=\{V\in \mfX(\bR^2) \ | \ V\text{ vanishes to order }2\text{ at the origin}\}.
    $$
    I claim $\mathcal{H}$ is a Hamiltonian module. Indeed, any $f\in C^\infty(\bR^2)$, the Hamiltonian vector field $V_f$ is given by
    $$
    V_f=r^2\frac{\partial f}{\partial y}\frac{\partial }{\partial x}-r^2\frac{\partial f}{\partial x}\frac{\partial }{\partial y}
    $$
    Thus, $V_f=r^2V_f^{can}$, where $V_f^{can}\in \mfX(\bR^2)$ is the Hamiltonian vector field with respect to the canonical symplectic form on $\bR^2$. Next given any vector field $W\in \mfX(\bR^2\setminus\{0\})$ and $x\in \bR^2\setminus\{0\}$, using bump functions we can find open neighbourhoods $A\subset B\substeq \bR^2\setminus\{0\}$ of $x$ and a vector field $V\in \mfX(\bR^2)$ such that
    $$
    V|_A=W|_A
    $$
    and so that $V|_{\bR^2\setminus B}=0$. By definition, $V\in \mathcal{H}$ as it vanishes to order $\infty$ at $0$. Thus, $\mathcal{H}$ is a Hamiltonian module.
    \end{egs}

For all that follows, we let $G$ be a connected Lie group, $J:(M,\omega)\to \mfg^*$ a proper Hamiltonian $G$-space, and $(\mathcal{L},\nabla)\to (M,\omega)$ an equivariant prequantum line bundle. We will be showing in this section that $L$ induces a stratified prequantum line bundle on the singular reduced space $M_0=J^{-1}(0)/G$. To do this, we first need to construct the Hamiltonian module. Recall that we saw in Corollary \ref{cor: equivariant ham vector fields can be restricted} the Hamiltonian vector fields of equivariant functions on $M$ were tangent to the strata of $J^{-1}(0)$. This result can be somewhat generalized. 

\begin{prop}\label{prop: tangent to J^-1(0) and equivariant implies stratified}
Let $V\in \mfX(M,J^{-1}(0))^G$, then $V|_{J^{-1}(0)}\in \mfX(\mathcal{S}_G(J^{-1}(0))^G$.
\end{prop}

\begin{proof}
Let $V\in \mfX(M,J^{-1}(0))^G$ and let $K\leq G$ be a compact subgroup. Since $V$ is $G$-invariant, we have $V|_{M_K}\in \mfX(M_K)$. So we just need to show that $V|_{M_K}$ is tangent to $J^{-1}(0)\cap M_K$. Observe that from Theorem \ref{thm: manifold of symmetry reduction gives strata} that 
$$
J^{-1}(0)\cap M_K=J_K^{-1}(0),
$$
where $J_K:M_K\to (\mathfrak{n}_G(K)/\mfk)^*$ is the induced momentum map from the free $N_G(K)/K$ action on $M_K$. Since the action of $N_G(K)/K$ on $M_K$ is free, it follows that for any $x\in J_K^{-1}(0)$ we have
$$
T_xJ_K^{-1}(0)=\ker(T_x J_K).
$$
In particular, $V|_{M_K}$ is tangent to $J_K^{-1}(0)$ if and only if $V|_{M_K}J_K$ along $J_K^{-1}(0)$. This follows immediately from the fact that $VJ=0$ along $J^{-1}(0)$ and that $J|_{M_K}=J_K$. Hence the result.
\end{proof}

\begin{cor}\label{cor: equivariant on J^-1 implies stratified}
The equivariant Zariski vector fields on $J^{-1}(0)$ coincide with the equivariant stratified vector fields. That is, $\mfX(J^{-1}(0))^G=\mfX(\mathcal{S}_G(J^{-1}(0)))^G$. 
\end{cor}

\begin{proof}
Fix $V\in \mfX(J^{-1}(0))^G$. Then since $J^{-1}(0)$ is closed, we can extend $V$ to a vector field $W\in \mfX(M)$. After averaging, we may assume $W\in \mfX(M)^G$. By assumption $W|_{J^{-1}(0)}=V$. Hence, by Proposition \ref{prop: tangent to J^-1(0) and equivariant implies stratified}, $V\in \mfX(\mathcal{S}_G(J^{-1}(0)))^G$. The reverse inclusion is obvious. 
\end{proof}

So we see that equivariant Zariski vector fields tangent to $J^{-1}(0)$ must be tangent to the strata of $J^{-1}(0)$ themselves. This is somewhat expected since equivariant vector fields on all of $M$ must themselves be tangent to the orbit-type strata. This now provides us with a natural class of vector fields to push down to the stratified quotient.

\begin{prop}\label{prop: pushforward equivariant Zariski}
Suppose $V\in \mfX(J^{-1}(0))^G$ and $\pi:J^{-1}(0)\to M_0$ the quotient map, then the fibrewise pushforward $\pi_*V$ defines a stratified vector field on $M_0$.
\end{prop}

\begin{proof}
Fix $V\in\mfX(J^{-1}(0))^G$. Since $V$ is $G$-invariant, we have $(\pi_0)_*V\in \mfX(M_0)$. Furthermore, by Corollary \ref{cor: equivariant on J^-1 implies stratified}, $V$ is tangent to the strata. Hence, if $S\in \mathcal{S}_G(J^{-1}(0))$ is an orbit-type stratum, we have
$$
(\pi_0|_S)_*(V|_S)\in \mfX(S/G).
$$
Now I claim that $(\pi_0)_*V$ is also stratified and that
\begin{equation}\label{eq: pushforward is functorial in restriction}
(\pi_0)_*V|_{S/G}=(\pi_0|_S)_*(V|_S)
\end{equation}
First, let us suppose $f\in C^\infty(M_0)$ satisfies $f|_{S/G}=0$. Lifting $f$ to $F\in C^\infty(J^{-1}(0))^G$ we see that $F|_S=0$. Hence, for any $x\in S$ we have
$$
T_x\pi_0(V)f=V_x(\pi_0^*f)=V_x(F)=0.
$$
Thus, $(\pi_0)_*V\in \mfX(\mathcal{S}_G(M_0))$. Now to show Equation (\ref{eq: pushforward is functorial in restriction}). Let $f\in C^\infty(S/G)$ be a smooth function and let $F\in C^\infty(S)^G$ be a lift. Observe that
\begin{align*}
((\pi_0)_*V)|_{S/G}(f|_{S/G})&=V|_{S}(\pi_0^*f|_{S/G})\\
&=V|_{S}(F|_{S})
\end{align*}
Furthermore,
\begin{align*}
(\pi_0|_S)_*(V|_S)(f|_{S/G})&=V|_S((\pi_0|_S)^*f|_{S/G})\\
&=(V|_S)(F|_S)
\end{align*}
These two expressions are equal and so equality in Equation (\ref{eq: pushforward is functorial in restriction}) holds.
\end{proof}

Due to the constructions above, we get a natural map
\begin{equation}\label{eq: extended push-forward}
\pi_*:\mfX(M,J^{-1}(0))^G\to \mfX(\mathcal{S}_G(M_0))
\end{equation}
given by restricting a vector field to $J^{-1}(0)$, then pushing it forward as in Proposition \ref{prop: pushforward equivariant Zariski}. 

\begin{lem}\label{lem: global pushforward of zariski}
    Let $f\in C^\infty(M)^G$ be an invariant smooth function. Then under the map in Equation (\ref{eq: extended push-forward}), we have
    $$
    \pi_*(V_f)=V_{f_0},
    $$
    where $f_0\in C^\infty(M_0)$ is the unique smooth function so that
    $$
    \pi_0^*f_0=f|_{J^{-1}(0)}.
    $$
\end{lem}

\begin{proof}
    Apply Proposition \ref{prop: pushforward of hamiltonian vf} stratum-wise.
\end{proof}

\begin{defs}[Hamiltonian Module for Singular Reduction] \label{def: Hamiltonian module for singular reduction}
Define the Hamiltonian module $\mathcal{H}_G(M_0)$ for $M_0$ to be the image of the map $\pi_*$ in Equation (\ref{eq: extended push-forward}). We will call elements of $\pi_*$ \textbf{liftable stratified vector fields}.
\end{defs}
\begin{remark}
In an earlier version of this thesis, I had a faulty proof that $\mfX(\mathcal{S}_G(M_0))=\mathcal{H}_G(M_0)$. I still believe that this is true, and my original proof goes through in the case that $M_0$ is a manifold with corners. For more general quotients like the ones being considered in this thesis, perhaps other arguments need to be put forward.
\end{remark}

\begin{theorem}\label{thm: Hamiltonian module for singular reduction}
$\mathcal{H}_G(M_0)$ in Definition \ref{def: Hamiltonian module for singular reduction} is indeed a Hamiltonian module.
\end{theorem}

\begin{proof}
By Lemma \ref{lem: global pushforward of zariski}, the Hamiltonian vector fields of $M_0$ lie in $\mathcal{H}(M_0)$. So we only have to show now that vector fields of the strata of $M_0$ admit extensions to elements of $\mathcal{H}(M_0)$.  Let $S\in \mathcal{S}_G(J^{-1}(0))$ be an orbit type stratum and $V\in \mfX(S/G)$ a compactly supported vector field. Then we can find an invariant lift $W\in \mfX(S)^G$. Fix $x\in S$ so that $\pi(x)$ lies in the support of $V$. Since $J^{-1}(0)$ is smoothly locally trivial, we can find an open neighbourhood $U\subseteq M$ of $x$ and a diffeomorphism to an open neighbourhood $\Omega\subseteq \bR^n\times \bR^m$ around $(0,0)$
$$
\phi:U\to \Omega\subseteq\bR^n\times \bR^m
$$ 
such that $\phi(x)=(0,0)$, $\phi(U\cap S)=\Omega\cap (\bR^n\times \{0\})$ and, if $S\subseteq \overline{R}$, then $\phi(U\cap R)=\Omega\cap (\bR^n\times L)$, where $L\subseteq \bR^m\setminus\{0\}$ is a submanifold. We can then push $W$ forward to $\Omega\cap (\bR^n\times \{0\})$ by $\phi$. The image of $W$ by $\phi$ can then be trivially extended to all of $\Omega$ by declaring
$$
(\phi_*W)_{(x,y)}=(\phi_*W)_{(x,0)}\in T_{(x,y)} \bR^n\times \bR^m. 
$$
Pushing back to $\Omega$, we have then extended $W$. Thus, in a neighbourhood of any point $x\in S$, we can extend the invariant lift $W$ to a neighbourhood of $x$ in $M$, which remains tangent to the strata of $J^{-1}(0)$. Using an invariant bump function, we can, after possibly shrinking $U$, obtain a $G$-invariant vector field $W\in\mfX(M,J^{-1}(0))^G$ with $\pi_*(W|_{U\cap S})=V|_{(U\cap S)/G}$. Hence the claim.
\end{proof}

\section{Singular Reduction of Cotangent Bundles}\label{sec: sing reduce cotangent}
In Section \ref{sec: non-sing reduction examples}, we discussed in great detail the symplectic reduction of cotangent bundles in the case where the base manifold has only one orbit-type. We now would like to generalize that discussion to the arbitrary case. This discussion is based on the work found in \cite{perlmutter_geometry_2007}. 

\ 

For all that follows, let $G$ be a connected Lie group, $Q$ a proper $G$-space and $J:T^*Q\to Q$ the Hamiltonian lift to the cotangent bundle. Write
\begin{itemize}
    \item $\tau_Q:T^*Q\to Q$ for the bundle projection map.
    \item $\theta_Q\in \Omega^1(T^*Q)$ for the Liouville $1$-form and $\omega_Q=-d\theta_Q$ for the canonical symplectic form.
\end{itemize}
Recall that for any $\alpha\in T^*Q$, we define $J(\alpha)\in \mfg^*$ by the pairing
$$
\bra J(\alpha),\xi\ket=\bra \alpha,\xi_Q(\tau(\alpha))\ket,
$$
where $\xi_Q\in \mfX(Q)$ is the fundamental vector field associated to $\xi\in \mfg^*$. In particular, since the fundamental vector fields span the tangent spaces to the orbits on $Q$, we have
$$
J^{-1}(0)=\bigsqcup_{q\in Q} (T_q(G\cdot q))^\circ.
$$
Hence, $J^{-1}(0)$ is a pseudobundle over $Q$. Note that since the orbit-type strata of $Q$ are saturated by $G$-orbits of the same dimension, it follows that for each orbit-type stratum $S\in \mathcal{S}_G(Q)$, the restriction $J^{-1}(0)|_S\to S$ is a $G$-equivariant vector bundle. Hence, $J^{-1}(0)$ can be partitioned canonically into equivariant smooth vector bundles
$$
J^{-1}(0)=\bigcup_{S\in \mathcal{S}_G(Q)}J^{-1}(0)|_S.
$$
However, it should be noted that despite the fact that this is a $G$-invariant partition of $J^{-1}(0)$ into embedded submanifolds, this is \textbf{not} the orbit-type stratification of $J^{-1}(0)$. Indeed, consider the following example.

\begin{egs}\label{eg: cotangent of r2}
    Consider the canonical action of $S^1$ on $\bR^2$ by rotations about the origin. Identifying $\bR^2=\bC$, this action is simply multiplication by elements of $S^1$:
    $$
    S^1\times \bR^2\to \bR^2;\quad (e^{it},x)\mapsto e^{it}x.
    $$
    Let $J:T^*\bR^2\to \bR$ be the Hamiltonian lift of this action. Using the canonical trivialization $T^*\bR^2=\bR^2\times \bR^2$, the $S^1$ action becomes
    $$
    e^{it}\cdot (x,y)=(e^{it}x,e^{it}y)
    $$
    and the momentum map becomes
    $$
    J(x,y)=x_1y_2-x_2y_1.
    $$
    where $x=(x_1,x_2)$ and $y=(y_1,y_2)$. Its easy to see that the orbit-type stratification of the base $\bR^2$ is given by
    $$
    \mathcal{S}_{S^1}(\bR^2)=\{\{0\}, \bR^2\setminus\{0\}\}
    $$
    Given the description of the momentum map, we see that restricting $J^{-1}(0)$ to the stratum $\{0\}\substeq \bR^2$ is 
    $$
    J^{-1}(0)|_{\{0\}}=\{0\}\times \bR^2.
    $$
    Clearly $J^{-1}(0)|_{\{0\}}$ is equivariantly isomorphic to $\bR^2$ and hence can be further subdivided into strata. 
\end{egs}

To obtain a nice description of the orbit-type stratification of $J^{-1}(0)$, let us introduce another pseudobundle.

\begin{defs}[Conormal Bundle]
    Define the \textbf{Conormal Bundle} to the orbit type stratification $\mathcal{S}_G(Q)$ on the base manifold to be the pseudobundle $\nu^*(\mathcal{S}_G(Q))$ defined fibre-wise by
    $$
    \nu^*_q(\mathcal{S}_G(Q)):=\bigcup_{q\in Q} (T_q\mathcal{S}_G(Q))^\circ\substeq T^*Q.
    $$
    Write $\nu^*(S):=\nu^*(\mathcal{S}_G(Q))|_S$ for any stratum $S\in \mathcal{S}_G(Q)$ and call $\nu^*(S)$ the conormal bundle to $S$.
\end{defs}
It's easy to see that if $S\in \mathcal{S}_G(Q)$ is an orbit-type stratum, then $\nu^*(S)$ is a $G$-equivariant vector bundle over $S$ and hence admits its own orbit-type stratification $\mathcal{S}_G(\nu^*(S))$. Let us now look at the local structure of both $J^{-1}(0)$ and $\nu^*(\mathcal{S}_G(Q))$. 

\ 

Recall the Slice Theorem (Theorem \ref{thm: slice theorem}), for any $q\in Q$ with $G_q=K$, there is a $G$-equivariant diffeomorphism between a neighbourhood of $q$ and a neighbourhood of $[e,0]$ in $G\times_K V$ for $V=\nu_q(Q,G)$. As we saw in Example \ref{eg: normal form for tangent bundle of proper space}, the cotangent bundle of slice neighbourhoods has the form
$$
T^*(G\times_K V)=G\times_K (V\times \mfk^\circ\times V^*)
$$
As we saw in Proposition \ref{prop: local normal form manifold of symmetry and orbit type}, if $S\in \mathcal{S}_G(Q)$ is the unique orbit-type stratum containing $q$, then in this neighbourhood, $S$ becomes identified with $G\times_K V^K$. Thus, we have
$$
T^*Q|_S=G\times_K(V^K\times \mfk^\circ\times V^*)
$$
and 
$$
T^*S=G\times_K (V^K\times \mfk^\circ\times (V^K)^*)
$$
Thus the canonical projection $T^*Q|_S\to T^*S$ is induced by the canonical projection $V^*\to (V^K)^*$ coming from the inclusion $V^K\into V^*$. Therefore, in this neighbourhood, we have
\begin{equation}\label{eq: local form of conormal bundle}
    \nu^*(S)=G\times_K(V^K\times \{0\}\times (V^K)^\circ).
\end{equation}

\begin{prop}
    Let $S\in \mathcal{S}_G(Q)$ be an orbit-type stratum and $\tau:\nu^*(S)\to S$ the conormal bundle. Suppose $S\subseteq Q_{(K)}$ for some compact $K\leq G$.
    \begin{itemize}
        \item[(1)] The orbit-type $\nu^*(S)_{(K)}$ is the image of the zero section $S\into \nu^*(S)$.
        \item[(2)] If $R\in \mathcal{S}_G(\nu^*(S))$ is an orbit-type stratum, then $R\substeq \nu^*(S)_{(H)}$ for some compact subgroup $H\leq K$. Furthermore, the restriction of the projection $\tau:\nu^*(S)\to S$ induces a $G$-equivariant fibre bundle $\tau|_R:R\to S$.
    \end{itemize}
\end{prop}

\begin{proof}
    \begin{itemize}
        \item[(1)] This follows from Equation (\ref{eq: local form of conormal bundle}). Indeed, the canonical projection $V^*\to (V^K)^*$ restricts to a linear isomorphism $(V^*)^K\to (V^K)^*$. Hence, the fixed point set of the annihilator $(V^K)^\circ$ under the action of $K$ is just the origin. In particular,
        $$
        \nu^*(S)_{(K)}=G\times_K(V^K\times \{0\}\times ((V^K)^\circ)^K)=G\times_K(V^K\times \{0\}\times \{0\})
        $$
        hence is the zero section.

        \item[(2)] Since the projection $\tau:\nu^*(S)\to S$ is equivariant, we must have $G_{\alpha}\subseteq G_{\tau(\alpha)}$ for any $\alpha\in \nu^*(S)$. Thus, it follows if $R\substeq \mathcal{S}_G(\nu^*(S))$ is an orbit-type stratum, then $R\substeq \nu^*(S)_{(H)}$ for some subgroup $H\leq K$. 

        \ 

        In the local form in Equation (\ref{eq: local form of conormal bundle}), we have
        $$
        \nu^*(S)_{(H)}=G\times_K(V^K\times \{0\}\times (V^K)^\circ_{(H)})
        $$
        The restriction of the projection $\tau:\nu^*(S)\to S$ locally has the form
        $$
        G\times_K(V^K\times \{0\}\times (V^K)^\circ_{(H)})\to G\times_K V^K;\quad [g,v,0,\lambda]\mapsto [g,v]
        $$
        Hence, restricting to connected components, we trivially get a $G$-equivariant fibre bundle over $S$. 
    \end{itemize}
\end{proof}

\begin{lem}\label{lem: hamiltonian structure pulls back to strata}
    Suppose $Q$ is equipped with a $G$-invariant metric and let $\iota_S:T^*S\into T^*Q|_S$ be the induced splitting of the short exact sequence
    $$
    0\to \nu^*(S)\to T^*Q|_S\to T^*S\to 0.
    $$
    Writing $J_S:T^*S\to S$ for the Hamiltonian lift of the $G$-action on $S$, we have the following.
    \begin{itemize}
        \item[(1)] $\iota_S^*\theta_Q=\theta_S$, where $\theta_S\in \Omega^1(T^*S)$ is the Liouville $1$-form.
        \item[(2)] The diagram commutes.
        $$
        \begin{tikzcd}
            T^*S\arrow[r,"J_S"]\arrow[d,"\iota_S"] & \mfg^*\\
            T^*Q\arrow[ur,"J"] &
        \end{tikzcd}
        $$
    \end{itemize}
\end{lem}

\begin{proof}
    \begin{itemize}
        \item[(1)] Let $x\in S$, $\alpha\in T_x^*S$, and $v\in T_\alpha(T^*S)$. Writing $\tau_S:T^*S\to S$ recall that
        $$
        \theta_S(v)=\bra \alpha,(\tau_S)_*v\ket.
        $$
        Now, we have
        $$
        \iota_S^*\theta_Q(v)=\theta_Q((\iota_S)_*v)=\bra \iota_S(\alpha),(\tau_Q)_*(\iota_S)_*v\ket
        $$
        Clearly $\tau_Q\circ \iota_S=\tau_S$ and hence 
        $$
        \bra \iota_S(\alpha),(\tau_Q)_*(\iota_S)_*v\ket=\bra \iota_S(\alpha),(\tau_S)_*v\ket.
        $$
        By definition of the splitting, since $(\tau_S)_*v\in TS$, we have
        $$
        \bra \iota_S(\alpha),(\tau_S)_*v\ket=\bra \alpha,(\tau_S)_*v\ket
        $$
        Hence the result.

        \item[(2)] Similar to (1).
    \end{itemize}
\end{proof}

Thus, if we define 
$$
T^*\mathcal{S}_G(Q):=\bigcup_{S\in \mathcal{S}_G(Q)}\iota_S(T^*S)
$$
we get the following.

\begin{prop}
    Suppose $Q$ has a $G$-invariant Riemannian metric. Then, $T^*Q=T^*\mathcal{S}_G(Q)\oplus \nu^*(\mathcal{S}_G(Q))$. That is, for each stratum $S\in \mathcal{S}_G(Q)$,
    $$
    T^*Q|_S=\iota_S(T^*S)\oplus \nu^*(S).
    $$
    Furthermore, 
    $$
    J^{-1}(0)|_S=J_S^{-1}(0)\oplus \nu^*(S)
    $$
\end{prop}

\begin{proof}
    Follows from the definition of $\iota_S$ and Lemma \ref{lem: hamiltonian structure pulls back to strata}.
\end{proof}

Now that we have these results in hand, we can state the following Theorem.

\begin{theorem}[{\cite{perlmutter_geometry_2007}}]\label{thm: strata for cotangent}
For every orbit-type stratum $S\in \mathcal{S}_G(Q)$ there exists a unique orbit-type stratum $Z_S\in \mathcal{S}_G(J^{-1}(0))$ with the following properties.
\begin{itemize}
    \item[(1)] If $\tau:T^*Q\to Q$ is the bundle projection, then $\tau(Z_S)=\overline{S}$. Furthermore, 
    $$
    \tau|_{Z_S}:Z_S\to \overline{S}
    $$
    is an open $G$-invariant surjective map.
    \item[(2)] For each $S,R\in \mathcal{S}_G(Q)$ with $S\leq R$, there exists an orbit-type stratum $C_S^R\in \mathcal{S}_G(\nu^*(S))$ so that 
    $$
    Z_R|_S=J_S^{-1}(0)\oplus C_S^R
    $$
    with $C_R^R\subseteq \nu^*(R)$ being the image of the zero section.
    \item[(3)] The map
    $$
    \mathcal{S}_G(Q)\to \mathcal{S}_G(J^{-1}(0));\quad S\mapsto Z_S
    $$
    is a bijection.
\end{itemize}
\end{theorem}

\begin{cor}
    Let $S\in \mathcal{S}_G(Q)$ be an orbit-type stratum of the base and $Z_S\in \mathcal{S}_G(J^{-1}(0))$ the corresponding stratum in Theorem \ref{thm: strata for cotangent}. Then $J_S^{-1}(0)\substeq Z_S$ as an open dense set.
\end{cor}

\begin{proof}
    Since $S$ is an open dense set in its closure $\overline{S}$, it follows from (1) that $Z_S|_S\substeq Z_S$ is an open dense set. From (2), $Z_S|_S=J_S^{-1}(0)\oplus \{0\}$. Hence the result.
\end{proof}

Now we would like to get a description of the reduced space. As we already know from Theorem \ref{thm: reduced symplectomorphism}, given an orbit-type stratum $S\in \mathcal{S}_G(Q)$, the reduced space $J_S^{-1}(0)/G$ is canonically symplectomorphic to the cotangent bundle $T^*(S/G)$. Since $J_S^{-1}(0)\substeq Z_S$ as an open dense set, we deduce the following.

\begin{lem}\label{lem: pullback of Liouville}
    Let $S\in \mathcal{S}_G(Q)$ be an orbit-type stratum and consider the canonical embedding $\phi:T^*(S/G)\into Z_S/G$ induced by the canonical symplectomorphism $T^*(S/G)\cong J_S^{-1}(0)/G$. Then there is a $1$-form $\alpha_S\in \Omega^1(Z_S)$ satisfying the following.
    \begin{itemize}
        \item[(1)] $(\pi_0|_{Z_S})^*\alpha_S=\theta_Q|_{Z_S}$.
        \item[(2)] $d\alpha_S=-\omega_{Z_S/G}$ where $\omega_{Z_S/G}\in \Omega^2(Z_S/G)$ is the canonical symplectic form coming from singular reduction.
        \item[(3)] $\phi^*\alpha_S=\theta_{S/G}$, where $\theta_{S/G}\in \Omega^1(T^*(S/G))$ is the Liouville $1$-form.
    \end{itemize}
\end{lem}

\begin{proof}
    \begin{itemize}
        \item[(1)] $\theta_Q|_{Z_S}$ is equivariant and vanishes on the $G$-orbits contained in $Z_S$. Hence, it immediately follows that $\theta_Q|_{Z_S}$ is basic. Hence, there exists $\alpha_S\in \Omega^1(Z_S/G)$ so that
        $$
        (\pi_0|_{Z_S})^*\alpha_S=\theta_Q|_{Z_S}
        $$

        \item[(2)] Observe that $(\pi_0|_{Z_S})^*d\alpha_S=d\theta_Q|_{Z_S}=-\omega_Q|_{Z_S}$. $\omega_{Z_S/G}$ is the unique $2$-form satisfying $(\pi_0|_{Z_S})^*\omega_{Z_S/G}=\omega_Q|_{Z_S}$ hence it follows.

        \item[(3)] Write $\tau_S:T^*S\to S$ for the bundle map and $\pi_S:S\to S/G$ for the quotient. Recall that the symplectomorphism $\phi:T^*(S/G)\to J_S^{-1}(0)/G$ was induced by the canonical map
        $$
        \widetilde{\phi}:\pi_S^{-1}T^*(S/G)\to J_S^{-1}(0)
        $$
        given by dualizing the derivative map $T\pi:TS\to \pi_S^{-1}T(S/G)$. Furthermore, recall from Lemma \ref{lem: reduced 1-form is canonical} that
        $$
        \pi_S^{-1}T^*(S/G)\substeq S\times T^*(S/G)
        $$
        and that we showed $\widetilde{\phi}^*\theta_S=pr_2^*\theta_{S/G}$, where $\theta_S\in \Omega^1(T^*S)$ is the Liouville $1$-form. Applying (1) then finishes the proof.
    \end{itemize}
\end{proof}

Thus, for each stratum $S\in \mathcal{S}_G(Q)$, the corresponding canonical stratum $Z_S/G\in \mathcal{S}_G((T^*Q)_0)$ contains the cotangent bundle $T^*(S/G)$ as an open dense set. Unless $S$ is a minimal stratum, $Z_S/G$ will be strictly larger than the cotangent bundle of $S/G$, as the following result shows.

\begin{theorem}[{\cite{perlmutter_geometry_2007}}]
Fix an orbit-type strata $S,R\in \mathcal{S}_G(M)$ satisfying $R<S$. Then $(Z_S|_R)/G$ is coisotropic in $Z_S/G$.
\end{theorem}

\cleardoublepage
\chapter{Singular Reduction of Prequantum Line Bundles}\label{ch: sing prequantum}

This chapter is dedicated to developing an elementary theory of stratified prequantum line bundles. Sniatycki \cite{sniatycki_differential_2013} laid much of the groundwork for this chapter, although my work will differ from his in imposing a smoothness condition of the prequantum connection. 

\section{Stratified Prequantum Line Bundles}\label{sec: strat prequantum}

We now define a prequantum line bundle over a symplectic stratified space. There is a natural definition that we could do as a generalization of prequantum line bundles of Poisson manifolds given by Vaisman \cite{vaisman_geometric_1991}. The only technical issue we unfortunately have to deal with is that we won't be able to to use every vector field when defining the prequantum connection. So, we will have to restrict ourselves to using Hamiltonian modules.

\begin{defs}[Stratified Prequantum Line Bundle]\label{def: strat prequantum}
Let $(X,\Sigma,\{\cdot,\cdot\})$ be a stratified symplectic space and $\mathcal{H}\subseteq \mfX(X)$ a Hamiltonian module. A $\mathcal{H}$ prequantum line bundle consists of a triple $(\mathcal{L},\nabla,\{\nabla^S\}_{S\in\Sigma})$ where
\begin{itemize}
    \item[(i)] $\mathcal{L}\to X$ is a differentiable complex line bundle.
    \item[(ii)] $\nabla:T\Sigma\times\Gamma(\mathcal{L})\to \mathcal{L}$ is a bilinear mapping such that for any $v\in T_x\Sigma$, $f\in C^\infty(X,\bC)$, and $\sigma\in \Gamma(\mathcal{L})$ we have
    $$
    \nabla_v(f\sigma)=v(f)\sigma(x)+f(x)\nabla_v\sigma
    $$
    \item[(iii)] $\nabla^S$ is a prequantum connection on $\mathcal{L}|_S\to S$ for each stratum $S\in \Sigma$
\end{itemize}
Satisfying
\begin{itemize}
    \item[(1)] For any $V\in \mathcal{H}$ and $\sigma\in \Gamma(\mathcal{L})$, the map
    $$
    \nabla_V\sigma:X\to \mathcal{L};\quad \mapsto \nabla_{V_x}\sigma
    $$
    is an element of $\Gamma(\mathcal{L})$.
    \item[(2)] For every $\sigma\in\Gamma(\mathcal{L})$ and $f,g\in C^\infty(X)$ we have
    $$
    [\nabla_{V_f},\nabla_{V_g}]\sigma-\nabla_{V_{\{f,g\}}}\sigma=-i \{f,g\}\sigma
    $$
    \item[(3)] If $S\in\Sigma$ is a stratum, $V\in \mathcal{H}(\Sigma)$, then
    $$
    (\nabla_{V}\sigma)|_S=\nabla^S_{V|_S}(\sigma|_S).
    $$
\end{itemize}
\end{defs}

\begin{egs}
If $(M,\omega)$ is a symplectic manifold and $\Sigma$ a stratification of $M$ by symplectic submanifolds, then a stratified prequantum line bundle over $(M,\Sigma,\{\cdot,\cdot\})$ is the same thing as a usual prequantum line bundle.
\end{egs}

\begin{egs}
Somewhere we could get potentially new and interesting is with Poisson manifolds where the symplectic foliation defines a stratification. Here the definition appears similar to that of a prequantum line bundle of a Poisson manifold as in \cite{vaisman_geometric_1991}, but there a prequantum line bundle need not quantize each symplectic leaf. 
\end{egs}

\begin{egs}\label{eg: non-trivial prequantum}
    Let 
    $$
    \pi_{can}=\frac{\partial}{\partial x}\wedge \frac{\partial}{\partial y}
    $$
    be the canonical symplectic Poisson structure on $\bR^2$ and let $\pi=r^2\pi_{can}$ as in Example \ref{eg: not symplectic Poisson}. As we saw in Example \ref{eg: first non-trivial symplectic stratified space}, the symplectic foliation $\mathscr{F}_\pi=\{\{0\},\bR^2\setminus\{0\}\}$ is a symplectic stratification of $\bR^2$ and in Example \ref{eg: non-trivial Ham module} the set of vector fields $\mathcal{H}$ vanishing to order $2$ at the origin defines a Hamiltonian module. Let us now use this module to define a prequantum line bundle.

    \ 

    Let $(r,\theta)$ be polar coordinates on $\bR^2$ defined by
    \begin{align*}
        x&= r\cos(\theta)\\
        y& =r\sin(\theta)
    \end{align*}
    On the stratum $S=\bR^2\setminus\{0\}$, let $d\theta\in \Omega^1(S)$ be the $1$-form defined by
    $$
    d\theta=\frac{-ydx+xdy}{r^2}.
    $$
    Observe that on $S$, we have
    $$
    d\bigg(\frac{r^2}{2}d\theta\bigg))=rdr\wedge d\theta=dx\wedge dy=\omega_{can},
    $$
    the canonical symplectic form on $\bR^2$. In Example \ref{eg: not symplectic Poisson} we saw that the symplectic form $\omega_S$ on $S$ induced by $\pi$ has the form
    $$
    \omega_S=\frac{\omega_{can}}{r^2}=\frac{dr\wedge d\theta}{r}
    $$
    Thus, we may define primitive $\alpha_S\in \Omega^1(S)$ by
    $$
    \alpha_S=\ln(r)d\theta.
    $$
    With this set-up we can now define the prequantum line bundle.

    \ 

    Let $\mathcal{L}=\bR^2\times \bC$ and identify $\Gamma(\mathcal{L})=C^\infty(\bR^2,\bC)$. Define the map
    $$
    \nabla:\mathcal{H}\times C^\infty(\bR^2,\bC)\to \bC
    $$
    by
    $$
    \nabla_vf=
    \begin{cases}
        0,\quad &v\in T_0\mathscr{F}_\pi\\
        v(f)+i\alpha(v)f(x),\quad & v\in T_xS
    \end{cases}
    $$
    for $f\in C^\infty(\bR^2,\bC)$. We can readily see that $\nabla$ induces prequantum connections on each of the strata. The connection $\nabla$ also pairs smoothly with elements of the Hamiltonian module. Indeed, if $V\in \mathcal{H}$, then by definition, there exists a neighbourhood $U\substeq \bR^2$ of the origin and a vector field $W\in \mfX(U)$ so that
    $$
    V|_U=r^2W|_U.
    $$
    Thus, over $U\setminus\{0\}$, we have
    $$
    \nabla_Vf=V(f)+i\alpha(V)f
    $$
    Observe that 
    $$
    \alpha(V)=\alpha(r^2W)=r^2\ln(r)d\theta(W)=\ln(r)(-ydx(W)+xdy(W))
    $$
    Writing 
    $$
    W=W_x\frac{\partial}{\partial x}+W_y\frac{\partial}{\partial y}
    $$
    we then obtain that
    $$
    \alpha(V)=\ln(r)(-yW_x+xW_y)
    $$
    Using that $x=r\cos(\theta)$ and $y=r\sin(\theta)$, we can then show that
    $$
    \lim_{r\to 0^+}\ln(r)(-yW_x+xW_y)=0.
    $$
    In particular, it follows that $\alpha(V)$ can be smoothly extended to the origin by $0$. In particular, $\nabla_Vf$ can be smoothly extended to the origin by $0$. Hence, $\nabla_Vf\in C^\infty(\bR^2,\bC)$. 
\end{egs}

Now just as with usual prequantum line bundles, we get a prequantum operator here as follows. Fix $f\in C^\infty(X,\bC)$ and define
\begin{equation}\label{eq: stratified prequantum operator}
\mathcal{Q}{pre}(f):\Gamma(\mathcal{L})\to \Gamma(\mathcal{L});\quad \sigma\mapsto -i\nabla_{V_f}\sigma+f\sigma.
\end{equation}

\begin{prop}
The prequantum operator in Equation (\ref{eq: stratified prequantum operator}) satisfies the following properties.
\begin{itemize}
    \item[(1)] $\mathcal{Q}_{pre}(1_X)=\id_{\Gamma(\mathcal{L})}$, where $1_X\in C^\infty(X)$ is the constant $1$ function.
    \item[(2)] If $f,g\in C^\infty(X)$, then 
    $$
    i[\mathcal{Q}_{pre}(f),\mathcal{Q}_{pre}(g)]=\mathcal{Q}_{pre}(\{f,g\}).
    $$
\end{itemize}
\end{prop}

\begin{proof}
\begin{itemize}
    \item[(1)] First, observe that $V_{1_X}=0$. Indeed, taking any $f\in C^\infty(X)$ we have
    $$
    \{1_X,f\}=\{1_X\cdot 1_X,f\}=2\{1_x,f\}
    $$
    which implies $\{1_x,f\}=0$ and hence $V_{1_x}=0$. Thus, since $\nabla_{V_{1_x}}$ is defined point-wise, it follows that $\nabla_{V_{1_x}}=0$. Thus, for any section $\sigma\in \Gamma(\mathcal{L})$ we have
    $$
    \mathcal{Q}_{pre}(1_X)\sigma=-i\nabla_{V_{1_X}}+1_X\sigma=\sigma.
    $$
    \item[(2)] Identical proof to Proposition \ref{prop: prequantum map is a lie algebra hom}.
\end{itemize}
\end{proof}

A final property of stratified prequantum line bundles is that we can glue together the various prequantum connections on each stratum to get a differentiable map between the complexified stratified tangent bundle and the prequantum line bundle.

\begin{prop}\label{prop: stratified prequantum applied to section}
For each stratum $S\in \Sigma$ and $\sigma\in \Gamma(\mathcal{L})$, we get a smooth map of complex vector bundles
$$
\nabla^S\sigma:T^\bC S\to \mathcal{L};\quad v\mapsto \nabla^S_v\sigma
$$
These piece together to a (not necessarily differentiable) map of complex pseudo bundles
\begin{equation}\label{eq: stratified connection applied to section}
\nabla\sigma:T^\bC\Sigma\to \mathcal{L}
\end{equation}
defined by
$$
\nabla\sigma|_{T^\bC S}=\nabla^S\sigma.
$$
\end{prop}

\begin{proof}
The first statement is simply Corollary \ref{cor: prequantum connection applied to sections} applied to each stratum. 
\end{proof}

\begin{remark}
I suspect the map in Equation (\ref{eq: stratified connection applied to section}) actually is differentiable, but my attempts to prove it thus far have proven fruitless.
\end{remark}

\section{The Singular Reduction of Prequantum Line Bundles}\label{sec: sing reduce prequantum}

We now show that if a prequantum line bundle is generated by its equivariant sections, then it descends to a stratified prequantum line bundle over the singular reduced space. Let us now fix a connected Lie group $G$, a proper Hamiltonian $G$-space $J:(M,\omega)\to \mfg^*$, and an equivariant prequantum line bundle $(\mathcal{L},\nabla)\to (M,\omega)$. Using ideas from Section \ref{sec: [Q,R]}, we can produce singular versions of the prequantum line bundle reduction of Guillemin and Sternberg.

\begin{prop}
Suppose $(\mathcal{L},\nabla)$ is equivariantly generated over $J^{-1}(0)$. That is, the map
$$
J^{-1}(0)\times \Gamma(\mathcal{L})^G\to \mathcal{L}|_{J^{-1}(0)};\quad (x,\sigma)\mapsto \sigma(x)
$$
is surjective. Then
\begin{itemize}
    \item[(i)] $\mathcal{L}_0:=(\mathcal{L}|_{J^{-1}(0)})/G$ is naturally a differentiable complex line bundle over $J^{-1}(0)/G$.
    \item[(ii)] The projection $\pi_\mathcal{L}:\mathcal{L}|_{J^{-1}(0)}\to \mathcal{L}_0$ induces a bijection
    $$
    (\pi_\mathcal{L})_!:\Gamma(\mathcal{L}|_{J^{-1}(0)})^G\to \Gamma(\mathcal{L}_0).
    $$
    \item[(iii)] The map
    $$
    \kappa:\Gamma(\mathcal{L})^G\to \Gamma(\mathcal{L}_0);\quad \sigma\mapsto (\pi_\mathcal{L})_!(\sigma|_{J^{-1}(0)})
    $$
    is a surjection.
\end{itemize}
\end{prop}

\begin{proof}
Each of these statements are just special cases of Theorem \ref{thm: equivariantly generated implies quotient is vector bundle}.
\end{proof}

We will now attempt to use $\kappa$ to define a prequantum connection on $\mathcal{L}_0$. Namely, given an equivariant section $\sigma\in \Gamma(\mathcal{L})^G$ and a stratified tangent vector $v\in T\mathcal{S}_G(J^{-1}(0))$, the projection $(\pi_0)_*v\in T\mathcal{S}_G(M_0)$ is a stratified tangent vector and $\kappa(\sigma)\in \Gamma(\mathcal{L}_0)$ is a section of the reduced bundle. Hence, we would like to define 
\begin{equation}\label{eq: reduced prequantum connection}
\nabla^0_{(\pi_0)_*v}\kappa(\sigma):=\pi_{\mathcal{L}}(\nabla_v\sigma).
\end{equation}
Working stratum-by-stratum and applying Theorem \ref{thm: can reduce non-singular} from the non-singular case, we prove the following.

\begin{lem}[{\cite{sniatycki_differential_2013}}]\label{lem: prequantum connection definition}
    Equation (\ref{eq: reduced prequantum connection}) gives a well-defined map
    $$
    \nabla^0:T\mathcal{S}_G(M_0)\times \Gamma(\mathcal{L}_0)\to \mathcal{L}_0;\quad (v,\sigma)\mapsto \nabla^0_v\sigma.
    $$
\end{lem}

Recall that we define the Hamiltonian module of the singular quantization $\mathcal{H}(M_0)$ to be the image of the map
$$
\pi_*:\mfX(M,J^{-1}(0))^G\to \mfX(\mathcal{S}_G(M_0))
$$
given by restricting vector fields to $J^{-1}(0)$, then point-wise differentiating. We would now like to show for any $V\in \mathcal{H}(M_0)$ and section $\sigma\in \Gamma(\mathcal{L}_0)$, that the natural pairing $\nabla^0_V\sigma$ is a smooth section of $\mathcal{L}_0$. First, we need a Lemma.

\begin{lem}\label{lem: agree on J^-1 means agree below}
Suppose $(\mathcal{L},\nabla)$ is equivariantly generated over $J^{-1}(0)$. Fix $\sigma_1,\sigma_2\in \Gamma(\mathcal{L})^G$ and $V_1,V_2\in\mfX(M,J^{-1}(0))^G$. If
$$
\pi_*(V_1)=\pi_*(V_2)\text{ and }\sigma_1|_{J^{-1}(0)}=\sigma_2|_{J^{-1}(0)},
$$
then 
$$
\nabla_{V_1}\sigma_1|_{J^{-1}(0)}=\nabla_{V_2}\sigma_2|_{J^{-1}(0)}.
$$
\end{lem}

\begin{proof}
Utilizing a similar trick as in Lemma \ref{lem: non-singular reduction of line bundle}, we have
\begin{align*}
\nabla_{V_1}\sigma_1-\nabla_{V_2}\sigma_2=\nabla_{V_1-V_2}\sigma_1+\nabla_{V_2}(\sigma_1-\sigma_2).
\end{align*}
Since $\mfX(M,J^{-1}(0))^G$ and $\Gamma(\mathcal{L})^G$ are $C^\infty(M)^G$ modules, it suffices to show 
$$
\nabla_V\sigma|_{J^{-1}(0)}=0
$$
for all $V\in \mfX(M,J^{-1}(0))^G$ and $\sigma\in\Gamma(\mathcal{L})^G$ whenever $\pi_*(V)=0$ or $\sigma|_{J^{-1}(0)}=0$. Working stratum-by-stratum, this reduces to the proof Theorem of part (i) of \ref{thm: non-singular [Q,R]=0}. 
\end{proof}

\begin{cor}\label{cor: Hamiltonian guys smooth when paired with sections}
    Given any vector fields $V\in \mathcal{H}(M_0)$ and any section $\sigma\in\Gamma(\mathcal{L}_0)$, the pairing $\nabla^0_V\sigma$ is a smooth section. More specifically, if $W\in \mfX(M,J^{-1}(0))^G$ and $\tau\in \Gamma(\mathcal{L})^G$ satisfy $\pi_*W=V$ and $\kappa(\tau)=\sigma$, then
    $$
    \nabla^0_V\sigma=\kappa(\nabla_W\tau).
    $$
\end{cor}

\begin{proof}
    Fixing $V,W,\sigma,\tau$ as in the statement, let $x\in J^{-1}(0)$ and $[x]\in M_0$ its image under the quotient map. Then, by Lemma \ref{lem: prequantum connection definition}
    $$
    \nabla^0_{V_{[x]}}\sigma=\pi_{\mathcal{L}}(\nabla_{W_x}\tau)
    $$
    By definition, $\kappa(\nabla_W\tau)([x])=\pi_{\mathcal{L}}(\nabla_{W_x}\tau)$
\end{proof}

\begin{cor}
    Given any two functions $f,g\in C^\infty(M_0)$ and any section $\sigma\in \Gamma(\mathcal{L}_0)$ we have
    $$
    [\nabla^0_{V_f},\nabla^0_{V_g}]\sigma-\nabla^0_{V_{\{f,g\}_0}}=-i\{f,g\}_0\sigma.
    $$
\end{cor}

\begin{proof}
    Let $F,G\in C^\infty(M)^G$ be lifts of $f$ and $g$ and $\tau\in \Gamma(\mathcal{L})^G$ be a lift of $\sigma$. Note that $\pi_*V_F=V_f$ and $\pi_*V_G=V_g$. Hence, successive applications of Corollary \ref{cor: Hamiltonian guys smooth when paired with sections} yields
    $$
    \nabla^0_{V_f}(\nabla^0_{V_g}\sigma)=\nabla^0_{V_f}\kappa(\nabla_{V_G}\tau)=\kappa(\nabla_{V_F}\nabla_{V_G}\tau)
    $$
    Thus,
    \begin{align*}
        [\nabla^0_{V_f},\nabla^0_{V_g}]\sigma-\nabla^0_{V_{\{f,g\}_0}}&=\kappa([\nabla_{V_F},\nabla_{V_G}]\tau-\kappa(\nabla_{V_{\{F,G\}}}\tau)\\
        &=\kappa(-i\{F,G\}\tau)\\
    \end{align*}
    Note for any $x\in J^{-1}(0)$, if $[x]\in M_0$ is its image under the quotient map $\pi_0:J^{-1}(0)\to M_0$, then by the fibre-wise linearity of $\pi_\mathcal{L}$ and the definition of the Poisson bracekt on $M_0$, we obtain
    \begin{align*}
        \kappa(-i\{F,G\}\tau)([x])&=\pi_\mathcal{L}(-i\{F,G\}(x)\tau(x))\\
        &=-i\{F|_{J^{-1}(0)},G|_{J^{-1}(0)}\}_{J^{-1}(0)}(x)\pi_\mathcal{L}(\tau(x))\\
        &=-i\{f,g\}_0\sigma(x)
    \end{align*}
    Hence the result.
\end{proof}

Now we need to establish the stratum-wise prequantum line bundles. 

\begin{prop}\label{prop: manifold of symmetry is hamiltonian}
Let $K\leq G$ be a compact subgroup and let $M_K$ denote the manifold of symmetry. Let $N_G(K)\leq G$ for the normalizer subgroup of $K$, $\mfn_G(K)=\text{Lie}(N_G(K))$ for its Lie algebra, and write
$$
J_K:M_K\to \mfn_G(K)^*
$$
for the induced Hamiltonian free Hamiltonian $N_G(K)$-space structure. Then $\mathcal{L}|_{M_K}\to M_K$ is canonically an equivariantly generated $N_G(K)$-equivariant prequantum line bundle over $J_K^{-1}(0)$.
\end{prop}

\begin{proof}
Since $M_K$ is a symplectic submanifold, the restriction $\mathcal{L}|_{M_K}\to M_K$ is automatically a prequantum line bundle. Since $\mathcal{L}$ is $G$-equivariant, it automatically follows that $\mathcal{L}|_{M_K}$ is $N_G(K)$ equivariant. Finally, to show $\mathcal{L}|_{M_K}$ is equivariantly generated over $J_K^{-1}(0)$, let $x\in J)K^{-1}(0)$. Then,
$$
(\mathcal{L}|_{M_k})_x^{N_G(K)_x}=\mathcal{L}_x^{N_G(K)_x}=\mathcal{L}_x^K=\mathcal{L}_x.
$$
Hence the result.
\end{proof}

\begin{cor}\label{cor: prequantum line bundle over the reduced strata}
Suppose $(\mathcal{L},\nabla)$ is equivariantly generated over $J^{-1}(0)$ and  $S\in \mathcal{S}_G(J^{-1}(0))$ be an orbit-type stratum, then $(\mathcal{L}|_S)/G\to S/G$ is canonically a prequantum line bundle over $S/G$ with prequantum connection $\nabla^S$ satisfying
$$
\pi_0^{-1}\nabla^S=\nabla|_S.
$$
\end{cor}

\begin{proof}
By Proposition \ref{prop: manifold of symmetry is hamiltonian}, $\mathcal{L}_0=(\mathcal{L}|_{J^{-1}(0)})/G$ is a differentiable complex line bundle over $J^{-1}(0)$. In particular, this implies the restriction to $S/G$, $\mathcal{L}_0|_{S/G}=(\mathcal{L}|_S)/G$ is a smooth complex line bundle over $S/G$. Using nearly the exact same proof as in Theorem \ref{thm: non-singular [Q,R]=0} part (ii), we can easily show that $\nabla|_S$ induces a connection $\nabla^S$ on $\mathcal{L}_0$ which satisfies
$$
\pi_0^{-1}\nabla^S=\nabla|_S.
$$
To show $\nabla^S$ is prequantum, fix $x\in S$ and write $K=G_x$. Then by Theorem \ref{thm: manifold of symmetry reduction gives strata}, the map
$$
J_K^{-1}(0)\into S
$$
induces a symplectomorphism
\begin{equation}
(M_K)_0=J_K^{-1}(0)/N_G(K)\to S/G.
\end{equation}
By Proposition \ref{prop: manifold of symmetry is hamiltonian}, $\mathcal{L}|_{M_K}\to M_K$ is equivariantly generated over $J_K^{-1}(0)$ and hence by Theorem \ref{thm: non-singular [Q,R]=0}, 
$$
(\mathcal{L}|_{M_K})_0:=(\mathcal{L}|_{J_K^{-1}(0)})/N_G(K)
$$
is canonically a complex line bundle over $(M_K)_0$ with prequantum connection $\nabla^K$ satisfying
$$
\pi_K^{-1}\nabla^K=\nabla|_{J_K^{-1}(0)},
$$
where $\pi_K:J_K^{-1}(0)\to (M_K)_0$ is the quotient map. It's then a straightforward matter of chasing diagrams to show the map of line bundles
$$
\mathcal{L}|_{J_K^{-1}(0)}\to \mathcal{L}|_S
$$
induces an isomorphism of line bundles
$$
(\mathcal{L}|_{M_K})_0\to \mathcal{L}_0|_{S/G}
$$
and that the pullback of $\nabla^S$ under this map is the same as $\nabla^K$. The result then follows.
\end{proof}

\begin{theorem}\label{thm: stratified reduction of prequantum}
Suppose $(\mathcal{L},\nabla)$ is an equivariantly generated prequantum line bundle over $J^{-1}(0)$ and let
$$
\mathcal{L}_0:=(\mathcal{L}|_{J^{-1}(0)})/G.
$$
Then there exists a canonical $\mathcal{H}_G(M_0)$-connection $\nabla^0$ and prequantum connections $\nabla^S$ on each $\mathcal{L}|_S\to S$, $S\in \mathcal{S}_G(M_0)$, such that $(\mathcal{L}_0,\nabla^0,\{\nabla^S\}_{S\in \mathcal{S}_G(M_0)})$ is a stratified prequantum line bundle over $M_0$.
\end{theorem}

\begin{proof}
Apply Theorem \ref{thm: Hamiltonian module for singular reduction} and Corollary \ref{cor: prequantum line bundle over the reduced strata}.
\end{proof}

\begin{egs}\label{eg: contangent prequantum still reducible for singular cotangent}
    Just like the non-singular case, the reduction of prequantum line bundles over cotangent bundles is particularly clear. Let $G$ be a connected Lie group, $Q$ a proper $G$-space, and $J:T^*Q\to \mfg^*$ the Hamiltonian lift. Writing $\theta_Q\in \Omega^1(T^*Q)$ for the Liouville $1$-form, the trivial bundle $\mathcal{L}=T^*Q\times \bC$ together with the connection $\nabla=L-i\theta_Q$ is an equivariant prequantum line bundle over $T^*Q$. Since $G$ only acts on the first factor of $\mathcal{L}$, it trivially follows that $\mathcal{L}$ is singularly reducible. Indeed, we have
    $$
    \mathcal{L}_0=(\mathcal{L}|_{J^{-1}(0)})/G=(T^*Q)_0\times \bC.
    $$

    \ 

    Now, recall from Theorem \ref{thm: strata for cotangent} that if $S\in \mathcal{S}_G(Q)$ is an orbit-type stratum of $Q$ then there is a unique orbit-type stratum $Z_S\in \mathcal{S}_G(J^{-1}(0))$ which contains $J_S^{-1}(0)$ as an open dense set, where $J_S:T^*S\to \mfg^*$ is the Hamiltonian lift of the $G$ action on $S$. Furthermore, we recall that the quotient $Z_S/G$ has a canonical $1$-form $\alpha_S\in \Omega^1(Z_S/G)$ so that $d\alpha_S=-\omega_S$, where $\omega_S\in \Omega^2(Z_S/G)$ is the reduced symplectic form and so that
    $$
    (\pi_0|_{Z_S})^*\alpha_S=\theta_Q|_{Z_S}
    $$
    for the reduction map $\pi_0:J^{-1}(0)\to (T^*Q)_0$. It then follows immediately that the reduced prequantum connection $\nabla^0$ on $\mathcal{L}_0$ restricts to the prequantum connection
    $$
    \nabla^0=L-i\alpha_S
    $$
    on $Z_S/G$. 
\end{egs}

As a final remark, I would like to speculate on future work. Let $\mfg_{J^{-1}(0)}\subseteq T\mathcal{S}_G(J^{-1}(0))$ be the subbundle of the stratified tangent bundle with fibres given by tangent spaces to orbits. It's straightforward to show that $\pi_*:T\mathcal{S}_G(J^{-1}(0))\to T\mathcal{S}_G(M_0)$ induces a differentiable equivariant map between pseudo bundles
$$
\pi_*:T\mathcal{S}_G(J^{-1}(0))/\mfg_{J^{-1}(0)}\to T\mathcal{S}_G(M_0)
$$
which is an isomorphism on each fibre. Taking by a quotient by $G$, we get a map
\begin{equation}\label{eq: stratified tangent given by quotient map}
\pi_*:(T\mathcal{S}_G(J^{-1}(0))/\mfg_{J^{-1}(0)})/G\to T\mathcal{S}_G(M_0).
\end{equation}
I conjecture that this map is in fact an isomorphism. If this is the case, then we see that we can extend $\nabla^0$ to all stratified vector fields.

\begin{cor}
Suppose the map in Equation (\ref{eq: stratified tangent given by quotient map}) is an isomorphism of pseudo bundles. Then for each $\sigma\in \Gamma(\mathcal{L}_0)$, then map
$$
\nabla\sigma:T^\bC \mathcal{S}_G(M_0)\to \mathcal{L}_0
$$
as in Proposition \ref{prop: stratified prequantum applied to section} is a differentiable map of complex pseudo bundles. Hence, $\nabla$ can be extended to a full connection
$$
\nabla:\mfX(\mathcal{S}_G(M_0))\times \Gamma(\mathcal{L}_0)\to \Gamma(\mathcal{L}_0).
$$
\end{cor}

\section{Singular Reduction of Prequantum Line Bundles over Toric Manifolds}\label{sec: sing prequantum toric}

Recall the construction of a prequantum line bundle on a toric manifold from Section \ref{section: non-singular quantization toric manifold}. Let $\Delta\subseteq \bR^n$ be a Delzant polytope defined by
$$
\Delta=\{x\in \bR^n \  | \ \bra x,v_i\ket\geq \lambda_i\},
$$
where $v_1,\dots,v_N\in \bR^n$ are the vertices and let
$$
J:\bC^N\to \bR^N;\quad (z_1,\dots,z_N)\mapsto \frac{1}{2}(|z_1|^2,\dots,|z_n|^2)-\lambda,
$$
where $\lambda=(\lambda_1,\dots,\lambda_N)$. As we saw, $\bC^N$ together with $J$ and the canonical $T^N$ action is a Hamiltonian $T^N$-space. Further recall that if we let $\mfk\subseteq \bR^N$ be the kernel of the linear map
\begin{equation}\label{eq: infinitesimal torus homomorphism}
\bR^N\to \bR^n;\quad e_i\mapsto v_i,
\end{equation}
$\iota:\mfk\into \bR^N$ the inclusion, and $K=\exp(\mfk)\subseteq T^N$, then $M_\Delta=(\iota^*\circ J)^{-1}(0)/K$ is canonically a Hamiltonian $T^N$ space with an induced effective $T^n$ action via the homomorphism $T^N\to T^n$ induced by Equation (\ref{eq: infinitesimal torus homomorphism}). For our own convenience, we will be viewing $M_\Delta$ as a Hamiltonian $T^N$ space.

\

As for the prequantum line bundle, recall that $\bC^N$ has a canonical $T^N$ equivariant prequantum line bundle $(\mathcal{L}_{can},\nabla^{can})\to \bC^N$, where
$$
\mathcal{L}_{can}=\bC^N\times \bC
$$
and 
$$
\nabla^{can}=d+i \sum_{j=1}^N \frac{r_j^2}{2}d\theta_j
$$
where $(r_j,\theta_j)$ are polar coordinates. As we saw, if both $\lambda$ and $\iota^*(\lambda)$ are integral, then we can equip $\mathcal{L}_{can}$ with a $T^N$ action via the homomorphism
\begin{equation}\label{eq: toric prequantum homomorphism}
    \rho:T^N\to S^1;\quad (e^{i\theta_1},\dots,e^{i\theta_N})\mapsto e^{-i\bra \lambda,(\theta_1,\dots,\theta_N)\ket}
\end{equation}
With this action, $(\mathcal{L}_{can},\nabla^{can})$ is reducible over $J^{-1}(0)$ and hence induces a prequantum line bundle 
$$
\mathcal{L}_\Delta:=(\mathcal{L}_{can}|_{(\iota^*\circ J)^{-1}(0)})/K
$$
over $M_\Delta$ with prequantum connection $\nabla^\Delta$. This is a $T^N$-equivariant (and hence $T^n$-equivariant) prequantum line bundle. Similar to our discussion of $M_\Delta$, we will be viewing $\mathcal{L}_\Delta$ as a $T^N$-equivariant prequantum line bundle for reasons of convenience.

\

Given that $(\mathcal{L}_\Delta,\nabla^\Delta)\to M_\Delta$ is $T^N$-equivariant, it automatically follows that it is also $H$-equivariant for any subgroup $H\leq T^N$. That is, if we fix $H\leq T^N$ and let
\begin{equation}
    J_{\Delta,H}:M_\Delta\to \mfh^*
\end{equation}
be the induced momentum map, then $(\mathcal{L}_\Delta,\nabla^\Delta)\to M_\Delta$ is an $H$-equivariant prequantum line bundle. We would like to now explore for which subgroups $\mathcal{L}_\Delta$ is equivariantly generated over $J_{\Delta,H}^{-1}(0)$.

\begin{theorem}\label{thm: reducible prequantum for toric}
    Let $H\leq T^N$ be a subtorus and write $\mfh=\text{Lie}(H)$. Further, write
    $$
    J_{\Delta,H}:M_\Delta\to \mfh^*
    $$
    for the induced momentum map. If $H\leq \ker(\rho)$, where $\rho:T^N\to S^1$ is the homomorphism if Equation (\ref{eq: toric prequantum homomorphism}), then $\mathcal{L}_\Delta\to M_\Delta$ is singularly reducible over $J_{\Delta,H}^{-1}(0)$.
\end{theorem}

\begin{proof}
Let $x\in (\iota^*\circ J)^{-1}(0)$ and $[x]\in M_\Delta$ its image under the quotient map $(\iota^*\circ J)^{-1}(0)\to M_\Delta$. We show that
$$
(\mathcal{L}_\Delta)_{[x]}^{H_{[x]}}=(\mathcal{L}_\Delta)_{[x]}
$$
To show this, let 
$$
\pi_\mathcal{L}:(\iota\circ J)^{-1}(0)\times \bC\to \mathcal{L}_\Delta
$$
be the quotient map. Then I claim for any $x\in (\iota\circ J)^{-1}(0)$ that 
$$
\pi_{\mathcal{L}}(\{x\}\times \bC^{(KH)_x})=(\mathcal{L}_\Delta)_{[x]}^{H_{[x]}}.
$$
This is an immediate consequence of the fact that $(KH)_x=H_{[x]}$ and the equivariance of $\pi_{\mathcal{L}}$. Furthermore, since $H$ acts trivially on $\bC$ and $K$ acts freely on $(\iota^*\circ J)^{-1}(0)$, we conclude the desired result.
\end{proof}

Note that although $H$ acts ``trivially'' on $\mathcal{L}_\Delta$ it need not act trivially on $M_\Delta$. Furthermore, there are no guarantees that it will have nice stabilizers, hence the reduced space $J_{\Delta,H}^{-1}(0)/H$ will be a stratified symplectic space in general.

\cleardoublepage
\chapter{Singular Reduction of Polarizations}\label{ch: sing polarization}

Now that we have explored the singular reduction of the prequantum line bundle, we now turn to the much more subtle matter of the polarization. As we will see, it is much easier to create examples where the prequantum line bundle can be reduced as in Theorem \ref{thm: stratified reduction of prequantum} than to find examples of reducible polarizations. I think this points to a potential for a more generalized (and presumably much more ill-behaved) generalization of the notion of a polarization which could be well-suited to things like cotangent bundles.

\section{Stratified Polarizations}\label{sec: strat polarization}

To begin with, we give a deceptively simple definition of a stratified polarization. I say deceptively simple because it can be quite the struggle to verify. 

\begin{defs}[Stratified Polarization]\label{def: strat pol}
Let $(X,\Sigma,\{\cdot,\cdot\})$ be a symplectic stratified space and $\mathcal{P}\subseteq T^\bC \Sigma$ a complex stratified distribution on $X$. Call $\mathcal{P}$ a stratified polarization if $\mathcal{P}|_S\subseteq T^\bC S$ is a polarization for each $S\in\Sigma$.

\

Given a polarization $\mathcal{P}$, define the polarized functions by
$$
\mathcal{O}_\mathcal{P}(X):=\{f\in C^\infty(X,\bC) \ | \ v(f)=0\text{ for all }v\in \mathcal{P}\}.
$$
\end{defs}

We can see this is the most obvious generalization of Definition \ref{def: polarization}. The assumption that $\mathcal{P}$ be \textbf{stratified} is actually quite difficult to verify in examples (and very often does not hold). In particular, it can be highly non-obvious if the frontier condition is met. Nonetheless, when it does hold, we get some powerful conclusions.

\begin{prop}\label{prop: stratified polarized functions}
Let $\mathcal{P}$ be a stratified polarization on a stratified symplectic space $(X,\Sigma,\{\cdot,\cdot\})$. 
\begin{itemize}
    \item[(i)] $\Gamma(\mathcal{P})$ is involutive. That is, if $V,W\in \Gamma(\mathcal{P})$, then $[X,Y]\in \Gamma(\mathcal{P})$.
    \item[(ii)] $\mathcal{O}_\mathcal{P}(X)$ is a Poisson subalgebra of $C^\infty(X,\bC)$.
\end{itemize}
\end{prop}

\begin{proof}
\begin{itemize}
    \item[(1)] Let $V,W\in \Gamma(\mathcal{P})$. Since $\Gamma(\mathcal{P})\subseteq \mfX^\bC(\Sigma)$ and $\mfX^\bC(\Sigma)$ is involutive, it follows that $[V,W]\in \mfX^\bC(\Sigma)$. In particular, for all $S\in\Sigma$, we have
    $$
    [V,W]|_S=[V|_S,W|_S].
    $$
    Since we are assuming $V|_S,W|_S\in\Gamma(\mathcal{P}|_S)$ and $\mathcal{P}|_S$ is involutive, it follows that $[V|_S,W|_S]\in\Gamma(\mathcal{P}|_S)$. Since $S$ was arbitrary, we conclude that $[V,W]\in\Gamma(\mathcal{P})$.

    \item[(2)] Let $f,g\in \mathcal{O}_{\mathcal{P}}(X)$, $x\in X$, $S\in \Sigma$ the unique stratum containing $x$, and $v\in \mathcal{P}_x$. We show $v(\{f,g\})=0$. Indeed, observe that since $v\in T^\bC S$ and the inclusion $S\into X$ is Poisson, we get
    $$
    v(\{f,g\})=v(\{f,g\}|_S)=v(\{f|_S,g|_S\}).
    $$
    Using the fact that $v\in T^\bC S$ once again, we have $v(f)=v(f|_S)$ and $v(g)=v(g|_S)$. Since we are assuming $f,g\in\mathcal{O}_\mathcal{P}(X)$, it follows that $f|_S,g|_S\in \mathcal{O}_{\mathcal{P}|_S}(S)$ and hence by Proposition \ref{prop: polarized functions}, $\{f|_S,g|_S\}\in \mathcal{O}_{\mathcal{P}|_S}(S)$. Hence,
    $$
    v(\{f,g\})=v(\{f,g\}|_S)=v(\{f|_S,g|_S\})=0.
    $$
    Since $x$, $S$, and $v$ were arbitrary, we conclude that $\{f,g\}\in \mathcal{O}_\mathcal{P}(X)$. 
\end{itemize}
\end{proof}

\begin{egs}
Let $(M,\omega)$ be a symplectic manifold equipped with its canonical stratification $\Sigma$ by connected components. If $\mathcal{P}\subseteq T^\bC M$ is any polarization, then it automatically is a stratified polarization.
\end{egs}

\begin{egs}\label{eg: stratified polarizations on toric manifolds}
For me, the most archetypal stratified polarizations arise from toric geometry. Let $\bT$ be a torus with Lie algebra $\mft$ and suppose
$$
J:(M,\omega)\to \mft^*
$$
is a toric manifold in the sense of Definition \ref{def: toric mfld}. Since $\bT$ is abelian, the orbit-type stratification $\mathcal{S}_\bT(M)$ is a stratification by symplectic submanifolds. As we have discussed in Example \ref{eg: stratification by symplectic manifolds not stratsymp}, $\mathcal{S}_\bT(M)$ is not a symplectic stratification since the Hamiltonian vector fields are not tangent to the strata. Nevertheless, we can illustrate the principles of a stratified polarization beautifully in this setting. 

\ 

First, a real stratified polarization. Let $\mathscr{F}_\bT$ denote the singular foliation of $M$ given by the $\bT$-orbits and define
$$
\mathcal{P}_\bR:=T\mathscr{F}_\bT\otimes \bC\substeq T^\bC\mathcal{S}_\bT(M).
$$
Since $\mathcal{P}_\bR$ is the complexification of the tangent bundle of a singular foliation, clearly it is involutive. To show now that $\mathcal{P}_\bR$ restricts to a polarization on each stratum, fix an orbit-type stratum $S\in \mathcal{S}_T(M)$ and let $x\in S$. Then the quotient group $N=\bT/\bT_x$ acts freely on $S$ and using the Delzant construction, we can easily see that $\dim(S)=2\dim(\bT/\bT_x)$. Hence, by Example \ref{eg: toric lagrangian foliation}, the quotient map
$$
\pi_S:S\to S/N
$$
is a Lagrangian fibration. In particular, $\ker(T\pi_S)\otimes \bC\substeq T^\bC S$ is a polarization and $\mathcal{P}_\bR|_S=\ker(T\pi_S)\otimes \bC$. Hence, modulo $\mathcal{S}_\bT(M)$ not exactly being a stratified symplectic space, $\mathcal{P}_\bR$ is a stratified polarization. Finally for the polarized functions, observe that 
$$
C^\infty_\mathcal{P}(M)=C^\infty(M)^\mft
$$
the $\mft$-invariant functions. Since the action of $T$ on $M$ is symplectic, it follows that $C^\infty(M)^\mft$ is a Poisson subalgebra. Note that since the leaves of the $\bT$-action are Lagrangian on an open dense set, it's straightforward to see that $C^\infty_\mathcal{P}(M)$ has a trivial Poisson bracket.

\

Thus, we have now the archetypal example of a real stratified polarization. Toric geometry also provides for us a natural family of stratified K\"{a}hler polarizations as well. Let $J$ be a $\bT$-invariant compatible complex structure on $M$. These exist due to the Delzant construction of $(M,\omega)$ from its moment polytope. Let $\langle\cdot,\cdot\rangle$ be the induced Riemannian metric. Pulling back the metric to any of the strata, we get another compatible metric and hence another K\"{a}hler structure.

\

Thus, for each $S\in \mathcal{S}_\bT(M)$ let $J_S$ be the induced compatible $\bT$-invariant complex structure and let $T^{1,0}_{J_S} S$ denote the holomorphic tangent bundle with respect to $J_S$. We define
$$
\mathcal{P}_\bC:=\bigcup_{S\in \mathcal{S}_T(M)} T^{1,0}_{J_S}S
$$
As we already saw in Example \ref{eg: stratified complex structures}, we can also realize
$$
\mathcal{P}_\bC=\bigcup_{x\in M}((T^{1,0}_J)_xM)^{\bT_x}
$$
which as we saw is a stratified complex distribution. From the description of $\mathcal{P}_\bC$ the union of holomorphic tangent bundles, $\mathcal{P}_\bC$ is eminently a stratified polarization. Note that if $f\in\mathcal{O}_{\mathcal{P}_\bC}(M)$ is a complex polarized function, then $f$ is antiholomorphic over an open dense set. This implies that $f$ must be globally antiholomorphic. Furthermore, since the inclusions of the strata $S\into M$ are antiholomorphic, we conclude that this must be all the polarized functions. That is, the $\mathcal{P}$ polarized functions coincide with the antiholomorphic functions on $M$. This is a Poisson subalgebra.
\end{egs}

\begin{egs}\label{eg: true non-trivial strat pols}
    We can somewhat massage the above example to get a true example of stratified polarizations. In particular, let us return to our test example of $\bR^2$ along with the Poisson structure
    $$
    \pi=r^2\frac{\partial}{\partial x}\wedge \frac{\partial}{\partial y}.
    $$
    For a real stratified polarization, let
    $$
    \partial_\theta:=-y\frac{\partial}{\partial x}+x\frac{\partial}{\partial y}\in \mfX(\bR^2).
    $$
    Letting $\mathcal{P}_\bR\subseteq T^\bC\bR^2$ be the image of the map
    $$
    \bR^2\times\bC\to T^\bC \bR^2;\quad (x,\lambda)\mapsto \lambda \partial_\theta(x)
    $$
    then $\mathcal{P}_\bR$ is a stratified polarization. This of course is the complexification of the tangent bundle of the foliation induced by the canonical $S^1$ action on $\bR^2$, hence is involutive. $\mathcal{P}_\bR$ vanishes at $0$ and is rank $1$ elsewhere, hence is Lagrangian. For complex, let $I_{can}$ be the canonical complex structure on $\bR^2$. Letting
    $$
    \mathcal{P}_\bC=\widetilde{T^{0,1}_{I_{can}}\bR^2}
    $$
    with respect to the $S^1$-action, we see that $\mathcal{P}_\bC$ is also a polarization.
\end{egs}

Recall that symplectic stratified spaces $(X,\Sigma,\{\cdot,\cdot\})$ are assumed to be Whitney regular. In particular, they are assumed to satisfy the Whitney A condition. This condition allows us to show the following.

\begin{prop}\label{prop: suffices to go to maximal strata}
Let $(X,\Sigma,\{\cdot,\cdot\})$ be a symplectic stratified space, $\mathcal{P}\subseteq T^\bC \Sigma$ a stratified polarization, and $f\in C^\infty(X)$. If $f|_S\in \mathcal{O}_{\mathcal{P}|_S}(S)$ for all maximal strata $S\in\Sigma$, then $f\in \mathcal{O}_{\mathcal{P}}(M)$.
\end{prop}

\begin{proof}
Fix $f\in C^\infty(X)$ which is polarized on the maximal strata and let $x\in X$ and $v\in \mathcal{P}_x$. We show $v(f)=0$.

\

Writing $R\in\Sigma$ for the stratum containing $x$, we can find a maximal stratum $S\in\Sigma$ with $R\subseteq \overline{S}$. Since $\mathcal{P}$ is assumed to be stratified, we have $\mathcal{P}|_R\subseteq \overline{\mathcal{P}|_S}$. Hence, we can find a sequence $\{v_k\}\subseteq \mathcal{P}|_S$ converging to $v$. Since $S$ is maximal, we have
$$
v_k(f)=0
$$
for all $k$. Thus, taking a limit we conclude
$$
v(f)=\lim_{k\to \infty}v_k(f)=\lim_{k\to \infty}0=0.
$$
Hence, $f$ is polarized.
\end{proof}

\section{The Singular Reduction of Polarizations}\label{sec: sing reduce polarization}

We are now finally ready to discuss the singular reduction of polarizations. In order to fix out notation, let $J:(M,\omega)\to \mfg^*$ be a proper Hamiltonian $G$-space for connected Lie group $G$, and $\mathcal{P}\substeq T^\bC M$ a $G$-invariant polarization. We will write $M_0=J^{-1}(0)/G$ for the reduced space and $\pi_0:J^{-1}(0)\to M_0$ for the quotient map. The rest of the following subsection will be devoted to proving the following result.

\begin{theorem}\label{thm: singular reduction of polarizations}
Suppose the subbundle
$$
\mathcal{P}_{J^{-1}(0)}:=\mathcal{P}\cap T^\bC \mathcal{S}_G(J^{-1}(0))
$$
is a stratified complex distribution over $J^{-1}(0)$, then 
$$
\mathcal{P}_0:=T\pi_0(\mathcal{P}_{J^{-1}(0)})\subseteq T^\bC \mathcal{S}_G(M_0)
$$
is a stratified polarization over $M_0$.
\end{theorem}

The main ingredient for the proof is to show the following.

\begin{lem}\label{lem: setup for singular reduction of polarizations}
In the setup of Theorem \ref{thm: singular reduction of polarizations}, for each $x\in J^{-1}(0)$ and stratum $S\in \mathcal{S}_G(J^{-1}(0))$, we have
$$
T\pi(\mathcal{P}_x\cap T^\bC_x S)\subset T_{[x]}^\bC(S/G)
$$
is a complex Lagrangian subspace.
\end{lem}

\begin{proof}
Fix $x\in J^{-1}(0)$ and stratum $S\in \mathcal{S}_G(J^{-1}(0))$ containing $x$. Let $K=G_x$ and  write $V=S\nu_x(M,G)$ for the symplectic normal space, we have by the local normal form in Theorem \ref{thm: manifold of symmetry reduction gives strata} a $K$-equivariant linear symplectomorphism
$$
T_xM\cong (\mfg/\mfk)\oplus (\mfg/\mfk)^* \oplus V.
$$
Write $W=(\mfg/\mfk)\oplus (\mfg/\mfk)^*$. Then $\mathcal{P}_x$ can be identified with a $K$-invariant complex Lagrangian subspace $L\subseteq W_\bC\oplus V_\bC$. Writing $F=(\mfg/\mfk)\oplus \{0\}\subseteq W$, we can also identify 
$$
T_x S=F\oplus V^K
$$
and $T_{[x]}(S/G)$ with the reduction of $W\oplus V$ along $F\oplus V^K$. Thus, we are in the situation of Theorem \ref{thm: linear model of singular reduction}. Hence, the image of $L$ in the reduction is a complex Lagrangian subspace. This image precisely corresponds to $T\pi(\mathcal{P}_x\cap T^\bC_x S)$. 
\end{proof}

Given Lemma \ref{lem: setup for singular reduction of polarizations}, we can immediately prove Theorem \ref{thm: singular reduction of polarizations}.

\begin{proof}{Of Theorem \ref{thm: singular reduction of polarizations}}
Fix a stratum $S\in \mathcal{S}_G(J^{-1}(0))$. Then
$$
\mathcal{P}_{J^{-1}(0)}|_S=\mathcal{P}\cap T^\bC S
$$
is a complex involutive distribution on $S$. Since $\mathcal{P}$ is $G$-invariant and using Lemma \ref{lem: setup for singular reduction of polarizations}, we conclude that the map
$$
T\pi:\mathcal{P}_{J^{-1}(0)}|_S\to T^\bC(S/G)
$$
is constant rank and constant over the fibres of $\pi:S\to S/G$. Hence, by Lemma \ref{lem: surjective submersion gives us distributions}, 
$$
\mathcal{P}_0|_S=T\pi(\mathcal{P}_{J^{-1}(0)}|_S)\subseteq T^\bC (S/G)
$$
is an involutive complex distribution. Once again, by Lemma \ref{lem: setup for singular reduction of polarizations}, the fibres of $\mathcal{P}_0|_{S/G}$ are complex Lagrangian subspace, hence $\mathcal{P}_0|_{S/G}$ is a polarization. 

\

Now to show $\mathcal{P}_0$ is a complex stratified distribution. We already have
$$
\mathcal{P}_0\subseteq T^\bC \mathcal{S}_G(J^{-1}(0))
$$
and that $\mathcal{P}_0|_S\subseteq T^\bC S$ is a smooth complex distribution. Since $\mathcal{P}_0$ inherits its differentiable structure and stratum-wise vector bundle structure from $T^\bC \mathcal{S}_G(M_0)$, we just need to show that the canonical decomposition of $\mathcal{P}_0$:
$$
\mathcal{P}_0=\bigsqcup_{S\in \mathcal{S}_G(M_0)} \mathcal{P}_0|_S
$$
is a stratification. To that end, suppose $S,R\in \mathcal{S}_G(J^{-1}(0))$ satisfy
$$
\mathcal{P}_0|_{S/G}\cap \overline{\mathcal{P}_0|_{R/G}}\neq\emptyset
$$
By continuity of the projection from $\mathcal{P}_0$ onto $M_0$, we have $(S/G)\cap \overline{R/G}\neq\emptyset$ and hence $S/G\subseteq \overline{R/G}$. Consequently, $S\subseteq \overline{R}$ as well. 

\

Now, fix $x\in S$ and $v\in (\mathcal{P}_0)_{[x]}$. We can find a lift $w\in \mathcal{P}_x\cap T_x^\bC S$. By assumption, $\mathcal{P}_{J^{-1}(0)}$ is a stratified distribution, hence we can find a sequence $\{w_k\}\subseteq \mathcal{P}_{J^{-1}(0)}|_R$ which converge to $w$. Applying the projection $\pi$ and setting $v_k=\pi(w_k)$, we have a sequence $\{v_k\}\subseteq \mathcal{P}_0|_R$ converging to $v$ and thus $\mathcal{P}_0|_S\subseteq \overline{\mathcal{P}_0|_R}$. 
\end{proof}

Motivated by our previous discussion in non-singular reduction and the sole condition applied to the polarization to ensure the reduction was a polarization again, we give the following definition.

\begin{defs}[Singular Reducibility]\label{def: singular reducibility}
Let $\mathcal{P}$ be a $G$-invariant polarization. Say $\mathcal{P}$ is singularly reducible if 
$$
\mathcal{P}_{J^{-1}(0)}=\mathcal{P}\cap T^\bC\mathcal{S}_G(J^{-1}(0))
$$
is a complex stratified distribution of $J^{-1}(0)$.
\end{defs}

\begin{egs}\label{eg: cant always reduce}
    As we shall go over below, we can obtain plenty of interesting examples of singular reducible polarizations by examining equivariant K\"{a}hler structures. Unfortunately, we shall see our first example now of a very natural equivariant polarization, that is not singularly reducible. Indeed, recall the Hamiltonian lift of the canonical $S^1$-action on $\bR^2$, $J:T^*\bR^2\to \bR$ from Example \ref{eg: cotangent of r2}. In this case, $J^{-1}(0)$ has two strata:
    $$
    J^{-1}(0)=\{(0,0)\}\sqcup S
    $$
    and $S$ can be further equivariantly decomposed
    \begin{equation}
    S=(\{0\}\times (\bR^2\setminus \{0\}))\sqcup R.
    \end{equation}
    Since $J(x,x)=0$, it's straightforward to see that
    $$
    R=\{(x,t x)\in \bR^2\times \bR^2 \ | \ x\neq 0\text{ and }tin \bR\}.
    $$
    The projection onto the first factor 
    $$
    pr_1:R\to \bR^2\setminus\{0\}
    $$
    clearly makes $R$ into a free $S^1$-equivariant vector bundle and hence the quotient $R/S^1\to (\bR^2\setminus\{0\})/S^1$ is a smooth rank $1$-vector bundle. Via our discussion in Section \ref{sec: sing reduce cotangent}, we can see that $R/S^1\cong T^*(0,\infty)$ via the map
    $$
    (\bR^2\setminus\{0\})/S^1\to (0,\infty);\quad [x]\mapsto \|x\|^2.
    $$
    Observe that $0$ is a regular value of $J$ restricted to $\bR^4\setminus\{(0,0)\}$ and so using this we can deduce that
    $$
    T_{(x,y)}S=\{(v,w)\in \bR^2\times \bR^2 \ | \ J(x,w)-J(y,v)=0\}
    $$
    for any $(x,y)\in S$. Using this description together with the fact that if $(x,y)\in S$ with $x\neq 0$ then $y$ is a multiple of $x$, we obtain
    $$
    T_{(x,y)}S=
    \begin{cases}
        \{(\lambda y,w) \ | \ \lambda\in \bR\text{ and }w\in \bR^2\},\quad x=0\\
        \{(v,tv+\lambda x) \ | \ \lambda\in \bR\text{ and }v\in \bR^2\},\quad y=tx\text{ for some }t\in \bR.
    \end{cases}
    $$
    Now let us see how the canonical polarization $\mathcal{P}$ intersects the complexification of the tangent bundle of $S$.

    \

    Writing $pr_1:\bR^2\times \bR^2\to \bR^2$ for the projection onto the first factor, the canonical real polarization in this case is
    $$
    \mathcal{P}=\ker(Tpr_1)\otimes \bC=\bR^2\times \bR^2\times \{0\}\times \bC^2.
    $$
    Using the description of the tangent space to $S$, we then obtain
    $$
    \mathcal{P}_{(x,y)}\cap T_{(x,y)}^\bC S=
    \begin{cases}
        \{0\}\times \bC^2,\quad x=0\\
        \{0\}\times \text{span}_\bC(x),\quad \text{else}.
    \end{cases}
    $$
    Thus, $\mathcal{P}\cap T^\bC S$ does not have constant rank and hence $\mathcal{P}$ is not singularly reducible. Nevertheless, notice something interesting here. If I let $\pi_S:S\to S/S^1$ denote the quotient map, then for any $(x,y)\in S$ we have
    $$
    \dim_\bC T_{(x,y)}\pi_S(\mathcal{P}_{(x,y)}\cap T^\bC_{(x,y)}S)=1
    $$
    How could this be? Over the open subset $R\substeq S$ where $x\neq 0$, this is clear as the orbits of the $S^1$ action intersect the cotangent fibres trivially. Over $\{0\}\times \bR^2$, it's easy to see that
    $$
    T^\bC_{(0,y)}S^1\cdot(0,y)\substeq \mathcal{P}_{(0,y)}\cap T^\bC_{(0,y)}S
    $$
    and hence
    $$
    T_{(0,y)}\pi_S(\mathcal{P}_{(0,y)}\cap T^\bC_{(0,y)}S\cong \bC^2/T^\bC_y (S^1\cdot y)
    $$
    which is clearly $1$-dimensional. As we shall see shortly, $\pi_*(\mathcal{P}\cap T^\bC R)$ is actually a polarization and hence if $\pi_*(\mathcal{P}\cap T^\bC S)$ is smooth, it also is a polarization over $S/S^1$. 
\end{egs}

\section{Preservation of Type Under Singular Reduction}\label{sec: type and reduction}

\begin{prop}
Let $K\leq G$ be a compact subgroup and let 
$$
J_K:M_K\to (\mfn_K(G)/\mfk)^*
$$
be the canonical momentum map on the manifold of symmetry $M_K$. Define
\begin{equation}
\mathcal{P}_K:=\mathcal{P}\cap T^\bC M_K.
\end{equation}
Then,
\begin{itemize}
    \item[(1)] $\mathcal{P}_K$ is an $N_G(K)/K$-invariant polarization on $M_K$.
    \item[(2)] If $\mathcal{P}$ is real, then so is $\mathcal{P}_K$.
    \item[(3)] If $\mathcal{P}$ is totally complex, then so is $\mathcal{P}_K$.
    \item[(4)] If $\mathcal{P}$ is singularly reducible, then $\mathcal{P}_K$ is reducible.
\end{itemize}
\end{prop}

\begin{proof}
Let $x\in J_K^{-1}(0)$ and let $S\subseteq \mathcal{S}_G(J^{-1}(0))$ be the orbit-type stratum containing $x$. 
$$
(\mathcal{P}_K)_x=\mathcal{P}_x\cap T_x^\bC M_K=\mathcal{P}_x\cap (T_x^\bC M)^K=\mathcal{P}_x^K.
$$
Hence, $\mathcal{P}_K$ is a Lagrangian complex distribution on $M_K$. Since both $\mathcal{P}$ and $T^\bC M_K$ are involutive, it follows that $\mathcal{P}_K$ is also involutive, hence a polarization. $\mathcal{P}_K$ is invariant under the $N_G(K)/K$ action because $\mathcal{P}$ is invariant. It is also reducible since 
$$
(\mathcal{P}_K)_x\cap T^\bC_x J_K^{-1}(0)=(\mathcal{P}_x\cap T^\bC_xS)^K.
$$
The statements about type being preserved for real or totally complex polarizations follows from the identification $(\mathcal{P}_K)_x=\mathcal{P}^K_x$ for each $x\in M_K$ and from the fact that complex conjugation preserves fixed point sets.
\end{proof}

\begin{prop}\label{prop: singular reduction of polarizations can be done on manifolds of symmetry}
Suppose $\mathcal{P}$ is a singularly reducible polarization, $S\in \mathcal{S}_G(J^{-1}(0))$ an orbit-type stratum, $x\in S$, and let $K=G_x$. Then, under the canonical symplectomorphism
$$
S/G\cong J_K^{-1}(0)/(N_G(K)/K)=:(M_K)_0
$$
the polarization $\mathcal{P}_0|_{S/G}$ gets mapped to $(\mathcal{P}_K)_0\subseteq T^\bC (M_K)_0$.
\end{prop}

\begin{proof}
This makes use of the local form.
\end{proof}

\begin{cor}
Suppose $\mathcal{P}$ is singularly reducible and let $S\in \mathcal{S}_G(M_0)$ be a stratum of the singular reduced space.
\begin{itemize}
    \item[(1)] If $\mathcal{P}$ is real, then $\mathcal{P}_0|_S$ is real.
    \item[(2)] If $\mathcal{P}$ is totally complex, then $\mathcal{P}_0|_S$ is totally complex.
    \item[(3)] If $\mathcal{P}=T^{0,1}_JM$ for some compatible $G$-invariant complex structure $J$ on $M$, then there exists compatible complex structure $J_S$ on $S$ so that $\mathcal{P}_0|_S=T^{0,1}_{J_S}S$.
\end{itemize}
\end{cor}

\begin{proof}
Each statement is a simply a stratum-by-stratum statement of Proposition \ref{prop: singular reduction of polarizations can be done on manifolds of symmetry}. 
\end{proof}

\section{When Polarizations Can Be Reduced}\label{sec: when we can reduce pols}

\begin{prop}
Suppose $\mathcal{P}$ satisfies the property that for any strata $S,R\in \mathcal{S}_G(J^{-1}(0))$ with $S\subseteq \overline{R}$ and any $(x,y)\in S\times R$ we have
$$
\dim_\bC (\mathcal{P}_x\cap T^\bC_x S)\leq \dim_\bC(\mathcal{P}_y\cap T^\bC_y R).
$$
Then $\mathcal{P}$ is singularly reducible if and only if $\mathcal{P}\cap T^\bC S$ is a complex distribution of $S$ for all $S\in\mathcal{S}_G(J^{-1}(0))$.
\end{prop}

\begin{proof}
Since $\mathcal{P}$ is fibre-wise Lagrangian, it automatically follows that $\mathcal{P}$ is closed as a subset of $T^\bC M$. Hence, a chart-wise application of Lemma \ref{lem: a stratified distribution can be reconstructed from the strata} completes the proof.
\end{proof}

We will finish now with a discussion on when an invariant polarization is singularly reducible.

\begin{lem}\label{lem: fibre-wise lagrangian decomposition}
Fix $x\in J^{-1}(0)$ and let $K=G_x$. Letting $S\in \mathcal{S}_G(J^{-1}(0))$ be the orbit-type stratum containing $x$. Then,
\begin{equation}\label{eq: fibre-wise lagrangian decomposition}
\mathcal{P}_x\cap T^\bC_x S=(\mathcal{P}_x\cap T^\bC_x S)^K+ \mathcal{P}_x\cap T^\bC_x(G\cdot x).
\end{equation}
\end{lem}

\begin{proof}
Clearly, since both $S$ and $\mathcal{P}$ are closed under the $G$-action, we have
$$
(\mathcal{P}_x\cap T^\bC_x S)^K+ \mathcal{P}_x\cap T^\bC_x(G\cdot x)\subseteq \mathcal{P}_x\cap T^\bC_x S.
$$
For the reverse containment, from the local normal form we have a $K$-equivariant linear symplectomorphism
\begin{equation}\label{eq: linear symplectomorphism to local model}
T_xM\cong V\oplus W,
\end{equation}
where $V=\mfk\times \mfk^*$ and $W=S\nu_x(M,G)$ with their canonical symplectic forms $\omega_V$ and $\omega_W$, respectively. Under this symplectomorphism, $T_x(G\cdot x)$ gets mapped to $F\oplus 0$, where $F=\mfk\times\{0\}\subseteq V$, and $T_x S$ gets mapped to $F\oplus W^K$.
Let $L\subseteq V^\bC\oplus W^\bC$ denote the $K$-invariant Lagrangian subspace $\mathcal{P}_x$ gets mapped to. Since both $C$ and $L$ are $K$-invariant, we automatically have
$$
L\cap C^\bC=(L\cap C^\bC)^K+L\cap C^\bC\cap ((V^\bC)^K\oplus (W^\bC)^K)^\omega.
$$
It's easy to see that
$$
C^\bC\cap ((V^\bC)^K\oplus (W^\bC)^K)^\omega=(F^\bC\cap ((V^\bC)^K)^{\omega_V})\oplus 0\subseteq F^\bC\oplus 0
$$
Hence,
$$
L\cap C^\bC\cap ((V^\bC)^K\oplus (W^\bC)^K)^\omega\subseteq L\cap F^\bC\oplus 0
$$
Translating back to $T^\bC_x M$ via the symplectomorphism from Equation (\ref{eq: linear symplectomorphism to local model}), we get the desired result.
\end{proof}

\begin{cor}
If $\mathcal{P}$ has constant rank intersections with the orbit foliation on the orbit-type strata of $M$ and for each $K\leq G$, the induced polarization $\mathcal{P}_K$ on $M_K$ is reducible in the sense of Definition \ref{def: singular reducibility}, then $\mathcal{P}$ is singularly reducible.
\end{cor}

\begin{proof}
Since each orbit-type stratum $S\in\mathcal{S}_G(J^{-1}(0))$ is a sweep of $J_K^{-1}(0)\subseteq M_K$ for some $K$, using Equation (\ref{eq: fibre-wise lagrangian decomposition}) from Lemma \ref{lem: fibre-wise lagrangian decomposition}, we see that $\mathcal{P}\cap T^\bC S$ has the form
$$
\mathcal{P}\cap T^\bC S=G\cdot (\mathcal{P}_K\cap T^\bC J_K^{-1}(0)+\mathcal{P}\cap \mfg_S|_{J_K^{-1}(0)})
$$
Using the $G$-invariance of $\mathcal{P}$ together with the assumptions, the result then follows.
\end{proof}

Now that we have these preliminary results, we can extend the result of Guillemin and Sternberg.

\begin{prop}
Suppose $\mathcal{P}$ is a positive $G$-invariant polarization. Then $\mathcal{P}$ is reducible. 
\end{prop}

\begin{proof}
Since the orbits of the $G$-action are isotropic, it immediately follows from Corollary \ref{cor: positive polarizations are reducible non-singular}that $\mathcal{P}_K$ is reducible for all compact $K\leq G$ and and $\mathcal{P}$ intersects the orbit foliation trivially. Thus, $\mathcal{P}$ is singularly reducible. 
\end{proof}

\begin{cor}
$G$-invariant K\"{a}hler polarizations are always singularly reducible and induce a stratum-by-stratum K\"{a}hler structure on the reduced space.
\end{cor}

\begin{remark}
Using what we developed above, we can recover Hyper K\"{a}hler reduction as discussed by Mayrand \cite{mayrand_stratified_2020} 
\end{remark}

\section{Non-Reducible Polarizations}\label{sec: non-reducible pols}

As we saw in Example \ref{eg: cant always reduce}, the canonical polarization of a cotangent bundle need not be reducible. Nevertheless, we do get something useful for the purposes of quantization. To set things up, let $G$ be a connected Lie group, $Q$ a proper $G$-space, and $J:T^*Q\to \mfg^*$ the Hamiltonian lift. Write $\tau:T^*Q\to Q$ for the bundle projection. Recall that the canonical polarization on $T^*Q$ has the form
$$
\mathcal{P}:=\ker(T\tau)\otimes \bC.
$$
This is not singularly reducible, but it is equivariant and hence we can define a pseudobundle on the reduction.

\begin{defs}\label{def: almost polarization on reduced cotangent bundle}
    Let $\pi_0:J^{-1}(0)\to (T^*Q)_0$ be the reduction map. Define pseudobundle $\mathcal{P}_0\substeq T^\bC \mathcal{S}_G((T^*Q)_0)$ by
    $$
    T^\bC \pi_0(\mathcal{P}\cap T^\bC\mathcal{S}_G(J^{-1}(0))).
    $$
\end{defs}

Using the fact that for any stratum $S\in \mathcal{S}_G(Q)$, we have a corresponding stratum $Z_S\in\mathcal{S}_G(J^{-1}(0))$ with the property that $T^*(S/G)\substeq Z_S/G$ as an open dense set, we can prove the following.

\begin{theorem}\label{thm: cotangent polarization almost reducible}
    For any stratum $S\in \mathcal{S}_G(Q)$, $\mathcal{P}_0|_{T^*(S/G)}$ agrees with the canonical polarization on $T^*(S/G)$.
\end{theorem}

\begin{proof}
    Since a choice of Riemannian metric induces an embedding of vector bundles $T^*S\into T^*Q|_S$, this reduces to the non-singular case, and hence we can apply Proposition \ref{prop: nonsing reduction of cotangent polarization} to complete the proof. 
\end{proof}

\begin{remark}
    Once again, it should be noted that $\mathcal{P}_0$ in Definition \ref{def: almost polarization on reduced cotangent bundle} does not define a stratified polarization on ($T^*Q)_0$ in the sense of Definition \ref{def: strat pol}. What we have shown in Theorem \ref{thm: cotangent polarization almost reducible} is that for each stratum of $(T^*Q)_0$ there is an open dense set on which $\mathcal{P}_0$ restricts to a polarization, but it need not define a polarization (or even a vector bundle) over the entire stratum. The reader is encourage to consult Example \ref{eg: cant always reduce} for an example of this phenomenon.
\end{remark}
\cleardoublepage
\chapter{Singular Reduction and Quantization}\label{ch: sing quant}

Now that we have laid the groundwork of stratified versions of symplectic manifolds, polarizations, and prequantum line bundles, we can now delve into a stratified version of quantization. As we shall see, this will be essentially the same as the classical definition of quantization.

\section{Quantization of Symplectic Stratified Spaces}\label{sec: quant symp strats}

To begin, let us set up the notion of stratified quantization data in analogy with Definition \ref{def: non-singular quantization data}.

\begin{defs}[Stratified Quantization Data]
Stratified quantization data will consist of four pieces of data.
\begin{itemize}
    \item[(1)] A symplectic stratified space $(X,\Sigma,\{\cdot,\cdot\})$
    \item[(2)] A Hamiltonian module $\mathcal{H}\subseteq \mfX(X)$.
    \item[(3)] A stratified polarization $\mathcal{P}\subseteq T^\bC\Sigma$.
    \item[(4)] A stratified prequantum line bundle $(\mathcal{L},\nabla,\{\nabla^S\}_{S\in \Sigma})$ over $X$.
\end{itemize}
Write $(X,\mathcal{H},\mathcal{P},\mathcal{L})$ for this data.
\end{defs}

\begin{egs}
    Usual quantization data in the sense Definition \ref{def: non-singular quantization data} are examples of stratified quantization data.
\end{egs}

\begin{defs}\label{def: stratified quantization}
    Given stratified quantization data $(X,\mathcal{H},\mathcal{P},\mathcal{L})$, define the quantization by
    \begin{equation}
        \mathcal{Q}(X,\mathcal{H},\mathcal{P},\mathcal{L})=\{\sigma\in \Gamma(\mathcal{L}) \ | \ \nabla_v\sigma=0\text{ for all }v\in\mathcal{P}\}.
    \end{equation}
    If there is no risk of confusion, we will simply write $\mathcal{Q}(X)$ for the quantization of $X$.
\end{defs}

\begin{egs}\label{eg: first non-trivial quantization}
    Through our discussions in Examples \ref{eg: first non-trivial symplectic stratified space}, \ref{eg: non-trivial Ham module}, \ref{eg: non-trivial prequantum}, and \ref{eg: true non-trivial strat pols}, we already have non-trivial examples of quantization data. 

    \ 

    Indeed, in Example \ref{eg: first non-trivial symplectic stratified space} we saw the Poisson structure
    $$
    \pi=r^2\frac{\partial}{\partial x}\wedge \frac{\partial}{\partial y}
    $$
    on $\bR^2$ induces a symplectic stratification via its symplectic foliation $\mathscr{F}_\pi=\{\{0\},\bR^2\setminus\{0\}\}$. Then in Example \ref{eg: non-trivial Ham module} we saw that this stratification comes equipped with a Hamiltonian module
    $$
    \mathcal{H}=\{V\in \mfX(\bR^2) \ | \ V\text{ vanishes to order }2\text{ at }0\}.
    $$
    In Example \ref{eg: non-trivial prequantum} we obtained a stratified prequantum line bundle $(\mathcal{L},\nabla)$ where
    $$
    \mathcal{L}=\bR^2\times \bC
    $$
    and $\nabla$ is given by
    $$
    \nabla|_{\{0\}}=0
    $$
    and
    $$
    \nabla|_{\bR^2\setminus\{0\}}=L+i\alpha
    $$
    for 
    $$
    \alpha=\ln(r)d\theta
    $$
    with $(r,\theta)$ being polar coordinates. Then, in Example \ref{eg: true non-trivial strat pols}, we saw that this symplectic stratified space has two natural polarizations, a real polarization $\mathcal{P}_\bR$ spanned by $\partial_\theta$ and a complex polarization $\mathcal{P}_\bC$ arising from the canonical complex structure on $\bR^2$.

    \ 

    What are then the quantizations? First, let us examine the real polarization. In this case, if $f\in \mathcal{Q}(\bR^2,\mathcal{P}_\bR)\substeq C^\infty(\bR^2,\bC)$, we see that since $\partial_\theta$ spans $\mathcal{P}_\bR$,
    $$
    0=\nabla_{\partial_\theta} f=\frac{\partial f}{\partial \theta}+i\alpha(\partial_\theta)f=\frac{\partial f}{\partial \theta}+i\ln(r)f.
    $$
    Hence, we must have
    \begin{equation}\label{eq: bohr-sommerfeld for quant}
        f=g(r)e^{-i\ln(r)\theta}.
    \end{equation}
    Note that since $\theta$ is $2\pi$-periodic, a function $f$ can only satisfy Equation (\ref{eq: bohr-sommerfeld for quant}) if $f=0$. Hence, $\mathcal{Q}(\bR^2,\mathcal{P}_\bR)=0$. Note that as is shown in \cite{sniatycki_bohr-sommerfeld_1975}, this is to be expected as $\mathcal{P}_\bR$ restricts to a Lagrangian toric fibration on $\bR^2\setminus\{0\}$. We can still get interesting quantizations in settings like this, but rather than using the classical definition of quantization as in Definition \ref{def: stratified quantization}, we should adopt either a cohomological approach as we discuss below or we could define the quantization in terms of distinguished leaves of $\mathcal{P}_\bR$ called \textbf{Bohr-Sommerfeld Leaves}. For this latter approach, we first find all the leaves $S\subseteq \bR^2$ of $\mathcal{P}_\bR$ which admit non-zero flat sections of the restricted bundle $(\mathcal{L}|_S,\nabla|_S)\to S$. In this case, it's easy to see that $\{0\}$ is a Bohr-Sommerfeld leaf as are all radius $r$ circles $S_r\subseteq \bR^2$ with $\ln(r)\in 2\pi \bZ$. Hence, the Bohr-Sommerfeld quantization would be
    $$
    \mathcal{Q}_{BS}(\bR^2,\mathcal{P}_\bR)=\bC\oplus_\bZ \bC,
    $$
    where I added the subscript ``BS'' to distinguish this quantization from the quantization obtained via Definition \ref{def: stratified quantization}. 
    I leave the curious reader to consult \cite{sniatycki_bohr-sommerfeld_1975} or \cite{hamilton_locally_2010} for more information on these two approaches.

    \ 

    As for the complex polarization, writing
    $$
    \frac{\partial}{\partial \overline{z}}=\frac{\partial}{\partial x}+i\frac{\partial}{\partial y},
    $$
    it's not too difficult to see that the polarization $\mathcal{P}_\bC$ from Example \ref{eg: true non-trivial strat pols} is spanned by
    $$
    r^2\frac{\partial}{\partial \overline{z}}
    $$
    Note that this lies in the complexification of the Hamiltonian module $\mathcal{H}$. Observe that over $\bR^2\setminus\{0\}$ we have
    $$
    \alpha\bigg(r^2\frac{\partial}{\partial \overline{z}}\bigg)=i\ln(r)z,
    $$
    where $z=x+iy$. Using this, we see that $f\in \mathcal{Q}(\bR^2,\mathcal{P}_\bC)\substeq C^\infty(\bR^2,\bC)$ if and only if
    $$
    0=\nabla_{r^2\frac{\partial}{\partial \overline{z}}}f=\frac{\partial f}{\partial \overline{z}}-\ln(r)zf
    $$
    Using this, it follows that $f=g e^{r\ln(r)}$ where $g$ is a holomorphic function on $\bR^2$. Since $e^{r\ln(r)}$ is a smooth function, we then get a canonical identification between $\mathcal{Q}(\bR^2,\mathcal{P}_\bC)$ and the set of holomorphic functions on $\bR^2$. Note that due to Proposition \ref{prop: suffices to go to maximal strata}, this is precisely the set of polarized functions $\mathcal{O}_{\mathcal{P}_\bC}(\bR^2)$. 
\end{egs}

\begin{remark}
    If for each $\sigma\in\Gamma(\mathcal{L})$ the map
    $$
    \nabla\sigma:T^\bC \Sigma \to \mathcal{L}
    $$
    is continuous, then we can define a sheaf
    $$
    \mathcal{S}(U)=\{\sigma\in\Gamma(\mathcal{L}|_U) \ | \ \nabla_v\sigma=0\text{ for all }v\in \mathcal{P}|_U\}
    $$
    In this case, we could then use other definitions of quantization like Kostant's cohomological wave functions
    $$
    \mathcal{Q}^\bullet(X)=\bigoplus_{i}H^i(X,\mathcal{S})
    $$
    making use of the sheaf cohomology of $\mathcal{S}$. For the purpose of this thesis, we will prefer the definition from Definition \ref{def: stratified quantization}.
\end{remark}

\begin{prop}\label{prop: quantizable on stratum implies quantizable on boundary}
Let $(X,\mathcal{H},\mathcal{P},\mathcal{L})$ be quantization data. Let $\mathcal{H}^\bC\substeq \mfX^\bC(\Sigma)$ denote the set of stratified vector fields of the form
$$
\mathcal{H}^\bC:=\{f_1V+f_2V_2 \ | \ f_1,f_2\in C^\infty(X,\bC), \ V_1,V_2\in \mathcal{H}\}.
$$
Furthermore, suppose the sections
$$
\mathcal{H}_\mathcal{P}:=\Gamma(\mathcal{P})\cap \mathcal{H}\otimes \bC
$$
spans $\mathcal{P}$. Finally, let $S,R\subseteq \Sigma$ be two strata with $S\subseteq \overline{R}$ and suppose $\sigma\in \Gamma(L)$ satisfies $\sigma|_R\in \mathcal{Q}(R)$. Then $\sigma|_S\in \mathcal{Q}(S)$. 
\end{prop}

\begin{proof}
Fix $x\in S$ and $v\in \mathcal{P}_x$. Choose a section $V\in \mathcal{H}_\mathcal{P}$ so that $V_x=v$. By assumption, we must have $(\nabla_V\sigma)|_R=\nabla^R_{V|_R}(\sigma|_R)=0$. Let $\{x_n\}\subseteq R$ be any sequence in $R$ converging to $S$. Then, by continuity of $\nabla_V\sigma$,
$$
0=\lim_{n\to \infty} \nabla_V\sigma(x_n)=\nabla_V\sigma(x)=\nabla^S_v (\sigma|_S).
$$
Since $x$ and $v$ were arbitrary, we conclude that $\sigma|_S\in \mathcal{Q}(S)$.
\end{proof}

\begin{cor}\label{cor: quantized iff quantized on regular strata}
In the set up of Proposition \ref{prop: quantizable on stratum implies quantizable on boundary}, given $\sigma\in\Gamma(\mathcal{L})$, we have $\sigma\in \mathcal{Q}(X)$ $\iff$ $\sigma|_S\in\mathcal{Q}(S)$ for each regular stratum $S\in \Sigma$.
\end{cor}

\begin{proof}
Every stratum is contained in the closure of some regular stratum. Now apply Proposition \ref{prop: quantizable on stratum implies quantizable on boundary}.
\end{proof}

\begin{remark}
    This is essentially what we already saw in Example \ref{eg: first non-trivial quantization}, where the quantization of the whole space was determined by the quantization of the open dense stratum. 
\end{remark}

\section{$[Q,R]=0$}\label{sec: [Q,R] sing}

We would now like to set up a general statement for quantization commuting with reduction. As we shall see, this is a general framework which we can apply cases like toric manifolds. 

Let's now return to the setting of singular reduction. For all that follows, we fix the following.
\begin{itemize}
    \item A connected Lie group $G$ with Lie algebra $\mfg$.
    \item A proper Hamiltonian $G$-space $J:(M,\omega)\to\mfg^*$
    \item A $G$-invariant polarization $\mathcal{P}\subseteq T^\bC M$.
\end{itemize}
Write $\pi:J^{-1}(0)\to J^{-1}(0)/G=M_0$ for the quotient map, $\mathcal{L}_0=(\mathcal{L}|_{J^{-1}(0)})/G$, and 
$$
\mathcal{P}_0=\pi_*(\mathcal{P}\cap T^\bC\mathcal{S}_G(J^{-1}(0))).
$$

Recall that if $(\mathcal{L},\nabla)$ is equivariantly generated over $J^{-1}(0)$, i.e. 
$$
J^{-1}(0)\times \Gamma(L)^G\to L|_{J^{-1}(0)};\quad (x,\sigma)\mapsto \sigma(x)
$$
is surjective, then we get a surjective map
\begin{equation}
\kappa:\Gamma(L)^G\to \Gamma(L_0)
\end{equation}

With this set up, we can now state the following.

\begin{theorem}\label{thm: sing [Q,R] map}
Suppose $(L,\nabla)$ is equivariantly generated over $J^{-1}(0)$ and suppose $\mathcal{P}$ is reducible at level zero (see Definition \ref{def: singular reducibility}). Then $\kappa$ induces a linear map
\begin{equation}\label{eq: singular [Q,R] map}
\kappa:\mathcal{Q}(M)^G\to \mathcal{Q}(M_0).
\end{equation}
\end{theorem}

\begin{proof}
Let $\sigma\in \mathcal{Q}(M)^G$, $x\in J^{-1}(0)$, and $v\in (\mathcal{P}_0)_[x]$ an element of the reduced polarization. We can lift $v$ to an element of $\widetilde{v}\in\mathcal{P}_x^{G_x}$. By construction,
$$
\nabla^0_v\kappa(\sigma)=\kappa(\nabla_{\widetilde{v}}\sigma)=\kappa(0)=0.
$$
Thus, $\kappa$ defines a map
$$
\kappa:\mathcal{Q}(M)^G\to \mathcal{Q}(M_0).
$$
\end{proof}

\begin{defs}[Quantization Commutes With Reduction]
If the map $\kappa$ in Equation (\ref{eq: singular [Q,R] map}) is a bijection, then we say quantization commutes with reduction. (In symbols, this is usually denoted $[Q,R]=0$).
\end{defs}

To end this chapter, we will now move on to prove that $[Q,R]=0$ for two of the major classes of examples examined throughout this thesis: toric manifolds and cotangent bundles.

\section{Toric Manifolds and $[Q,R]=0$}\label{sec: toric [Q,R]}
Let $\Delta\substeq \bR^n$ be a Delzant polytope and $J_\Delta:(M_\Delta,\omega_\Delta)\to \bR^n$ be the induced toric $T^n$-manifold from Delzant's construction as detailed in Section \ref{sub: delzant construction}. Writing
$$
\Delta =\{x\in \bR^n  \ | \ \bra x,v_i\ket\geq \lambda_i\text{ for all }i=1,\dots,N\},
$$
let $k$ be the kernel of the map $\bR^N\to \bR^n$ mapping $e_i$ to $v_i$ for each of the standard basis vectors $e_i\in \bR^N$, $\iota:\mfk\into \bR^N$, the inclusion, and $\lambda=(\lambda_1,\dots,\lambda_N)\in \bR^N$. Then, if we define
$$
J:\bC^N\to \bR^N;\quad (z_1,\dots,z_N)\mapsto \frac{1}{2}(|z_1|^2,\dots,|z_N|^2)-\lambda,
$$
it follows that
$$
M_\Delta=(\iota^*\circ J)^{-1}(0)/K
$$
is a toric $T^N/K$ manifold with moment map
$$
J_\Delta:M_\Delta \to \bR^N;\quad [x]\mapsto J(x),
$$
where $K=\exp(\mfk)\subseteq T^N$. We also saw in Section \ref{section: non-singular quantization toric manifold} that the canonical complex structure $I^{can}$ on $\bC^N$ is reducible with respect to the Hamiltonian action of $K$ and hence descends to a $T^N/K$ equivariant K\"{a}hler structure $I_\Delta$.

\ 

Now for the prequantum line bundle. As we saw in Equation (\ref{eq: toric prequantum homomorphism}), if $\lambda$ and $\iota^*(\lambda)$ are integral, then we get an induced homomorphism
$$
\rho:T^N\to S^1
$$
which makes the canonical prequantum line bundle $\bC^N\times \bC\to \bC$ into a reducible prequantum line bundle, allowing us to define prequantum line bundle
$$
\mathcal{L}_\Delta = ((\iota^*\circ J)^{-1}(0)\times \bC)/K
$$
over $M_\Delta$. Furthermore, as we saw, for any subtorus $H\leq \ker(\rho)$, $\mathcal{L}_\Delta$ is singularly reducible with respect to the induced action of $H$ on $M_\Delta$. Write
$$
J_H:M_\Delta\to \mfh^*
$$
for this induced Hamiltonian $H$-structure. Furthermore, the canonical K\"{a}hler structure $I_\Delta$ on $M_\Delta$ coming from Delzant's construction is equivariant hence reducible as well. Thus, we have a quantization map
$$
\mathcal{Q}(M_\Delta)^H\to \mathcal{Q}((M_\Delta)_0),
$$
where $(M_\Delta)_0:=J_H^{-1}(0)/H$ is the symplectic reduction of $M_\Delta$ with respect to the action of $H$. We show now that this is a bijection, provided some assumptions are met. Let us first discuss an alternate construction for $M_\Delta$.

\ 

For each facet $F\subseteq \Delta\subseteq \bR^n$, define
$$
I_F:=\{i\in \{1,\dots,n\} \ | \ \bra x,v_i\ket=\lambda_i\text{ for all }x\in F\}.
$$
Using this, we can define a closed subset $Z_F\subseteq \bC^N$ by
$$
Z_F:=\{(z_1,\dots,z_n)\in \bC^N \ | \ z_i=0\text{ for all }i\in I_F\}.
$$
Now set
$$
U_\Delta:=\bC^N\setminus \bigcup_{F} Z_F.
$$
Let $K_\bC\subseteq (\bC^\times)^N$ be the complexification of $K$. Then $K_\bC$ naturally acts on $U_\Delta$. Furthermore, $J^{-1}(0)\subseteq U_\Delta$.

\begin{prop}[{\cite{kirwan_momentum_1998}}]\label{prop: complexification of toric}
The inclusion $(\iota^*\circ J)^{-1}(0)\subseteq U_\Delta$ induces a diffeomorphism
$$
M_\Delta \to U_\Delta/K_\bC.
$$
\end{prop}

This can also be extended to the equivariant line bundle as well. Recall that the defining homomorphism for $\mathcal{L}_\Delta$ is given by
$$
\rho:T^N\to S^1;\quad (t_1,\dots,t_N)\mapsto t_1^{-\lambda_1}\cdots t_N^{-\lambda_N},
$$
where $\lambda=(\lambda_1,\dots,\lambda_N)\in \bZ^N$. This trivially extends to the complexifications
$$
\rho_\bC:(\bC^\times)^N\to \bC^\times;\quad (z_1,\dots,z_N)\mapsto z_1^{-\lambda_1}\cdots z_N^{-\lambda_N}
$$
Thus, we get a $T^N$-equivariant homomorphism of line bundles
$$
(\iota^*\circ J)^{-1}(0)\times \bC\into U_\Delta\times \bC,
$$
where $U_\Delta\times \bC$ has the $(\bC^\times)^N$ action induced by $\rho_\bC$. Since $K_\bC$ acts freely on the left-hand side, we obtain in conjunction with Proposition \ref{prop: complexification of toric} the following result

\begin{prop}\label{prop: complexification of line toric}
    The inclusion $(\iota^*\circ J)^{-1}(0)\subseteq U_\Delta$ induces an isomorphism of complex line bundles
    $$
    \begin{tikzcd}
        \mathcal{L}_\Delta\arrow[r]\arrow[d] & (U_\Delta\times \bC)/K_\bC\arrow[d]\\
        M_\Delta\arrow[r] & U_\Delta/K_\bC
    \end{tikzcd}
    $$
    where the bottom arrow is the diffeomorphism from Proposition \ref{prop: complexification of toric}.
\end{prop}

For this reason, the action of $(\bC^\times)^N$ descends to an action on $\mathcal{L}_\Delta \to M_\Delta$, extending the action of $T^N$. Hence $H_\bC$, the complexification of $H$, naturally acts on $\mathcal{L}_\Delta\to M_\Delta$. 

\begin{prop}\label{prop: complexified infinitesimal formula}
    Let $\sigma\in\Gamma(\mathcal{L}_\Delta)$. Then, $i\xi\in i\mfh\subseteq \text{Lie}(H_\bC)$ acts on $\sigma$ by the following formula
    $$
    (i\xi)\cdot \sigma=\nabla^\Delta_{I_{\Delta}(\xi_{M_\Delta})}\sigma
    $$
    In particular, if $\sigma$ is holomorphic, then
    $$
    (i\xi)\cdot \sigma=i(\xi\cdot \sigma).
    $$
\end{prop}

\begin{proof}
    Fix $\xi\in \mfh$ and $\sigma\in \Gamma(\mathcal{L})$. By Proposition \ref{prop: equivariant sections give all sections below}, the section $\sigma$ lifts to a $K_\bC$-equivariant section of the trivial bundle $U_\Delta\times \bC\to U_\Delta$, which can be identified with a $K_\bC$ equivariant function $f:U_\Delta \to \bC$ with respect to the extended $K_\bC$ action on $\bC$. Since $\sigma$ is obtained by restricting $f$ to $(\iota^*\circ J)^{-1}(0)$, we will show
    $$
    (i\xi)\cdot f|_{(\iota^*\circ J)^{-1}(0)}=\nabla^{can}_{I^{can}(\xi_{\bC^N})}f|_{(\iota^*\circ J)^{-1}(0)}
    $$
    Recall that after letting $(r_1,\theta_1,\dots,r_N,\theta_N)$ denote polar coordinates on $\bC^N$, we have
    $$
    I^{can}(\partial_{\theta_j})=\partial_{r_j}
    $$
    In particular, writing $\xi=(\xi_1,\dots,\xi_N)\in \bR^N$, we have
    $$
    I^{can}(\xi_{\bC^N})=\sum_{j=1}^n \xi_j I^{can}(\partial_{\theta_j})=\sum_{j=1}^n \xi_j \partial_{r_j}.
    $$
    Hence,
    \begin{align*}
    \nabla^{can}_{I^{can}(\xi_{\bC^N})}f&=I^{can}(\xi_{\bC^N})f+\frac{i}{2}\sum_{j=1}^Nr_j^2d\theta_j(I^{can}(\xi_{\bC^N}))f\\
    &=\sum_{j=1}^N\xi_j \partial_{r_j}f+\frac{i}{2}\sum_{j=1}^Nr_j^2d\theta_j\bigg(\sum_{k=1}^N\xi_k \partial_{r_k}\bigg)f\\
    &=\sum_{j=1}^N\xi_j \partial_{r_j}f.
    \end{align*}

    Conversely, if we fix $z=(z_1,\dots,z_N)\in J^{-1}(0)$ then
    \begin{align*}
    (i\xi)\cdot f(z)&=\frac{d}{dt}\bigg|_{t=0}\rho(\exp(it\xi))f(\exp(-it\xi)\cdot z)\\
    &=\frac{d}{dt}\bigg|_{t=0}e^{t\bra \lambda,\xi\ket}f(e^{t\xi_1}z_1,\dots,e^{t\xi_N}z_N)\\
    &=\bra \lambda,\xi\ket f(z)+\sum_{j=1}^N\xi_j\partial_{r_j}f(z)
    \end{align*}
    Note that by definition, since $H\leq \ker(\rho)$ it follows that $\bra\lambda,\xi\ket\in 2\pi \bZ$. However, being in the kernel is a linear relationship at the level of Lie algebras and hence it follows that $\bra\lambda,\xi\ket=0$. 

    \ 

    Now to show the formula for holomorphic sections. Suppose $\sigma\in \Gamma(\mathcal{L}_\Delta)$ is holomorphic and observe that
    $$
    I_\Delta(\xi_{M_\Delta}+iI_\Delta(\xi_{M_\Delta}))=-i(\xi_{M_\Delta}+iI_\Delta(\xi_{M_\Delta}))
    $$
    Thus, $\xi_{M_\Delta}+iI_\Delta(\xi_{M_\Delta})$ is a section of the anti-holomorphic tangent bundle $T^{0,1}_{I_\Delta}M_\Delta$ and so by definition of $\sigma$ being holomorphic, we have
    $$
    \nabla_{\xi_{M_\Delta}+iI_\Delta(\xi_{M_\Delta})}\sigma=0.
    $$
    Hence by the linearity in the first entry of the connection, 
    $$
    \nabla_{\xi_{M_\Delta}}\sigma=-i\nabla_{I_\Delta(\xi_{M_\Delta})}\sigma
    $$
    and hence the result follows.
\end{proof}

\begin{cor}\label{cor: polar + real equivariant implies complex equivariant}
    If $\sigma\in \mathcal{Q}(M_\Delta)^H$, then $\sigma$ is $H_\bC$-equivariant.
\end{cor}

\begin{proof}
    Fix $\sigma\in\mathcal{Q}(M_\Delta)^H$. Since $H$ is a subtorus of the standard torus $(S^1)^N$, we have $H_\bC=H\cdot \exp(i\mfh)$. And thus, since $H$ is assumed to be connected, it suffices to show if $\xi_1,\xi_2\in \mfh$, then
    $$
    (\xi_1+i\xi_2)\cdot \sigma=0.
    $$
    In this case, we easily see that $\xi_1\cdot \sigma=0$ since $\sigma$ is $H$-equivariant. Furthermore, since $\sigma$ is by definition holomorphic, it follows we have from Proposition \ref{prop: complexified infinitesimal formula} that
    $$
    (i\xi_2)\cdot \sigma=i(\xi_2\cdot \sigma)=0.
    $$
    Thus, $\sigma$ is $H_\bC$-equivariant.
\end{proof}

Using this result, we can now nearly prove quantization commutes with reduction.

\begin{theorem}\label{thm: [Q,R] for toric}
Suppose
$$
M_\Delta^s:=H_\bC\cdot J_H^{-1}(0)
$$
is an open subset. Then
$$
\kappa:\mathcal{Q}(M_\Delta)^H\to \mathcal{Q}((M_\Delta)_0)
$$
as in Theorem \ref{thm: sing [Q,R] map} is injective.
\end{theorem}

\begin{proof}
    If $\sigma\in \mathcal{Q}(M_\Delta)^H$ satisfies $\kappa(\sigma)=0$, then we have $\sigma|_{J_H^{-1}(0)}=0$ since fibre-wise the map
    $$
    \mathcal{L}_\Delta|_{J_H^{-1}(0)}\to (\mathcal{L}_\Delta)_0
    $$
    is a bijection. By Corollary \ref{cor: polar + real equivariant implies complex equivariant}, $\sigma$ is $H_\bC$ equivariant and thus $\sigma|_{M_\Delta^s}=0$. Thus, we have a holomorphic section vanishing on an open set and so $\sigma=0$. Hence, $\kappa$ is injective.
\end{proof}

\section{Cotangent Bundles and $[Q,R]=0$}\label{sec: sing cotangent [Q,R]}

Let $G$ be a connected Lie group, $Q$ a proper $G$-space, and $J:T^*Q\to \mfg^*$ the Hamiltonian lift. We almost have a canonical choice of stratified quantization data on the reduced space $(T^*Q)_0$. In particular, letting $\mathcal{L}=T^*Q\times \bC$ be the trivial bundle with canonical prequantum connection $\nabla=L-i\theta_Q$, where $\theta_Q\in \Omega^1(T^*Q)$ is the Liouville $1$-form, we have from Example \ref{eg: contangent prequantum still reducible for singular cotangent}, that $\mathcal{L}$ is singularly reducible over $J^{-1}(0)$ and hence we get a canonical prequantum line bundle
$$
\mathcal{L}_0=(T^*Q)_0\times \bC.
$$
Furthermore, letting $\tau_Q:T^*Q\to Q$ be the bundle map, we have a canonical polarization
$$
\mathcal{P}=\ker(T\tau_Q)\otimes \bC
$$
As we saw in Example \ref{eg: cant always reduce}, $\mathcal{P}$ is in general not reducible. Nevertheless, we still get a pseudobundle
$$
\mathcal{P}_0=T^\bC\pi_0(\mathcal{P}\cap T^\bC\mathcal{S}_G(J^{-1}(0)))
$$
which restricts to a polarization on an open dense set inside each stratum of $(T^*Q)_0$. Hence, we can still define
$$
\mathcal{Q}((T^*Q)_0)=\{\sigma\in \Gamma(\mathcal{L}_0) \ | \ \nabla_v\sigma=0\text{ for all }v\in \mathcal{P}_0\}
$$

Note that in this setting, the sections of both $\mathcal{L}$ and $\mathcal{L}_0$ can be identified with complex valued differentiable functions on $T^*Q$ and $(T^*Q)_0$, respectively. Further, recall from Theorem \ref{thm: quantization of cotangent} that $\mathcal{Q}(T^*Q)=\tau_Q^*C^\infty(Q,\bC)$ and $\mathcal{Q}(T^*Q)^G=\tau_Q^*C^\infty(Q)^G$. Using the identification $\tau_Q^*C^\infty(Q)=C^\infty(Q)$ and $C^\infty(Q,\bC)^G=C^\infty(Q/G,\bC)$, we can view $\mathcal{Q}(T^*Q)^G$ as the set of smooth complex-valued functions on the quotient $Q/G$. As it turns out, this insight allows us to show quantization commutes with reduction here as well.

\begin{theorem}\label{thm: sing [Q,R] cotangent}
    The map $\mathcal{Q}(T^*Q)^G\to \mathcal{Q}((T^*Q)_0)$ is a bijection. In particular, $\mathcal{Q}((T^*Q)_0)$ can be identified with $C^\infty(Q/G)$. 
\end{theorem}

\begin{proof}
    First, let us show $\kappa$ is injective. Let $f\in \mathcal{Q}(T^*Q)^G$ satisfy $\kappa(f)=0$. Then $f\in \tau_Q^*C^\infty(Q)^G$ and $f|_{J^{-1}(0)}=0$. Since $f$ is constant on cotangent fibres and all cotangent fibres intersect $J^{-1}(0)$ non-trivially, it follows that $f=0$. Hence $\kappa$ is injective.

    \ 

    Now for surjective. Let $0_Q:Q\into T^*Q$ be the zero section. Since $J^{-1}(0)$ is a pseudobundle, the image of the zero section lies in $J^{-1}(0)$ and is smooth and equivariant. In particular, we have
    $$
    0_Q^*C^\infty(J^{-1}(0))^G\substeq C^\infty(Q)^G.
    $$
    Hence, we get a differentiable map between subcartesian spaces
    $$
    \overline{0_Q}:Q/G\to (T^*Q)_0
    $$
    Letting $f\in \mathcal{Q}((T^*Q)_0)$ and $\pi_Q:Q\to Q/G$, we get a function $\pi_Q^*0_Q^*f\in C^\infty(Q)$. Now let
    $$
    h=\tau_Q^*\pi_Q^*\overline{0_Q}^*(f)\in \mathcal{Q}(T^*Q)^G
    $$
    I claim $\kappa(h)=f$. Indeed, if we write $\pi_0:J^{-1}(0)\to (T^*Q)_0$ for the reduction map, it suffices to show
    $$
    h|_{J^{-1}(0)}=\pi_0^*f
    $$
    and so in turn in suffices to show for any orbit-type stratum $S\in \mathcal{S}_G(Q)$ that $h|_{Z_S}=\pi_0^*f|_{Z_S}$ where $Z_S\in \mathcal{S}_G(J^{-1}(0))$ is the orbit-type stratum from Theorem \ref{thm: strata for cotangent}. Recall that we showed $J_S^{-1}(0)\substeq Z_S$ as an open dense set, where $J_S:T^*S\to \mfg^*$ is the Hamiltonian lift of the $G$-action on $S$. Hence, it suffices to show 
    $$
    h|_{J_S^{-1}(0)}=\pi_0^*f|_{J_S^{-1}(0)}.
    $$
    This follows immediately from the non-singular reduction statement in Theorem \ref{theorem: non-singular [Q,R] cotangent bundles}. Hence, $\kappa(h)=f$ and thus $\kappa$ is a bijection.
\end{proof}

\cleardoublepage
\bibliographystyle{amsalpha}
\bibliography{thesis}

\newcommand{\etalchar}[1]{$^{#1}$}
\providecommand{\bysame}{\leavevmode\hbox to3em{\hrulefill}\thinspace}
\providecommand{\MR}{\relax\ifhmode\unskip\space\fi MR }
\providecommand{\MRhref}[2]{%
  \href{http://www.ams.org/mathscinet-getitem?mr=#1}{#2}
}
\providecommand{\href}[2]{#2}
\begin{thebibliography}{PROSD07}

\bibitem[AFT17]{ayala_local_2017}
David Ayala, John Francis, and Hiro~Lee Tanaka, \emph{Local structures on stratified spaces}, Advances in Mathematics \textbf{307} (2017), 903--1028 (en).

\bibitem[Arn78]{arnold_mathematical_1978}
V.~I. Arnold, \emph{Mathematical {Methods} of {Classical} {Mechanics}}, fourth ed., Graduate {Texts} in {Mathematics}, no.~60, Springer-Verlag, Berlin, 1978.

\bibitem[AS61]{aronszajn_theory_1961}
Nachman Aronszajn and K.~T. Smith, \emph{Theory of {Bessel} potentials. {I}}, Annales de l'Institut Fourier \textbf{11} (1961), 385--475 (en).

\bibitem[Bie75]{bierstone_lifting_1975}
Edward Bierstone, \emph{Lifting isotopies from orbit spaces}, Topology \textbf{14} (1975), no.~3, 245--252.

\bibitem[CM18]{crainic_orbispaces_2018}
Marius Crainic and João~Nuno Mestre, \emph{Orbispaces as differentiable stratified spaces}, Letters in Mathematical Physics \textbf{108} (2018), no.~3, 805--859 (en), arXiv:1705.00466 [math].

\bibitem[Del88]{delzant_hamiltoniens_1988}
Thomas Delzant, \emph{Hamiltoniens périodiques et images convexes de l'application moment}, Bulletin de la Société Mathématique de France \textbf{116} (1988), no.~3, 315--339 (en).

\bibitem[Dir82]{dirac_principles_1982}
P.~A.~M. Dirac, \emph{The {Principles} of {Quantum} {Mechanics}}, fourth edition, fourth edition ed., International {Series} of {Monographs} on {Physics}, Oxford University Press, Oxford, New York, February 1982.

\bibitem[DS08]{da_silva_lectures_2008}
Ana~Cannas Da~Silva, \emph{Lectures on {Symplectic} {Geometry}}, Lecture {Notes} in {Mathematics}, vol. 1764, Springer, Berlin, Heidelberg, 2008.

\bibitem[Dui]{duistermaat_dynamical_nodate}
J.J. Duistermaat, \emph{Dynamical {Systtems} with {Symmetries}}.

\bibitem[ER90]{emmrich_orbifolds_1990}
C.~Emmrich and H.~Römer, \emph{Orbifolds as configuration spaces of systems with gauge symmetries}, Communications in Mathematical Physics \textbf{129} (1990), no.~1, 69--94 (en).

\bibitem[GGK02]{guillemin_moment_2002}
Victor Guillemin, Viktor Ginzburg, and Yael Karshon, \emph{Moment {Maps}, {Cobordisms}, and {Hamiltonian} {Group} {Actions}}, Mathematical {Surveys} and {Monographs}, vol.~98, American Mathematical Society, 2002.

\bibitem[Got95]{gotay_full_1995}
Mark~J. Gotay, \emph{On a {Full} {Quantization} of the {Torus}}, Quantization, {Coherent} {States}, and {Complex} {Structures} (J.-P. Antoine, S.~Twareque Ali, W.~Lisiecki, I.~M. Mladenov, and A.~Odzijewicz, eds.), Springer US, Boston, MA, 1995, pp.~55--62 (en).

\bibitem[GR12]{grabowski_graded_2012}
Janusz Grabowski and Mikołaj Rotkiewicz, \emph{Graded bundles and homogeneity structures}, Journal of Geometry and Physics \textbf{62} (2012), no.~1, 21--36.

\bibitem[Gro46]{groenewold_principles_1946}
H.~J. Groenewold, \emph{On the principles of elementary quantum mechanics}, Physica \textbf{12} (1946), no.~7, 405--460.

\bibitem[GS82]{guillemin_geometric_1982}
V.~Guillemin and S.~Sternberg, \emph{Geometric quantization and multiplicities of group representations}, Inventiones mathematicae \textbf{67} (1982), no.~3, 515--538 (en).

\bibitem[GS84]{guillemin_symplectic_1984}
Victor Guillemin and Shlomo Sternberg, \emph{Symplectic {Techniques} in {Physics}}, Cambridge University Press, 1984.

\bibitem[Gui94]{guillemin_moment_1994}
Victor Guillemin, \emph{Moment {Maps} and {Combinatorial} {Invariants} of {Hamiltonian} \${T}{\textasciicircum}n\$-spaces}, Progress in {Mathematics}, vol. 122, Birkhäuser, Boston, MA, 1994 (en).

\bibitem[Hal13]{hall_quantum_2013}
Brian~C. Hall, \emph{Quantum {Theory} for {Mathematicians}}, Graduate {Texts} in {Mathematics}, vol. 267, Springer, New York, NY, 2013 (en).

\bibitem[Ham08]{hamilton_quantization_2008}
Mark~D. Hamilton, \emph{The quantization of a toric manifold is given by the integer lattice points in the moment polytope}, February 2008, arXiv:0708.2710 [math].

\bibitem[Ham10]{hamilton_locally_2010}
Mark Hamilton, \emph{Locally toric manifolds and singular {Bohr}-{Sommerfeld} leaves}, Memoirs of the {American} {Mathematical} {Society}, vol. 207, American Mathematical Society, September 2010 (en), ISSN: 0065-9266, 1947-6221 Issue: 971.

\bibitem[HHS{\etalchar{+}}98]{hazod_normalizers_1998}
W.~Hazod, K.~H. Hofmann, H.-P. Scheffler, M.~Wünster, and H.~Zeuner, \emph{Normalizers of compact subgroups, the existence of commuting automorphisms, and applications to operator semistable measures}, Journal of Lie Theory \textbf{8} (1998), 189--209.

\bibitem[Hud24]{hudson_multiplicative_2024}
Daniel Hudson, \emph{On {Multiplicative} {Weightings} for {Lie} {Groupoids} and {Lie} {Algebroids}}, Ph.{D}., University of Toronto (Canada), Canada -- Ontario, CA, 2024, ISBN: 9798342756303.

\bibitem[IZ13]{iglesias-zemmour_diffeology_2013}
Patrick Iglesias-Zemmour, \emph{Diffeology}, Mathematical {Surveys} and {Monographs}, vol. 185, American Mathematical Society, 2013.

\bibitem[JRS11]{jotz_singular_2011}
M.~Jotz, T.~S. Ratiu, and J.~Sniatycki, \emph{Singular {Reduction} of {Dirac} {Structures}}, Transactions of the American Mathematical Society \textbf{363} (2011), no.~6, 2967--3013, Publisher: American Mathematical Society.

\bibitem[Kir62]{kirillov_unitary_1962}
A.~A. Kirillov, \emph{Unitary representations of nilpotent {Lie} groups}, Russian Mathematical Surveys \textbf{17} (1962), no.~4, 53 (en), Publisher: IOP Publishing.

\bibitem[Kir98]{kirwan_momentum_1998}
Frances Kirwan, \emph{Momentum maps and reduction in algebraic geometry}, Differential Geometry and its Applications \textbf{9} (1998), no.~1, 135--171.

\bibitem[KL23]{karshon_vector_2023}
Yael Karshon and Eugene Lerman, \emph{Vector {Fields} and {Flows} on {Subcartesian} {Spaces}}, November 2023, arXiv:2307.10959 [math].

\bibitem[Kos09]{kostant_orbits_2009}
Bertram Kostant, \emph{Orbits, {Symplectic} {Structures} and {Representation} {Theory}}, Collected {Papers}: {Volume} {I} 1955-1966 (Bertram Kostant, Anthony Joseph, Shrawan Kumar, and Michèle Vergne, eds.), Springer, New York, NY, 2009, pp.~482--482 (en).

\bibitem[Lee13]{lee_introduction_2013}
John~M. Lee, \emph{Introduction to smooth manifolds}, 2nd ed ed., Graduate texts in mathematics, no. 218, Springer, New York ; London, 2013, OCLC: ocn800646950.

\bibitem[LT10]{lytchak_curvature_2010}
Alexander Lytchak and Gudlaugur Thorbergsson, \emph{Curvature explosion in quotients and applications}, Journal of Differential Geometry \textbf{85} (2010), no.~1, 117--140, Publisher: Lehigh University.

\bibitem[Mar75]{marshall_calculus_1975}
Charles~D. Marshall, \emph{Calculus on {Subcartesian} {Spaces}}, Journal of Differential Geometry \textbf{10} (1975), 551--574.

\bibitem[Mat73]{mather_stratifications_1973}
John~N. Mather, \emph{Stratifications and {Mappings}}, Dynamical {Systems} (M.~M. Peixoto, ed.), Academic Press, January 1973, pp.~195--232.

\bibitem[Mat77]{mather_differentiable_1977}
\bysame, \emph{Differentiable invariants}, Topology \textbf{16} (1977), no.~2, 145--155.

\bibitem[Mat12]{mather_notes_2012}
John Mather, \emph{Notes on {Topological} {Stability}}, Bulletin of the American Mathematical Society \textbf{49} (2012), no.~4, 475--506 (en).

\bibitem[Max19]{maxim_intersection_2019}
Laurenţiu~G. Maxim, \emph{Intersection {Homology} \& {Perverse} {Sheaves}: with {Applications} to {Singularities}}, Graduate {Texts} in {Mathematics}, vol. 281, Springer International Publishing, Cham, 2019 (en).

\bibitem[May20]{mayrand_stratified_2020}
Maxence Mayrand, \emph{Stratified {Hyperkähler} {Spaces} {From} {Semisimple} {Lie} {Algebras}}, Transformation Groups \textbf{25} (2020), no.~1, 191--215 (en).

\bibitem[Mey73]{meyer_symmetries_1973}
Kenneth~R. Meyer, \emph{Symmetries and {Integrals} in {Mechanics}†}, Dynamical {Systems} (M.~M. Peixoto, ed.), Academic Press, January 1973, pp.~259--272.

\bibitem[Mic08]{michor_topics_2008}
Peter~W. Michor, \emph{Topics in {Differential} {Geometry}}, Graduate {Studies} in {Mathematics}, vol.~93, American Mathematical Society, 2008.

\bibitem[Miy23]{miyamoto_basic_2023}
David Miyamoto, \emph{The {Basic} de {Rham} {Complex} of a {Singular} {Foliation}}, International Mathematics Research Notices \textbf{2023} (2023), no.~8, 6364--6401.

\bibitem[MMO{\etalchar{+}}07]{marsden_hamiltonian_2007}
Jerrold~E. Marsden, Gerard Misiolek, Juan-Pablo Ortega, Matthew Perlmutter, and Tudor~S. Ratiu, \emph{Hamiltonian {Reduction} by {Stages}}, Lecture {Notes} in {Mathematics}, vol. 1913, Springer, Berlin, Heidelberg, 2007 (en).

\bibitem[Mol88]{molino_riemannian_1988}
Pierre Molino, \emph{Riemannian foliations}, Progress in {Mathematics}, vol.~73, Birkhäuser Boston, Inc., Boston, MA, 1988. \MR{932463}

\bibitem[Mol24]{mol_stratification_2024}
Maarten Mol, \emph{Stratification of the transverse momentum map}, Compositio Mathematica \textbf{160} (2024), no.~3, 518--585 (en).

\bibitem[MR25]{miyamoto_srfs_2025}
David Miyamoto and Ethan Ross, \emph{{SRFs} and {Stratifications} (in preparation)}, 2025.

\bibitem[MS99]{meinrenken_singular_1999}
Eckhard Meinrenken and Reyer Sjamaar, \emph{Singular {Reduction} and {Quantization}}, Topology \textbf{38} (1999), no.~4, 699--762.

\bibitem[MW74]{marsden_reduction_1974}
Jerrold Marsden and Alan Weinstein, \emph{Reduction of symplectic manifolds with symmetry}, Reports on Mathematical Physics \textbf{5} (1974), no.~1, 121--130.

\bibitem[Nes03]{nestruev_smooth_2003}
Jet Nestruev, \emph{Smooth {Manifolds} and {Observables}}, second ed., Graduate {Texts} in {Mathematics}, Springer, 2003.

\bibitem[NN57]{newlander_complex_1957}
A.~Newlander and L.~Nirenberg, \emph{Complex {Analytic} {Coordinates} in {Almost} {Complex} {Manifolds}}, Annals of Mathematics \textbf{65} (1957), no.~3, 391--404, Publisher: [Annals of Mathematics, Trustees of Princeton University on Behalf of the Annals of Mathematics, Mathematics Department, Princeton University].

\bibitem[Pal61]{palais_existence_1961}
Richard~S. Palais, \emph{On the {Existence} of {Slices} for {Actions} of {Non}-{Compact} {Lie} {Groups}}, Annals of Mathematics \textbf{73} (1961), no.~2, 295--323, Publisher: Annals of Mathematics.

\bibitem[Pfl01]{pflaum_analytic_2001}
Markus~J. Pflaum, \emph{Analytic and {Geometric} {Study} of {Stratified} {Spaces}}, Lecture {Notes} in {Mathematics}, vol. 1768, Springer, Berlin, Heidelberg, 2001 (en).

\bibitem[PROSD07]{perlmutter_geometry_2007}
Matthew Perlmutter, Miguel Rodríguez-Olmos, and M.~Esmeralda Sousa-Dias, \emph{On the geometry of reduced cotangent bundles at zero momentum}, Journal of Geometry and Physics \textbf{57} (2007), no.~2, 571--596.

\bibitem[PTW17]{posthuma_resolutions_2017}
Hessel~B. Posthuma, Xiang Tang, and Kirsten J.~L. Wang, \emph{Resolutions of proper {Riemannian} {Lie} groupoids}, June 2017, arXiv:1706.07843 [math].

\bibitem[Rad12]{radeschi_low_2012}
Marco Radeschi, \emph{Low dimensional {Singular} {Riemannian} {Foliations} in spheres}, March 2012, arXiv:1203.6113 [math].

\bibitem[Ros24]{ross_stratified_2024}
Ethan Ross, \emph{Stratified vector bundles: {Examples} and constructions}, Journal of Geometry and Physics (2024), 105114.

\bibitem[Saf23]{safronov_shifted_2023}
Pavel Safronov, \emph{Shifted geometric quantization}, Journal of Geometry and Physics \textbf{194} (2023), 104992.

\bibitem[Sch75]{schwarz_smooth_1975}
Gerald~W. Schwarz, \emph{Smooth functions invariant under the action of a compact {Lie} group}, Topology \textbf{14} (1975), no.~1, 63--68 (en).

\bibitem[Sch80]{schwarz_lifting_1980}
\bysame, \emph{Lifting smooth homotopies of orbit spaces}, Publications Mathématiques de l'Institut des Hautes Études Scientifiques \textbf{51} (1980), no.~1, 37--132 (en).

\bibitem[Sch23]{schoeneberg_answer_2023}
Torsten Schoeneberg, \emph{Answer to "{Example} of a non-connected {Lie} subgroup where the normalizer of {Lie} algebra is not the {Lie} algebra of the normalizer"}, June 2023.

\bibitem[Seg68]{segal_equivariant_1968}
Graeme Segal, \emph{Equivariant {K}-theory}, Institut des Hautes Études Scientifiques. Publications Mathématiques (1968), no.~34, 129--151. \MR{234452}

\bibitem[Sik67]{sikorski_abstract_1967}
Roman Sikorski, \emph{Abstract {Covariant} {Derivative}}, Colloquium Mathematicae \textbf{18.1} (1967), 251 --272.

\bibitem[SL91]{sjamaar_stratified_1991}
Reyer Sjamaar and Eugene Lerman, \emph{Stratified {Symplectic} {Spaces} and {Reduction}}, Annals of Mathematics \textbf{134} (1991), no.~2, 375--422, Publisher: Annals of Mathematics.

\bibitem[Sou67]{souriau_quantification_1967}
Jean-Marie Souriau, \emph{Quantification géométrique. {Applications}}, Annales de l'institut Henri Poincaré. Section A, Physique Théorique \textbf{6} (1967), no.~4, 311--341 (en).

\bibitem[Sou70]{souriau_structure_1970}
J.-M. Souriau, \emph{Structure of {Dynamical} {Systems}: {A} {Symplectic} {View} of {Physics}}, Birkhäuser, 1970.

\bibitem[Ste74]{stefan_accessible_1974}
P.~Stefan, \emph{Accessible {Sets}, {Orbits}, and {Foliations} with {Singularities}}, Proceedings of the London Mathematical Society \textbf{s3-29} (1974), no.~4, 699--713.

\bibitem[Sus73]{sussmann_orbits_1973}
Héctor~J. Sussmann, \emph{Orbits of {Families} of {Vector} {Fields} and {Integrability} of {Distributions}}, Transactions of the American Mathematical Society \textbf{180} (1973), 171--188, Publisher: American Mathematical Society.

\bibitem[SW83]{sniatycki_reduction_1983}
Jedrzej Sniatycki and Alan Weinstein, \emph{Reduction and quantization for singular momentum mappings}, Letters in Mathematical Physics \textbf{7} (1983), no.~2, 155--161, Publisher: Kluwer Academic Publishers.

\bibitem[Tho69]{thom_ensembles_1969}
R.~Thom, \emph{Ensembles et morphismes stratifiés}, Bulletin of the American Mathematical Society \textbf{75} (1969), no.~2, 240--284 (en).

\bibitem[Vai91]{vaisman_geometric_1991}
Izu Vaisman, \emph{On the geometric quantization of {Poisson} manifolds}, Journal of Mathematical Physics \textbf{32} (1991), no.~12, 3339--3345.

\bibitem[Vis05]{vistoli_grothendieck_2005}
Angelo Vistoli, \emph{Grothendieck topologies, fibered categories and descent theory.}, Fundamental algebraic geometry: {Grothendieck}'s {FGA} explained, Mathematical surveys and monographs, no. v. 123, American Mathematical Society, Providence, R.I, 2005, OCLC: ocm61362228.

\bibitem[Wei71]{weinstein_symplectic_1971}
Alan Weinstein, \emph{Symplectic manifolds and their lagrangian submanifolds}, Advances in Mathematics \textbf{6} (1971), no.~3, 329--346.

\bibitem[Wei83]{weinstein_local_1983}
\bysame, \emph{The local structure of {Poisson} manifolds}, Journal of Differential Geometry \textbf{18} (1983), no.~3, 523--557, Publisher: Lehigh University.

\bibitem[Wey16]{weyl_classical_2016}
Hermann Weyl, \emph{The {Classical} {Groups}: {Their} {Invariants} and {Representations}}, The {Classical} {Groups}, Princeton University Press, June 2016 (en).

\bibitem[Whi15]{whitney_local_2015}
Hassler Whitney, \emph{Local {Properties} of {Analytic} {Varieties}}, Local {Properties} of {Analytic} {Varieties}, Princeton University Press, December 2015, pp.~205--244 (en).

\bibitem[Zim24]{zimhony_commutative_2024}
Yoav Zimhony, \emph{Commutative {Control} {Data} for {Smoothly} {Locally} {Trivial} {Stratified} {Spaces}}, Transformation Groups (2024) (en).

\bibitem[Śn75]{sniatycki_bohr-sommerfeld_1975}
J.~Śniatycki, \emph{Bohr-{Sommerfeld} conditions in geometric quantization}, Reports on Mathematical Physics \textbf{7} (1975), no.~2, 303--311.

\bibitem[Śn13]{sniatycki_differential_2013}
\bysame, \emph{Differential {Geometry} of {Singular} {Spaces} and {Reduction} of {Symmetry}}, New {Mathematical} {Monographs}, Cambridge University Press, Cambridge, 2013.

\end{thebibliography}
\cleardoubleevenpage

\end{document}